\documentclass[pdflatex,sn-mathphys-num,12pt]{sn-jnl}

\usepackage{graphicx}%
\graphicspath{{./}}
\usepackage{multirow}%
\usepackage{amsmath,amssymb,amsfonts}%
\usepackage{amsthm}%
\usepackage{mathrsfs}%
\usepackage[title]{appendix}%
\usepackage{xcolor}%
\usepackage{textcomp}%
\usepackage{manyfoot}%
\usepackage{booktabs}%
\usepackage{listings}%

\newtheorem{theorem}{Theorem}
\newtheorem{proposition}[theorem]{Proposition}%
\newtheorem{definition}{Definition}%

\newtheorem{lemma}{Lemma}

\newtheorem{assumption}{Assumption}
\newtheorem{statement}{Statement}[section]

\theoremstyle{thmstyletwo}%
\newtheorem{example}{Example}%
\newtheorem{remark}{Remark}%

\usepackage{geometry}
\usepackage{enumerate}
\usepackage{amssymb,mathrsfs,amsthm,multirow}
\usepackage{extarrows}
\usepackage{mathtools}
\usepackage{xcolor}
\usepackage{graphicx}
\usepackage{subfigure}
\usepackage{chngcntr}
\usepackage[inline]{enumitem}
\usepackage{amsmath,cases,multirow,url}
\usepackage{bm}
\usepackage[ruled,linesnumbered]{algorithm2e}

\newcommand{\dd}{\mathsf {d\kern -0.07em l}} 

\usepackage{color,colortbl,xcolor}
\usepackage{hyperref}
\usepackage{adjustbox}
\usepackage{dsfont}
\usepackage{ulem}
\let\ULforem\normalem

\SetAlgoNlRelativeSize{1}
\providecommand{\mathbbm}[1]{\mathds{#1}}

\def\bbe{{\mathbb{E}}}

\newcommand{\bec}{\begin{center}}
\newcommand{\enc}{\end{center}}

\def\bgeqn#1\edeqn{\begin{align}#1\end{align}}
\def\bgeq#1\edeq{\begin{align*}#1\end{align*}}

\begin{document}

\title[Maximum Utility Split Method for Utility Preference Elicitation]{Maximum Utility Split Method for Utility Preference Elicitation}

\author[1]{\fnm{Bo} \sur{Chen}}\email{chenbo1211@stu.xjtu.edu.cn}

\author[1]{\fnm{Jia} \sur{Liu}}\email{jialiu@xjtu.edu.cn}

\author[2]{\fnm{Huifu} \sur{Xu}}\email{hfxu@se.cuhk.edu.hk}

\affil[1]{\orgdiv{School of Mathematics and Statistics}, \orgname{Xi'an Jiaotong University},
\city{Xi'an}, \state{Shaanxi}, \country{China}}

\affil[2]{\orgdiv{Department of Systems Engineering and Engineering Management}, \orgname{The Chinese University of Hong Kong},
\city{Hong Kong}, \country{China}}

\abstract{
{
Preference elicitation has been extensively studied in behavioral economics, decision analysis and marketing.
} 
In this paper, we propose a new { choice-based approach}
called maximum utility split (MUS) scheme, which is built on the random utility split (RUS) scheme but with a notable difference: one lottery is designed with two fixed outcomes but with varying probability assigned to the outcomes, and the other has a deterministic outcome specifically chosen at the point where the range between the largest and smallest possible utility values is maximized, based on the elicited preference information. 
Consequently, the probability of the random lottery is set such that the range of the ambiguity set of utility functions is reduced by half at the point.
Under some moderate conditions, we show that MUS can successively generate a sequence of such questionnaires and effectively reduce the ambiguity set of utility functions, eventually converging to the true utility function as the number of questionnaires increases. 
The main challenge is to effectively identify the point with the largest utility range for a given ambiguity set constructed from elicited preference information. 
Based on the structure of the ambiguity set, we propose an interval-based algorithm which identifies each certain-outcome lottery by solving a sequence of linear programs.
Moreover, to deal with the case where elicitation terminates before the ambiguity set reduces to a singleton, we demonstrate how to figure out a nominal utility function by solving two optimization programs.
These identify, respectively, the smallest and largest utility functions under the Kantorovich metric within the ambiguity set, after which we identify {
an appropriate estimate of the true utility function
which lies between the two extremes}. 
Finally, numerical results demonstrate the efficiency of the MUS method and the performance of a robo-advisor system based on MUS-type queries and the nominal utility elicitation approach. 
While the main discussions focus on concave utility functions, 
we also demonstrate how the MUS approach can be extended to accommodate general non-concave utility functions, particularly S-shaped ones.
}

\keywords{
Preference elicitation, 
Design of lotteries, 
Maximum utility split scheme, Piecewise linear utility function, 
Robo-advisor
}

\maketitle




\section{Introduction}\label{sec:Intro}

The expected utility theory established by~\citet{von1947theory} has been a dominant normative and descriptive model of decision-making. It states that any set of preferences that a decision maker (DM, for short) may have among uncertain/risky prospects can be characterized by an expected utility function if the preferences satisfy certain reasonable axioms (i.e., completeness, transitivity, continuity, and independence). 
Specifically, there exists a utility function $u: 
\mathbb{R}\to\mathbb{R}
$ such that
the DM prefers the prospect $X$ over the prospect $Y$ if and only if
$
\bbe[u(X)]\geq \bbe[u(Y)]. 
$
The utility function is unique up to a positive linear transformation.

In practice, the true utility function that captures a DM's preferences is often unknown. Several preference elicitation methods have subsequently been proposed in behavioral economics, marketing, and decision analytics, see, e.g., 
\citet{ca-book,toubia2003fast,toubia2004polyhedral,saure2019ellipsoidal,toubia2007probabilistic,
hu2024distributional}
and references therein.
One of the most popular methods is to elicit the DM's preferences with paired gambling approaches for preference comparisons~\citep{Far84}.
It uses the elicited information to identify the values of the utility function at a discrete set of points and construct an approximate utility function via some interpolation methods, see, for instance, \citet{ClR01}.

\citet{armbruster2015decision} argue that the classical interpolation approach for finding an approximate utility function has some drawbacks not only because it is often difficult to identify a non-parametric utility function purely based on the DM's preferences over pairwise comparison lotteries,  
but also because it could be risky to use a single approximate utility function without considering other plausible utility functions nearby. 
Consequently, they propose an alternative approach: instead of trying to find a single approximate Von Neumann-Morgenstern (VNM) utility function, they propose to use elicited information of a DM's preferences such as preferring certain lotteries over other lotteries, being risk averse over gains, and risk taking over losses, to construct an ambiguity set of plausible utility functions, and then to base the optimal decision on the worst-case utility function from the ambiguity set.
The approach is called preference robust optimization (PRO) as it follows the general philosophy of robust optimization and can be traced back to earlier work by~\citet{maccheroni2002maxmin}, 
where the author derives necessary and sufficient conditions for characterizing the kind of worst-case utility-based decision-making.
\citet{
hu2015robust} propose a max-min utility-based PRO model for single-attribute decision-making under uncertainty where ambiguity and inconsistency in utility assessments occur. 
By viewing the utility function as a cumulative distribution function of a random variable, they consider the case that the ambiguity set of utility functions can be represented by a system of moment-type conditions and demonstrate that the maximin utility-based PRO problem can be reformulated as a linear programming problem. 
Over the past few years, PRO has received increasing attention in stochastic programming and risk management where information on the true utility function or risk measure representing a DM's utility/risk preferences is incomplete, see~\cite{haskell2016ambiguity,
DeL18,
JLI21,
haskell2022preference, hu2017optimization, guo2024utility,hu2024distributional,
liu2025preference}. 

While PRO models effectively address the modelling risk in decision-making (by using an approximate utility function or risk measure with data fitting/interpolation), 
{they}
also bring in some side effects: the decision based on the worst-case utility or risk measure may be over-conservative if the ambiguity set is large. This raises a question as to how to reduce the size of the ambiguity set {prior to or in the process of decision-making}. 
This kind of issue exists in generic minimax/maximin robust optimization problems.
There are several ways to tackle the issue
depending on available information. 
One is to use available empirical information, e.g., a DM is risk-taking or risk averse, to specify the shape of utility functions such as convexity, concavity, and S-shapedness, as well as the property of marginal utility via Lipschitz continuity. 
This information may help to specify a class of utility functions, some of which may be parametric, see~\cite{tversky1992advances,abdellaoui2000parameter,bleichrodt2000parameter}.
The other is to elicit preference information from 
{the}
DM and use the elicited information to reduce the size of the ambiguity set.  
Random utility split (RUS) and random relative utility split (RRUS)  
are the most popular preference elicitation methods.
In the discrete choice models for customers' choice behavior, \citet{toubia2004polyhedral} propose a polyhedral cut method for eliciting customers' preferences characterized by multivariate linear utility functions.  
\citet{VYMDR20} propose an objective-oriented method for eliciting a multivariate linear utility function.  
In a more recent development, \citet{zhang2025modified} propose a modified polyhedral method to elicit a DM’s nonlinear univariate utility function, which does not rely on explicit information about the shape structure, Lipschitz modulus, and the inflection point of the utility.
{In the case when the true utility function belongs to a parametric family of utility functions, a Bayesian approach can be adopted by assigning a prior distribution to the parameters and updating it using Bayes' theorem as more information about the 
DM's 
preferences is elicited through the learning process; 
see~\cite{toubia2013dynamic,saure2019ellipsoidal}.
}

While preference elicitation methods described above may help to reduce the size of the ambiguity set of utility functions/risk measures, the main issue to be addressed is effectiveness. 
On the one hand, we need to establish a theoretical guarantee which ensures the convergence of an elicitation method, reducing the ambiguity set to a singleton as the number of questionnaires increases. 
On the other hand, the actual number of questionnaires should be reasonable  because in practice, one cannot ask a DM too many 
{questions}
to elicit the DM's preferences, as it may be either infeasible or too expensive. 
{ 
By “expensive”, we mean 
that preference elicitation requires time, effort, and, in some applications, financial resources on the modeler’s part, while also imposing a 
participation and cognitive burden on the DM~\cite{kostric2024generating}. Moreover, excessive questioning may compromise the quality of the elicited information~\cite{bradley1994use}.}
As far as we are concerned, there is no theoretical guarantee that RUS and RRUS will reduce an ambiguity set of utility functions to a singleton as the number of questionnaires increases, nor does the polyhedral cut method in~\cite{zhang2025modified}. 
In robo-advisor systems, interaction between machine and user might improve the efficiency of preference elicitation, but the cost of each question is often high.
\citet{Alsabah2021}, \citet{Cui2022} and \citet{Yu2023} address the issue by considering users with a mean-risk form of preferences and restricting elicitation to the value of the risk-aversion coefficient. 
\citet{chen2026robo} extend the robo-advisor system from the mean-risk framework to the VNM utility function framework.

In this paper, we propose a maximum utility split (MUS) scheme, which is a specific adaptive preference questionnaire generation method for pairwise comparisons. 
It aims at improving the efficiency of the existing RUS method. 
We begin by considering an ambiguity set of monotonically increasing, concave, and Lipschitz continuous univariate utility functions within a non-parametric framework, where the DM's preference follows VNM expected utility theory, and then extend the discussion to 
{the}
non-concave 
case particularly when the true utility function is S-shaped.
The proposed MUS scheme is then utilized to adaptively determine the parameters of a specifically designed pairwise comparison question so that the resulting elicited preference information on the DM can significantly reduce the size of the ambiguity set. 
We also consider the issue of identifying a nominal utility function from the set, which serves as the ``best" representation of the DM's preference so far. 
Figure~\ref{fig:flowchart-preference-learning} 
provides a flowchart which illustrates the overall preference elicitation procedure in a single interaction with the DM. 

\begin{figure}[htbp]
    \centering
    \includegraphics[width=0.9\textwidth]{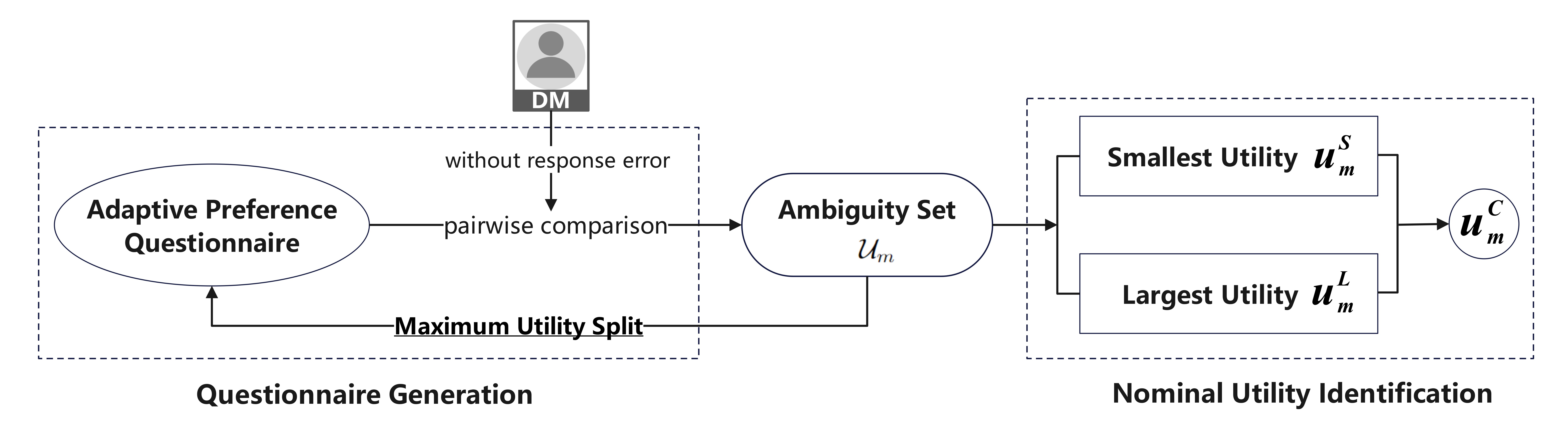}
    \caption{Flowchart of the preference elicitation process during the $m$-th interaction with the DM.}
    \label{fig:flowchart-preference-learning}
\end{figure}

The main contributions of this paper can be summarized as follows: 
\begin{itemize}
\item {\bf 
Methodology.}
We propose a new approach, called maximum utility split (MUS) scheme, which is built on RUS but with a notable difference.  
We begin with an ambiguity set of plausible utility functions, 
and then design a pair of lotteries like the RUS scheme: 
one lottery is constructed with two fixed outcomes but with varying probabilities for the outcomes to be determined, and the other has a deterministic outcome. 
However, unlike RUS, which selects the outcome of the deterministic lottery randomly, MUS chooses the outcome specifically at 
a point  where the difference between the largest 
and smallest 
possible utility values 
among the ambiguity set of utility functions is the 
biggest. 
Consequently, the probability of the random lottery is 
designed such that the range of the ambiguity set of utility functions is reduced by half at the point. 
{By presenting the specifically designed query to the DM, we update the ambiguity set based on the DM's response.}
Under some moderate conditions, we prove that MUS may successively generate a sequence of paired lotteries and eventually reduce 
{the sequence of the ambiguity sets} of utility functions to the true one as the number of questionnaires increases {(Theorem~\ref{thm:mus-converg})}.

\item \textbf{Tractable formulation of MUS.} 
The main challenge is to effectively identify 
a point 
in the domain of the utility functions 
where the ambiguity set of utility functions attains its largest range.
We formulate this as a maximization 
problem where the objective function 
is the difference between the largest 
and smallest 
possible utility values of the ambiguity set. Unfortunately,  
the objective function is neither 
a concave function, nor the difference of two concave functions. 
However, by exploiting the specific structure of the ambiguity set, we constructively demonstrate the existence and uniqueness of two piecewise linear utility functions 
(called the lower bound function in Theorem~\ref{thm:mus-lower-bound} and the upper bound function in Theorem~\ref{thm:mus-upper-bound}) that attain the pointwise minimum and maximum, respectively, on each interval between any two consecutive points in the support set formed by the previous MUS queries.  
{Thus, identifying a point with the maximum range over each interval amounts to identifying a point at which the largest gap between the lower and upper bound functions over the interval is attained.}
Consequently, we 
propose an interval-based algorithm which allows one to identify the outcome of the deterministic lottery by solving a sequence of linear programs {(Algorithm~\ref{alg:mus-1})}.
Moreover, to deal with the case where elicitation terminates before the ambiguity set is reduced to a singleton, we demonstrate how to determine a nominal utility function by solving two programs: a linear minimization program and a second-order cone-constrained maximization program, which respectively identify the smallest and the largest utility functions in the ambiguity set under the Kantorovich metric 
{(Proposition~\ref{prop:lower-bound-tight} and Theorem~\ref{thm:upper-bound-re-N+1}).} 
We then identify a utility function lying between them {(Theorem~\ref{thm:opt-nominal}).}




\item \textbf{Application in a robo-advisor system.}
We have carried out some numerical tests of the proposed MUS on an academic example and 
find that the new method displays 
{a}
much faster reduction of the ambiguity set of utility functions in comparison with RUS, RRUS and polyhedral method. 
We have also
applied 
the MUS to a robo-advisor system
for optimal decision-making in portfolio optimization.
{In our view, the proposed
iterative preference elicitation 
process is particularly suitable for this kind of adaptive
human-machine interaction system}.
Unlike the mean-risk models~\citep{Alsabah2021,Cui2022,Yu2023}, here we adopt 
the VNM expected utility framework over a general functional space 
which allows the system to cover a 
broader 
class of users.
Our preliminary numerical results show that 
the 
MUS approach outperforms existing preference elicitation approaches as observed in the numerical test on the academic example.

\end{itemize}

The rest of the paper is organized as follows. 
Section~\ref{sec:questionnaire}
introduces the MUS 
for 
adaptive questionnaire generation and analyses its convergence,
Section~\ref{sec-bounds}
derives a semi-closed form solution 
to a key optimization problem in MUS,
Section~\ref{sec:utility}
discusses how to identify a nominal utility function at the end of MUS. 
Section~\ref{sec-extensions} extends the discussion to general 
{non-concave}
case particularly the S-shaped true utility functions.
Section~\ref{sec:numerical} 
reports numerical tests 
of the MUS method and finally
Section~\ref{sec:conclusion} concludes.

\section{Adaptive preference questionnaire generation}\label{sec:questionnaire}

The design and generation of questionnaires presented to the DM, whose responses reveal the preferences, are critical to the efficiency of preference elicitation. 
Pairwise comparison (choosing between two items/lotteries) is the most widely used form of questionnaire, as it is easy to implement in the interaction with the DM~\citep{eggers2021choice}. 
In practice, preference information and pairwise comparison data from the DM are often limited. 
Thus, we construct and maintain an ambiguity set of plausible utility functions using available information about the DM.  
Moreover, \citet{mcfadden1972conditional} suggests that efficient questionnaire design should involve a sequential process that adapts to information inferred from prior responses. 
We propose a maximum utility split (MUS) scheme to generate pairwise comparison questions based on the current ambiguity set.

In this section, we first introduce how to construct the ambiguity set of plausible utility functions by using both the pairwise comparison data under the VNM expected utility framework and the available
generic information of the DM's risk-averse attitude. 
We then review two popular adaptive questionnaire generation methods for pairwise comparison---RUS and RRUS, both involving a random lottery with two outcomes and a deterministic lottery with one outcome. 
Motivated by RUS, where the random lottery has two fixed outcomes and the deterministic lottery is randomly chosen, we propose the MUS method to specifically choose the deterministic lottery at the point 
{with}
the greatest difference between the largest and smallest possible utility values in the current ambiguity set. 
We provide both qualitative and quantitative convergence results for the MUS method in Theorem~\ref{thm:mus-converg}.

\subsection{Specification of the true utility function}

{
We begin by specifying a class of utility functions characterized by generic information available prior to preference elicitation, such as risk aversion and properties of marginal utility.}
Let 
$\mathcal{L}^p([0,1])$ denote the set of real-valued functions $u: [0,1]\to \mathbb{R}$ integrable up to the $p$-th order. 
Let 
$$
\begin{aligned}
\mathcal{U}_{cv}:= \{ u\in \mathcal{L}^1([0,1])\mid \
&u \text{ is monotonically increasing, concave and Lipschitz continuous with }\\
&\text{modulus bounded by $L$}, 
u(0)=0,\ u(1)=1\},
\end{aligned}
$$
which can be further represented as 
\begin{align}\label{eq:U_c}
\nonumber
\mathcal{U}_{cv}:= \Big\{ & u\in\  \mathcal{L}^1([0,1]) \mid  
 u^{'}_{+}(y)\geq0,\ 
 {\forall y\in[0,1)},\
 u^{'}_{-}(y)\geq0,\ 
 {\forall y\in(0,1]},\ 
 u(0)=0,\ u(1)=1,\\
& \text{Lip}(u)\leq L,\ u^{'}_{+}(y)\leq u^{'}_{-}(y), \forall y\in
{(}0,1{)},\ 
u^{'}_{-}(y_2)\leq u^{'}_{+}(y_1), 
\forall 0\leq y_1<y_2\leq1 \Big\},\quad
\end{align}
where $u^{'}_{+}(y)$ ($u^{'}_{-}(y)$, resp.) denotes the right (left, resp.) derivative of $u$ at $y$; $\text{Lip}(u)\leq L$ indicates that $u$ is Lipschitz continuous { 
over $[0,1]$ with modulus bounded 
by a constant  $L$. 
The normalization condition on $u$ requires 
$L\geq 1$.}
{
Lipschitz continuity is a mild condition satisfied by many utility functions of practical interest, see e.g.,~\cite[Remark~1]{chen2026eliciting}.
In this paper, we will need this condition 
to ensure convergence of MUS (Theorem~\ref{thm:mus-converg}). Moreover, we will need information on $L$ for 
implementation of MUS (Algorithm~\ref{alg:mus-1}).
}
Note that if we set $\mathcal{U}_{cv}(y):=\{u(y): u\in \mathcal{U}_{cv}\}$, then $\mathcal{U}_{cv}(\cdot)$ is a set-valued mapping from $[0,1]$ to $[0,1]$. 
{Consequently,  
we can define the range of $\mathcal{U}_{cv}$ at point $y$ 
by 
$\mathcal{U}_{cv}(y)= \left[\min\limits_{u\in \mathcal{U}_{cv}} u(y),\  \max\limits_{u\in \mathcal{U}_{cv}} u(y)\right] $, $\forall y\in[0,1]$. 
}
Our aim is to reduce the range (graph) of $\mathcal{U}_{cv}(\cdot)$ over $[0,1]$ through preference elicitation.

\begin{assumption}
\label{assump:VM}
The DM's preferences can be represented by the Von Neumann-Morgenstern (VNM) expected utility theory.
{The true VNM utility function $u^*$ remains unchanged throughout the preference elicitation process.}
Information on 
$u^*$
{that}
characterizes the DM's preferences is incomplete, and the available information before preference elicitation ensures that the true utility function lies in the set $\mathcal{U}_{cv}$.
\end{assumption}

Under Assumption~\ref{assump:VM}, the true utility function is defined over $[0,1]$ and normalized over the interval. 
This is because we focus on the case that the DM's incomes are non-negative and the DM is risk averse on gains. We call $\mathcal{U}_{cv}$ the initial ambiguity set. 
The main purpose of this paper is how to design paired gambling questionnaires for the DM to choose so that the preference information inferred from the choices is effectively used to reduce the size of $\mathcal{U}_{cv}$ as quickly as possible, eventually to a singleton. 
To this end, {we adopt the pseudo-metrics introduced in~\cite{guo2021utility}
for quantifying 
the distance between normalized utility functions and use some specific ones 
to quantify the range of ${\cal U}_{cv}(\cdot)$ and its reduction as more and more preference information is obtained}. 
Let $\mathscr{G}$ be a set of bounded measurable functions defined over $[0,1]$. 
For $u,v\in\mathcal{U}_{cv}$, define the semi-distance between $u$ and $v$ by
\begin{equation}\label{eq:compute-semi-dist}
    \dd_{\mathscr{G}} (u,v):=\sup_{g \in \mathscr{G}}
 \left|\int_0^1 g(z)du(z)- \int_0^1 g(z)dv(z) \right|,
\end{equation}
where the integrals are well-defined in the Lebesgue-Stieltjes sense~\citep{hildebrandt1963introduction}. 
Here $g$ might be viewed as a test function and $\dd_{\mathscr{G}} (u,v)=0$
means that the test functions in $\mathscr{G}$ do not distinguish $u$ from $v$ even if $u\neq v$. 
Since the utility functions in $\mathcal{U}_{cv}$ are normalized with $u(0)=0, u(1)=1$, $\dd_{\mathscr{G}} (u,v)$ resembles the pseudo-metric of $\zeta$-structure in probability theory, see~\cite{guo2024utility}. 
In this paper, we are interested in two cases:
\begin{equation*}
\mathscr{G}=\mathscr{G}_{L}:=\left\{g: [0,1]\to\mathbb{R} \mid \ g\text{ is Lipschitz continuous with modulus bounded by 1}\right\}
\label{eq:G-kant}
\end{equation*}
and
\begin{equation*}
\mathscr{G} = \mathscr{G}_{I} :=\left\{ g:=\mathbbm{1}_{(0, z]}(\cdot) \mid \
 \mathbbm{1}_{(0, z]}(s):=
	1 \; \text{if }\; s\in (0, z] \; \text{and} \;
	0 \; \text{otherwise},\; 
    {z\in[0,1]}
	\right\}.
\end{equation*}
The former corresponds to the Kantorovich metric, denoted by $\dd_K(u,v)$,
and the latter corresponds to the uniform Kolmogorov metric, denoted by $\dd_I(u,v)$. 
It is well-known that $\dd_{I}(u,v)=\|u-v\|_\infty$.
Since $u,v$ are in the set of normalized utility functions $\mathcal{U}_{cv}$, 
$$
\int_0^1 g(z)d u(z) - \int_0^1 g(z)d v(z) =
\int_0^1 (g(z)-
g({1})
)d u(z) - \int_0^1 (g(z)-
g({1})
)d v(z).
$$
We can replace $\mathscr{G}_L$ with 
\bgeqn 
\mathscr{G}_{L}:=\left\{g: [0,1]\to\mathbb{R} \mid 
g({1})
=0,\ g\text{ is Lipschitz continuous with modulus bounded by 1}
\right\}.
\label{eq:G-kant-a}
\edeqn

\subsection{RUS and RRUS}\label{sec:RUS-RRUS}


With the initial ambiguity set $\mathcal{U}_{cv}$ in place, we now move on to discuss how to design a sequence of paired gambling questionnaires, denoted by $\{(W_{m},Y_{m})\}_{m=1,2,\dots}$, to use them to elicit the DM's preferences and subsequently update the ambiguity set sequentially.
Following the literature in behavioural economics and PRO, we consider a random lottery with two outcomes and a deterministic lottery with one outcome: 
\begin{equation}\label{eq:W-Y}
W_{m}=\left\{
\begin{aligned}
    r_1^{m}\quad & \text{ w.p. }\; 1-p^{m},\\
    r_3^{m}\quad & \text{ w.p. }\; p^{m},
\end{aligned}
\right.
\quad\text{and}\quad
Y_{m}=r_2^{m}\quad \text{ w.p. } 1,
\end{equation}
where $r_1^{m},r_2^{m},r_3^{m}\in[0,1]$ are parameters representing outcomes, and $p^{m}\in(0,1)$ is the corresponding probability. 
{This simple comparison format has roots in early experiments on utility assessment with human participants~\cite{becker1964measuring}, which is easy to understand and to answer for human DMs.}
We can write $W_{m}$ and $Y_{m}$ as random variables mapping from $(\Omega,{\cal F}, \mathbb{P})$ to $[0,1]$,   
{and their expected utilities are}
\begingroup
\begin{equation*}
\mathbb{E}[u(W_m)]=(1-p^m)u(r_1^m)+p^m u(r_3^m),
\quad \mathbb{E}[u(Y_m)]=u(r_2^m),
\end{equation*}
\endgroup
where the expectation is taken w.r.t.~the probability measure $\mathbb{P}$ of the gambling. 
Let $\mathcal{U}_0 :=\mathcal{U}_{cv}$ {be the initial ambiguity set} and 
$\mathcal{U}_m$ be the ambiguity set updated with the paired gambling questionnaires of 
$\{(W_{k},Y_{k})\}_{k=1,2,\cdots,m}$ for $m=1,2,\cdots$. 
{With} the $(m+1)$-th query, the ambiguity set $\mathcal{U}_m$ is updated to
\begingroup
\begin{equation*}
\begin{aligned}
\mathcal{U}_{m+1}=\Big\{u\in\mathcal{U}_m\mid\;&Z_{m+1}\cdot u(r_2^{m+1})\geq Z_{m+1}\cdot\big[(1-p^{m+1})u(r_1^{m+1})+p^{m+1}u(r_3^{m+1})\big]\Big\},
\end{aligned}
\end{equation*}
\endgroup
after the DM's preference for the pair $(W_{m+1},Y_{m+1})$ is observed, 
where $Z_{m+1}=1$ if $Y_{m+1}$ is preferred and $-1$ otherwise, $m=0,1,\cdots.$ 
{If we specify 
$r_1^{m+1}=0$ and $r_3^{m+1}=1$, the normalization $u(0)=0$ and $u(1)=1$ implies that the update simplifies to}
\begingroup
\begin{equation}\label{eq:U_m-MUS-0}
\mathcal{U}_{m+1}=\Big\{u\in\mathcal{U}_m\mid Z_{m+1}\cdot u(r_2^{m+1})\geq Z_{m+1}\cdot p^{m+1}\Big\}.
\end{equation}
\endgroup
Throughout the paper, we make the following assumption.
\begin{assumption}
 There is no error in the process of preference elicitation, whether in the DM's response, in the observation of the DM's response by the modeler, or in data concerning $p^{m}$ and $r_2^m$ (e.g.~rounding errors). 
This means that the DM prefers $W_m$ to $Y_m$ if and only if $\mathbb{E}\left[u^*(W_m)\right] \geq \mathbb{E}\left[u^*(Y_m)\right]$. 
\end{assumption}
{
Preference elicitation error has been widely discussed in the literature. For instance, \citet{toubia2007probabilistic} illustrate that an erroneous response in deterministic polyhedral elicitation can exclude the true preference parameters from the retained polyhedron, thereby preventing subsequent estimates from converging to the truth. Similarly, \citet{bertsimas2013learning} explain that errors in early responses can affect the selection of subsequent adaptive questions, making it more difficult to identify the DM's true utilities. In addition to erroneous responses, \citet{armbruster2015decision} argue that elicitation error may also arise from model misspecification. In the RUS, RRUS, and MUS settings considered below, if the DM provides an erroneous response, the elicitation procedure may similarly lead to ambiguity sets that exclude the true utility function from subsequent stages of the elicitation process.}

Like ${\cal U}_{cv}$, we may view ${\cal U}_m(\cdot)$ as a set-valued mapping from $[0,1]$ to $[0,1]$.
Throughout the paper, we refer to $\mathcal{U}_{m}$ as a set of utility functions and a set-valued mapping interchangeably depending on the context.
Note that since $\mathcal{U}_{m}$ is a convex set of functions, the range of $\mathcal{U}_{m}$ at point $y$ can be written as $\mathcal{U}_{m}(y)= \left[\min\limits_{u\in \mathcal{U}_{m}} u(y),\ \max\limits_{u\in \mathcal{U}_{m}} u(y)\right] $, $\forall y\in[0,1]$, see Figure~\ref{fig:motivation} for the range of $\mathcal{U}_{0}:=\mathcal{U}_{cv}$ over [0,1].
The key question is how to set $r_1^{m+1}$, $r_2^{m+1}$, $r_3^{m+1}$, and $p^{m+1}$ such that the range of $\mathcal{U}_{m+1}$ is reduced from $\mathcal{U}_{m}$ efficiently.

\vspace{0.4cm}
\uline{\textbf{Random utility split (RUS) and Random relative utility split (RRUS).}}
The RUS scheme fixes $r^{m+1}_1$ and $r^{m+1}_3$ as the worst and best possible returns respectively for all $m\in\mathbb{Z}$, i.e., $r^{m+1}_1=0$ and $r^{m+1}_3=1$. 
In that case, $u(r^{m+1}_1)\equiv0$ and $u(r^{m+1}_3)\equiv1$ by the normalization condition in the ambiguity set. 
The RUS scheme then selects $r_2^{m+1}$ randomly from $[0,1]$ following a uniform distribution, and sets $p^{m+1}$ to reduce by half ({bisect}) the range of $\mathcal{U}_{m}$ at $r^{m+1}_2$, making $p^{m+1}u(r^{m+1}_3)+(1-p^{m+1})u(r^{m+1}_1)$ the midpoint of the range
as in~\eqref{eq:pm}. 
The randomness in setting $r_2^{m+1}$ facilitates implementation but also reduces the efficiency of the elicitation process if it is poorly selected at a point where the range of ${\cal U}_m$ is narrow.  
The RRUS scheme modifies RUS by allowing $r_1^{m+1}$ and $r_3^{m+1}$ to be selected randomly from the interval $[0,1]$ following a uniform distribution and then setting $r_2^{m+1}=(r_1^{m+1}+r_3^{m+1})/2$. 
Like RUS, the RRUS sets $p^{m+1}$ to halve the range of ${\cal U}_m$ at $r_2^{m+1}$.

\subsection{Maximum utility split (MUS)}\label{sec:MUS}

The MUS scheme to be proposed in this paper resembles RUS by setting $r^{m}_1=0$ and $r^{m}_3=1$ for all $m\in\mathbb{Z}$, thus $u(r^{m}_1)\equiv0$ and $u(r^{m}_3)\equiv1$. 
In this case,
\begin{equation}\label{eq:U_m-MUS}
\mathcal{U}_m=\Big\{u\in\mathcal{U}_{cv}\mid Z_k\cdot u(r_2^k)\geq Z_k\cdot p^k,\ k=1,2,\ldots,m\Big\}.  
\end{equation}
However, it differs from RUS by taking a different strategy to set $r_2^m$.
To reduce the range of $\mathcal{U}_m$ more efficiently, the MUS scheme selects $r^{m+1}_2$ where the range of ${\cal U}_m$ is largest, and this can be identified by solving the following optimization problem: 
\begin{equation}\label{eq:MUS-r2-initial}
r^{m+1}_2=\mathop{\arg\max}\limits_{r_2\in [0,1]} \left[\max\limits_{u\in \mathcal{U}_{m}}u(r_2)-\min\limits_{u\in \mathcal{U}_{m}}u(r_2)\right].
\end{equation}
Once $r_2^{m+1}$ is identified, we set the probability $p^{m+1}$ 
such that the newly added pairwise comparison constraint $Z_{m+1}\cdot u(r_2^{m+1})\geq Z_{m+1}\cdot p^{m+1}$ in~\eqref{eq:U_m-MUS-0} halves the range $\{u(r^{m+1}_2)\mid u\in \mathcal{U}_{m}\}$, i.e., 
\bgeqn \label{eq:pm}
p^{m+1}=\frac{1}{2}\left[\max\limits_{u\in \mathcal{U}_{m}} u(r^{m+1}_2) + \min\limits_{u\in \mathcal{U}_{m}}u(r^{m+1}_2)\right].
\edeqn
{Analogous to RUS and RRUS, we adopt the specific forms of query $(W_m,Y_m)$ to perform the 
bisection step for identifying $p^{m+1}$. }
Note that the maximum range can be represented by the Kolmogorov distance, that is, $\max\limits_{u_1,u_2 \in \mathcal{U}_m} \dd_I(u_1,u_2)$.
Let
\bgeqn 
\label{eq:Kom-Kant-distance}
\dd_I(\mathcal{U}_m):= \max\limits_{u_1,u_2 \in \mathcal{U}_m} \dd_I(u_1,u_2)
\quad \text{and} \quad
\dd_K(\mathcal{U}_m):= \max\limits_{u_1,u_2 \in \mathcal{U}_m} \dd_K(u_1,u_2).
\edeqn

\begin{example} 
We use a simple example to illustrate the difference between RUS and MUS. 
Consider the case that the DM's true utility function is $u(y)=\frac{1}{1-e^{-9.5}}(1-e^{-9.5y})$, see the black curve in Figure~\ref{fig:motivation}. 
The function is monotonically increasing, concave, and Lipschitz continuous with modulus $L=10$, and it is normalized over $[0,1]$. 
This information allows us to identify the initial ambiguity set ${\cal U}_0:=\mathcal{U}_{cv}$. 
We can determine the lower bound and the upper bound of utility functions in the ambiguity set ${\cal U}_0$, see the green line OP and the blue line OTP in Figure~\ref{fig:motivation}.

Under the MUS scheme, the largest gap between the upper bound and the lower bound is attained at point $y=0.1$.
Thus, we set $r_2^{\rm MUS}=0.1$ and $p^{\rm MUS}=(1+0.1)/2=0.55$. 
See point M(0.1,0.55) in Figure~\ref{fig:motivation}. 
On the other hand, under the RUS scheme, we randomly choose $r_2^{\rm RUS}=0.7$,
resulting in $p^{\rm RUS}=(1+0.7)/2=0.85$. 
See point R(0.7,0.85) in Figure~\ref{fig:motivation}. 
From the figure, we can see that the MUS scheme moves the lower bound OP up to OMP, whereas the RUS scheme moves the lower bound OP to ORP. 
Obviously, the MUS scheme reduces the range of ${\cal U}_0$ (from triangle area OTPO to area OTPMO) by a significantly larger portion compared to RUS (from OTPO to area OTPRO). 

\begin{figure}[htbp]
    \centering
    \includegraphics[width=0.9\linewidth]{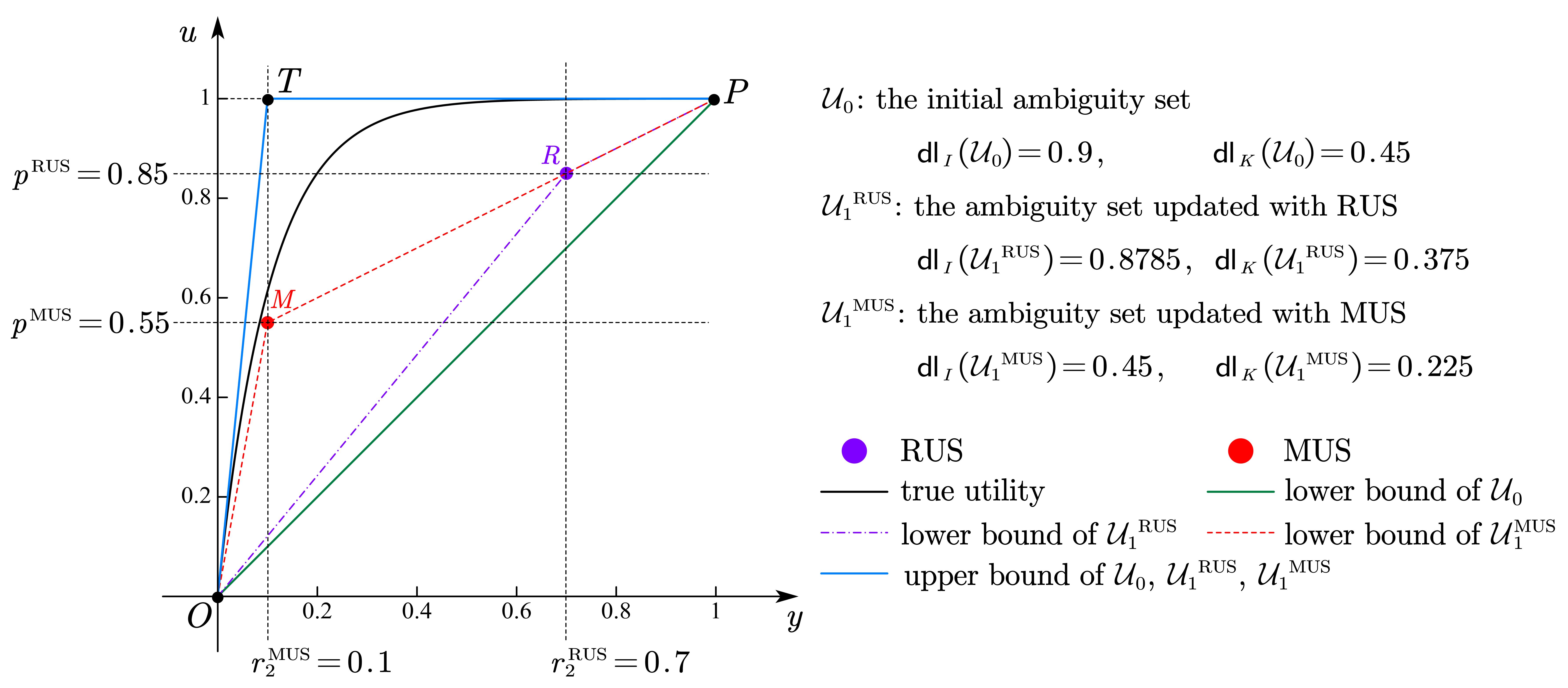}
    \caption{
    Illustration of MUS. 
    } 
    \label{fig:motivation}
\end{figure}
\end{example}

Parts~(i)~and~(ii) of the next theorem characterize the structural properties of ${\cal U}_m$ under the Kolmogorov norm and Kantorovich norm.
We need them in theoretical analysis in Sections~\ref{sec-bounds}~and~\ref{sec:utility}. 
Part~(iii) states the convergence of the MUS scheme deterministically, whereas Part~(iv) 
and Part~(v) 
quantify
the rate of convergence.

\begin{theorem}\label{thm:mus-converg}
Let $\{W_m,Y_m\}_{m=1}^{\infty}$ be a sequence of queries for pairwise comparisons structured as in~\eqref{eq:W-Y} and generated by the MUS scheme,  
let
\begin{equation*}
\mathcal{U}_m:=\Big\{u\in\mathcal{U}_{cv}\mid Z_k\cdot u(r_2^k)\geq Z_k\cdot p^k,\ k=1,\ldots,m\Big\},
\end{equation*}
where $Z_k$ is the choice between $W_k$ and 
{$Y_k$}
according to the true VNM utility $u^*$ of the DM with no response error.
Then the following assertions hold.
\begin{itemize}

\item[(i)] $\mathcal{U}_m$ is compact under the topology induced by the Kolmogorov 
{(supremum)}
norm.

\item[(ii)]
For any $u,v\in {\cal U}_m$, $\dd_K(u,v)$ is the area between curves of $u$ and $v$. 

\item[(iii)]  
The sequence of the ambiguity sets $\{{\cal U}_m\}$ converges to the singleton $\{u^*\}$ under the Hausdorff distance equipped with the Kolmogorov norm.

\item[(iv)]For any ${\epsilon>0}$, there exists 
$
{m_0}:=(3\lceil 4L/\epsilon\rceil+2)-(2\lceil4L/\epsilon\rceil+1)\lfloor\log_2(\epsilon)\rfloor$
such that 
\begin{align}
\label{eq:thm1-iv}
\sup_{u\in \mathcal{U}_{
{m}-1}} \|u-u^*\|_\infty < \epsilon,
\ \  \text{for all} \ \ 
{m}\geq 
{m_0}, 
\end{align}
where $\lceil \cdot \rceil$ and $\lfloor \cdot \rfloor$ denote respectively
rounding up and down to the nearest integers.

\item[(v)]
{
For any $m\geq 2$,
\begin{equation}
\sup_{u\in\mathcal U_{m-1}}\|u-u^*\|_\infty
<\frac{40L\log_2 m}{m}.
\end{equation}
}
\end{itemize}
\end{theorem}

\begin{proof}

\underline{Part~(i)}. By definition, $u\in\mathcal{U}_m$ is non-decreasing over $[0, 1]$ with $u(0) = 0$, $u(1) = 1$, and $u$ is globally Lipschitz continuous with a uniformly bounded Lipschitz modulus $L$. The monotonically increasing property and the normalization condition ensure the uniform boundedness of the utility functions in the set; the 
{global}
Lipschitz continuity guarantees equicontinuity of the class of functions.
By the Arzelà-Ascoli Theorem (see e.g.,~\cite[Theorem~2.3]{brown1993topological}), $\mathcal{U}_m$ is a 
relatively 
compact set, that is, it is contained by a compact set in the space of continuous functions
{on $[0,1]$ under the supremum norm.}
To show the compactness of the set, it suffices to show that the set is closed.
Let $\{u_k\} \subseteq \mathcal{U}_m$ be a sequence converging to $u$ 
{in the supremum norm.}
The uniform convergence ensures continuity of $u$. For any fixed points $x, y \in [0, 1]$,
$$
\left|u_k(x) - u_k(y)\right| \leq L |x - y|, \; \forall k,
$$
where $L$ is the global Lipschitz modulus. 
By driving $k$ to infinity, we obtain
$$
|u(x) - u(y)| \leq L |x - y|,
$$
which means that $u$ is also Lipschitz continuous with modulus bounded by $L$. Moreover, since $u_k$ is a monotonically increasing, concave, Lipschitz continuous, and normalized function, and satisfies all pairwise comparison constraints in $\mathcal{U}_m$, which are all linear, its limit is also a function in $\mathcal{U}_{cv}$ satisfying these linear constraints. This shows $u \in \mathcal{U}_m$, and hence $\mathcal{U}_m$ is closed.

\vspace{0.1cm}
\underline{Part~(ii)}. Since $u$ and $v$ are normalized utility functions with $u(0)=v(0)=0$ and $u(1)=v(1)=1$, and the utility functions are non-decreasing, we may view them as cumulative distribution functions of some random variables distributed over $[0,1]$. 
The conclusion then follows from the fact that the Kantorovich distance between two probability distributions is equal to the area 
{between}
the cdfs of the two distributions, see 
discussions at Page~424 in~\cite{gibbs2002choosing}.

\vspace{0.1cm}
\underline{Part~(iii)}. 
It suffices to show 
$$
\lim_{m\to\infty}\sup_{u\in \mathcal{U}_{m-1}} \|u-u^*\|_\infty =0
$$
or equivalently
$$
\lim_{m\to\infty} \sup_{r_2\in [0,1]} 
D^m(r_2) :=\max\limits_{u\in \mathcal{U}_{m-1}}u(r_2)-\min\limits_{u\in \mathcal{U}_{m-1}}u(r_2)=0. 
$$
Under the maximum utility split scheme, 
\bgeq 
r_2^m=\mathop{\arg\max}\limits_{r_2\in[0,1]} D^m(r_2)
\quad \text{and} \quad 
D^m(r^m_2)\geq D^m(r_2)\geq 0,\;
\forall r_2\in [0,1].
\edeq
Thus it is enough to show that $D^m(r^m_2)\to0$ as $m\to\infty$. 
By definition
\begin{equation}\label{eq-mono-dm}
0\leq D^{m+1}(r_2)\leq D^m(r_2)< 1,\;
\forall r_2\in[0,1], 
\end{equation} 
for $m\in\mathbb{Z}_{+}$.
By the definition of $r_2^m$,
\begin{equation}\label{eq-mono-dm-r2}
0\leq D^{m+1}(r_2^{m+1})\leq D^m(r_2^{m+1})\leq \max\limits_{r_2\in[0,1]}D^m(r_2)=D^m(r_2^m)<1,\; \forall m\in\mathbb{Z}_+, 
\end{equation} 
which implies that the sequence $\{D^m(r^m_2)\}_{m=1}^{\infty}$ is bounded and monotonically decreasing. 
Let $\varepsilon$ denote its limit. We show next that $\varepsilon=0$.
Assume for the sake of a contradiction that $\varepsilon>0$. 
Then 
\begin{equation}\label{proof:eq-contradiction}
D^{m}(r_2^m)\geq\varepsilon>0,\ \forall m \in\mathbb{Z}_{+}.
\end{equation}
Since the sequence $\{r_2^m\}_{m=1}^{\infty}$ is bounded (lies in the interval $[0,1]$), there exists a converging subsequence $\{r_2^{m_k}\}_{k=1}^{\infty}\subseteq\{r_2^m\}_{m=1}^{\infty}$, and a number $s\in\mathbb{Z}_{+}$ such that
\begin{equation*}
\left|r_2^{m_{k+1}}-r_2^{m_{k}}\right|<\frac{\varepsilon}{4L},\; \forall k \geq s,
\end{equation*}
where $L>0$ is the Lipschitz modulus of the utility functions defined as in~\eqref{eq:U_c}. 
Let 
\bgeq 
I_1^m(r_2):=\min\limits_{u\in\mathcal{U}_{m-1}}u(r_2)
\quad \text{and} \quad  I_2^m(r_2):=\max\limits_{u\in\mathcal{U}_{m-1}}u(r_2),\ \forall m \in\mathbb{Z}_{+}.
\edeq
Then $D^m(r_2)=I_2^m(r_2)-I_1^m(r_2)$, $\forall r_2\in[0, 1]$, for $m \in\mathbb{Z}_{+}$. 
We consider two cases of each adjacent pair in the converging subsequence $\{r_2^{m_k}\}_{k=1}^{\infty}\subseteq\{r_2^m\}_{m=1}^{\infty}$.

\vspace{0.1cm}
Case~1: $r_2^{m_{k+1}}\geq r_2^{m_{k}}$. By the compactness of $\mathcal{U}_{m_{k+1}-1}$, we can find $\hat{u}\in\mathop{\arg\min}\limits_{u\in\mathcal{U}_{m_{k+1}-1}}u(r_2^{m_{k+1}})$ such that $\hat{u}(r_2^{m_{k+1}})=I_1^{m_{k+1}}(r_2^{m_{k+1}})$.
By the monotonicity of $\hat{u}\in\mathcal{U}_{m_{k+1}-1}$, we have
\begin{equation}\label{eq-proof-le1-mon}
I_1^{m_{k+1}}(r_2^{m_{k+1}})=\hat{u}(r_2^{m_{k+1}}) \geq \hat{u}(r_2^{m_{k}}) \geq \min\limits_{u\in\mathcal{U}_{m_{k+1}-1}}u(r_2^{m_{k}}) = I_1^{m_{k+1}}(r_2^{m_{k}}). 
\end{equation}
Moreover, we can find $\tilde{u}\in\mathop{\arg\max}\limits_{u\in\mathcal{U}_{m_{k+1}-1}}u(r_2^{m_{k+1}})$ such that $\tilde{u}(r_2^{m_{k+1}})=I_2^{m_{k+1}}(r_2^{m_{k+1}})$.
By the Lipschitz continuity of $\tilde{u}$, $\tilde{u}(r_2^{m_{k+1}}) - \tilde{u}(r_2^{m_{k}}) \leq L (r_2^{m_{k+1}}- r_2^{m_{k}})$. 
On the other hand, by the definition of $I_2^{m_{k+1}}(\cdot)$, 
$I_2^{m_{k+1}}(r_2^{m_{k}})=\max\limits_{u\in\mathcal{U}_{m_{k+1}-1}} u(r_2^{m_{k}}) \geq \tilde{u}(r_2^{m_{k}})$. 
Thus
\begin{equation}\label{eq-proof-le1-lip}
I_2^{m_{k+1}}(r_2^{m_{k+1}}) - I_2^{m_{k+1}}(r_2^{m_{k}}) =  \tilde{u}(r_2^{m_{k+1}}) - \tilde{u}(r_2^{m_{k}}) + \tilde{u}(r_2^{m_{k}})  - I_2^{m_{k+1}}(r_2^{m_{k}}) \leq L (r_2^{m_{k+1}}- r_2^{m_{k}}).
\end{equation}
Combining~\eqref{eq-proof-le1-mon}~and~\eqref{eq-proof-le1-lip}, we obtain
\begin{align*}
 D^{m_{k+1}}(r_2^{m_{k+1}}) & = I_2^{m_{k+1}}(r_2^{m_{k+1}}) -I_1^{m_{k+1}}(r_2^{m_{k+1}}) \\
&\leq I_2^{m_{k+1}}(r_2^{m_{k}})+L(r_2^{m_{k+1}}-r_2^{m_{k}})-I_1^{m_{k+1}}(r_2^{m_{k}}) = D^{m_{k+1}}(r_2^{m_{k}})+L(r_2^{m_{k+1}}-r_2^{m_{k}}). 
\end{align*}

\vspace{0.1cm}
Case~2: $r_2^{m_{k+1}}<r_2^{m_{k}}$. 
By the compactness of $\mathcal{U}_{m_{k+1}-1}$, and  monotonicity and Lipschitz continuity of every utility function in the set, we can derive as in Case~1 that 
\begin{equation*}
I_2^{m_{k+1}}(r_2^{m_{k+1}})\leq I_2^{m_{k+1}}(r_2^{m_{k}}), 
\quad {\rm and}\quad
I_1^{m_{k+1}}(r_2^{m_{k+1}})\geq I_1^{m_{k+1}}(r_2^{m_{k}})-L(r_2^{m_{k}}-r_2^{m_{k+1}}),
\end{equation*}
which give rise to
\begin{equation*}
D^{m_{k+1}}(r_2^{m_{k+1}})\leq I_2^{m_{k+1}}(r_2^{m_{k}})-\left[I_1^{m_{k+1}}(r_2^{m_{k}})-L(r_2^{m_{k}}-r_2^{m_{k+1}})\right] = D^{m_{k+1}}(r_2^{m_{k}})+L(r_2^{m_{k}}-r_2^{m_{k+1}}). 
\end{equation*}
Summarizing the discussions of the two cases, we conclude that
\begin{equation}\label{proof:eq-D-diff}
0\leq D^{m_{k+1}}(r_2^{m_{k+1}})-D^{m_{k+1}}(r_2^{m_{k}})\leq L\left|r_2^{m_{k+1}}-r_2^{m_{k}}\right|<\frac{\varepsilon}{4},\; \forall k\geq s.
\end{equation}

Under the maximum utility split scheme, we can set $p^m :=\frac{1}{2}\left[I_1^m(r^{m}_2) +I_2^m(r^{m}_2)\right]$ to halve the range of utility values $\{u(r^m_2)\mid u\in \mathcal{U}_{m-1}\}$ at point $r^{m}_2$. Combining this with~\eqref{eq:U_m-MUS}, 
we have 
\begin{equation}\label{proof:eq-cut-half}
D^{m+1}(r_2^m)=\frac{1}{2}D^m(r_2^m),
\; {\rm for} \;m\in\mathbb{Z}_{+}.
\end{equation}
Combining~\eqref{eq-mono-dm}, \eqref{proof:eq-D-diff}~and~\eqref{proof:eq-cut-half}, we have
\begin{align}
D^{m_{k+1}}(r_2^{m_{k+1}}) 
<\frac{\varepsilon}{4} + D^{m_{k+1}}\left(r_2^{m_{k}}\right)  
\leq \frac{\varepsilon}{4} + D^{m_{k}+1}\left(r_2^{m_{k}}\right)
=\frac{\varepsilon}{4}+\frac{1}{2}D^{m_{k}}\left(r_2^{m_{k}}\right), \label{eq-deduce}
\end{align}
where the second inequality follows from~\eqref{eq-mono-dm} and the fact that $m_{k+1} \geq m_{k}+1 $. 
Thus, we can deduce from~\eqref{eq-deduce} that 
\begin{align*}
D^{m_{k+1}}(r_2^{m_{k+1}}) 
<\frac{\varepsilon}{4}+\frac{1}{2}D^{m_{k}}\left(r_2^{m_{k}}\right)\nonumber
&< \frac{\varepsilon}{4}+\frac{\varepsilon}{4\times 2}+\cdots+\frac{\varepsilon}{4\times 2^{k-s}}+\frac{1}{2^{k-s+1}}D^{m_s}\left(r_2^{m_s}\right)\nonumber\\
&< \frac{\varepsilon}{4}+\frac{\varepsilon}{4\times 2}+\cdots+\frac{\varepsilon}{4\times 2^{k-s}}+\frac{1}{2^{k-s+1}} \nonumber\\
&=(1-\frac{1}{2^{k-s+1}})\frac{\varepsilon}{2}+\frac{1}{2^{k-s+1}}=\frac{\varepsilon}{2} +(1-\frac{\varepsilon}{2})\frac{1}{2^{k-s+1}}, 
\end{align*}
for any $k>s$, which implies that
\begin{equation*}
D^{m_{k+1}}(r_2^{m_{k+1}})<\frac{\varepsilon}{2} +(1-\frac{\varepsilon}{2})\frac{1}{2^{k-s+1}}<\varepsilon,\; \forall k > \max\Big\{\log_2{\frac{2-\varepsilon}{\varepsilon}}+s-1,\ s\Big\}, 
\end{equation*}
a contradiction to~\eqref{proof:eq-contradiction}.

\vspace{0.1cm}
\underline{Part~(iv)}. 
We begin by summarizing some key results established in the proof of Part~(iii). 
\begin{itemize}

\item 
$\{D^m(r_2^{m})\}$ monotonically decreases and converges to 0 with 
\begin{align}
& 0\leq D^{m_1}(r_2)\leq D^{m_2}(r_2)< 1, \;
\forall r_2\in[0,1],\; \text{for}\; m_1>m_2, 
\label{eq-mono-dm-1}\\
& 0\leq D^{m_1}(r_2^{m_1})\leq D^{m_2}(r_2^{m_2})<1,\; \forall m_1>m_2. 
\label{eq-mono-dm-2}
\end{align}

\item 
For any two integers $m_1$ and $m_2$, we have 
\begin{equation}
\label{proof:eq-D-diff-1}
0\leq D^{m_{1}}(r_2^{m_{1}})-D^{m_{1}}(r_2^{m_{2}})\leq L\left|r_2^{m_{1}}-r_2^{m_{2}}\right|.
\end{equation}

\item Under the MUS scheme, for any integer $m>0$, 
\begin{equation}
\label{proof:eq-cut-half-a}
D^{m+1}(r_2^m)=\frac{1}{2}D^m(r_2^m). 
\end{equation}

\end{itemize}

Next, we construct two sub-sequences for later use 
{in the}
proof.

\begin{statement}\label{state:cong-rate}
Let $\{r_2^m\}_{m=1}^{\infty}\subset [0,1]$ be a given sequence.
For any fixed positive numbers $L,\delta>0$ with $\delta/L<1$,
there exist two sub-sequences $\{r_2^{m^1_k}\}_{k=1}^{\infty},
\{r_2^{m^2_k}\}_{k=1}^{\infty}\subseteq\{r_2^m\}_{m=1}^{\infty}$ such that 
\begin{equation}
\label{proof:eq-r-convg-2}
m^1_{k}<m^2_{k}<m^1_{k+1}, 
\;
m^1_1\leq \lceil L/\delta\rceil+1,
\; 
m^1_{k+1}-m^1_{k}\leq 2\lceil L/\delta\rceil+1,
\; \forall k\in\mathbb{Z}_+
\end{equation}
and
\begin{equation}\label{proof:eq-r-convg}
\left|r_2^{m^1_{k}}-r_2^{m^2_{k}}\right|\leq\frac{\delta}{L},\; \forall k\in\mathbb{Z}_+.
\end{equation}
\end{statement}

\noindent {\bf Proof of the statement.} 
We begin by dividing the sequence $\{r_2^m\}_{m=1}^{\infty}$ into non-overlapping 
batches/groups of equal size $M:=(\lceil L/\delta\rceil+1)$ as follows: 
\begin{equation*}
{\cal R}^{k} :=\left\{r_2^{(k-1)M+1},r_2^{(k-1)M+2},\ldots,r_2^{kM}\right\},\; \forall k\in\mathbb{Z}_+,
\end{equation*}
which corresponds to a series of index sets $[1,2,\ldots,M]$, $[M+1,\ldots,2M]$, $\cdots$.
In other words, ${\cal R}^{1}$
is composed of the first $M$ numbers of sequence
$\{r_2^m\}_{m=1}^{\infty}$, and 
${\cal R}^{2}$
is composed of the next $M$ numbers of the sequence, etc. 
By the definition, we can see immediately: 
(a) ${\cal R}^{k}\cap  {\cal R}^{j}=\emptyset$,
(b) $\left| {\cal R}^{k}\right|=M$, where $\left|{\cal S}\right|$ denotes the cardinality of the set ${\cal S}$,
(c) $\cup_{k=1}^\infty {\cal R}^{k} = \cup_{m=1}^\infty \left\{r_2^m\right\}$,
(d) for any two elements
$r_2^{i},\ r_2^{j}\in 
{\cal R}^{k}$, 
$|i - j|\leq M$.


Next, we claim that there exist $r_2^{i}$ and $r_2^{j}$ ($i\neq j$) from ${\cal R}^{k}$ 
such that 
$\left|r_2^{i}-r_2^{j}\right|\leq\frac{\delta}{ L}.
$
To see how such $i$ and $j$ 
may exist, we label the points in ${\cal R}^{k}$ in an increasing order of value of $r_2$ and denote them by 
$\{a_i\}_{i=1,\ldots,M}$ for 
simplicity of notation, where $0\leq a_i\leq1$. The minimum distance between any two consecutive points must not exceed $\frac{1}{M-1
}$, because, otherwise, we would have 
$$
a_{M}=\sum_{i=1}^{M-1}(a_{i+1}-a_i)+a_1\geq\sum_{i=1}^{M-1}(a_{i+1}-a_i)>(M-1)\cdot\frac{1}{M-1}=1,
$$
a contradiction! 

Now, we label the first point in $\{i,j\}$ as $m_k^1$ and the second point as $m_k^2$. 
By repeating the process described 
above for $k=1,2,\cdots$,
we generate 
two sub-sequences 
$\{r_2^{m^1_k}\}_{k=1}^{\infty},
\{r_2^{m^2_k}\}_{k=1}^{\infty}\subseteq\{r_2^m\}_{m=1}^{\infty}$.

Let $r_2^{m_k^1}, r_2^{m_k^2}
\in {\cal R}^{k}$,
and $r_2^{m_{k+1}^1},
r_2^{m_{k+1}^2}
\in {\cal R}^{k+1}$. 
By properties stated in (a)--(d),
we have
\begin{equation*}
m^1_{k+1}-m^1_{k}\leq 2M-1 
\quad \text{and} \quad
m^1_k<m^2_k<m^1_{k+1}.
\end{equation*}
This completes the proof of Statement~\ref{state:cong-rate}. 
\hfill $\Box$

By Statement~\ref{state:cong-rate}, 
$$
m^1_k<m^1_k+1\leq m^2_k<m^1_{k+1}.
$$
Thus by the results established in the proof of Part~(iii),
we have 
\begin{align}\label{eq-deduce-a}
D^{m^1_{k+1}}\Big(r_2^{m^1_{k+1}}\Big)
\overset{\text{by~\eqref{eq-mono-dm-2}}}{\leq}
D^{m^2_k}\Big(r_2^{m^2_k}\Big)
\overset{\text{by~\eqref{proof:eq-D-diff-1}~and~\eqref{proof:eq-r-convg}}}{\leq}
\delta+D^{m^2_k}\Big(r_2^{m^1_k}\Big)
&\overset{\text{by~\eqref{eq-mono-dm-1}}}{\leq}
\delta+D^{m^1_k+1}\Big(r_2^{m^1_k}\Big)\nonumber\\
&\overset{\text{by~\eqref{proof:eq-cut-half-a}}}{=}
\delta+\frac{1}{2}D^{m^1_k}\Big(r_2^{m^1_k}\Big).
\end{align}
By deducing from~\eqref{eq-deduce-a}, we have 
\begin{align*}
D^{m^1_{k+1}}(r_2^{m^1_{k+1}}) 
\leq\delta+\frac{1}{2}D^{m^1_{k}}\Big(r_2^{m^1_{k}}\Big)
&\leq\delta+\frac{1}{2}\left[\delta+\frac{1}{2}D^{m^1_{k-1}}\Big(r_2^{m^1_{k-1}}\Big)\right] \nonumber\\
&\leq\delta+\frac{1}{2}\delta+\cdots+\frac{1}{2^{k-1}}\delta+\frac{1}{2^k}D^{m^1_{1}}\Big(r_2^{m^1_{1}}\Big)\nonumber\\
&\leq \Big(2-\frac{1}{2^{k-1}}\Big)\delta+\frac{1}{2^k}, 
\end{align*}
for $k\in\mathbb{Z}_+$. 
By~\eqref{proof:eq-r-convg-2}, we can estimate that 
\begin{equation*}
m^1_{k+1}\leq (2k+1)\lceil L/\delta\rceil+k+1,\ \forall k\in\mathbb{Z}_+. 
\end{equation*}
Together with~\eqref{eq-mono-dm-2}, we obtain
\begin{align*}
D^{(2k+1)\lceil L/\delta\rceil+k+1}\Big(r_2^{(2k+1)\lceil L/\delta\rceil+k+1}\Big) 
\leq D^{m^1_{k+1}}(r_2^{m^1_{k+1}}) 
&\leq \left(2-\frac{1}{2^{k-1}}\right)\delta+\frac{1}{2^k}. 
\end{align*}

For \(\epsilon\geq1\), by~\eqref{eq-mono-dm-r2}, the conclusion follows from \(\sup\limits_{u\in\mathcal U_{m-1}}\|u-u^*\|_\infty\leq D^m(r_2^m)<1\leq\epsilon\). For \(0<\epsilon<1\), 
let $\delta=\frac{\epsilon}{4}$ and $k=\lceil1-\log_2(\epsilon)\rceil$. 
Let $m_0:=(2k+1)\lceil L/\delta\rceil+k+1=(3\lceil 4L/\epsilon\rceil+2)-(2\lceil 4L/\epsilon\rceil+1)\lfloor\log_2(\epsilon)\rfloor$.
Then for any 
$m\geq m_0$, 
we have 
\begin{align*}
D^{
    {m}
    }(r_2^
    {m}
    ) \leq
D^{m^1_{k+1}}\Big(r_2^{m^1_{k+1}}\Big) 
&\leq \left(2-\frac{1}{2^{k-1}}\right)\delta+\frac{1}{2^k}< 2\delta+\frac{1}{2^k}\leq 2\cdot\frac{\epsilon}{4}+\frac{1}{2^{1-\log_2(\epsilon)}}
=\epsilon.
\end{align*}

\vspace{0.1cm}
{
\underline{Part~(v)}.
Set 
\(\epsilon :=\frac{40L\log_2 m}{m}\) in Part~(iv)
for any fixed integer 
\(m\geq2\). 
We consider two cases. 
Case~1: $\epsilon\geq 1$,
and Case~2: $0<\epsilon< 1$.
Observe that by~\eqref{eq-mono-dm-r2},
$
\sup\limits_{u\in\mathcal U_{m-1}}\|u-u^*\|_\infty \leq D^m(
r_2^m)
< 1$ for any $m$.
In Case~1, we have 
 \(\epsilon\geq1\) and \(
 \sup\limits_{u\in\mathcal U_{m-1}}\|u-u^*\|_\infty
 < 1\leq \epsilon
\). In Case~2,
it follows by~\eqref{eq:thm1-iv}
that 
$
\sup\limits_{u\in\mathcal U_{m-1}}\|u-u^*\|_\infty
<\epsilon$
for all $m\geq m_0$.
Thus, it suffices to show that $m\geq m_0$ under the specified relationship between $\epsilon$ and $m$.
By substituting $\epsilon$
into the definition of $m_0$ (see Part~(iv) of the theorem), 
we have 
\[
\begin{aligned}
m_0& = \Big(3\lceil 4L/\epsilon\rceil+2\Big)-\Big(2\lceil4L/\epsilon\rceil+1\Big)\lfloor\log_2(\epsilon)\rfloor 
\\
& \leq\left[3\left(\frac{4L}{\epsilon}+1\right)+2\right]
-\left[2\left(\frac{4L}{\epsilon}+1\right)+1\right]
\left(\log_2\epsilon-1\right) \quad \text{(since $\lceil x\rceil \leq x+1$ and $\lfloor x\rfloor \geq x-1$)} \\
& =\left(\frac{20L}{\epsilon}+8\right)
+\left(\frac{8L}{\epsilon}+3\right)
\log_2(\frac{1}{\epsilon}) \\
& \leq \frac{28L}{\epsilon}
+\left(\frac{11L}{\epsilon}\right)
\log_2(\frac{1}{\epsilon}) \qquad \text{(since $1\leq \frac{L}{\epsilon}$ for $0<\epsilon<1, L\geq 1$)} \\
& \leq\frac{28L}{\epsilon}\log_2m +\left(\frac{11L}{\epsilon}\right)
\log_2(\frac{1}{\epsilon}) \quad \text{(since $\log_2 m\geq 1$ for $m\geq 2$)} \\
& \leq\frac{39L}{\epsilon}\log_2m \qquad \text{(since $\log_2 m\geq \log_2(\frac{1}{\epsilon})$ for 
$\frac1\epsilon=\frac{m}{40L\log_2m}\leq m$)} \\
& =\frac{39}{40}m <m. 
\end{aligned}
\]
Thus, \(
\sup\limits_{u\in\mathcal U_{m-1}}\|u-u^*\|_\infty
<\epsilon=\frac{40L\log_2m}{m}\) holds for all $m\geq 2$.
}
The proof is complete.
\end{proof}



\section{Solution of the optimization problem~\texorpdfstring{\eqref{eq:MUS-r2-initial}}{(\ref*{eq:MUS-r2-initial})}}\label{sec-bounds}

A key step in the MUS scheme is to solve 
problem~\eqref{eq:MUS-r2-initial} as follows: 
\begin{equation}\label{eq:mus-f(r_2)-initial}
 \mathop{\max}\limits_{r_2\in [0,1]} f(r_2):= \left[\max\limits_{u\in \mathcal{U}_{m}}u(r_2)-\min\limits_{u\in \mathcal{U}_{m}}u(r_2)\right]. 
\end{equation}
For each $r_2\in[0,1]$, 
\begin{equation}\label{eq:mus-upper-lower-def}
    \bar{f}(r_2)=\max\limits_{u\in \mathcal{U}_{m}}u(r_2)
    \quad\text{ and }\quad
    \underline{f}(r_2)=\min\limits_{u\in \mathcal{U}_{m}}u(r_2). 
\end{equation}
By the definition, the graphs of $\bar{f}$ and $ \underline{f}$ correspond to the upper and lower bounds of the graph of ${\cal U}_m$. 
Since every $u\in {\cal U}_m$ is concave, $\underline{f}$ is concave (Theorem~\ref{thm:mus-lower-bound}). 
However, $\bar{f}$ is not necessarily concave, and $f(\cdot)$ is generally non-concave.
Moreover, the inner maximization and minimization problems are non-parametric 
with respect to $u\in\mathcal{U}_m$.
Although $\bar{f}$ is not globally concave over $[0,1]$, we will show in this section that $\bar{f}$ is piecewise locally concave within $N_m-1$ sub-intervals (Theorem~\ref{thm:mus-upper-bound}) and this will enable us to derive a semi-closed form of $f(\cdot)$ 
(Theorem~\ref{thm:mus-semi-f}) by solving at most $O(N_m)$ sub-problems (see~\eqref{eq:calcu-max-alpha}--\eqref{eq:calcu-min-alpha}~and~\eqref{eq:calcu-max-beta}--\eqref{eq:calcu-min-beta}), 
{where $N_m$ is the number of supporting points of the first $m$ MUS queries.} 
Consequently, we will propose an 
efficient algorithm (Algorithm~\ref{alg:mus-2}) to solve the non-concave, non-parametric min-max optimization problem~\eqref{eq:mus-f(r_2)-initial}.

\subsection{Guidance on the main results and developments}\label{sec-MUS-guid}

To facilitate reading, we provide guidance on the main results and the developments leading to these results, before moving on to the detailed proofs. 
Below are the key setups leading to the main results/findings.

\begin{itemize}

\item 
Consider the MUS and the ambiguity set $\mathcal{U}_m$ defined as in~\eqref{eq:U_m-MUS}. 
Let 
$$
\mathcal{S}_m :=\{0\}\cup \bigcup\limits_{k=1}^{m}\{r_2^k\}\cup\{1\}
$$
{be the support set constituted by the first $m$ MUS queries,}
and $N_m:=|\mathcal{S}_m|\leq m+2$ 
{be the cardinality of $\mathcal{S}_m$.}
By sorting the elements of $\mathcal{S}_m$ in an increasing order, we use 
\begin{equation}
\label{eq:Y_m}
\mathbb{Y}_m:= \left\{{y}^m_{j}\right\}_{j=1,\ldots,N_m}
\end{equation}
to denote the set of points 
sorted 
{in}
increasing order such that for $k=1,\ldots,m$, 
$y^m_{j_k}=r_2^k$, for some $j_k\in\{1,2,\ldots,N_m\}$, see the horizontal coordinates in Figure~\ref{fig:mus-illustration}.
To ease the discussion, we call $\mathcal{S}_m$ and $\mathbb{Y}_m$ the set of elicited points.

\item 
Denote the maximum/minimum utility values attained by the utility functions in ${\cal U}_m$ at each $y^m_k$ in $\mathbb{Y}_m$:  
\begin{equation}\label{eq:mus-def-alpha-max}
\bar{\alpha}^m_{k}:=\max\limits_{u\in\mathcal{U}_m}u(y^m_{k}),\quad \underline{\alpha}^m_k:=\min\limits_{u\in\mathcal{U}_m}u(y^m_k),
\quad
k=1,\ldots,N_m. 
\end{equation}
The normalization condition of $u$ in $\mathcal{U}_m$ ensures that $\bar{\alpha}^m_1=0$, $\underline{\alpha}^m_1=0$, $\bar{\alpha}^m_{N_m}=1$ and $\underline{\alpha}^m_{N_m}=1$, 
see the red points and purple points in Figure~\ref{fig:mus-illustration}.

\item 
Define the subset of utility functions that attain the maximum value $\bar{\alpha}^m_k$ at $y^m_k$, i.e., 
\begin{equation}\label{eq:mus-def-U_m^k}
\mathcal{U}_m^k:=\Big\{u\in\mathcal{U}_m\mid u(y^m_k)=\bar{\alpha}^m_k\Big\}, \; k=1,\ldots,N_m.
\end{equation}
Define 
the maximum right derivative and the minimum left derivative of the utility functions in $\mathcal{U}^k_m$ at point $y^m_k$ 
and in $\mathcal{U}^{k+1}_m$ at point $y^m_{k+1}$, 
respectively:
\begin{equation}\label{eq:mus-def-beta-max-min}
\bar{\beta}^m_k:=\max\limits_{u\in\mathcal{U}^k_m} u^{'}_{+}(y^m_k),\quad 
\underline{\beta}^m_{k+1}:=\min\limits_{u\in\mathcal{U}^{k+1}_m} u^{'}_{-}(y^m_{k+1}), \quad k=1,\ldots,N_m-1. 
\end{equation}
{
Attainment of 
the maximum and minimum 
will be discussed in 
Proposition~\ref{prop:mus-max-min-beta}.
}


\end{itemize}

\begin{figure}[htbp]
    \centering
    \subfigure
[$\bar{\beta}^m_k=\underline{\beta}^m_{k+1}=\frac{\bar{\alpha}^m_{k+1}-\bar{\alpha}^m_k}{y^m_{k+1}-y^m_k}$]
    {
    \label{fig:one-piece}
    \includegraphics[width=0.45\linewidth]{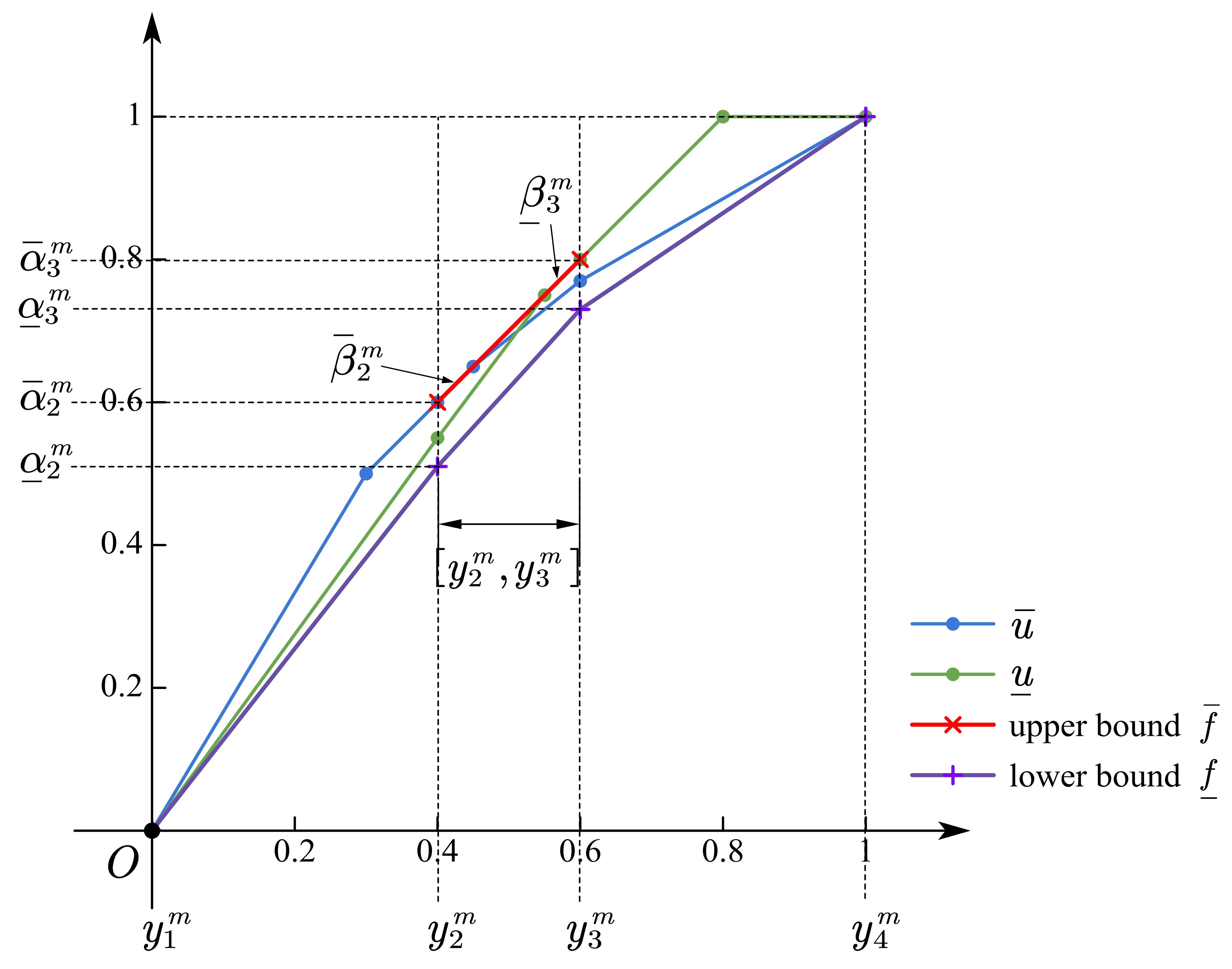}
    }
    \quad
    \subfigure
    [$\bar{\beta}^m_k>\frac{\bar{\alpha}^m_{k+1}-\bar{\alpha}^m_k}{y^m_{k+1}-y^m_k}>\underline{\beta}^m_{k+1}$]
    {
    \label{fig:two-pieces}
    \includegraphics[width=0.45\linewidth]{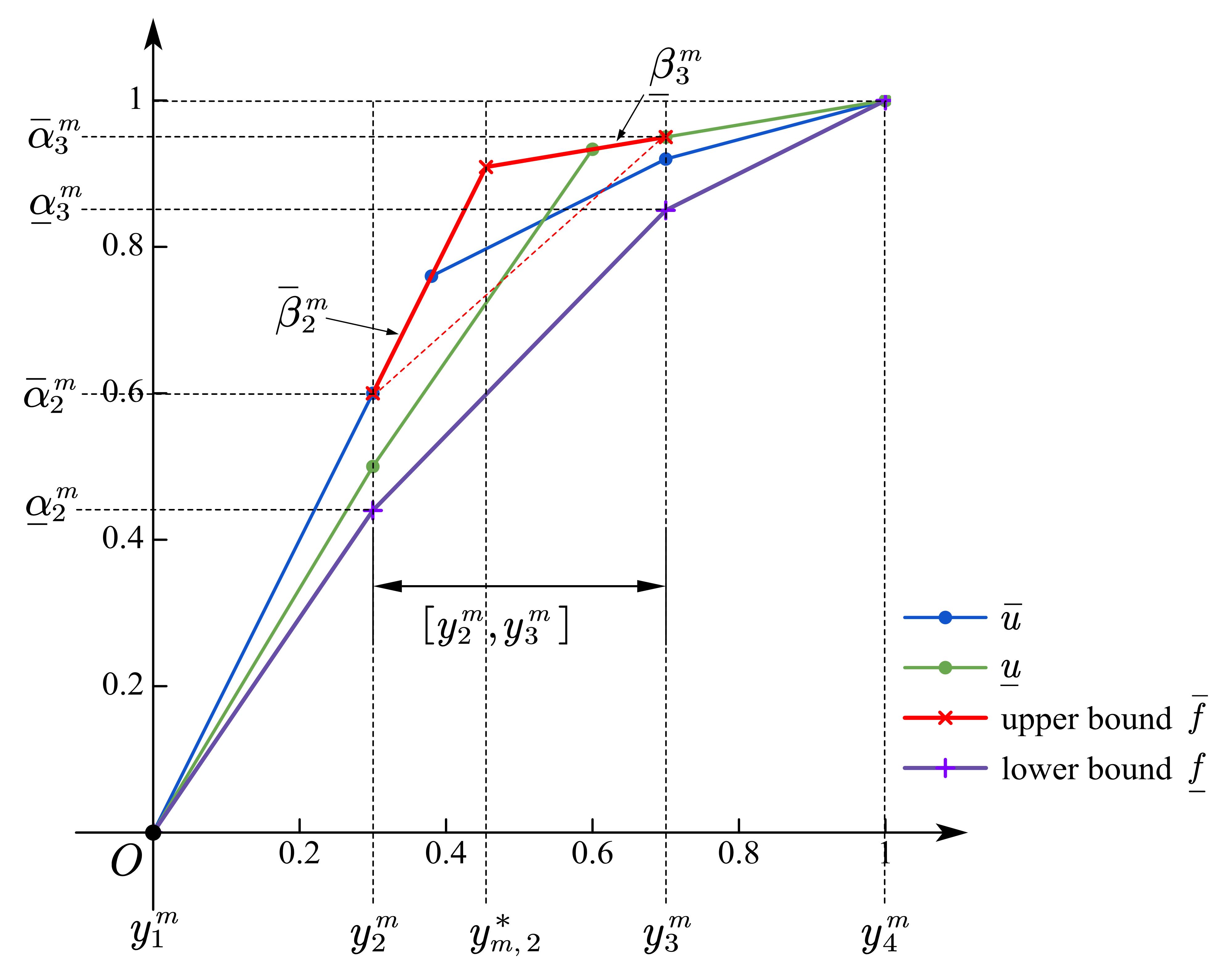}
    }
    \caption{Illustration of the graphs 
    of the lower bound function $\underline{f}$ and the upper bound function $\bar{f}$ over interval 
    $[y_2^m,y_3^m]$. $\bar{f}$ is either linear 
    or concave comprising two linear pieces over the interval depending on the relationship between 
    {$\bar{\beta}^m_2$ and $\underline{\beta}^m_{3}$}.
    The red segment is the graph of $\bar{f}$ restricted to interval $[y^m_2, y^m_3]$, which is constructed from $\bar{u}$ and $\underline{u}$.}
    \label{fig:mus-illustration}
\end{figure}

\noindent
\textbf{Main findings/results}

\begin{itemize}

\item By connecting the minimum utility values $\underline{\alpha}^m_k$ at elicited points $y_k^m$ for $k=1,\ldots,N_m$ (defined in~\eqref{eq:mus-def-alpha-max}), we can obtain a piecewise linear utility function that is contained in $\mathcal{U}_m$ and is defined as in~\eqref{eq:mus-min-utility},  which is exactly the lower bound $\underline{f}$ in~\eqref{eq:mus-upper-lower-def}. This result will be stated in Theorem~\ref{thm:mus-lower-bound}. The result is in alignment with the intuition that the minimum of all concave functions in the ambiguity set $\mathcal{U}_m$ is still concave. 
See the purple line in Figure~\ref{fig:one-piece} and Figure~\ref{fig:two-pieces}.

\item Unlike the lower bound function $\underline{f}$, there may not exist a concave utility function in $\mathcal{U}_m$ which coincides with $\bar{f}$ over $[0,1]$. 
However, Theorem~\ref{thm:mus-upper-bound} states that $\bar{f}$ is piecewise concave with one or two linear pieces within each interval  $\left[y^m_{k},y^m_{k+1}\right]$, $k=1,\ldots,N_m-1$. 
The specific structure within each interval depends on 
{
how $\bar{\beta}^m_k$ and $\underline{\beta}^m_{k+1}$ compare
with the slope of 
the line
segment connecting 
points $(y^m_k,\bar{\alpha}^m_k)$ and $(y^m_{k+1},\bar{\alpha}^m_{k+1})$. 
The key observation is Lemma~\ref{lemma:mus-beta-inequal}, which brackets this slope:
$\bar{\beta}^m_k\geq \frac{\bar{\alpha}^m_{k+1}-\bar{\alpha}^m_k}{y^m_{k+1}-y^m_k}\geq \underline{\beta}^m_{k+1}$. 
We then distinguish two cases, according to whether the two inequalities are 
strict or not.}
\begin{itemize}
\vspace{0.1cm}
\item 
{If $\bar{\beta}^m_k=\frac{\bar{\alpha}^m_{k+1}-\bar{\alpha}^m_k}{y^m_{k+1}-y^m_k}$ or $\underline{\beta}^m_{k+1}=\frac{\bar{\alpha}^m_{k+1}-\bar{\alpha}^m_k}{y^m_{k+1}-y^m_k}$,}
then by Lemma~\ref{lemma:mus-beta-equal}, 
the two linear pieces coincide in $\left[y^m_{k},y^m_{k+1}\right]$. 
Consequently, $\bar{f}$ over $\left[y^m_{k},y^m_{k+1}\right]$ can be represented by a segment connecting the maximum utility values $\bar{\alpha}_k^m$ and $\bar{\alpha}_{k+1}^m$ at $y^m_k$ and $y^m_{k+1}$, see the red segment in Figure~\ref{fig:one-piece}. 
\vspace{0.1cm}
\item 
If $\bar{\beta}^m_k>\frac{\bar{\alpha}^m_{k+1}-\bar{\alpha}^m_k}{y^m_{k+1}-y^m_k}>\underline{\beta}^m_{k+1}$, then 
the two linear pieces intersect at a point $y^*_{m,k}\in \left(y^m_{k},y^m_{k+1}\right)$ 
{as stated in Theorem~\ref{thm:mus-upper-bound}}. 
In this case, $\bar{f}$ exhibits piecewise linear concavity over $[y^m_{k},y^m_{k+1}]$, see the red segment in Figure~\ref{fig:two-pieces}. 

\end{itemize} 

\end{itemize}

To this end, we need an intermediate technical result on the property of $\mathcal{U}_m$, which relies on the particular structure of the lotteries as defined in~\eqref{eq:W-Y}. 
It plays an important role in the proofs of the main findings above.

\begin{lemma}\label{lemma:mus-u-U_m}

For any fixed $u\in\mathcal{U}_{cv}$, $u\in\mathcal{U}_m$ if and only if at each $y^m_{j}$, $j=1,\ldots,N_m$, there exist two utility functions $u^1_j,u^2_j\in\mathcal{U}_m$ (depending on $y^m_{j}$) such that 
\bgeqn 
\label{eq:u_1uu_2atr_2-lem4.1}
u^1_j(y^m_{j})\leq u(y^m_j)\leq u^2_j(y^m_{j}). 
\edeqn 
In a particular case, if there exists $\hat{u}\in {\cal U}_m$ such that $u(y_j^m)= \hat{u}(y_j^m)$ for $j=1,\cdots, N_m$, then $u\in {\cal U}_m$.
\end{lemma}
\begin{proof}
\uline{The if part.}
By definition (see~\eqref{eq:U_m-MUS}), 
it suffices to show that
\begin{equation}
\label{eq:ambiguity-set-mus-lemma-pf}
Z_k\cdot u(r_2^k)\geq Z_k\cdot p^k,\quad {\rm for}\;  k=1,\ldots,m.
\end{equation}
Let $\mathbb{Y}_m$ be defined as in~\eqref{eq:Y_m} and $y^m_{j_k}\in \mathbb{Y}_m$ be such that $r_2^k=y^m_{j_k}$ for some $j_k\in\{1,2,\ldots,N_m\}$. 
\uline{Case~1:} $Z_k=1$.
Since $u^1_{j_k}\in\mathcal{U}_m$, $u^1_{j_k}(y^m_{j_k})=u^1_{j_k}(r_2^k)\geq p^k$, for $k=1,\ldots,m$. 
By~\eqref{eq:u_1uu_2atr_2-lem4.1}, 
i.e., $u^1_{j_k}(y^m_{{j_k}})\leq u(y^m_{j_k})\leq u^2_{j_k}(y^m_{{j_k}})$, we have 
\bgeqn 
\label{eq:ambiguity-set-mus-lemma-pf-c}
u(r_2^k)=u(y^m_{j_k})\geq u^1_{j_k}(y^m_{j_k})=u^1_{j_k}(r_2^k)\geq p^k. 
\edeqn 
\uline{Case~2:} $Z_k=-1$.
Since $u^2_{j_k}\in\mathcal{U}_m$, 
$u^2_{j_k}(y^m_{j_k})=u^2_{j_k}(r_2^k) \leq p^k$ for $k=1,\ldots,m$. 
By~\eqref{eq:u_1uu_2atr_2-lem4.1}, i.e., $u^1_{j_k}(y^m_{{j_k}})\leq u(y^m_{j_k})\leq u^2_{j_k}(y^m_{{j_k}})$, we have 
\bgeqn 
\label{eq:ambiguity-set-mus-lemma-pf-d}
u(r_2^k)=u(y^m_{j_k})\leq u^2_{j_k}(y^m_{j_k})=u^2_{j_k}(r_2^k) \leq p^k.
\edeqn 
A combination of~\eqref{eq:ambiguity-set-mus-lemma-pf-c}~and~\eqref{eq:ambiguity-set-mus-lemma-pf-d}
gives rise to~\eqref{eq:ambiguity-set-mus-lemma-pf}.

\uline{The only if part} can be observed by taking $u^1_j=u^2_j=u$, for all $j=1,\ldots,N_m$. 
\end{proof}

Let 
\bgeq 
{\mathcal{U}_{m}(y^m_j)}
:= [\min\limits_{u\in {\cal U}_m} u(y^m_{j}),\  \max\limits_{u\in {\cal U}_m} u(y^m_{j})]
\quad \text{and}\quad 
{\cal U}_j^m :=
\{u\in{\cal U}_{cv}: u(y^m_j)\in 
{\mathcal{U}_{m}(y^m_j)}
\}.
\edeq 
Then ${\cal U}_m = \cap_{j} {\cal U}_j^m$.
We may view $u\in {\cal U}_j^m$ as a ``feasibility check'' of whether $u\in {\cal U}_m$ at point $y_j^m$.
Let
$I_{m,j}:=[u^1_j(y^m_{j}),\ u^2_j(y^m_{j})]$. 
Since $I_{m,j}\subseteq 
{\mathcal{U}_{m}(y^m_j)}
$, then
$
\{u\in{\cal U}_{cv}: u(y^m_j)\in I_{m,j}\}\subseteq {\cal U}_j^m. 
$
The lemma states that $u\in {\cal U}_m$ if and only if 
it passes the ``feasibility check'' at each $y^m_{j}$, $j=1,\ldots,N_m$.



\subsection{Lower bound function \texorpdfstring{$\underline{f}$}{f (lower)}}
\label{sec:lower-bound-f}

We begin by deriving a semi-closed form of $\underline{f}$, corresponding to the first bullet point of the main findings.
The next theorem states this.

\begin{theorem}[Semi-closed form of $\underline{f}$]
\label{thm:mus-lower-bound}

Consider a piecewise linear function $\underline{u}_{N_m}$ whose graph is the connection of two consecutive points $(y^m_j,\underline{\alpha}^m_j)$ and $(y^m_{j+1},\underline{\alpha}^m_{j+1})$ with a line segment for $j=1,\ldots,N_m-1$, i.e., 
\begin{equation}\label{eq:mus-min-utility}
    \underline{u}_{N_m}(y)
    =\left\{ 
    \begin{aligned}
    &0 && {\rm for} \quad y=0,\\
    &\frac{\underline{\alpha}^m_{j+1}-\underline{\alpha}^m_{j}}{y^m_{j+1}-y^m_{j}}(y-y^m_{j})+\underline{\alpha}^m_{j} && {\rm for} \quad y^m_{j} < y \leq y^m_{j+1},\ j=1,\ldots,N_m-1,\\
    & 1 && {\rm for} \quad y=1.
    \end{aligned}
    \right.
\end{equation}
Then $\underline{u}_{N_m}\in \mathcal{U}_m$ and $\underline{f}=\underline{u}_{N_m}$, i.e., $\underline{f}(r_2)=\underline{u}_{N_m}(r_2)$ for all $r_2\in[0,1]$. 
\end{theorem}
\begin{proof}
We proceed with the proof in two steps:
first $\underline{u}_{N_m}\in\mathcal{U}_m$, and then $\underline{f}=\underline{u}_{N_m}$. 

\vspace{0.1cm}
\noindent
\uline{Step~1.}
Since each $u\in {\cal U}_m$ is a concave utility function and the min operation preserves the concavity, $\underline{f}$ is concave over $[0,1]$. 
The concavity of $\underline{f}$ ensures concavity of $\underline{u}_{N_m}$ because the two functions coincide at breakpoints $y^m_j$. 
The min operation also preserves the monotonicity of $u$, which means that
$\underline{\alpha}^m_{j}\leq \underline{\alpha}^m_{j+1}$. This, in turn, ensures monotonicity of $\underline{u}_{N_m}$. 
Moreover, since 
\begin{align*}
|\underline{\alpha}^m_{k+1}-\underline{\alpha}^m_k|
=
\left|\min\limits_{u\in\mathcal{U}_m}u(y^m_{k+1})-
\min\limits_{u\in\mathcal{U}_m}u(y^m_k)\right|
\leq
\max\limits_{u\in\mathcal{U}_m}\left|u(y^m_{k+1})-
u(y^m_k)\right| \leq L|y^m_{k+1}-y^m_{k}|,
\end{align*}
$\underline{u}_{N_m}$ is Lipschitz continuous with modulus $L$. 
This shows $\underline{u}_{N_m}\in\mathcal{U}_{cv}$. 
Furthermore, let $u^1_j,u^2_j\in\mathop{\arg\min}\limits_{u\in\mathcal{U}_m}u(y^m_j)$. 
Then $u^1_j(y^m_{j})=\underline{u}_{N_m}(y^m_{j})= u^2_j(y^m_{j})$, for $j=1,\ldots,N_m$.  
By Lemma~\ref{lemma:mus-u-U_m}, $\underline{u}_{N_m}\in\mathcal{U}_m$.

\vspace{0.1cm}
\noindent
\uline{Step~2.}
Since $\underline{u}_{N_m}\in\mathcal{U}_m$, we have $\underline{f}(r_2)=\min\limits_{u\in\mathcal{U}_m}u(r_2)\leq \underline{u}_{N_m}(r_2)$, $\forall r_2\in[0,1]$. 
Thus it suffices to show that $\underline{u}_{N_m}(r_2)\leq \underline{f}(r_2)$.
For any $r_2\in[y^m_j, y^m_{j+1}]$, let $\lambda'\in[0,1]$ be such that $r_2=\lambda' y^m_j+(1-\lambda')y^m_{j+1}$. 
By the concavity of $\underline{f}$, 
\begin{align*}
\underline{f}(r_2)=\underline{f}(\lambda' y^m_j+(1-\lambda')y^m_{j+1})
&\geq \lambda'\underline{f}(y^m_j)+(1-\lambda')\underline{f}(y^m_{j+1})\\
&=\lambda' \underline{u}_{N_m}(y^m_j)+(1-\lambda')\underline{u}_{N_m}(y^m_{j+1})\\
&=\underline{u}_{N_m}(\lambda' y^m_j+(1-\lambda')y^m_{j+1})=\underline{u}_{N_m}(r_2). 
\end{align*}
The inequality holds for $j=1,\ldots,N_m-1$.
\end{proof}

From~\eqref{eq:mus-min-utility}, we can see $\underline{f}$ is concave and piecewise linear. We say it is a semi-closed form because the values of $\underline{\alpha}^m_k$ are implicitly defined. 


\subsection{Upper bound function \texorpdfstring{$\bar{f}$}{f (upper)}}\label{sec:upper-bound-f}

The specification of $\bar{f}$ relies on determining the maximum utility values $\bar{\alpha}^m_k$ in~\eqref{eq:mus-def-alpha-max}, and the maximum right derivative $\bar{\beta}^m_k$ and minimum left derivative $\underline{\beta}^m_{k+1}$ in~\eqref{eq:mus-def-beta-max-min}. 
There are inherent relationships among $\bar{\alpha}^m_k$, $\bar{\alpha}^m_{k+1}$, $\bar{\beta}^m_k$, and $\underline{\beta}^m_{k+1}$ for $k=1,\ldots,N_m-1$, which arise from the monotonicity, concavity, and Lipschitz continuity of utility functions in $\mathcal{U}_m$. 
Lemma~\ref{lemma:mus-beta-inequal} establishes an important inequality that bounds the two derivatives from below and above based on the 
connection of adjacent points with maximum utility values. 
Lemma~\ref{lemma:mus-beta-equal} further characterizes the case when equality holds.

\begin{lemma}
\label{lemma:mus-beta-inequal}

Let $\bar{\beta}^m_k$ and $\underline{\beta}^m_{k+1}$ be defined as in~\eqref{eq:mus-def-beta-max-min}.
Then 
$$
\bar{\beta}^m_k
\geq\frac{\bar{\alpha}^m_{k+1}-\bar{\alpha}^m_{k}}{y^m_{k+1}-y^m_k}\geq \underline{\beta}^m_{k+1}
\quad \text{\rm for}\quad  k=1,\ldots,N_m-1.
$$
\end{lemma}

\begin{lemma}\label{lemma:mus-beta-equal}
For $k=1,\ldots,N_m-1$, 
let $\bar{\beta}^m_k$ and $\underline{\beta}^m_{k+1}$ be defined as in~\eqref{eq:mus-def-beta-max-min}. Then
\begin{equation*}
\bar{\beta}^m_k
=\frac{\bar{\alpha}^m_{k+1}-\bar{\alpha}^m_{k}}{y^m_{k+1}-y^m_k} 
\quad 
\Longleftrightarrow
\quad
\underline{\beta}^m_{k+1}
=\frac{\bar{\alpha}^m_{k+1}-\bar{\alpha}^m_{k}}{y^m_{k+1}-y^m_k}. 
\end{equation*}
\end{lemma}


{With the two lemmas,}
we are
ready to derive a semi-closed form of $\bar{f}$.
The key point is that, for a fixed interval $\left[y^m_{k},y^m_{k+1}\right]$, $k=1,\ldots,N_m-1$, we use the relationships among $\bar{\alpha}^m_k$, $\bar{\alpha}^m_{k+1}$, $\bar{\beta}^m_k$, and $\underline{\beta}^m_{k+1}$ in Lemmas~\ref{lemma:mus-beta-inequal}~and~\ref{lemma:mus-beta-equal}, to construct a piecewise linear utility function $\hat{u}\in\mathcal{U}_m$, which coincides with $\bar{f}$ in this fixed interval, but not necessarily in the other parts of $[0,1]$ ($\hat{u}$ and $\bar{f}$ may differ over the other areas of the interval $[0,1]$). 
By concatenating all the $\hat{u}$ with one or two pieces within each interval, we derive the upper bound function $\bar{f}$ with at most $2(N_m-1)$ pieces.

\begin{theorem}[Semi-closed form of $\bar{f}$]
\label{thm:mus-upper-bound}
Let $\bar{f}(r_2)$ be defined as in~\eqref{eq:mus-upper-lower-def}. 
Then $\bar{f}(r_2)$ is piecewise linear and 
concave over $[y^m_k,y^m_{k+1}]$ with the following structure for $k=1,\ldots,N_m-1:$

\begin{itemize}
    
\item[(i)] if  $\bar{\beta}^m_k=\underline{\beta}^m_{k+1}$, then  
\bgeqn 
\label{eq:upper-f-one-piece}
\bar{f}(y)=
\bar{\alpha}^m_k+\frac{\bar{\alpha}^m_{k+1}-\bar{\alpha}^m_k}{y^m_{k+1}-y^m_{k}}(y-y^m_k)
\quad
\text{\rm for}\quad  y_k^m\leq y \leq y^m_{k+1}; 
\edeqn

\item[(ii)] if $\bar{\beta}^m_k>\underline{\beta}^m_{k+1}$, then 
\begin{equation}
\label{eq:upper-f-two-piece}
\bar{f}(y)=\left\{ 
\begin{aligned}
& \bar{\alpha}^m_k+\bar{\beta}^m_k(y-y^m_k) 
&& \text{\rm for}\quad  y_k^m\leq y \leq y^*_{m,k},\\
&\bar{\alpha}^m_{k+1}+\underline{\beta}^m_{k+1}(y-y_{k+1}^m) 
&& \text{\rm for}\quad y^*_{m,k}< y \leq y_{k+1}^m, 
\end{aligned}
\right.
\end{equation}
where
\bgeqn 
\label{eq:y^*_mk}
y^*_{m,k}=\frac{\bar{\alpha}^m_{k+1}-\bar{\alpha}^m_k-\underline{\beta}^m_{k+1}y^m_{k+1}+\bar{\beta}^m_ky^m_k}{\bar{\beta}^m_k-\underline{\beta}^m_{k+1}}\in (y^m_k,y^m_{k+1}).
\edeqn 

\end{itemize}

\end{theorem}


From~\eqref{eq:upper-f-one-piece} and~\eqref{eq:upper-f-two-piece}, we can see that $\bar{f}$ is piecewise linear and piecewise concave over the interval $[0,1]$. 
However, it is not necessarily 
globally 
concave and thus does not necessarily belong to $\mathcal{U}_m$. 
We say it is a semi-closed form, as the values of $\bar{\alpha}^m_k$, 
$\bar{\beta}^m_k$, and 
$\underline{\beta}^m_{k+1}$  
are all implicitly defined.

The proof of
Theorem~\ref{thm:mus-upper-bound} 
involves 
Lemmas~\ref{lemma:mus-beta-inequal}~and~\ref{lemma:mus-beta-equal}, 
whereas the proofs of the latter
require 
{an}
additional intermediate technical 
{result}
{established in Proposition~\ref{prop:mus-max-min-beta} below}. 
To improve the readability, we 
put 
all of them in 
a separate section, Section~\ref{sec-ec-proof-lemma2}.

\subsection{Semi-closed form of \texorpdfstring{$f$}{f}}

With the semi-closed forms of $\bar{f}$ and $\underline{f}$ in the preceding sections, we are ready to state the semi-closed form of $f$ which follows from a combination of Theorems~\ref{thm:mus-lower-bound}~and~\ref{thm:mus-upper-bound}.

\begin{theorem}\label{thm:mus-semi-f}

Let $f(r_2)$ be defined as in~\eqref{eq:mus-f(r_2)-initial}. 
Then $f(r_2)$ is piecewise linear and piecewise concave with at most two linear pieces in each interval $\left[y^m_{k},y^m_{k+1}\right]$, for $k=1,\ldots,N_m-1$. 
Moreover, 
\begin{equation}\label{eq:mus-semi-closed-form-f}
    f(r_2)=\left\{ 
    \begin{aligned}
    &\bar{\beta}^m_k(r_2-y_k^m)+\bar{\alpha}^m_k-\left[\frac{\underline{\alpha}^m_{k+1}-\underline{\alpha}^m_{k}}{y^m_{k+1}-y^m_{k}}(r_2-y^m_{k})+\underline{\alpha}^m_{k}\right] 
    && \text{\rm for }\quad r_2\in [y_k^m,
    y^*_{m,k}],\\
    &\underline{\beta}^m_{k+1}(r_2-y_{k+1}^m)+\bar{\alpha}^m_{k+1}-\left[\frac{\underline{\alpha}^m_{k+1}-\underline{\alpha}^m_{k}}{y^m_{k+1}-y^m_{k}}(r_2-y^m_{k})+\underline{\alpha}^m_{k} \right]
    && \text{\rm for }\quad  
    r_2\in (y^*_{m,k},
    y_{k+1}^m],
    \end{aligned}
    \right.
\end{equation}
and $k=1,\ldots,N_m-1$, where $y^*_{m,k}$ is defined by~\eqref{eq:y^*_mk} if $\bar{\beta}^m_k>\underline{\beta}^m_{k+1}$, and is set trivially to $y^m_k$ if $\bar{\beta}^m_k=\underline{\beta}^m_{k+1}$. 

\end{theorem}

\begin{remark}\label{remark:thm4}

By Theorem~\ref{thm:mus-semi-f}, we conclude that the optimal solution of the maximization problem 
\bgeq
\mathop{\max}\limits_{r_2\in [y^m_k, y^m_{k+1}]} f(r_2)=\mathop{\max}\limits_{r_2\in [y^m_k, y^m_{k+1}]}\left[\max\limits_{u\in \mathcal{U}_{m}}u(r_2)-\min\limits_{u\in \mathcal{U}_{m}}u(r_2)\right]
\edeq 
lies in the finite set $\{y^m_k,y^*_{m,k},y^m_{k+1}\}$ due to the piecewise linear structure of $f(r_2)$ in~\eqref{eq:mus-semi-closed-form-f}.  
Thus, in order to find an optimal solution to problem~\eqref{eq:mus-f(r_2)-initial}, it suffices to compute $f(\cdot)$ at $y^m_k, y^*_{m,k}, y^m_{k+1}$ for $k=1,\ldots,N_m-1$, and then identify the one that yields the maximum value. 
We will develop the computational procedures in Section~\ref{sec-3-6} based on this observation. 
\end{remark}


\subsection{Computing \texorpdfstring{$\bar{\alpha}^m_{k}$, $\underline{\alpha}^m_k$, $\bar{\beta}^m_k$ and $\underline{\beta}^m_{k+1}$}{alpha and beta}}
\label{sec:semi-closed-form-f}

In the preceding sections, we have derived a semi-closed form of the non-concave objective function of problem~\eqref{eq:mus-f(r_2)-initial}, which significantly facilitates the calculation of the optimal value and the optimal solution of problem~\eqref{eq:mus-f(r_2)-initial}. 
However, we have yet to address how to compute $\bar{\alpha}^m_{k}$, $\underline{\alpha}^m_k$, $\bar{\beta}^m_k$, and $\underline{\beta}^m_{k+1}$, defined in~\eqref{eq:mus-def-alpha-max}~and~\eqref{eq:mus-def-beta-max-min}. Each of these tasks corresponds to an infinite-dimensional functional optimization problem. 
In this section, we discuss these issues. 
By~\eqref{eq:N-piecewise-utility} and Proposition~\ref{prop:mus-max-min-beta}, we prove 
that the maximum (or minimum) utility values and the maximum right (or minimum left) derivatives at the elicited points $y_j^m$, $j=1,\ldots,N_m$, can be attained by some piecewise linear utility functions.
Before presenting the main results, we first introduce the concept of piecewise-linear lower approximation utility functions with pre-given breakpoints.

\begin{definition}[Piecewise-Linear Lower Approximation (PLA)]\label{def:PLU}

Given the ambiguity set $\mathcal{U}_m$ defined in~\eqref{eq:U_m-MUS} and the set of elicited points $\mathbb{Y}_m=\{y^m_j\}_{j=1,\ldots,N_m}$ defined in~\eqref{eq:Y_m}, 
define a piecewise-linear lower approximation (PLA) operator $\underline{\mathcal{A}}: \mathcal{L}^1([0,1])\to\mathcal{L}^1([0,1])$ such that for each $u\in\mathcal{U}_m$, 
\begin{equation*}
    \underline{\mathcal{A}}u(y)=\left\{ 
    \begin{aligned}
    &0 && \text{\rm for}\quad  y=y^m_1=0,\\
    &\frac{u(y^m_{j+1})-u(y^m_j)}{y^m_{j+1}-y^m_{j}}(y-y^m_{j})+u(y^m_j) && \text{\rm for}\quad  y^m_{j} < y \leq y^m_{j+1},\ j=1,\ldots,N_m-1,\\
    & 1 &&\text{\rm for}\quad  y=y^m_{N_m}=1.
    \end{aligned}
    \right.
\end{equation*}
Let ${\cal U}_m^{\underline{\mathcal{A}}} =\{\underline{\mathcal{A}}u: u\in {\cal U}_m\}$ denote the set of all piecewise-linear lower approximation utility functions defined in this way. 

\end{definition}

By the definition, $\underline{\mathcal{A}} u$
coincides with $u$ over 
$\mathbb{Y}_m$, i.e., $\underline{\mathcal{A}} u(y^m_j) = u(y^m_j)$, $j=1,\ldots,N_m$. 
Thus, we may view $\underline{\mathcal{A}} u$ as a {\em piecewise linear approximation (PLA)} of $u$ and $\mathcal{U}_m^{\underline{\mathcal{A}}}$ as an approximation of ${\cal U}_m$. 
In~\cite{guo2024utility}, the authors quantify the errors of PLA by virtue of Hoffman's lemma. 
In this paper, we do not require such quantification as we only use $\underline{\mathcal{A}} u$ to calculate maximum/minimum utility values of the utility functions in ${\cal U}_m$ and the extreme derivatives at the specified breakpoints.
Note that $\underline{\mathcal{A}} u$ preserves monotonicity, concavity, and Lipschitz continuity of $u$. 
If we select $u^1_j=u^2_j=u$ for $j = 1,\dots,N_m$, 
then it follows from Lemma~\ref{lemma:mus-u-U_m} that $\underline{\mathcal{A}} u\in\mathcal{U}_m$. 
Furthermore, we have the fact that 
\begin{align*}
{u}(y)
&\geq(1-\frac{y-y_{j}^m}{y_{j+1}^m-y_{j}^m}) {u}(y_{j}^m) + \frac{y-y_{j}^m}{y_{j+1}^m-y_{j}^m} {u}(y_{j+1}^m)\nonumber\\
&=\frac{{u}(y_{j+1}^m)-{u}(y_{j}^m)}{y_{j+1}^m-y_{j}^m} (y-y^m_j)+{u}(y_{j}^m) = \underline{\mathcal{A}} u(y), 
\; \forall y\in[y^m_j,y^m_{j+1}],\;\; j=1,\ldots,N_m-1.
\end{align*}
The inequality means that $\underline{\mathcal{A}} u(y)$ provides a lower bound for $u(y)$ over $[0,1]$.

For any $u\in\mathcal{U}_m$ and its PLA $\underline{\mathcal{A}} u\in\mathcal{U}_m^{\underline{\mathcal{A}}}$, denote by $\alpha^m
{(u)}
:=\left[\alpha^m_1{(u)},\ldots,\alpha^m_{N_m}{(u)}\right]\in\mathbb{R}^{N_m}$ the utility values at the 
elicited points
in $\mathbb{Y}_m$, and let $\beta^m{(u)}:=\left[\beta^m_1{(u)},\ldots,\beta^m_{N_m-1}{(u)}\right]\in\mathbb{R}^{N_m-1}$, 
where $\alpha^m_j{(u)}:=\underline{\mathcal{A}}u({y}^m_j)=u(y^m_j),\ j=1,\ldots,N_m$, and 
\bgeqn 
\label{eq:beta_j^m}
\beta^m_j{(u)}:=
\frac{\alpha^m_{j+1}{(u)}-\alpha^m_j{(u)}}{{y}^m_{j+1}-{y}^m_j}=
{\frac{u({y}^m_{j+1})-u({y}^m_{j})}{{y}^m_{j+1}-{y}^m_j}}
,\; j=1,\ldots,N_m-1.
\edeqn 
Then any function $\underline{\mathcal{A}}u$ in $\mathcal{U}_m^{\underline{\mathcal{A}}}$ can be structured as 
\begin{equation}\label{eq:N-piecewise-utility}
    \underline{\mathcal{A}}u(y)
    =\left\{ 
    \begin{aligned}
    &0 && \text{for}\quad y= {y}^m_{1}=0,\\
    &\beta   a^m_j{(u)}(y-{y}^m_j)+\alpha^m_j{(u)} && \text{for}\quad {y}^m_j< y \leq{y}^m_{j+1},\ j=1,\ldots,N_m-1,\\
    & 1 && \text{for}\quad y={y}^m_{N_m}=1.\\
    \end{aligned}
    \right.
\end{equation}

By the definition, 
\begin{equation}\label{eq:prop-calcu-f(r2)}
\max\limits_{u\in\mathcal{U}_m}u(y^m_j)=\max\limits_{u\in\mathcal{U}_m^{\underline{\mathcal{A}}}}u(y^m_j),\quad
\min\limits_{u\in\mathcal{U}_m}u(y^m_j)=\min\limits_{u\in\mathcal{U}_m^{\underline{\mathcal{A}}}}u(y^m_j),\;
\quad\text{for}\; j=1,\ldots,N_m. 
\end{equation}


\subsubsection{Computing \texorpdfstring{$\bar{\alpha}^m_{k}$ and $\underline{\alpha}^m_k$}{alpha}}
\label{subsec:semi-closed-form-f-alpha}

We can use the relationships in~\eqref{eq:prop-calcu-f(r2)} to develop tractable reformulations for computing $\bar{\alpha}^m_{k}$ and $\underline{\alpha}^m_k$. Specifically, we can recast $\max\limits_{u\in\mathcal{U}_m^{\underline{\mathcal{A}}}}u(y^m_j)$ as a linear program: 
\begin{subequations}\label{eq:calcu-max-alpha}
\begin{align}
\bar{f}(y^m_j)
=\max\limits_{\alpha^m,\beta^m}\ &\alpha^m_{j}\\
\text{s.t.}\ \ \  
\label{const:mus-pairwise}
& Z_k \cdot \sum\limits_{j=1}^{{N}_m}\mathbbm{1}_{({y}^m_j=r_2^k)}\alpha^m_{j}\geq Z_k \cdot p^k,\ k=1,\ldots,m,\\
\label{const:mus-beta-alpha}
&\beta^m_j=(\alpha^m_{j+1}-\alpha^m_j)/({y}^m_{j+1}-{y}^m_j),\ j=1,\ldots,{N}_m-1,\\
\label{const:mus-concave}
&\beta^m_{j+1} \leq \beta^m_j,\ j=1,\ldots,{N}_m-2, \\
\label{const:mus-lipschitz}
& 0\leq\beta^m_{j}\leq L,\ j=1,\ldots,{N}_m-1, \\
&
\alpha^m_1=0,\ \alpha^m_{{N}_m}=1,\ 
\alpha^m\in\mathbb{R}^{{N}_m},\ \beta^m\in\mathbb{R}^{{N}_m-1}_+.
\label{const:mus-domain}
\end{align}
\end{subequations}
Likewise, since $\underline{f}(y^m_j)=\min\limits_{u\in\mathcal{U}_m}u(y^m_j)=\min\limits_{u\in\mathcal{U}_m^{\underline{\mathcal{A}}}}u(y^m_j)$,  
we may calculate $\underline{f}(y^m_j)$ by solving a linear program: 
\begin{subequations}\label{eq:calcu-min-alpha}
\begin{align}
\underline{f}(y^m_j)
=\min\limits_{\alpha^m,\beta^m}\ &\alpha^m_{j}\\
\text{s.t.}\ \ \  & \eqref{const:mus-pairwise}-\eqref{const:mus-domain},
\end{align}
\end{subequations}
where~\eqref{const:mus-pairwise} requires the piecewise linear function constructed by $\alpha^m$ and $\beta^m$ to be consistent with the DM's pairwise comparison preferences, which are obtained from the response to questionnaire in 
{the}
form of~\eqref{eq:W-Y}; 
\eqref{const:mus-beta-alpha} reflects the relation between $\alpha^m$ and $\beta^m$; \eqref{const:mus-concave}--\eqref{const:mus-lipschitz} impose shape constraints on the piecewise linear function, including monotonicity, concavity, and Lipschitz continuity. 
\eqref{const:mus-domain} is the normalization constraint. That is, all constraints in~\eqref{eq:calcu-max-alpha}/\eqref{eq:calcu-min-alpha} ensure that the piecewise linear utility functions constructed by the optimal $\alpha^m$ and $\beta^m$ belong to the ambiguity set $\mathcal{U}_m^{\underline{\mathcal{A}}}$. 
We denote the optimal value of problem~\eqref{eq:calcu-max-alpha} as $\bar{\alpha}^m_j$, and the optimal value of problem~\eqref{eq:calcu-min-alpha} as $\underline{\alpha}^m_j$, $j=1, \ldots, N_m$.


\subsubsection{Computing \texorpdfstring{$\bar{\beta}^m_k$ and $\underline{\beta}^m_{k+1}$}{beta}}
\label{subsec:semi-closed-form-f-beta}

{
Next, we establish the attainment of 
the maximum right derivative $\bar{\beta}^m_k$ and the minimum left derivative $\underline{\beta}^m_{k+1}$ defined
in~\eqref{eq:mus-def-beta-max-min} and develop tractable formulations for computing them.
}
\begin{lemma}[Compactness of ${\cal U}_m^k$]
\label{lemma:ambiguity-set-compact}
Let $\mathcal{U}_m^k$, $k=1,\ldots,N_m$, be defined as in~\eqref{eq:mus-def-U_m^k}.
Then $\mathcal{U}_m^k$ is a compact set under the 
{supremum}
norm topology. 
\end{lemma}

\begin{proof}
By the compactness of $\mathcal{U}_m$ stated in Part~(i) of Theorem~\ref{thm:mus-converg}, and the continuity of $u(y_k^m)$ in $u$, there exists a $u\in\mathcal{U}_m$ such that $u(y^m_k)=\bar{\alpha}^m_k$. 
This implies that $\mathcal{U}^k_m$ is not empty.
Moreover, it is easy to verify that $\mathcal{U}^k_m$ is closed under the 
{supremum}
norm topology. 
{As a closed subset of the compact set $\mathcal{U}_m$,
$\mathcal{U}^k_m$ is compact.}
\end{proof}


{
\begin{proposition}[Attainment and computation of the extreme derivatives in~\eqref{eq:mus-def-beta-max-min}]\label{prop:mus-max-min-beta}
Let $\mathcal U_m^k$ be defined as in~\eqref{eq:mus-def-U_m^k}.
Denote $\beta_0^m(u):=L$ and $\beta_{N_m}^m(u):=0$ for every $u\in\mathcal U_m$ for notational convenience, and let $\bar\beta_k^m$ and $\underline\beta_{k+1}^m$ be defined as in~\eqref{eq:mus-def-beta-max-min}, i.e.,
\bgeqn 
\bar\beta_k^m:=\max\limits_{u\in\mathcal U_m^k}u'_+(y_k^m),\quad
\underline\beta_{k+1}^m:=\min\limits_{u\in\mathcal U_m^{k+1}}u'_-(y_{k+1}^m),\quad \text{\rm for} \quad k=1,\ldots,N_m-1.
\edeqn 
Then the following assertions hold.

\begin{itemize}

\item[(i)]    Both maximum and minimum are attained. 

\item[(ii)] For each \(k=1,\ldots,N_m-1\), $\bar{\beta}^m_k$ and $\underline{\beta}^m_{k+1}$ admit the following characterizations{\rm :} 
\begin{equation}\label{eq:relation-left-right-derivative-u-u_M^N}
\bar\beta_k^m=
\begin{cases}
\displaystyle\max\limits_{u\in\mathcal U_m^k}\beta_k^m(u),
&\text{\rm if }\; \beta_k^m(u)=\beta_{k+1}^m(u)\;\;\text{\rm for all }u\in\mathcal U_m^k,\\[2mm]
\displaystyle\max\limits_{u\in\mathcal U_m^k}\beta_{k-1}^m(u),
&\text{\rm otherwise},
\end{cases}
\end{equation}
\begin{equation}\label{eq:relation-left-right-derivative-u-u_M^N-LR-end}
\underline\beta_{k+1}^m=
\begin{cases}
\displaystyle\min\limits_{u\in\mathcal U_m^{k+1}}\beta_k^m(u),
&\text{\rm if }\; \beta_{k-1}^m(u)=\beta_k^m(u)\;\;\text{\rm for all }u\in\mathcal U_m^{k+1},\\[2mm]
\displaystyle\min\limits_{u\in\mathcal U_m^{k+1}}\beta_{k+1}^m(u),
&\text{\rm otherwise}.
\end{cases}
\end{equation}

\item[(iii)] In particular, if $L>1$, then
\begin{equation}
\label{eq:prop5-endpoint}
\bar\beta_1^m=L
\quad \text{\rm and}\quad
\underline\beta_{N_m}^m=0.
\end{equation}
\end{itemize}

\end{proposition}

\begin{proof}

\uline{Parts~(i)~and~(ii).}
We prove only the attainability and characterization of 
$\bar\beta_k^m$ 
in~\eqref{eq:relation-left-right-derivative-u-u_M^N}, since the corresponding results for the minimum left derivative follow analogously. 
Fix $k\in\{1,\ldots,N_m-1\}$ and let $\beta_j^m(u)$ be defined as in~\eqref{eq:beta_j^m}. 
The concavity, monotonicity, and Lipschitz continuity of $u\in\mathcal U_m^k$ imply
\[
L=\beta_0^m(u)\ge\beta_1^m(u)\ge\cdots\ge\beta_{N_m-1}^m(u)\ge\beta_{N_m}^m(u)=0 
\quad\text{and}\quad
u'_+(y_k^m)\le
\beta_{k-1}^m(u). 
\]
By the definition of $\beta_j^m(u)$, 
each
mapping \(u\mapsto\beta_j^m(u)\) is continuous on \(\mathcal U_m^k\) under the supremum norm
in that
\[
|\beta_j^m(u)-\beta_j^m(v)|
\le \frac{2\|u-v\|_\infty}{y_{j+1}^m-y_j^m},
\quad \forall \ u, v\in\mathcal U_m^k,
\quad \text{for}\;\; j=1,\ldots,N_m-1,
\]
while \(\beta_0^m(\cdot)\) and \(\beta_{N_m}^m(\cdot)\) are constant. 
We then distinguish two cases.

\vspace{0.1cm}
\noindent\underline{Case~1:
$\beta_k^m(u)=\beta_{k+1}^m(u)$ for every $u\in\mathcal U_m^k$.}
By concavity (when \(k\leq N_m-2\)), and by monotonicity (when \(k=N_m-1\)), \(u\) is affine on \([y_k^m,y_{k+1}^m]\).
Hence, 
\[
u'_+(y_k^m)=\beta_k^m(u)=\frac{u(y_{k+1}^m)-u(y_k^m)}{y_{k+1}^m-y_k^m},\quad \forall u\in\mathcal U_m^k.
\]
Thus the right-derivative objective $u'_+(y_k^m)$ in~\eqref{eq:mus-def-beta-max-min} coincides with
the 
continuous function $\beta_k^m(u)$. 
Moreover, by Lemma~\ref{lemma:ambiguity-set-compact}, $\mathcal U_m^k$ is 
compact. Thus, 
the maximum $\max\limits_{u\in\mathcal U_m^k}u'_+(y_k^m)$ is attained 
and 
$$
\max\limits_{u\in\mathcal U_m^k}u'_+(y_k^m)=\max\limits_{u\in\mathcal U_m^k}\beta_k^m(u)=\max\limits_{u\in\mathcal U_m^k}\frac{u(y_{k+1}^m)-u(y_k^m)}{y_{k+1}^m-y_k^m}, 
$$ 
which establishes the first case of~\eqref{eq:relation-left-right-derivative-u-u_M^N}.

\vspace{0.1cm}
\noindent\underline{Case 2: there exists $v\in\mathcal U_m^k$ such that $\beta_k^m(v)>\beta_{k+1}^m(v)$.}
By the compactness of $\mathcal U_m^k$ in Lemma~\ref{lemma:ambiguity-set-compact} and the continuity of $\beta_{k-1}^m(\cdot)$, there exists
\[
w\in\operatorname*{arg\,max}_{u\in\mathcal U_m^k}\beta_{k-1}^m(u).
\]

\noindent
It follows that 
\begin{equation}
\label{eq:beta_ineq}
\beta_{k-1}^m(w)\ge\beta_{k-1}^m(v)\ge\beta_k^m(v)>\beta_{k+1}^m(v)\geq0. 
\end{equation}
Define 
\begin{equation}
    \label{eq:wild_u}
\widehat u(y):=
\begin{cases}
\underline{\mathcal A}w(y),&
\text{for}\quad y\in[0, y_k^m),\\[1mm]
\displaystyle\min\!\left\{
\bar\alpha_k^m+\beta_{k-1}^m(w)(y-y_k^m),v(y_{k+1}^m)+\beta_{k+1}^m(v)(y-y_{k+1}^m)
\right\},&
\text{for} \quad y\in [y_k^m,y_{k+1}^m],\\[1mm]
\underline{\mathcal A}v(y),&
\text{for} \quad
y\in (y_{k+1}^m,1], 
\end{cases} 
\end{equation}
where $\underline{\mathcal A}w$ and $\underline{\mathcal A}v$ are the 
PLAs
of $w$ and $v$ 
over $[0,1]$, respectively, 
and both belong to $\mathcal{U}_m$.
Since $w,v\in\mathcal{U}_m^k$, then
$\underline{\mathcal A}w(y_k^m)=w(y_k^m)=v(y_k^m)=\underline{\mathcal A}v(y_k^m)=\bar\alpha_k^m$.  
Using this relationship and the inequalities
in~\eqref{eq:beta_ineq}, 
we can easily show that
$\widehat u(y_k^m)=\underline{\mathcal A}w(y_k^m)=\bar\alpha_k^m$ and $\widehat u(y_{k+1}^m)=v(y_{k+1}^m)=\underline{\mathcal A}v(y_{k+1}^m)$. 
This shows 
$\widehat u$ is continuous
at $y_k^m$ and $y_{k+1}^m$.
Moreover, 
the inequalities 
in~\eqref{eq:beta_ineq} imply that the two affine functions 
$\bar\alpha_k^m+\beta_{k-1}^m(w)(y-y_k^m),v(y_{k+1}^m)+\beta_{k+1}^m(v)(y-y_{k+1}^m)$
in~\eqref{eq:wild_u} 
intersect at a unique point \(z\in(y_k^m,y_{k+1}^m]\), so we can write
\[
\widehat u (y)=
\begin{cases}
\bar\alpha_k^m+\beta_{k-1}^m(w)(y-y_k^m),&
\text{for} \quad y\in [y^m_k,z],
\\
v(y_{k+1}^m)+\beta_{k+1}^m(v)(y-y_{k+1}^m),& 
\text{for} \quad 
y\in [z, y^m_{k+1}], 
\end{cases}
\]
with $\widehat u'_+(y^m_k)=\beta_{k-1}^m(w)$. 
See Figure~\ref{fig:compute-beta-fig} for an illustration.
We claim that $\widehat u \in\mathop{\arg\max}\limits_{u\in\mathcal U_m^k}u'_+(y_k^m)$. 

\begin{figure}[htbp]
    \centering
    \includegraphics[width=0.5\linewidth]{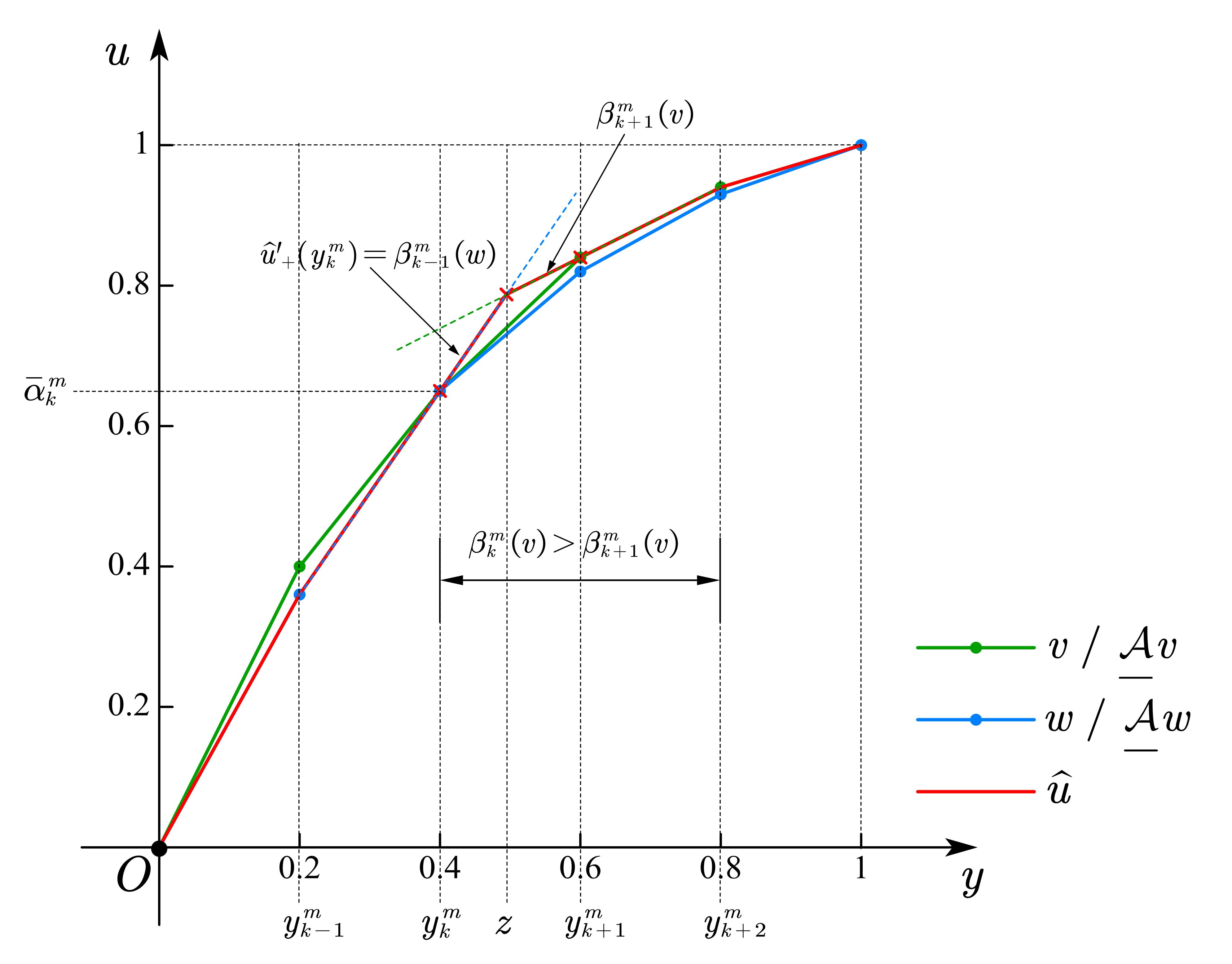}
    \caption{Geometric interpretation of the second case in~\eqref{eq:relation-left-right-derivative-u-u_M^N}.}
    \label{fig:compute-beta-fig}
\end{figure}

First, we show that $\widehat u\in\mathcal U_m^k$. 
By the relation $\beta_{k-1}^m(w)>\beta_{k+1}^m(v)\geq0$ in~\eqref{eq:beta_ineq}, 
\(\widehat u\) is concave and non-decreasing on \([y_k^m,y_{k+1}^m]\).  
Since $\widehat u'_+(y^m_k)=\beta_{k-1}^m(w)=(\underline{\mathcal A}w)'_-(y^m_k)=\widehat u'_-(y^m_k)$ 
and $\widehat u'_-(y^m_{k+1})\geq\beta_{k+1}^m(v)=(\underline{\mathcal A}v)'_+(y^m_{k+1})=\widehat u'_+(y^m_{k+1})$, 
together with the corresponding properties of 
$\underline{\mathcal A}w$ and $\underline{\mathcal A}v$ in $\mathcal{U}_m$, 
this implies that \(\widehat u\) is monotonically increasing, Lipschitz continuous, concave and normalized on $[0,1]$. 
This shows $\widehat{u}\in\mathcal{U}_{cv}$. 
Moreover, 
\[
\widehat u(y_j^m)=
\begin{cases}
w(y_j^m),&\text{for}\quad j=1,\ldots,k,\\
v(y_j^m),&\text{for}\quad j=k+1,\ldots,N_m.
\end{cases}
\]

\noindent
Let \(u_j^1=u_j^2:=w\in\mathcal{U}_m\) 
such that $u_j^1(y^m_j)=w(y^m_j)=\widehat u(y^m_j)=u_j^2(y^m_j)$
for \(j=1,\ldots,k\), and let \(u_j^1=u_j^2:=v\in\mathcal{U}_m\) 
such that $u_j^1(y^m_j)=v(y^m_j)=\widehat u(y^m_j)=u_j^2(y^m_j)$
for \(j=k+1,\ldots,N_m\). 
Then we obtain \(\widehat u\in\mathcal U_m\) by Lemma~\ref{lemma:mus-u-U_m}.
Together with $\widehat u(y_k^m)=\bar\alpha_k^m$, this yields $\widehat u\in\mathcal U_m^k$.

Second, 
by concavity of $u\in\mathcal U_m^k$ and 
the definition of $w$, 
\[
u'_+(y_k^m)\le\beta_{k-1}^m(u)\le\beta_{k-1}^m(w)=\widehat u'_+(y_k^m),
\quad \forall u\in\mathcal U_m^k.  
\]
Thus, \(\widehat u\) attains the maximum right derivative at $y^m_k$ within $\mathcal{U}_m^k$, and 
$$
\max\limits_{u\in\mathcal U_m^k}u'_+(y_k^m)=\widehat u'_+(y_k^m)=\beta_{k-1}^m(w)=\max\limits_{u\in\mathcal U_m^k}\beta_{k-1}^m(u). 
$$
This establishes the second case of~\eqref{eq:relation-left-right-derivative-u-u_M^N}. 

\vspace{0.1cm}
\underline{Part~(iii).}
The thrust of the proof is to demonstrate
that $\bar{\beta}_1^m$ and $\underline{\beta}_{N_m}^m$ are attained in
the second case in both~\eqref{eq:relation-left-right-derivative-u-u_M^N}~and~\eqref{eq:relation-left-right-derivative-u-u_M^N-LR-end}. 
We first
show by induction that $\mathcal{U}_m$ 
contains a strictly concave utility function for any finite $m\in\mathbb{Z}_+$. 
The initial ambiguity set \(\mathcal U_0=\mathcal{U}_{cv}\) contains the strictly concave utility function 
$v_0(y)=y+\eta y(1-y)$ for any fixed $\eta\in(0,\min\{1,L-1\})$. 
Suppose that \(v_{m-1}\in\mathcal U_{m-1}\) is strictly concave. 
At the new query point \(r_2^m\) selected by MUS, we have $\mathcal{U}_m=\mathcal{U}_{m-1}\bigcap\{u\in\mathcal{U}_{cv}: Z_m(u(r_2^m)-p^m)\geq0\}$, where $Z_m$ is the DM's answer to the $m$-th query. 
We claim that there exists \(w\in\mathcal U_{m-1}\) such that $Z_m(w(r_2^m)-p^m)>0$.
Assume without loss of generality that $Z_m=1$ and assume for the sake of a contradiction that such $w$ does not exist.
Then $u(r_2^m)\leq p^m$ for all $u\in\mathcal U_{m-1}$, which 
contradicts the fact that
\(p^m\) is the midpoint of 
the range $\{u(r_2^m): u\in\mathcal U_{m-1}\}$.
We now use $w$ and $v_{m-1}$
to construct a strictly concave function
$v_m\in {\cal U}_m$.
Let 
$v_m=(1-\lambda) w+ \lambda v_{m-1}$.
Then the convex structure of ${\cal U}_{m-1}$ ensures that $v_m\in {\cal U}_{m-1}$ for any $\lambda\in [0,1]$.
Moreover,
$
v_m(r_2^m)=(1-\lambda) w(r_2^m)+\lambda v_{m-1}(r_2^m)\geq p^m, 
$
for sufficiently small \(\lambda\in(0,1)\). 
Thus,
\(v_m\in\mathcal U_m\). 

We are now ready to show 
that $\bar{\beta}_1^m$ and $\underline{\beta}_{N_m}^m$ are attained in
the second case in both~\eqref{eq:relation-left-right-derivative-u-u_M^N}~and~\eqref{eq:relation-left-right-derivative-u-u_M^N-LR-end}, and subsequently~\eqref{eq:prop5-endpoint}. 
By the definition of $\mathcal{U}_m^k$ in~\eqref{eq:mus-def-U_m^k},
the normalization condition 
(\(u(0)=0\) and \(u(1)=1\))
implies 
\(\mathcal U_m^1=\mathcal U_m^{N_m}=\mathcal U_m\). 
Thus, 
\(v_m\in \mathcal U_m =\mathcal U_m^1=\mathcal U_m^{N_m}\).
Moreover, the strict concavity of $v_m$
implies that $\beta_1^m(v_m)>\beta_2^m(v_m)$ and $\beta_{N_m-2}^m(v_m)>\beta_{N_m-1}^m(v_m)$. 
By~\eqref{eq:relation-left-right-derivative-u-u_M^N}, the first inequality implies that 
$\bar\beta_1^m=\max\limits_{u\in\mathcal U_m}\beta_0^m(u)=L
$, where the last equality is due to our specification $\beta_0^m(u):=L$.
Likewise, it follows by the second inequality and~\eqref{eq:relation-left-right-derivative-u-u_M^N-LR-end} that
$\underline\beta_{N_m}^m=\min\limits_{u\in\mathcal U_m}\beta_{N_m}^m(u)=0$ under the specification $\beta_{N_m}^m(u):=0$.
\end{proof}}




{
To compute 
$\bar{\beta}^m_k$ and $\underline{\beta}^m_{k+1}$ 
using Proposition~\ref{prop:mus-max-min-beta}, 
we first determine which case applies to each extreme derivative in~\eqref{eq:relation-left-right-derivative-u-u_M^N}~and~\eqref{eq:relation-left-right-derivative-u-u_M^N-LR-end}, respectively. 
Denote $\beta_0^m:=\beta_0^m(u)=L$ and $\beta_{N_m}^m:=\beta_{N_m}^m(u)=0$. 
For each fixed $k=1,\ldots,N_m-1$, define two indicators $\mathcal{I}_k^{+}$ and $\mathcal{I}_k^{-}$ by the following linear programs: 
\par\medskip\noindent
\begin{minipage}[t]{0.49\linewidth}
\noindent
$\displaystyle\begin{array}{@{}r@{\;}l@{}}
\makebox[95pt][r]{$\displaystyle \mathcal I_k^+:=\max_{\alpha^m,\beta^m}$}&\beta_k^m-\beta_{k+1}^m\\[2mm]
\mathrm{s.t.}&\eqref{const:mus-pairwise}\text{--}\eqref{const:mus-domain},\\[1mm]
&\alpha_k^m=\bar\alpha_k^m.
\end{array}$
\end{minipage}\hfill
\begin{minipage}[t]{0.49\linewidth}
\noindent
$\displaystyle\begin{array}{@{}r@{\;}l@{}}
\makebox[95pt][r]{$\displaystyle \mathcal I_k^-:=\max_{\alpha^m,\beta^m}$}&\beta_{k-1}^m-\beta_k^m\\[2mm]
\mathrm{s.t.}&\eqref{const:mus-pairwise}\text{--}\eqref{const:mus-domain},\\[1mm]
&\alpha_{k+1}^m=\bar\alpha_{k+1}^m.
\end{array}$
\end{minipage}
\par\medskip\noindent
$\mathcal{I}_k^{+}$ and $\mathcal{I}_k^{-}$ 
are nonnegative on their respective feasible sets. 
\(\mathcal I_k^+=0\) if and only if \(\beta_k^m(u)=\beta_{k+1}^m(u)\) for all \(u\in\mathcal U_m^k\), while \(\mathcal I_k^-=0\) if and only if \(\beta_{k-1}^m(u)=\beta_k^m(u)\) for all \(u\in\mathcal U_m^{k+1}\). 
These are precisely the conditions 
in~\eqref{eq:relation-left-right-derivative-u-u_M^N}~and~\eqref{eq:relation-left-right-derivative-u-u_M^N-LR-end}, respectively. 
Set $r_k^+=k$ if $\mathcal I_k^+=0$, and $r_k^+=k-1$ if $\mathcal I_k^+>0$.
Likewise, set $r_k^-=k$ if $\mathcal I_k^-=0$, and $r_k^-=k+1$ if $\mathcal I_k^->0$. 
With these indices fixed, Proposition~\ref{prop:mus-max-min-beta} yields the following linear programs for computing the extreme derivatives: 
\par\medskip\noindent
\begin{minipage}[t]{0.49\linewidth}
\refstepcounter{equation}\label{eq:calcu-max-beta}
\noindent
$\displaystyle\begin{array}{@{}r@{\;}l@{}}
\makebox[95pt][r]{$\displaystyle \bar\beta_k^m=\max_{\alpha^m,\beta^m}$}&\beta_{r_k^+}^m\\[2mm]
\mathrm{s.t.}&\eqref{const:mus-pairwise}\text{--}\eqref{const:mus-domain},\\[1mm]
&\alpha_k^m=\bar\alpha_k^m,
\end{array}$\hfill(\theequation)
\end{minipage}\hfill
\begin{minipage}[t]{0.49\linewidth}
\refstepcounter{equation}\label{eq:calcu-min-beta}
\noindent
$\displaystyle\begin{array}{@{}r@{\;}l@{}}
\makebox[95pt][r]{$\displaystyle \underline\beta_{k+1}^m=\min_{\alpha^m,\beta^m}$}&\beta_{r_k^-}^m\\[2mm]
\mathrm{s.t.}&\eqref{const:mus-pairwise}\text{--}\eqref{const:mus-domain},\\[1mm]
&\alpha_{k+1}^m=\bar\alpha_{k+1}^m.
\end{array}$\hfill(\theequation)
\end{minipage}
\par\medskip\noindent 
The constraints in~\eqref{eq:calcu-max-beta}~and~\eqref{eq:calcu-min-beta} ensure that the piecewise linear utility functions constructed from feasible \((\alpha^m,\beta^m)\) belong to \(\mathcal U_m^k\) and \(\mathcal U_m^{k+1}\), respectively. 
Each 
$\bar\beta_k^m$ or $\underline\beta_{k+1}^m$ 
requires at most one linear program to decide the objective index $r_k^+$ or $r_k^-$, and one to compute its value. 
}

\subsection{Algorithm for solving problem~\texorpdfstring{\eqref{eq:mus-f(r_2)-initial}}{(\ref*{eq:mus-f(r_2)-initial})}}
\label{sec-3-6}

\subsubsection{Interval-Based Algorithm for Solving MUS Problem}
 
By Remark~\ref{remark:thm4}, we can obtain the optimal value by exhaustively evaluating and comparing the values of $f(\cdot)$ at the points in the finite set $\{y^m_k,y^*_{m,k}, y^m_{k+1}\}$ for $k=1,\ldots,N_m-1$. 
This motivates us to go through intervals $[y^m_k,y^m_{k+1}]$ one by one, calculating $y^*_{m,k}$ specified in~\eqref{eq:y^*_mk}, then evaluating $f(\cdot)$ at $y^m_k, y^*_{m,k}, y^m_{k+1}$, and finally identifying the one with the largest function value. 
Algorithm~\ref{alg:mus-1} is proposed to 
implement the procedures.


{
To 
implement Algorithm~\ref{alg:mus-1}
for generating  a new MUS query, we need to solve at most \(6(N_m-2)\) LPs to compute the extreme utility values $\bar{\alpha}^m_{k}$, $\underline{\alpha}^m_k$ defined in~\eqref{eq:mus-def-alpha-max} and the extreme derivatives $\bar{\beta}^m_k$, $\underline{\beta}^m_{k+1}$ defined in~\eqref{eq:mus-def-beta-max-min}. 
Once these quantities are 
obtained,
the remaining interval search requires \(O(N_m)\) arithmetic operations and comparisons. 
Each of these LPs, after eliminating the slope variables $\beta_j^m$, has at most \(N_m-2\) variables and \(2N_m-2\) inequality constraints. 
Let \(T_{\mathrm{LP}}\) denote a uniform bound on the 
complexity for solving one of these LPs. 
Then the
total 
complexity of 
Algorithm~\ref{alg:mus-1}
is 
$
T_{\mathrm{Alg1}}=O(N_mT_{\mathrm{LP}}+N_m).
$
Since each LP requires
about
\(O(\sqrt{N_m}\log N_m)\) iterations and 
\(O(N_m)\) arithmetic operations per iteration 
(see~\cite[Section~11.5.3]{vandenberghe2004convex} 
and~\cite[Section~9.7.2]{vandenberghe2004convex} for standard estimates), then we can provide a rough estimate of  
$T_{\mathrm{LP}}=O(N_m\cdot\sqrt{N_m}\log N_m)
=O(N_m^{3/2}\log N_m)$,
and 
hence the overall 
complexity of Algorithm~\ref{alg:mus-1} is
$T_{\mathrm{Alg1}}
=O(N_m\cdot N_m^{3/2}\log N_m+N_m)
=O(N_m^{5/2}\log N_m).
$
}

\normalem
\begin{algorithm}[htbp]
\caption{Interval-Based 
Algorithm for Solving MUS Problem~\eqref{eq:mus-f(r_2)-initial}\label{alg:mus-1}}
\IncMargin{2em}
\DontPrintSemicolon
\KwIn{$\mathcal{U}_m, \{y^m_j\}_{j=1,\ldots,N_m}$,\ $\bar{\beta}^m_1=L$,\ $\underline{\beta}^m_{N_m}=0$}
\KwOut{$r_2^{m+1}, \ p^{m+1}$}
\textbf{Initialize:} $r_2^{m+1} \leftarrow 0$, $f(r_2^{m+1})\leftarrow 0$, $p^{m+1}\leftarrow 0$
    
\For{$j=1,\ldots,N_m$}{
Compute $\bar{\alpha}^m_j$ by solving linear program~\eqref{eq:calcu-max-alpha}

Compute $\underline{\alpha}^m_j$ by solving linear program~\eqref{eq:calcu-min-alpha}
}

\For{$j=2,\ldots,N_m-1$}{
Compute $\bar{\beta}^m_j$ by solving linear program~\eqref{eq:calcu-max-beta}

Compute $\underline{\beta}^m_j$ by solving linear program~\eqref{eq:calcu-min-beta}
}

\For{$j=1,\ldots,N_m-1$}{
\If{$\bar{\alpha}^m_j-\underline{\alpha}^m_j>f(r_2^{m+1})$}{
$r_2^{m+1}\leftarrow y^m_j$, 
$p^{m+1}\leftarrow (\bar{\alpha}^m_j+\underline{\alpha}^m_j)/2$, 
$f(r_2^{m+1})\leftarrow\bar{\alpha}^m_j-\underline{\alpha}^m_j$
}

\If{$\bar{\beta}^m_j>\underline{\beta}^m_{j+1}$}{
$y^*_{m,j}\leftarrow \frac{\bar{\alpha}^m_{j+1}-\bar{\alpha}^m_j-\underline{\beta}^m_{j+1}y^m_{j+1}+\bar{\beta}^m_j y^m_j}{\bar{\beta}^m_j-\underline{\beta}^m_{j+1}}$, 
$\alpha^*_{m,j}\leftarrow \bar{\beta}^m_j(y^*_{m,j}-y^m_j)+\bar{\alpha}^m_j$

\If{$\alpha^*_{m,j}-\frac{\underline{\alpha}^m_{j+1}-\underline{\alpha}^m_j}{y^m_{j+1}-y^m_j}(y^*_{m,j}-y^m_j)-\underline{\alpha}^m_{j}>f(r_2^{m+1})$}{
$r_2^{m+1}\leftarrow y^*_{m,j}$, 
$p^{m+1}\leftarrow (\alpha^*_{m,j}+\frac{\underline{\alpha}^m_{j+1}-\underline{\alpha}^m_j}{y^m_{j+1}-y^m_j}(y^*_{m,j}-y^m_j)+\underline{\alpha}^m_{j})/2$, 
$f(r_2^{m+1})\leftarrow\alpha^*_{m,j}-\frac{\underline{\alpha}^m_{j+1}-\underline{\alpha}^m_j}{y^m_{j+1}-y^m_j}(y^*_{m,j}-y^m_j)-\underline{\alpha}^m_{j}$
}
}
}
Return $r_2^{m+1}$ and $p^{m+1}$
\end{algorithm}
\ULforem

\subsubsection{Improved Interval-Based Algorithm}

Moreover, since $f(\cdot)$ is piecewise linear with at most two segments over each interval $[y^m_k,y^m_{k+1}]$ by Theorem~\ref{thm:mus-semi-f}, it is unimodal on each interval. 
Thus, we can improve the exhaustive search in Algorithm~\ref{alg:mus-1} by comparing the lower and upper bounds of $f$ in each interval and eliminating the intervals that cannot contain the optimal solution.

We consider a two-step mechanism: 
(a) we first eliminate the sub-interval $[y^m_j,y^m_{j+1}]$ that cannot contain the optimal $r_2^{m+1}$ of~\eqref{eq:mus-f(r_2)-initial}, based on the values of $\bar{\alpha}^m_j$, $\underline{\alpha}^m_j$, $\bar{\alpha}^m_{j+1}$ and $\underline{\alpha}^m_{j+1}$, $j=1,\ldots,N_m-1$, 
(b) we then perform further analysis and calculations on the remaining sub-intervals as described in Algorithm~\ref{alg:mus-1}. 
For any $r_2\in[y^m_j,y^m_{j+1}]$, $j=1,\ldots,N_m-1$, by the monotonicity of $u\in\mathcal{U}_m$, we have 
$$
\bar{f}(r_2)=\max\limits_{u\in\mathcal{U}_m}u(r_2)\leq\max\limits_{u\in\mathcal{U}_m}u(y^m_{j+1})=\bar{f}(y^m_{j+1})
$$
and
$$
\underline{f}(r_2)=\min\limits_{u\in\mathcal{U}_m}u(r_2)\geq\min\limits_{u\in\mathcal{U}_m}u(y^m_j)=\underline{f}(y^m_j). 
$$ 
Thus, 
$$
f(r_2)=\max\limits_{u\in\mathcal{U}_m}u(r_2)-\min\limits_{u\in\mathcal{U}_m}u(r_2)\leq \bar{f}(y^m_{j+1})-\underline{f}(y^m_j).
$$
We then have  
$$
\max\{f(y^m_j),\ f(y^m_{j+1})\}\leq \max\limits_{r_2\in[y^m_j, y^m_{j+1}]}f(r_2) \leq \bar{f}(y^m_{j+1})-\underline{f}(y^m_j), 
$$
which provides a lower bound and an upper bound for $\max\limits_{r_2\in[y^m_j, y^m_{j+1}]}f(r_2)$, $j=1,\ldots,N_m-1$.

Let
$$
{f}_l:=\mathop{\max}\limits_{j=1,\ldots,N_m-1}\left\{\max\{f(y^m_j),f(y^m_{j+1})\}\right\}
$$
represent the maximum lower bound of $\max\limits_{r_2\in\left[y^m_j, y^m_{j+1}\right]} f(r_2)$ across all sub-intervals. 
Then, we denote the set of sub-intervals $\left\{[y^m_{\tau_i},y^m_{\tau_{i}+1}]\right\}_{i=1,\ldots,I_m}\subseteq\left\{\left[y^m_{j},y^m_{j+1}\right]\right\}_{j=1,\ldots,N_m-1}$ with $\tau_i\in\{1,2,\ldots,N_m-1\}$ and $I_m\leq N_m-1$, such that $\bar{f}(y^m_{\tau_i+1})-\underline{f}(y^m_{\tau_i})\geq{f}_l$, $\forall i=1,\ldots,I_m$, and search further for optimal $r_2^{m+1}$ on these selected sub-intervals. 
That means we filter out any sub-interval for which the upper bound of $\max\limits_{r_2\in[y^m_j, y^m_{j+1}]} f(r_2)$ is even smaller than the maximum lower bound ${f}_l$.
In other words, we select the sub-intervals that may contain the optimal $r_2^{m+1}$ to the maximization problem~\eqref{eq:mus-f(r_2)-initial}. 
In practice, the number of such sub-intervals selected is usually smaller than the total number of sub-intervals within the domain $[0,1]$, thus improving the computational efficiency. 
Now we are ready to present the improved high-efficiency Algorithm~\ref{alg:mus-2} for solving~\eqref{eq:mus-f(r_2)-initial}, which generates the parameters $r_2^{m+1}$ and $p^{m+1}$ under the maximum utility split scheme.


\normalem
\begin{algorithm}
\caption{Improved Interval-Based 
Algorithm for Solving MUS Problem~\eqref{eq:mus-f(r_2)-initial}\label{alg:mus-2}}
\IncMargin{2em}
\DontPrintSemicolon
\KwIn{$\mathcal{U}_m, \{y^m_j\}_{j=1,\ldots,N_m}$,\ $\bar{\beta}^m_1=L$,\ $\underline{\beta}^m_{N_m}=0$}
\KwOut{$r_2^{m+1}, \ p^{m+1}$}
\textbf{Initialize:} $r_2^{m+1} \leftarrow 0$, $f(r_2^{m+1})\leftarrow 0$, $p^{m+1}\leftarrow 0$
    
\For{$j=1,\ldots,N_m$}{
Compute $\bar{\alpha}^m_j$ by solving the linear program \eqref{eq:calcu-max-alpha}

Compute $\underline{\alpha}^m_j$ by solving the linear program \eqref{eq:calcu-min-alpha}
}

${f}_l\leftarrow \mathop{\max}\limits_{j=1,\ldots,N_m-1}\left\{\max\{\bar{\alpha}^m_j-\underline{\alpha}^m_j,\bar{\alpha}^m_{j+1}-\underline{\alpha}^m_{j+1}\}\right\}$

\For{$j=2,\ldots,N_m-1$}{
    \If{$\bar{\alpha}^m_{j+1}-\underline{\alpha}^m_j\geq f_l$}{
        Compute $\bar{\beta}^m_j$ by solving the linear program~\eqref{eq:calcu-max-beta}
    }
    \If{$\bar{\alpha}^m_j-\underline{\alpha}^m_{j-1}\geq f_l$}{
        Compute $\underline{\beta}^m_j$ by solving the linear program~\eqref{eq:calcu-min-beta}
    }
}

\For{$j=1,\ldots,N_m-1$}{
\If{$\bar{\alpha}^m_{j+1}-\underline{\alpha}^m_j\geq{f}_l$}{
$\kappa\leftarrow\mathop{\arg\max}\limits_{i=j,j+1} \bar{\alpha}^m_i-\underline{\alpha}^m_i$

\If{$\bar{\alpha}^m_{\kappa}-\underline{\alpha}^m_{\kappa}>f(r_2^{m+1})$}{
$r_2^{m+1}\leftarrow y^m_{\kappa}$, 
$p^{m+1}\leftarrow (\bar{\alpha}^m_{\kappa}+\underline{\alpha}^m_{\kappa})/2$, 
$f(r_2^{m+1})\leftarrow \bar{\alpha}^m_{\kappa}-\underline{\alpha}^m_{\kappa}$
}

\If{$\bar{\beta}^m_j>\underline{\beta}^m_{j+1}$}{
$y^*_{m,j}\leftarrow \frac{\bar{\alpha}^m_{j+1}-\bar{\alpha}^m_j-\underline{\beta}^m_{j+1}y^m_{j+1}+\bar{\beta}^m_j y^m_j}{\bar{\beta}^m_j-\underline{\beta}^m_{j+1}}$, $\alpha^*_{m,j}\leftarrow \bar{\beta}^m_j(y^*_{m,j}-y^m_j)+\bar{\alpha}^m_j$

\If{$\alpha^*_{m,j}-\frac{\underline{\alpha}^m_{j+1}-\underline{\alpha}^m_j}{y^m_{j+1}-y^m_j}(y^*_{m,j}-y^m_j)-
{\underline{\alpha}^m_j}
>f(r_2^{m+1})$}{
$r_2^{m+1}\leftarrow y^*_{m,j}$, 
$p^{m+1}\leftarrow (\alpha^*_{m,j}+\frac{\underline{\alpha}^m_{j+1}-\underline{\alpha}^m_j}{y^m_{j+1}-y^m_j}(y^*_{m,j}-y^m_j)+
{\underline{\alpha}^m_j}
)/2$, 
$f(r_2^{m+1})\leftarrow \alpha^*_{m,j}-\frac{\underline{\alpha}^m_{j+1}-\underline{\alpha}^m_j}{y^m_{j+1}-y^m_j}(y^*_{m,j}-y^m_j)-
{\underline{\alpha}^m_j}
$
}
}
}
}
Return $r_2^{m+1}$ and $p^{m+1}$
\end{algorithm}
\ULforem

\section{Proofs of Lemma~\ref{lemma:mus-beta-inequal}, Lemma~\ref{lemma:mus-beta-equal}
and Theorem~\ref{thm:mus-upper-bound}}
\label{sec-ec-proof-lemma2}

{
We now prove Lemmas~\ref{lemma:mus-beta-inequal}~and~\ref{lemma:mus-beta-equal}, 
based on the attainment of the maximum right derivative and the minimum left derivative stated in Proposition~\ref{prop:mus-max-min-beta}.
}

\subsection{Proof of Lemma~\ref{lemma:mus-beta-inequal}.}
We only prove the first inequality, that is, 
\bgeqn 
\label{eq:beta-bar-m}
\bar{\beta}^m_k=\max\limits_{u\in\mathcal{U}^k_m} u^{'}_{+}(y^m_k)\geq\frac{\bar{\alpha}^m_{k+1}-\bar{\alpha}^m_{k}}{y^m_{k+1}-y^m_k},
\edeqn 
for each fixed $k\in\{1,2,\ldots,N_m-1\}$, because the second inequality can be proved analogously. 
{By 
Proposition~\ref{prop:mus-max-min-beta},} 
we know the maximum in~\eqref{eq:beta-bar-m}
is attainable. 
Assume for the sake of a contradiction that there exists a function $\bar{u}\in\mathop{\arg\max}\limits_{u\in\mathcal{U}^k_m} u^{'}_{+}(y^m_k)$ with 
\begin{equation*}
\bar{u}(y^m_k)=\bar{\alpha}^m_k
\quad\text{and}\quad
\bar{u}^{'}_{+}(y^m_k)=\bar{\beta}^m_k,
\end{equation*}
such that $\bar{\beta}^m_k=\bar{u}^{'}_{+}(y^m_k)<\frac{\bar{\alpha}^m_{k+1}-\bar{\alpha}^m_{k}}{y^m_{k+1}-y^m_k}$. 
We will prove that there exists a function $\check{u}\in\mathcal{U}^k_m$ with $\check{u}(y^m_k)=\bar{\alpha}^m_k$, such that 
\begin{equation*}
\check{u}^{'}_{+}(y^m_k)=\frac{\bar{\alpha}^m_{k+1}-\bar{\alpha}^m_{k}}{y^m_{k+1}-y^m_k}>\bar{\beta}^m_k=\bar{u}^{'}_{+}(y^m_k),  
\end{equation*}
which is a contradiction to the assumption that $\bar{u}\in\mathop{\arg\max}\limits_{u\in\mathcal{U}^k_m} u^{'}_{+}(y^m_k)$.

\vspace{0.1cm}
{\underline{\bf Construction of $\boldsymbol{\check{u}(y)}$.}}
Due to the non-emptiness of $\mathcal{U}^{k+1}_m$, we can select $\tilde{u}\in\mathcal{U}^{k+1}_m$ with 
\begin{equation*}
\tilde{u}(y^m_{k+1})=\bar{\alpha}^m_{k+1}=\max\limits_{u\in\mathcal{U}_m}u(y^m_{k+1}).
\end{equation*}
Since $\bar{\alpha}^m_k$ is the maximum utility value at $y^m_k$, we have $\tilde{u}(y^m_k)\leq\bar{\alpha}^m_k= \max\limits_{u\in\mathcal{U}_m}u(y^m_{k})$. 
Let 
\begin{equation}
\label{eq:alpha-beta-k-check}
\check{\alpha}_k=\bar{\alpha}^m_{k}, \quad \check{\alpha}_{k+1}=\bar{\alpha}^m_{k+1},\quad \check{\beta}_k=\frac{\check{\alpha}_{k+1}-\check{\alpha}_k}{y^m_{k+1}-y^m_k}.
\end{equation}
For $j=k-1,k-2,\ldots,1$, let
\begin{equation}\label{eq:lemma3-check-left}
\check{\alpha}_{j}=\min\Big\{\bar{u}(y^m_{j}),\ \check{\alpha}_{j+1}-\check{\beta}_{j+1}(y^m_{j+1}-y^m_{j})\Big\}, \quad
\check{\beta}_j=\frac{\check{\alpha}_{j+1}-\check{\alpha}_j}{y^m_{j+1}-y^m_j}, 
\end{equation}
and for $j=k+1,k+2,\ldots,N_m-1$, let
\begin{equation}\label{eq:lemma3-check-right}
\check{\alpha}_{j+1}=\min\Big\{\tilde{u}(y^m_{j+1}),\ \check{\alpha}_{j}+
\check{\beta}_{j-1}(y^m_{j+1}-y^m_{j})\Big\}, 
\quad 
\check{\beta}_j=\frac{\check{\alpha}_{j+1}-\check{\alpha}_j}{y^m_{j+1}-y^m_j}.   
\end{equation}
Let
\begin{equation}\label{eq:check-u}
    \check{u}(y)=\left\{ 
    \begin{aligned}
    & \check{\alpha}_{N_m} && \text{for} \quad y=y^m_{N_m}, \\
    &
    \check{\alpha}_j +\check{\beta}_j (y-y^m_{j})
    &&  \text{for} \quad y^m_{j} \leq y < y^m_{j+1},\;\; j=N_m-1, N_m-2,\ldots,1. 
    \end{aligned}
    \right.
\end{equation}
See Figure~\ref{fig:mus-construct} for an illustration. 
By the construction, $\check{u}$ is piecewise linear and continuous over $[y^m_1,y^m_{N_m}]$ with $\check{u}(y^m_j)=\check{\alpha}_j$ for $j=1,\ldots,N_m$, $\check{u}(y^m_k)=\bar{\alpha}^m_k=\bar{u}(y^m_k)$, $\check{u}(y^m_{k+1})=\bar{\alpha}^m_{k+1}=\tilde{u}(y^m_{k+1})$, and 
$$
\check{u}^{'}_{+}(y^m_k)=\check{\beta}_{k}=\frac{\bar{\alpha}^m_{k+1}-\bar{\alpha}^m_k}{y^m_{k+1}-y^m_k}>\bar{\beta}^m_k=\bar{u}^{'}_{+}(y^m_k).
$$
Next, we show $\check{u}\in\mathcal{U}_m$.

\begin{figure}[htbp]
    \centering
    \includegraphics[width=0.7\linewidth]{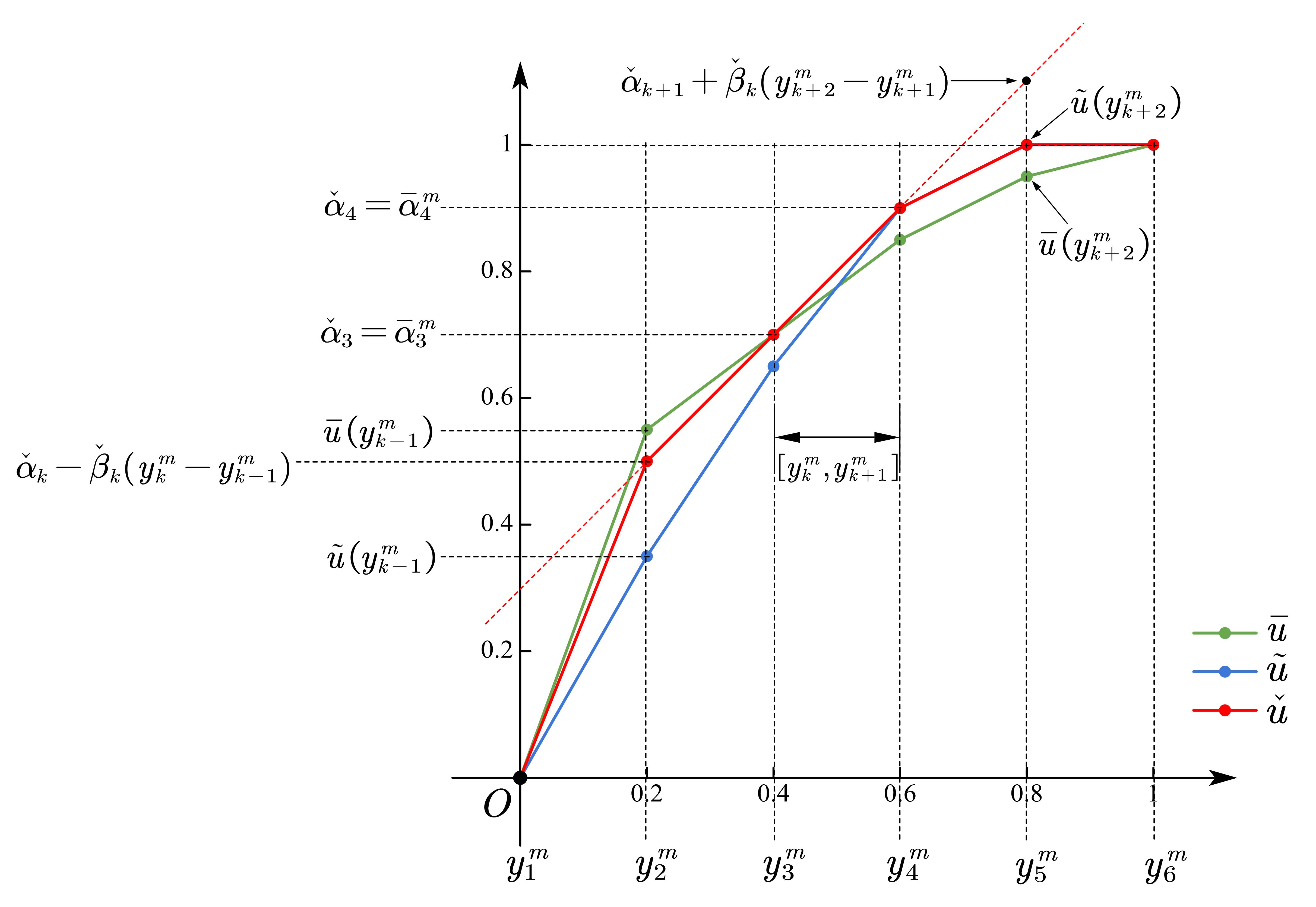}
    \caption{Consider a case where $N_m=6$ and $k=3$. \rm The right derivative of the green curve at $y_k^m$ indicates that
    $\bar{\beta}^m_k
    :=\max\limits_{u\in\mathcal{U}^k_m} u^{'}_{+}(y^m_k)=\bar{u}^{'}_{+}(y^m_k)<\frac{\bar{\alpha}^m_{k+1}-\bar{\alpha}^m_{k}}{y^m_{k+1}-y^m_k}$, which is assumed for a contradiction.
    The red curve illustrates how $\check{u}$ may be constructed. 
    The construction and the proof show that $\check{u}\in {\cal U}_m^k$ and it attains the maximum of  $\max\limits_{u\in\mathcal{U}^k_m} u^{'}_{+}(y^m_k)$, which contradicts the assumption that $\bar{u}$ attains the maximum. 
    Note that $\bar{u}$ and $\tilde{u}$ 
    do not have to be piecewise linear, 
    we plot them piecewise linearly for simplicity.}
    \label{fig:mus-construct}
\end{figure}

\vspace{0.1cm}
{\underline{\bf Monotonicity of $\boldsymbol{\check{u}(y)}$.}}
Due to the piecewise linear structure of $\check{u}$ over the interval $[y^m_1,y^m_{N_m}]$, it suffices to show that $\check{\alpha}_{1}\leq\check{\alpha}_2\leq\cdots\leq\check{\alpha}_{N_m}$ for the monotonicity of $\check{u}$. 
For $j=k+1,k+2,\ldots,N_m$, if $\check{\alpha}_{j}\geq\check{\alpha}_{j-1}$, then by the monotonicity of $\tilde{u}$ and~\eqref{eq:lemma3-check-right}, 
\begin{equation*}
\tilde{u}(y^m_{j+1})\geq\tilde{u}(y^m_j)\geq\check{\alpha}_j
\quad\text{and}\quad
\check{\alpha}_{j}+\check{\beta}_{j-1}(y^m_{j+1}-y^m_{j})=\check{\alpha}_{j}+\frac{\check{\alpha}_j-\check{\alpha}_{j-1}}{y^m_j-y^m_{j-1}}(y^m_{j+1}-y^m_{j})\geq\check{\alpha}_j. 
\end{equation*}
Thus $\check{\alpha}_{j+1}=\min\{\tilde{u}(y^m_{j+1}),\ \check{\alpha}_{j}+\check{\beta}_{j-1}(y^m_{j+1}-y^m_{j})\}\geq\check{\alpha}_j$. 
Moreover, $\check{\alpha}_{k+1}=\bar{\alpha}^m_{k+1}\geq\bar{\alpha}^m_k=\check{\alpha}_k$. 
By induction, we have the monotonicity of $\check{u}$ over interval $[y^m_k,y^m_{N_m}]$. 
The monotonicity of $\check{u}$ over $[y^m_1,y^m_k]$ can be established by the monotonicity of $\bar{u}$ and the definition of $\check{\alpha}_{j}$, $\check{\beta}_{j}$ in~\eqref{eq:lemma3-check-left} analogously.

\vspace{0.1cm}
{\underline{\bf Concavity of $\boldsymbol{\check{u}(y)}$.}}
Due to the piecewise linear structure of $\check{u}$ within $[y^m_1,y^m_{N_m}]$, the concavity of $\check{u}$ reduces to $\check{\beta}_1\geq\check{\beta}_2\geq\cdots\geq\check{\beta}_{N_m-1}$, where $\check{\beta}_j=\check{u}^{'}_{+}(y^m_j)=\check{u}^{'}_{-}(y^m_{j+1})$, $j=1,\ldots,k-1, k+1,\ldots,N_m-1$, and $\check{\beta}_{k}=\frac{\bar{\alpha}^m_{k+1}-\bar{\alpha}^m_k}{y^m_{k+1}-y^m_k}=\check{u}^{'}_{+}(y^m_{k})=\check{u}^{'}_{-}(y^m_{k+1})$. 
By the definition of $\check{\beta}_j$ and $\check{\alpha}_j$ in~\eqref{eq:lemma3-check-left} and~\eqref{eq:lemma3-check-right}, 
\begin{align*}
\check{\beta}_j=\frac{\check{\alpha}_{j+1}-\check{\alpha}_j}{y^m_{j+1}-y^m_j}&\leq \frac{\check{\alpha}_{j}+\check{\beta}_{j-1}(y^m_{j+1}-y^m_{j})-\check{\alpha}_j}{y^m_{j+1}-y^m_j}=\check{\beta}_{j-1} \quad \text{for}\;\;  j=k+1,k+2,\ldots,N_m-1, 
\end{align*}
and 
\begin{align*}
\check{\beta}_{j}=\frac{\check{\alpha}_{j+1}-\check{\alpha}_{j}}{y^m_{j+1}-y^m_j}\geq \frac{\check{\alpha}_{j+1}-\left[\check{\alpha}_{j+1}-\check{\beta}_{j+1}(y^m_{j+1}-y^m_{j})\right]}{y^m_{j+1}-y^m_j}=\check{\beta}_{j+1}\quad \text{for}\;\;  j=k-1,k-2,\ldots,1.
\end{align*}

\vspace{0.1cm}
{\underline{\bf $\boldsymbol{\check{u}(y)}$ satisfying pairwise comparison constraints.}}
We use Lemma~\ref{lemma:mus-u-U_m} to prove that $\check{u}$ satisfies the pairwise comparison constraints in $\mathcal{U}_m$ as in~\eqref{eq:U_m-MUS}. 
We do so by finding proper $u^1_j$ and $u^2_j$ in $\mathcal{U}_m$ such that $u^1_j(y^m_{j}) \leq \check{u}(y^m_j) \leq u^2_j(y^m_{j})$ for each $j = 1, \ldots, N_m$. 
The key step is to determine the value $\check{\alpha}_j=\check{u}(y^m_j)$ via~\eqref{eq:lemma3-check-left} for $j<k$, which is based on the comparison between $\bar{u}(y^m_j)$ and $\check{\alpha}_{j+1} - \check{\beta}_{j+1}(y^m_{j+1} - y^m_{j})$, and to determine the value $\check{\alpha}_{j+1}=\check{u}(y^m_{j+1})$ via~\eqref{eq:lemma3-check-right} for $j> k$, which is based on the comparison between $\tilde{u}(y^m_{j+1})$ and $\check{\alpha}_{j}+\check{\beta}_{j-1}(y^m_{j+1}-y^m_{j})$. 
Before conducting the case-by-case analysis, we summarize a key statement as follows.

\begin{statement}
\label{state:deduction}
Let $\check{\alpha}_j$, $j=1,\ldots,N_m$, be defined as in~\eqref{eq:lemma3-check-left} and~\eqref{eq:lemma3-check-right}. 
The following assertions hold. 

\begin{enumerate}[label=(\roman*)]
\item 
For $j=k-2,k-3,\ldots,1$, if $\bar{u}(y^m_{j+1})\leq
\check{\alpha}_{j+2}-\check{\beta}_{j+2}(y^m_{j+2}-y^m_{j+1})
$,
then $\bar{u}(y^m_{j})\leq 
\check{\alpha}_{j+1}-\check{\beta}_{j+1}(y^m_{j+1}-y^m_{j})
$.

\item 
For $j=k+2,k+3,\ldots,N_m-1$, if $\tilde{u}(y^m_{j})\leq \check{\alpha}_{j-1}+\check{\beta}_{j-2}(y^m_{j}-y^m_{j-1})$, then $\tilde{u}(y^m_{j+1})\leq \check{\alpha}_{j}+\check{\beta}_{j-1}(y^m_{j+1}-y^m_{j})$.
\end{enumerate}
\end{statement}

\noindent {\bf Proof of the statement.} 
\uline{Part~(i).}
By~\eqref{eq:lemma3-check-left}
the inequality 
$\bar{u}(y^m_{j+1})\leq 
\check{\alpha}_{j+2}-\check{\beta}_{j+2}(y^m_{j+2}-y^m_{j+1})
$ 
implies that $\check{\alpha}_{j+1}=\bar{u}(y^m_{j+1})$. 
By the concavity of $\bar{u}$ and the fact that $\check{\alpha}_{j+2}\leq\bar{u}(y^m_{j+2})$, 
\begin{equation*}
\bar{u}^{'}_{-}(y^m_{j+1})\geq\bar{u}^{'}_{+}(y^m_{j+1})\geq\frac{\bar{u}(y^m_{j+2})-\bar{u}(y^m_{j+1})}{y^m_{j+2}-y^m_{j+1}}\geq\frac{\check{\alpha}_{j+2}-\check{\alpha}_{j+1}}{y^m_{j+2}-y^m_{j+1}}=\check{\beta}_{j+1}.  
\end{equation*}
Together with the concavity of $\bar{u}$ and the fact that $\check{\alpha}_{j+1}=\bar{u}(y^m_{j+1})$ and $\bar{u}^{'}_{-}(y^m_{j+1})\geq\check{\beta}_{j+1}$, we have 
$$
\bar{u}(y^m_{j})\leq\bar{u}(y^m_{j+1})-\bar{u}^{'}_{-}(y^m_{j+1})(y^m_{j+1}-y^m_{j})\leq
\check{\alpha}_{j+1}-\check{\beta}_{j+1}(y^m_{j+1}-y^m_{j}). 
$$

\noindent
\uline{Part~(ii).}
The inequality $\tilde{u}(y^m_{j})\leq \check{\alpha}_{j-1}+\check{\beta}_{j-2}(y^m_{j}-y^m_{j-1})$
implies that $\check{\alpha}_j=\tilde{u}(y^m_j)$. 
By the concavity of $\tilde{u}$ and the fact that $\check{\alpha}_{j-1}\leq\tilde{u}(y^m_{j-1})$, 
$$
\tilde{u}^{'}_{+}(y^m_j)\leq\tilde{u}^{'}_{-}(y^m_j)\leq\frac{\tilde{u}(y^m_j)-\tilde{u}(y^m_{j-1})}{y^m_j-y^m_{j-1}}\leq\frac{\check{\alpha}_j-\check{\alpha}_{j-1}}{y^m_j-y^m_{j-1}}=\check{\beta}_{j-1}. 
$$
By the concavity of $\tilde{u}$ and 
the facts $\check{\alpha}_j=\tilde{u}(y^m_j)$ and $\tilde{u}^{'}_{+}(y^m_j)\leq\check{\beta}_{j-1}$, we have
$$
\tilde{u}(y^m_{j+1})\leq\tilde{u}(y^m_j)+\tilde{u}^{'}_{+}(y^m_j)(y^m_{j+1}-y^m_j)\leq\check{\alpha}_j+\check{\beta}_{j-1}(y^m_{j+1}-y^m_j). 
$$ 
This completes the proof of Statement~\ref{state:deduction}. 
\hfill $\Box$

\vspace{0.1cm}
First, we focus on $j=k-1,k-2,\ldots,1$ within the sub-horizon $[y^m_1,y^m_k]$. We then consider the satisfaction of pairwise comparison constraints. 
By Statement~\ref{state:deduction}, we find that if $\bar{u}(y^m_{j+1}) \leq \check{\alpha}_{j+2} - \check{\beta}_{j+2}(y^m_{j+2} - y^m_{j+1})$ holds for some $j < k-1$, which is equivalent to
$\check{u}(y^m_{j+1}) = \check{\alpha}_{j+1} = \bar{u}(y^m_{j+1})$, then $\check{u}(y^m_{i}) = \check{\alpha}_{i} = \bar{u}(y^m_{i})$ for all $i = 1, \ldots, j+1$.
Based on this observation, we can divide the situation over $[y^m_1,y^m_k]$ into the following three sub-cases.

\vspace{0.1cm}
\noindent
\uline{Case~1:} $\bar{u}(y^m_{k-1})\leq\check{\alpha}_k-\check{\beta}_k(y^m_k-y^m_{k-1})$. In this case, by the above induction relationship, 
\begin{equation*}
\check{u}(y^m_j)=\check{\alpha}_{j}=\bar{u}(y^m_j) 
\quad\text{for all}\quad j=k-1,k-2,\ldots,1. 
\end{equation*}
We let $u^1_j=u^2_j=\bar{u}$ so that $u^1_j(y^m_j)=\bar{u}(y^m_j)\leq
\check{u}(y^m_j)\leq\bar{u}(y^m_j)=u^2_j(y^m_j)$, $j=1,\ldots,k-1$.

\vspace{0.1cm}
\noindent
\uline{Case~2:} $\check{\alpha}_{j+1}-\check{\beta}_{j+1}(y^m_{j+1}-y^m_{j})<\bar{u}(y^m_{j})$, for all $j=k-1,k-2,\ldots,1$.
In this case, 
\begin{equation*}
\check{u}(y^m_j)=\check{\alpha}_j=\check{\alpha}_{j+1}-\check{\beta}_{j+1}(y^m_{j+1}-y^m_{j})
\end{equation*}
and
\begin{equation*}
\check{\beta}_j=\frac{\check{\alpha}_{j+1}-\check{\alpha}_j}{y^m_{j+1}-y^m_j}=\frac{\check{\alpha}_{j+1}-\left[\check{\alpha}_{j+1}-\check{\beta}_{j+1}(y^m_{j+1}-y^m_{j})\right]}{y^m_{j+1}-y^m_j}=\check{\beta}_{j+1},
\end{equation*}
for $j=k-1,k-2,\ldots,1$. 
Thus, all $\check{\beta}_j$ equal to $\check{\beta}_k=\frac{\bar{\alpha}^m_{k+1}-\bar{\alpha}^m_k}{y^m_{k+1}-y^m_k}$. 
Considering the concavity of $\tilde{u}$ and the fact that $\tilde{u}(y^m_{k+1})=\bar{\alpha}^m_{k+1}$ and $\tilde{u}(y^m_{k})\leq \bar{\alpha}^m_k$, we have
$$
\tilde{u}^{'}_{-}(y^m_{j+1})\geq\tilde{u}^{'}_{+}(y^m_{j+1})\geq\tilde{u}^{'}_{+}(y^m_{k})\geq\frac{\tilde{u}(y^m_{k+1})-\tilde{u}(y^m_{k})}{y^m_{k+1}-y^m_k}\geq\frac{\bar{\alpha}^m_{k+1}-\bar{\alpha}^m_k}{y^m_{k+1}-y^m_k}=\check{\beta}_k=\check{\beta}_{j+1}, 
$$
for $j=k-1,k-2,\ldots,1$. 
Then, for $j=k-1,\ldots,1$, by the concavity of $\tilde{u}$ and $\tilde{u}(y^m_{k})\leq \bar{\alpha}^m_k$, we have 
$$
\tilde{u}(y^m_{j})\leq \tilde{u}(y^m_{j+1})-\tilde{u}^{'}_{-}(y^m_{j+1})(y^m_{j+1}-y^m_{j})\leq\check{\alpha}_{j+1}-\check{\beta}_{j+1}(y^m_{j+1}-y^m_{j})=\check{\alpha}_{j}=\check{u}(y^m_{j}). 
$$
In this case, we let $u^1_{j}=\tilde{u}$ and $u^2_{j}=\bar{u}$, such that $u^1_{j}(y^m_{j})=\tilde{u}(y^m_{j})\leq \check{u}(y^m_{j})\leq\bar{u}(y^m_{j})=u^2_{j}(y^m_{j})$, $j=1,2,\ldots,k-1$.

\vspace{0.1cm}
\noindent
\uline{Case~3:} there exists $i\in\{2,3,\ldots,k-2\}$, such that $\check{\alpha}_{j+1}-\check{\beta}_{j+1}(y^m_{j+1}-y^m_{j})<\bar{u}(y^m_{j})$ for $j=k-1,$ 
$k-2,\ldots,i$, and $\bar{u}(y^m_{i-1})\leq\check{\alpha}_{i}-\check{\beta}_{i}(y^m_{i}-y^m_{{i-1}})$.
In this case, by the induction relationship in Statement~\ref{state:deduction}, we have $\check{u}(y^m_{j})=\check{\alpha}_{{j}}=\bar{u}(y^m_{j})$ for $j=i-1,i-2,\ldots,1$. (similar to Case~1.) 
Moreover, for $j=k-1,k-2,\ldots,i$, applying the same analysis in Case~2, we have $\tilde{u}(y^m_j)\leq\check{u}(y^m_j)\leq\bar{u}(y^m_j)$. 
In this case, we let $u^1_{j}=\tilde{u}$ and $u^2_{j}=\bar{u}$ for $j=i,i+1,\ldots,k-1$, and $u^1_j=u^2_j=\bar{u}$ for $j=1,2,\ldots,i-1$. 
Then, for all $j=1,\ldots,k-1$, $u^1_j(y^m_j)\leq\check{u}(y^m_j)\leq u^2_j(y^m_j)$.

Based on the three cases above, we conclude that there exist $u^1_1,\ldots,u^1_{k-1}\in\mathcal{U}_m$ and $u^2_1,\ldots,u^2_{k-1}\in\mathcal{U}_m$, such that $u^1_j(y^m_{j})\leq \check{u}(y^m_j)\leq u^2_j(y^m_{j})$, $j=1,\ldots,k-1$, 
and $u^1_{k}=u^2_{k}=\bar{u}$ such that $u^1_{k}(y^m_k)=\bar{u}(y^m_k)\leq\check{u}(y^m_k)\leq\bar{u}(y^m_k)=u^2_{k}(y^m_k)$. 
Similarly, we can divide the discussion on sub-horizon $[y^m_{k+1},y^m_{N_m}]$ into three cases by the second argument in Statement~\ref{state:deduction}, and select $u^1_{k+1}, u^1_{k+2},\ldots,u^1_{N_m}, u^2_{k+1}, u^2_{k+2},\ldots, u^2_{N_m}$ as either $\tilde{u}\in\mathcal{U}_m$ or $\bar{u}\in\mathcal{U}_m$, such that $u^1_j(y^m_{j})\leq \check{u}(y^m_j)\leq u^2_j(y^m_{j})$, $j=k+1,\ldots,N_m$. 

Before applying Lemma~\ref{lemma:mus-u-U_m}, we still need to validate the normalization and Lipschitz continuity of $\check{u}$, which rely on the existence of $u^1_j$ and $u^2_j$, $j=1,\ldots,N_m$.

\vspace{0.1cm}
{\underline{\bf Normalization of $\boldsymbol{\check{u}(y)}$.}}
Since $u^1_1$, $u^2_1$, $u^1_{N_m}$, and $u^2_{N_m}$ coincide with either $\bar{u}$ or $\tilde{u}$, and the latter are normalized with $\bar{u}(y^m_1)=0$, $\bar{u}(y^m_{N_m})=1$, and $\tilde{u}(y^m_1)=0$, $\tilde{u}(y^m_{N_m})=1$, it follows that $\check{u}(y^m_1)\in[u^1_1(y^m_1),u^2_1(y^m_1)]=\{0\}$ and $\check{u}(y^m_{N_m})\in[u^1_{N_m}(y^m_{N_m}), u^2_{N_m}(y^m_{N_m})]=\{1\}$.

\vspace{0.1cm}
{\underline{\bf Lipschitz continuity of $\boldsymbol{\check{u}(y)}$.}}
Since $\check{u}$ is piecewise linear, concave and monotonically increasing, the Lipschitz continuity $\check{u}$ reduces to $\check{\beta}_1 \leq L$, where $L$ is the Lipschitz modulus specified in $\mathcal{U}_{cv}$ (see~\eqref{eq:U_c}). 
From the sub-case analysis above, it always holds that $\check{u}(y^m_2) \leq \bar{u}(y^m_2)$. 
Moreover, we have $\check{u}(y^m_1)=0= \bar{u}(y^m_1)$.
Then, by the concavity and Lipschitz continuity of $\bar{u}$, we have
$$
\check{\beta}_1=\frac{\check{\alpha}_2-\check{\alpha}_1}{y^m_2-y^m_1}=\frac{\check{u}(y^m_2)-\check{u}(y^m_1)}{y^m_2-y^m_1}\leq\frac{\bar{u}(y^m_2)-\bar{u}(y^m_1)}{y^m_2-y^m_1}\leq\bar{u}^{'}_{+}(y^m_1)\leq L. 
$$

\vspace{0.1cm}
{\underline{\bf Summary.}}
Since $\check{u}$ is monotonically increasing, concave, Lipschitz continuous, and normalized, we have $\check{u}\in\mathcal{U}_{cv}$. 
Moreover, there exist $u^1_1,\ldots,u^1_{N_m},u^2_1,\ldots,u^2_{N_m}\in\mathcal{U}_m$ such that $u^1_j(y^m_{j})\leq \check{u}(y^m_j)\leq u^2_j(y^m_{j})$ for $j=1,\ldots,N_m$. 
Then, by Lemma~\ref{lemma:mus-u-U_m}, $\check{u}\in\mathcal{U}_m$. 
Recall that $\check{u}(y^m_k)=\bar{\alpha}^m_k$, thus $\check{u}\in\mathcal{U}^k_m$. 
The existence of $\check{u}\in\mathcal{U}^k_m$ with $\check{u}^{'}_{+}(y^m_k)=\frac{\bar{\alpha}^m_{k+1}-\bar{\alpha}^m_k}{y^m_{k+1}-y^m_k}>\bar{\beta}^m_k$ 
leads to a contradiction as desired. 
This shows 
$\bar{\beta}^m_k\geq\frac{\bar{\alpha}^m_{k+1}-\bar{\alpha}^m_k}{y^m_{k+1}-y^m_k}$.
\hfill $\Box$

\subsection{Proof of Lemma~\ref{lemma:mus-beta-equal}}
We only prove ``$\Longrightarrow$'' as ``$\Longleftarrow$'' can be proved analogously. 
{By 
Proposition~\ref{prop:mus-max-min-beta}}, 
the optimal solution sets $\mathop{\arg\max}\limits_{u\in\mathcal{U}^k_m} u^{'}_{+}(y^m_k)$ and $\mathop{\arg\min}\limits_{u\in\mathcal{U}^{k+1}_m} u^{'}_{-}(y^m_{k+1})$ are both non-empty. 
By Lemma~\ref{lemma:mus-beta-inequal}, we have that $\underline{\beta}^m_{k+1}=\min\limits_{u\in\mathcal{U}^{k+1}_m} u^{'}_{-}(y^m_{k+1})\leq\frac{\bar{\alpha}^m_{k+1}-\bar{\alpha}^m_{k}}{y^m_{k+1}-y^m_k}$ always holds. 
Assume for the sake of a contradiction that 
\begin{equation*}
\max\limits_{u\in\mathcal{U}^k_m} u^{'}_{+}(y^m_k)=\frac{\bar{\alpha}^m_{k+1}-\bar{\alpha}^m_{k}}{y^m_{k+1}-y^m_k}
\quad\text{and}\quad
\min\limits_{u\in\mathcal{U}^{k+1}_m} u^{'}_{-}(y^m_{k+1})<\frac{\bar{\alpha}^m_{k+1}-\bar{\alpha}^m_{k}}{y^m_{k+1}-y^m_k}. 
\end{equation*}
Then, there exists a function $\underline{u}\in\mathcal{U}^{k+1}_m$ with $\underline{u}(y^m_{k+1})=\bar{\alpha}^m_{k+1}$ such that $\underline{u}^{'}_{-}(y^m_{k+1})<\frac{\bar{\alpha}^m_{k+1}-\bar{\alpha}^m_{k}}{y^m_{k+1}-y^m_k}$.
We claim that there exists a function $\check{u}\in\mathcal{U}^k_m$ with $\check{u}(y^m_k)=\bar{\alpha}^m_k$ such that 
\begin{equation*}
\check{u}^{'}_{+}(y^m_k)>\frac{\bar{\alpha}^m_{k+1}-\bar{\alpha}^m_{k}}{y^m_{k+1}-y^m_k}=\max\limits_{u\in\mathcal{U}^k_m} u^{'}_{+}(y^m_k),
\end{equation*}
which is a contradiction.

Consider a function $\bar{u}\in\mathop{\arg\max}\limits_{u\in\mathcal{U}^k_m} u^{'}_{+}(y^m_k)$ with $\bar{u}(y^m_{k})=\bar{\alpha}^m_k$ and $\bar{u}^{'}_{+}(y^m_k)=\frac{\bar{\alpha}^m_{k+1}-\bar{\alpha}^m_{k}}{y^m_{k+1}-y^m_k}$. We construct a function: 
\begin{equation}\label{eq:lemma4-check-u}
    \check{u}(y)=\left\{ 
    \begin{aligned}
    &\underline{u}(y) && \text{for}\quad y^m_{k+1}\leq y \leq y^m_{N_m},\\
    &p(y)&& \text{for}\quad y^m_{k}\leq y < y^m_{k+1},\\
    &\check{\beta}_j (y-y^m_{j})+\check{\alpha}_j && \text{for}\quad y^m_{j} \leq y < y^m_{j+1},\;\; j=k-1,k-2,\ldots,1,\\
    \end{aligned}
    \right.
\end{equation}
where 
$$
p(y)=\min\left\{\underline{u}^{'}_{+}(y^m_k)(y-y^m_k)+\bar{\alpha}^m_k,\ \underline{u}^{'}_{-}(y^m_{k+1})(y-y^m_{k+1})+\bar{\alpha}^m_{k+1}\right\},
$$
and $\check{\beta}_j$ and $\check{\alpha}_j$ are defined in a recursive way as: 
$$
\check{\alpha}_{k-1}=\min\left\{\bar{u}(y^m_{k-1}),\ p(y^m_k)-p^{'}_{+}(y^m_k)(y^m_k-y^m_{k-1})\right\},\quad \check{\beta}_{k-1}=\frac{p(y^m_{k})-\check{\alpha}_{k-1}}{y^m_{k}-y^m_{k-1}},
$$
and for $j=k-2,k-3,\ldots,1$, $\check{\alpha}_{j}$ and $\check{\beta}_{j}$ are defined as in~\eqref{eq:lemma3-check-left}. 
By the definition, $\check{u}(y^m_{j})=\underline{u}(y^m_{j})$ for $j=k+1,k+2,\ldots,N_m$, and $\check{u}(y^m_j)=\check{\alpha}^m_j$ for $j=k-1,k-2,\ldots,1$. 
Analogous to the discussion in the proof of Lemma~\ref{lemma:mus-beta-inequal}, we have that 
$\check{u}$ is monotonically increasing, concave, and Lipschitz continuous on $[y^m_{k+1},y^m_{N_m}]$ by the corresponding properties of $\underline{u}$. 
Then, the key point to validate $\check{u}\in\mathcal{U}_{cv}$ is to examine the monotonicity, concavity, and Lipschitz continuity within the interval $[y^m_k,y^m_{k+1}]$.

Recall that $\underline{u}^{'}_{-}(y^m_{k+1})<\frac{\bar{\alpha}^m_{k+1}-\bar{\alpha}^m_{k}}{y^m_{k+1}-y^m_k}$, we have 
\begin{equation*}
    \bar{\alpha}^m_k=\frac{\bar{\alpha}^m_{k+1}-\bar{\alpha}^m_{k}}{y^m_{k+1}-y^m_k}(y^m_k-y^m_{k+1})+\bar{\alpha}^m_{k+1}<\underline{u}^{'}_{-}(y^m_{k+1})(y^m_k-y^m_{k+1})+\bar{\alpha}^m_{k+1},
\end{equation*}
indicating that $\check{u}(y^m_k)=p(y^m_k)=\bar{\alpha}^m_k$. 
By the concavity of $\underline{u}$, 
\begin{equation}\label{eq:inequal-1}
    \underline{u}^{'}_{-}(y^m_{k+1})\leq\frac{\underline{u}(y^m_{k+1})-\underline{u}(y^m_k)}{y^m_{k+1}-y^m_k}\leq\underline{u}^{'}_{+}(y^m_k),
\end{equation}
and by the fact that $\underline{u}(y^m_k)\leq\bar{\alpha}^m_k$, we then have 
\begin{equation*}
    \bar{\alpha}^m_{k+1}=\underline{u}(y^m_{k+1})=\frac{\underline{u}(y^m_{k+1})-\underline{u}(y^m_k)}{y^m_{k+1}-y^m_k}(y^m_{k+1}-y^m_{k})+\underline{u}(y^m_k)\leq\underline{u}^{'}_{+}(y^m_k)(y^m_{k+1}-y^m_{k})+\bar{\alpha}^m_k, 
\end{equation*}
indicating that $\check{u}(y^m_{k+1})=p(y^m_{k+1})=\bar{\alpha}^m_{k+1}=\underline{u}(y^m_{k+1})$. 
Thus, 
$$
\check{u}(y^m_k)=\bar{\alpha}^m_k\leq \bar{\alpha}^m_{k+1}=\check{u}(y^m_{k+1})=\underline{u}(y^m_{k+1}),
$$
implying the continuity and monotonicity of $\check{u}$ over the interval $[y^m_k,y^m_{k+1}]$. 
We then have the overall continuity and monotonicity of $\check{u}$.

Moreover, by~\eqref{eq:inequal-1} and the facts $\underline{u}(y^m_{k+1})=\bar{\alpha}^m_{k+1}$, $\underline{u}(y^m_k)\leq\bar{\alpha}^m_{k}$, we have
\begin{equation}\label{eq:inequal-2}
    \underline{u}^{'}_{+}(y^m_k)\geq\frac{\underline{u}(y^m_{k+1})-\underline{u}(y^m_k)}{y^m_{k+1}-y^m_k}\geq\frac{\bar{\alpha}^m_{k+1}-\bar{\alpha}^m_{k}}{y^m_{k+1}-y^m_k}.
\end{equation}
Then, we show in a Statement that at least one of the inequalities in~\eqref{eq:inequal-2} holds strictly, under the condition that $\underline{u}^{'}_{-}(y^m_{k+1})<\frac{\bar{\alpha}^m_{k+1}-\bar{\alpha}^m_{k}}{y^m_{k+1}-y^m_k}$.

\begin{statement}\label{stat-2}
  Given the assumption that $\underline{u}^{'}_{-}(y^m_{k+1})<\frac{\bar{\alpha}^m_{k+1}-\bar{\alpha}^m_{k}}{y^m_{k+1}-y^m_k}$, we have 
  $\underline{u}^{'}_{+}(y^m_k)>\frac{\bar{\alpha}^m_{k+1}-\bar{\alpha}^m_{k}}{y^m_{k+1}-y^m_k}.$ 
\end{statement}

\noindent {\bf Proof of the statement.} 
We apply the contradiction by assuming that $\underline{u}(y^m_k)=\bar{\alpha}^m_k$ and $\underline{u}^{'}_{+}(y^m_k)=\frac{\bar{\alpha}^m_{k+1}-\bar{\alpha}^m_{k}}{y^m_{k+1}-y^m_k}$ hold simultaneously.  
By the concavity of $\underline{u}$ and the facts $\underline{u}(y^m_k)=\bar{\alpha}^m_k$, $\underline{u}^{'}_{+}(y^m_k)=\frac{\bar{\alpha}^m_{k+1}-\bar{\alpha}^m_{k}}{y^m_{k+1}-y^m_k}$, we have that the tangent line of $\underline{u}$ at $y^m_k$ is above $\underline{u}$: 
\begin{align*}
\underline{u}(y)
\leq \underline{u}(y^m_k)+\underline{u}^{'}_{+}(y^m_k)(y-y^m_k)
=\bar{\alpha}^m_k+\frac{\bar{\alpha}^m_{k+1}-\bar{\alpha}^m_{k}}{y^m_{k+1}-y^m_k}(y-y^m_k):=L(y),\; \forall y\in[y^m_k,y^m_{k+1}],
\end{align*}
and that the secant line connecting $(y^m_k,\underline{u}(y^m_k))$ and $(y^m_{k+1},\underline{u}(y^m_{k+1}))$ is below $\underline{u}$: 
\begin{align*}
\underline{u}(y)
\geq \underline{u}(y^m_k)+\frac{\underline{u}(y^m_{k+1})-\underline{u}(y^m_k)}{y^m_{k+1}-y^m_k}(y-y^m_k)
=\bar{\alpha}^m_k+\frac{\bar{\alpha}^m_{k+1}-\bar{\alpha}^m_{k}}{y^m_{k+1}-y^m_k}(y-y^m_k):=S(y),
\end{align*}
for any $y\in[y^m_k,y^m_{k+1}]$. 
We then have $\underline{u}(y)=L(y)=S(y)$, $\forall y\in[y^m_k,y^m_{k+1}]$. 
Thus 
$$
\underline{u}^{'}_{-}(y^m_{k+1})=\frac{\bar{\alpha}^m_{k+1}-\bar{\alpha}^m_{k}}{y^m_{k+1}-y^m_k},
$$
which contradicts the condition $\underline{u}^{'}_{-}(y^m_{k+1})<\frac{\bar{\alpha}^m_{k+1}-\bar{\alpha}^m_{k}}{y^m_{k+1}-y^m_k}.$
Therefore, we have $\underline{u}^{'}_{+}(y^m_k)>\frac{\bar{\alpha}^m_{k+1}-\bar{\alpha}^m_{k}}{y^m_{k+1}-y^m_k}$. 
This completes the proof of Statement~\ref{stat-2}.
\hfill $\Box$

Then, $\check{u}$ inherits the properties from $\underline{u}$ that
\begin{equation*}
    \check{u}^{'}_{+}(y^m_k)=p^{'}_{+}(y^m_k)=\underline{u}^{'}_{+}(y^m_k)>\frac{\bar{\alpha}^m_{k+1}-\bar{\alpha}^m_{k}}{y^m_{k+1}-y^m_k}>\underline{u}^{'}_{-}(y^m_{k+1})=p^{'}_{-}(y^m_{k+1})=\check{u}^{'}_{-}(y^m_{k+1}),
\end{equation*}
which means $\check{u}(y)=p(y)$ is a two-piecewise linear concave function within $[y^m_k,y^m_{k+1}]$, with one breakpoint lying in $(y^m_k,y^m_{k+1})$. 
In addition, it is true that $\check{u}^{'}_{+}(y^m_k)>\frac{\bar{\alpha}^m_{k+1}-\bar{\alpha}^m_{k}}{y^m_{k+1}-y^m_k}=\bar{u}^{'}_{+}(y^m_k)=\max\limits_{u\in\mathcal{U}^k_m} u^{'}_{+}(y^m_k)$. 
Moreover, by the concavity of $\underline{u}$ and~\eqref{eq:lemma4-check-u}, we have 
$$
\check{u}^{'}_{-}(y^m_{k+1})=\underline{u}^{'}_{-}(y^m_{k+1})\geq\underline{u}^{'}_{+}(y^m_{k+1})=\check{u}^{'}_{+}(y^m_{k+1}),
$$
and by the definition of $\check{\beta}_{k-1}$ and $\check{\alpha}_{k-1}$, we have
\begin{align*}
\check{u}^{'}_{-}(y^m_{k})=\check{\beta}_{k-1}=\frac{p(y^m_{k})-\check{\alpha}_{k-1}}{y^m_{k}-y^m_{k-1}}
\geq \frac{p(y^m_{k})-\left[p(y^m_k)-p^{'}_{+}(y^m_k)(y^m_k-y^m_{k-1})\right]}{y^m_{k}-y^m_{k-1}}=p^{'}_{+}(y^m_k)=\check{u}^{'}_{+}(y^m_{k}). 
\end{align*}
Thus, $\check{u}^{'}_{-}(y^m_{k})\geq\check{u}^{'}_{+}(y^m_{k})>\check{u}^{'}_{-}(y^m_{k+1})\geq\check{u}^{'}_{+}(y^m_{k+1})$, and we have the overall concavity of $\check{u}$.

The satisfaction of pairwise comparison constraints, normalization and Lipschitz continuity of $\check{u}$ can be validated following a proof process analogous to Lemma~\ref{lemma:mus-beta-inequal}. We then finally confirm that $\check{u}\in\mathcal{U}_m$. 
Recall that $\check{u}(y^m_k)=\bar{\alpha}^m_k$, so $\check{u}\in\mathcal{U}^k_m$. The existence of $\check{u}\in\mathcal{U}^k_m$ such that $\check{u}^{'}_{+}(y^m_k)>\frac{\bar{\alpha}^m_{k+1}-\bar{\alpha}^m_{k}}{y^m_{k+1}-y^m_k}=\max\limits_{u\in\mathcal{U}^k_m} u^{'}_{+}(y^m_k)$ provides a contradiction, which implies the desired conclusion $\min\limits_{u\in\mathcal{U}^{k+1}_m} u^{'}_{-}(y^m_{k+1})=\frac{\bar{\alpha}^m_{k+1}-\bar{\alpha}^m_{k}}{y^m_{k+1}-y^m_k}$. 
\hfill $\Box$

\subsection{Proof of Theorem~\ref{thm:mus-upper-bound}}\label{appd-thm-upper}

Consider the $k$-th interval $[y^m_k,y^m_{k+1}]$, $k=1,\ldots,N_m-1$. By Lemmas~\ref{lemma:mus-beta-inequal}~and~\ref{lemma:mus-beta-equal}, we have one of the following two cases: 
\bgeq 
\text{(a)} \quad \bar{\beta}^m_k>\frac{\bar{\alpha}^m_{k+1}-\bar{\alpha}^m_{k}}{y^m_{k+1}-y^m_k}>\underline{\beta}^m_{k+1} 
\quad \text{or} \quad  \text{(b)}  \quad \bar{\beta}^m_k=\frac{\bar{\alpha}^m_{k+1}-\bar{\alpha}^m_{k}}{y^m_{k+1}-y^m_k}=\underline{\beta}^m_{k+1}. 
\edeq

\noindent
\underline{Case~(a).}
We can find that the two linear functions
\bgeq 
f_1(y)=\bar{\beta}^m_k(y-y_k^m)+\bar{\alpha}^m_k\quad  
\text{and}\quad 
f_2(y)=\underline{\beta}^m_{k+1}(y-y_{k+1}^m)+\bar{\alpha}^m_{k+1}
\edeq 
have a unique intersection point at $y^*_{m,k}\in (y^m_k,y^m_{k+1})$ as defined in~\eqref{eq:y^*_mk}. 

\noindent
\underline{Case~(b).}
$f_1(y)$ and $f_2(y)$ coincide over the interval $[y^m_k,y^m_{k+1}]$. 
We then choose a dummy intersection point $y^*_{m,k}=(y^m_k+y^m_{k+1})/2$.

{By the attainment result established in 
Proposition~\ref{prop:mus-max-min-beta},} 
the optimal solution sets $\mathop{\arg\max}\limits_{u\in\mathcal{U}^k_m} u^{'}_{+}(y^m_k)$ and $\mathop{\arg\min}\limits_{u\in\mathcal{U}^{k+1}_m} u^{'}_{-}(y^m_{k+1})$ are both non-empty. 
We select a function $\bar{u}\in\mathop{\arg\max}\limits_{u\in\mathcal{U}^k_m} u^{'}_{+}(y^m_k)$ with $\bar{u}(y^m_k)=\bar{\alpha}^m_k$ and $\bar{u}^{'}_{+}(y^m_k)=\bar{\beta}^m_k$, and a function $\underline{u}\in\mathop{\arg\min}\limits_{u\in\mathcal{U}^{k+1}_m} u^{'}_{-}(y^m_{k+1})$ with $\underline{u}(y^m_{k+1})=\bar{\alpha}^m_{k+1}$ and $\underline{u}^{'}_{-}(y^m_{k+1})=\underline{\beta}^m_{k+1}$. 
We then construct a function $\hat{u}$ based on $y^*_{m,k}$, $\bar{u}$ and $\underline{u}$ as:
\begin{equation}\label{eq:mus-max-utility}
    \hat{u}(y)=\left\{ 
    \begin{aligned}
    & \underline{u} (y) && \text{for} \quad y_{k+1}^m<y\leq y^m_{N_m},\\
    &\underline{\beta}^m_{k+1}(y-y_{k+1}^m)+\bar{\alpha}^m_{k+1} &&\text{for} \quad y^*_{m,k}< y \leq y_{k+1}^m,\\
    &\bar{\beta}^m_k(y-y_k^m)+\bar{\alpha}^m_k && \text{for} \quad y_k^m\leq y \leq y^*_{m,k},\\
    &\bar{u}(y) && \text{for} \quad y^m_1\leq y< y_k^m.
    \end{aligned}
    \right.
\end{equation}
We claim that 
\bgeq 
\bar{f}(r_2)=\max\limits_{u\in\mathcal{U}_m}u(r_2)=\hat{u}(r_2),\quad \forall r_2\in[y^m_k,y^m_{k+1}].
\edeq

First, we show that $\hat{u}\in\mathcal{U}_m$. 
It is obvious that $\hat{u}$ is monotonically increasing, Lipschitz continuous, concave and normalized due to the same properties shared by $\underline{u}$ and $\bar{u}$ in $\mathcal{U}_m$.
This shows $\hat{u}\in\mathcal{U}_{cv}$. 
By setting $u^1_j=u^2_j=\bar{u}$, $j=1,\ldots,k$ and $u^1_j=u^2_j=\underline{u}$, for $j=k+1,\ldots,N_m$, we assert by Lemma~\ref{lemma:mus-u-U_m} that $\hat{u}\in\mathcal{U}_m$.
Second, we show that 
\bgeq 
\hat{u}(y)\geq u(y),\quad \forall y\in[y^m_k,y^m_{k+1}]\quad \text{and}\quad
\forall 
u\in\mathcal{U}_m. 
\edeq
Since
$\hat{u}(y^m_k)=\bar{\alpha}^m_k:=\max\limits_{u\in\mathcal{U}_m}u(y^m_k)$ and $\hat{u}(y^m_{k+1})=\bar{\alpha}^m_{k+1}:=\max\limits_{u\in\mathcal{U}_m}u(y^m_{k+1})$, the inequality holds at $y^m_k$ and $y^m_{k+1}$. 
Thus, we are left to show the inequality over $(y^m_k, y^m_{k+1})$.

Assume for the sake of a contradiction that there exist $\tilde{u}\in\mathcal{U}_m$ and $\tilde{y}\in(y^m_k,y^m_{k+1})$ such that $\tilde{u}(\tilde{y})>\hat{u}(\tilde{y})$.  
Since $y_k^m< y^*_{m,k}< y_{k+1}^m$, we go through two cases: (a-1) $y^m_k<\tilde{y}\leq y^*_{m,k}$ and (b-1) $y^*_{m,k}< \tilde{y}<y^m_{k+1}$. 
In what follows, we focus only on (a-1) since (b-1) can be discussed analogously.

We claim that in this case there accordingly exists a function $\check{u}\in\mathcal{U}^k_m$ with $\check{u}(y^m_k)=\bar{\alpha}^m_k$ such that $\check{u}^{'}_{+}(y^m_k)>\bar{u}^{'}_{+}(y^m_k)=\max\limits_{u\in\mathcal{U}^k_m} u^{'}_{+}(y^m_k)$, which leads to a contradiction. 
We construct the function $\check{u}$ as:  
\begin{equation}\label{eq:check-u-part-2}
    \check{u}(y)=\left\{ 
    \begin{aligned}
    &\tilde{u}(y) && \text{for}\quad  y^m_{k+1}\leq y \leq y^m_{N_m},\\
    &\frac{\tilde{u}(y^m_{k+1})-\tilde{u}(\tilde{y})}{y^m_{k+1}-\tilde{y}}(y-y^m_{k+1})+\tilde{u}(y^m_{k+1}) && \text{for}\quad\tilde{y}\leq y < y^m_{k+1},\\
    &\frac{\tilde{u}(\tilde{y})-\bar{\alpha}^m_k}{\tilde{y}-y^m_k}(y-y^m_k)+\bar{\alpha}^m_k && \text{for}\quad y^m_{k}\leq y <\tilde{y},\\
    &\check{\beta}_j (y-y^m_{j})+\check{\alpha}_j && \text{for}\quad y^m_{j} \leq y < y^m_{j+1},\;\; j=k-1,k-2,\ldots,1,\\
    \end{aligned}
    \right.
\end{equation}
where $\check{\beta}_j$ and $\check{\alpha}_j$ are defined in a recursive way as: 
$$
\check{\alpha}_{k-1}=\min\left\{\bar{u}(y^m_{k-1}),\ \bar{\alpha}^m_k-\frac{\tilde{u}(\tilde{y})-\bar{\alpha}^m_k}{\tilde{y}-y^m_k}(y^m_k-y^m_{k-1})\right\},\quad \check{\beta}_{k-1}=\frac{\bar{\alpha}^m_k-\check{\alpha}_{k-1}}{y^m_{k}-y^m_{k-1}}, 
$$
and for $j=k-2,k-3,\ldots,1$, $\check{\beta}_j$ and $\check{\alpha}_j$ are defined as in~\eqref{eq:lemma3-check-left}. 
According to the construction above, it is evident that $\check{u}(y^m_{j})=\tilde{u}(y^m_{j})$ for $j=k+1,k+2,\ldots,N_m$, and $\check{u}(y^m_j)=\check{\alpha}^m_j$ for $j=k-1,k-2,\ldots,1$. 
$\check{u}$ is naturally monotonically increasing, concave, and Lipschitz continuous on $[y^m_{k+1},y^m_{N_m}]$ by the corresponding properties of $\tilde{u}$. 
Analogous to the discussion in the proof of Lemma~\ref{lemma:mus-beta-inequal}, we have that $\check{u}$ is monotonically increasing, concave, and Lipschitz continuous on $[y^m_1,y^m_k]$. 
Then, the key point to validate $\check{u}\in\mathcal{U}_{cv}$ is to examine the monotonicity, concavity, and Lipschitz continuity within interval $[y^m_k,y^m_{k+1}]$.

By~\eqref{eq:check-u-part-2}, we know that $\check{u}(y)$ is continuous over $[y^m_k,y^m_{k+1}]$ with $\check{u}(y^m_k)=\bar{\alpha}^m_k$, $\check{u}(\tilde{y})=\tilde{u}(\tilde{y})$ and $\check{u}(y^m_{k+1})=\tilde{u}(y^m_{k+1})$. 
Recall that $\tilde{u}(\tilde{y})>\hat{u}(\tilde{y})=\bar{\beta}^m_k(\tilde{y}-y_k^m)+\bar{\alpha}^m_k\geq\bar{\alpha}^m_k$, we have $\check{u}(y^m_k)=\bar{\alpha}^m_k\leq\tilde{u}(\tilde{y})\leq\tilde{u}(y^m_{k+1})=\check{u}(y^m_{k+1})$ by the monotonicity of $\tilde{u}$. 
Thus, we have the overall continuity and monotonicity of $\check{u}$. 
Moreover, by the definition of $\check{\alpha}_{k-1}$ and $\check{\beta}_{k-1}$, we have 
$$
\check{u}^{'}_{-}(y^m_k)=\check{\beta}_{k-1}=\frac{\bar{\alpha}^m_k-\check{\alpha}_{k-1}}{y^m_{k}-y^m_{k-1}}\geq \frac{\bar{\alpha}^m_k-\left[\bar{\alpha}^m_k-\frac{\tilde{u}(\tilde{y})-\bar{\alpha}^m_k}{\tilde{y}-y^m_k}(y^m_k-y^m_{k-1})\right]}{y^m_{k}-y^m_{k-1}}=\frac{\tilde{u}(\tilde{y})-\bar{\alpha}^m_k}{\tilde{y}-y^m_k}=\check{u}^{'}_{+}(y^m_k). 
$$
Considering that $\tilde{u}(\tilde{y})>\bar{\beta}^m_k(\tilde{y}-y_k^m)+\bar{\alpha}^m_k$, by Lemma~\ref{lemma:mus-beta-inequal}, we have $\check{u}^{'}_{+}(y^m_k)=\frac{\tilde{u}(\tilde{y})-\bar{\alpha}^m_k}{\tilde{y}-y^m_k}>\bar{\beta}^m_k\geq\frac{\bar{\alpha}^m_{k+1}-\bar{\alpha}^m_{k}}{y^m_{k+1}-y^m_k}$, and $\frac{\bar{\alpha}^m_{k+1}-\bar{\alpha}^m_{k}}{y^m_{k+1}-y^m_k}(\tilde{y}-y^m_{k})+\bar{\alpha}^m_{k}\leq\bar{\beta}^m_k(\tilde{y}-y^m_{k})+\bar{\alpha}^m_{k}$. 
Then, by the inequalities above and the fact $\tilde{u}(\tilde{y})\leq\tilde{u}(y^m_{k+1})\leq\bar{\alpha}^m_{k+1}$, we have
\begin{align*}
\check{u}^{'}_{+}(y^m_k)&=\frac{\tilde{u}(\tilde{y})-\bar{\alpha}^m_k}{\tilde{y}-y^m_k}>\frac{\bar{\alpha}^m_{k+1}-\bar{\alpha}^m_{k}}{y^m_{k+1}-y^m_k}
=\frac{\bar{\alpha}^m_{k+1}-\left[\bar{\alpha}^m_{k+1}-\frac{\bar{\alpha}^m_{k+1}-\bar{\alpha}^m_{k}}{y^m_{k+1}-y^m_k}(y^m_{k+1}-\tilde{y})\right]}{y^m_{k+1}-\tilde{y}}\\
&=\frac{\bar{\alpha}^m_{k+1}-\left[\bar{\alpha}^m_{k}+\frac{\bar{\alpha}^m_{k+1}-\bar{\alpha}^m_{k}}{y^m_{k+1}-y^m_k}(\tilde{y}-y^m_{k})\right]}{y^m_{k+1}-\tilde{y}}\\
&\geq \frac{\bar{\alpha}^m_{k+1}-\left[\bar{\beta}^m_k(\tilde{y}-y_k^m)+\bar{\alpha}^m_k\right]}{y^m_{k+1}-\tilde{y}}\\
&>\frac{\bar{\alpha}^m_{k+1}-\tilde{u}(\tilde{y})}{y^m_{k+1}-\tilde{y}}\geq\frac{\tilde{u}(y^m_{k+1})-\tilde{u}(\tilde{y})}{y^m_{k+1}-\tilde{y}}=\check{u}^{'}_{-}(y^m_{k+1}).  
\end{align*}
Further, by the concavity of $\tilde{u}$, we have 
$$
\check{u}^{'}_{-}(y^m_{k+1})=\frac{\tilde{u}(y^m_{k+1})-\tilde{u}(\tilde{y})}{y^m_{k+1}-\tilde{y}}\geq\tilde{u}^{'}_{-}(y^m_{k+1})\geq\tilde{u}^{'}_{+}(y^m_{k+1})=\check{u}^{'}_{+}(y^m_{k+1}). 
$$
Thus, $\check{u}^{'}_{-}(y^m_{k})\geq\check{u}^{'}_{+}(y^m_{k})>\check{u}^{'}_{-}(y^m_{k+1})\geq\check{u}^{'}_{+}(y^m_{k+1})$, and we have the overall concavity of $\check{u}$.

The satisfaction of pairwise comparison constraints, normalization and Lipschitz continuity can be validated following
a proof analogous to that of Lemma~\ref{lemma:mus-beta-inequal}. We then confirm that $\check{u}\in\mathcal{U}_m$. 
Recall that $\check{u}(y^m_k)=\bar{\alpha}^m_k$, so $\check{u}\in\mathcal{U}^k_m$. The existence of $\check{u}\in\mathcal{U}^k_m$ such that $\check{u}^{'}_{+}(y^m_k)>\bar{\beta}^m_k=\max\limits_{u\in\mathcal{U}^k_m} u^{'}_{+}(y^m_k)$ provides a contradiction, which implies the desired conclusion 
$\hat{u}(y)\geq u(y),\ \forall y\in[y^m_k,y^m_{k+1}],\ \forall u\in\mathcal{U}_m.$ 
\hfill $\Box$

\section{Identifying a nominal utility function}
\label{sec:utility}

As we discussed in Section~\ref{sec:MUS}, the MUS approach will effectively reduce the ambiguity set ${\cal U}_m$ to a singleton (which is the true utility function) as $m$ goes to infinity. 
In practice, however, it is unlikely to ask the DM unlimited number of questions, either because it is too expensive or because it is simply unrealistic.  
In that case, ${\cal U}_m$ remains a set. 
{While one could formulate a PRO model directly over the ambiguity set ${\cal U}_m$, as studied in~\cite{armbruster2015decision,hu2017optimization,guo2024utility}, such worst-case oriented formulations are often observed to produce rather conservative decisions in practice. In this paper, our focus is on preference elicitation rather than on constructing PRO models.} 
It 
{therefore}
might be desirable to identify a so-called {\em nominal utility function} from the set, which is treated as the best 
approximate
utility function to describe the DM's preference given the preference information obtained so far. 
For instance, in robo-advisor systems, the robo-advisor may use the available information about the DM's preferences to identify an approximate nominal utility function and make a recommendation based on it~\cite{chen2026robo}.  
A natural way is to choose the nominal utility function  that lies between 
the upper bound function $\bar{f}$ and the lower bound function $\underline{f}$, e.g., $(\bar{f}+\underline{f})/2$. Unfortunately,  
such a function does not lie in ${\cal U}_m$ because $\bar{f}\not\in {\cal U}_m$
in general.
This motivates us to 
use Kantorovich distance to work out 
``the largest'' utility function in ${\cal U}_m$ as a replacement for the upper bound function 
$\bar{f}$. 


\subsection{Smallest utility function in \texorpdfstring{${\cal U}_m$}{Um}}
\label{sec:lower-bound}

By Theorem~\ref{thm:mus-lower-bound}, the uniform lower bound function $\underline{f}$ belongs to $\mathcal{U}_m$, and thus serves as 
a good estimate of the possible smallest utility function in the set. 
By definition, we can also see that 
$\underline{f}$ is the smallest utility function in the sense of Kantorovich distance $\dd_K(\cdot,\cdot)$. The next proposition states this.

\begin{proposition}[{Uniqueness of the smallest utility function $u^S_m$}]
\label{prop:lower-bound-tight}

Let $\mathcal{U}_m$ 
be defined as in~\eqref{eq:U_m-MUS} 
and 
$\underline{f}$ be defined as in~\eqref{eq:mus-min-utility}. 
Then 
$\underline{f}$ is the unique optimal 
solution to problem 
\begin{equation}
\label{eq:smallest-u_m}
\min\limits_{u\in
\mathcal{U}_m
}\ 
\dd_{K}(u,0) := \int_{0}^{1} u(y)dy.
\end{equation}  
\end{proposition}

\begin{proof}

By Theorem~\ref{thm:mus-lower-bound}, 
$
\underline{f}\in
\mathcal{U}_m$. 
Moreover, since
$\underline{f}(r_2)=\min\limits_{u\in\mathcal{U}_m}u(r_2)$, then 
\begin{equation}
\label{eq:smallest-u_m-u_m-low}
u(y)\geq \underline{f}
(y),\;\; \forall y\in[0,1],\;\; \forall u\in\mathcal{U}_m. 
\end{equation}
Thus,
$\dd_K(u,0)=\int_0^1 u(y)dy\geq\int_0^1 
\underline{f}(y)dy=\dd_K(\underline{f}
,0)$ for all $u\in \mathcal{U}_m$, which implies that  
$\underline{f}$ is an optimal solution to problem~\eqref{eq:smallest-u_m}. To show the uniqueness of the solution, we assume for the sake of a contradiction that there exists $\hat{u}\in\mathcal{U}_m$ such that $\hat{u}\neq \underline{f}
$ and $\dd_K(\hat{u},0)=\dd_K(\underline{f}
,0)$.  Then $\hat{u}$ must satisfy~\eqref{eq:smallest-u_m-u_m-low} and there exists at least one point 
$\hat{y}\in (0,1)$ such that $\hat{u}(\hat{y})>\underline{f}(\hat{y})$. Since both 
$\hat{u}$ and $\underline{f}$ are continuous, then
there exists an open neighborhood of $\hat{y}$, written ${\cal N}(\hat{y})$, 
within $[0,1]$ such that
$\hat{u}(y)>\underline{f}
(y)$, $\forall y\in  {\cal N}(\hat{y})$.
Consequently, we have
\bgeq
\dd_K(\hat{u},0)-\dd_K(\underline{f},0)
=\int_0^1 [\hat{u}(y)-\underline{f}(y)]dy\geq 
\int_{{\cal N}(\hat{y})} 
[\hat{u}(y)-\underline{f}(y)]dy>0,
\edeq
a contradiction. 
\end{proof}

By the proposition, we can set $u^S_m:=\underline{f}$ as the smallest utility function in ${\cal U}_m$, 
which is uniquely determined by the current ambiguity set ${\cal U}_m$. 

The Kantorovich distance between $u_m^S$ and a uniform lower bound zero-valued function $0$ is essentially the area 
between 
the curve of $u_m^S$ and the horizontal axis. 
Since $u_m^S=\underline{f}$,
then 
it is also a solution to problem $\min\limits_{u\in\mathcal{U}_m} \dd_{I}(u,0)$. 
However, we cannot use the latter to define 
the smallest utility function because its solution is usually not unique. 
This is because the Kolmogorov 
distance (
{supremum}
norm) is determined 
at a single point where the two utility function values have largest gap,
without capturing the difference elsewhere, 
see~\cite[Example~1]{liu2025preference} for
an 
illustration.
 
Note that  we can compute $u^S_m=\underline{f}$ as in~\eqref{eq:mus-min-utility} where
$ \underline{\alpha}^m_k$ can be obtained 
by solving 
optimization problems~\eqref{eq:calcu-min-alpha} for $k = 1,\ldots,N_m$. 
Alternatively, we can compute the value of $ \underline{\alpha}^m_k$
for $k = 1, \ldots, N_m$ all at once by solving the following optimization problem: 
\begin{subequations}\label{eq:lower-bound-re-LP}
\begin{align}
\min\limits_{\alpha^m, \beta^m} & \sum_{j=1}^{N_m-1}\frac{1}{2}\left(y^m_{j+1}-y^m_j\right)\left(\alpha^m_{j+1}+\alpha^m_{j}\right)\\
\text{s.t.}\
& \eqref{const:mus-pairwise}-\eqref{const:mus-domain}.  
\end{align}
\end{subequations}
This is based on the fact that the lower bound $\underline{f}$ is attained by an $(N_m-1)$-piecewise linear functions: 
\begin{equation*}
\displaystyle\min\limits_{u\in \mathcal{U}_m}\int_{0}^{1} u(y)dy
\quad
\Longleftrightarrow
\quad    
\displaystyle\min\limits_{u\in \mathcal{U}_m^{\underline{\mathcal{A}}}}\int_{0}^{1} u(y)dy. 
\end{equation*}




\subsection{Largest utility function in \texorpdfstring{${\cal U}_m$}{Um}}
\label{sec:upper-bound}

As we mentioned earlier,  
$\bar{f}\not\in {\cal U}_m$
in general, we must find an alternative way to 
figure out the largest utility function. Here we do so by solving 
\begin{equation}\label{eq:upper-bound-initial-integral}
{\rm (A)} \quad \max\limits_{u\in\mathcal{U}_m}
\dd_{K}(u,0):= 
\int_{0}^{1} u(y)dy.
\end{equation}
We use $u^L_m$ to denote the
optimal solution where the superscript indicates ``largest''.

{
\begin{theorem}[Uniqueness of the largest utility function $u^L_m$]\label{thm:largest-unique}
Let 
\bgeqn 
\mathcal U_m:= \left\{u\in\mathcal U_{cv}:Z_k \cdot u(r_2^k)\geq Z_k \cdot p^k,\quad k=1,\ldots,m\right\}
\edeqn 
be defined as in~\eqref{eq:U_m-MUS}.
Then the 
program
$(A)$ in~\eqref{eq:upper-bound-initial-integral} has a unique optimal solution.
\end{theorem}


\begin{proof}
We first show existence.
Since 
$\mathcal{U}_m$ is compact under the topology of the supremum norm by 
Theorem~\ref{thm:mus-converg}, 
and the mapping $u\mapsto\int_0^1u(y)\,dy$ is continuous in the supremum norm, then
by Weierstrass theorem the maximum in~\eqref{eq:upper-bound-initial-integral} is attainable.


\vspace{0.1cm}
Next, we show uniqueness. 
Assume for the sake of a contradiction 
that there exist two distinct optimal solutions
\(u_1,u_2\in\mathcal U_m\). 
Let 
\(w=(u_1+u_2)/2\). 
The convexity of set $\mathcal U_m$
ensures that 
$w\in\mathcal U_m$,
and the linearity of the objective function in problem~\eqref{eq:upper-bound-initial-integral} implies that
$w$ is also optimal. 
In what follows, we will 
provide a constructive proof showing 
that
$u_1=u_2=w$ over $[0,1]$ 
and hence a contradiction.
The thrust of the proof is to 
construct a specific utility function 
in ${\cal U}_m$ such that the function
attains a greater objective value than 
$w$, contradicting the optimality of $w$.
We proceed the rest of proof in three main steps.
We adopt the convention 
that 
\(v'_-(0):=L\) and \(v'_+(1):=0\) for 
$v=u_1,u_2,w$. 


\vspace{0.2cm}
\noindent
\uline{Step~1. Construct a function 
$T\in {\cal U}_{cv}$ which majorizes
$w$ over $[0,1]$.}
To this end, we
consider 
 a specific subset of
$\{r_2^k, k=1,\cdots,m\}$:
\bgeqn 
S_w:=\{0,1\}\cup\left\{r_2^k: k\in\{1,\ldots,m\},\ Z_k=-1,\ 0<r_2^k<1,\ w(r_2^k)=p^k \right\}, 
\edeqn
where, by the definition of ${\cal U}_m$, $Z_k=-1$ implies $u(r_2^k)\leq p^k$ for all $u\in\mathcal{U}_m$.  The cardinality of $S_w$ is at most $m+2$.
We use the set of
points 
to divide the interval $[0,1]$ into 
mutually exclusive 
sub-intervals.
Let  $[a,b]$
be the first subinterval 
such that $u_1\neq u_2$,
i.e., $\exists y\in(a,b),\ u_1(y)\ne u_2(y)$.
Observe that $u_1(a)=w(a)=u_2(a)
$ and $u_1(b)=w(b)=u_2(b)$.
Define
\begin{equation}
\label{eq:thm7-T}
T(y):=
\begin{cases}
\min\{w(a)+
w'_-(a)(y-a),\ w(b)+
w'_+(b)(y-b)\},& \; \text{for}\quad y\in(a,b),\\
w(y),&\; \text{for}\quad y\notin(a,b).
\end{cases}
\end{equation}
By the concavity of 
$w$, $T(y)\geq w(y)$ for all $y\in (a,b)$.
Moreover, it is easy to verify that 
\(T\in\mathcal U_{cv}\). We skip the details.

\vspace{0.2cm}
\noindent
\uline{Step~2. Show that $T$ coincides with $w$ over $[a,b]$.} We consider two cases.

\begin{itemize}
   
\item Case~1: $T\in \mathcal{U}_m$.
In this case, $T(r_2^k)\leq p^k$, for all $k$ 
with 
$Z_k=-1$ and $r_2^k\in(a,b)$. 
If $T> w$ at some $y\in[a,b]$, then 
by 
continuity 
\[
\int_0^1 T(y)\,dy >\int_0^1w(y)\,dy,
\]
contradicting the optimality of $w$.
Thus, we must have 
\bgeqn 
\label{eq:stat-t-w}
T(y)= w(y), \forall y\in [a,b].
\edeqn 

\item Case~2: $T\notin \mathcal{U}_m$.
In this case, there exists some $r_2^k\in (a,b)$ such that 
$
T(r_2^k)>p^k
>w(r_2^k)$ and $Z_k=-1$.
We will  show that equality in~\eqref{eq:stat-t-w} also holds. To this end, 
for $\varepsilon\in(0,1)$, 
we construct a new concave function which lies between $T$ and $w$: 
\[
T_\varepsilon(y):=
\begin{cases}
w(y)+\varepsilon[T(y)-w(y)],& \text{for}\quad y\in(a,b),\\
w(y),&\text{for}\quad y\notin(a,b).
\end{cases}
\]
It is easy to see that  $T\geq T_\varepsilon \geq w$ on $(a,b)$ and $T= T_\varepsilon = w$ on $[0,1]\setminus(a,b)$. Moreover, by the convexity of $\mathcal{U}_{cv}$,
$T_\varepsilon=(1-\varepsilon)w+\varepsilon T\in\mathcal U_{cv}$. 
Since $w\in\mathcal{U}_m$, 
for each \(k\in\{1,\ldots,m\}\) with \(Z_k=1\), 
we have $T_\varepsilon(r_2^k)\geq w(r_2^k)\geq p^k$. 
If \(Z_k=-1\), then \(T_\varepsilon(r_2^k)=w(r_2^k)\le p^k\) for \(r_2^k\in [0,1]\backslash (a,b)\). 
For \(r_2^k\in(a,b)\), the strict 
inequality $w(r_2^k)<p^k$ allows us 
to set $\varepsilon$ sufficiently small such that 
\bgeqn 
T_\varepsilon(r_2^k)
=w(r_2^k)+\varepsilon[T(r_2^k)-w(r_2^k)]
\le p^k.
\edeqn 
Consequently, $T_\varepsilon$ satisfies all the pairwise comparison constraints in ${\cal U}_m$, so $T_\varepsilon\in\mathcal{U}_m$.  
If $T\neq w$, then
continuity and $T\geq w$ imply
\[
\int_0^1T_\varepsilon(y)\,dy
=\int_0^1w(y)\,dy+\varepsilon\int_a^b[T(y)-w(y)]\,dy
>\int_0^1w(y)\,dy,
\]
contradicting the optimality of $w$.
Again, this shows~\eqref{eq:stat-t-w} holds. 
\end{itemize}

\vspace{0.1cm}
\noindent
\uline{Step~3. Show that 
equality in~\eqref{eq:stat-t-w} implies
$u_1=u_2=w$ on $[a,b]$, and hence a contradiction.}  
The equality in~\eqref{eq:stat-t-w} implies that $w$
is also piecewise linear over $[a,b]$ as $T$. 
We consider two cases.  

\begin{itemize}
\item 
Case~1:  
The two affine functions in~\eqref{eq:thm7-T} coincide or intersect at endpoint $a$ or $b$. 
In this case, $T$ is linear on \([a,b]\) and so is $w$.
By concavity and the common endpoint values, for \(i=1,2\),
\bgeqn 
\label{eq:u_i=w}
u_i(y)\geq
\frac{b-y}{b-a}u_i(a)+\frac{y-a}{b-a}u_i(b)
=
\frac{b-y}{b-a}w(a)+\frac{y-a}{b-a}w(b)
=w(y),\;\; 
\forall y\in[a,b].
\edeqn 
Since \((u_1(y)+u_2(y))/2=w(y)\) for all $y\in [0,1]$, 
equality must hold in~\eqref{eq:u_i=w}.
Thus, \(u_1=u_2=w\) on \([a,b]\).

\item 
Case~2: The two affine functions in~\eqref{eq:thm7-T} 
intersect at a unique point $c\in(a,b)$. 
In this case, 
\(T\) and \(w\) are piecewise linear with two pieces on \([a,b]\).
Since \([a,b]\) is the first subinterval determined by \(S_w\) on which \(u_1\) and \(u_2\) differ, 
we have $u_1=u_2=w$ on $[0,a]$. 
Consequently, $u_1(a)=u_2(a)=w(a)$ and $u'_{1,-}(a)=u'_{2,-}(a)=w'_-(a)$. 
By concavity, 
\bgeqn 
\label{eq:u_i=w-case 2}
u_i(y)\leq u_i(a)+u'_{i,-}(a)(y-a)
=w(a)+w'_-(a)(y-a)=w(y),\;\; 
\forall y\in[a,c], 
\edeqn 
for $i=1,2$.
Since \((u_1(y)+u_2(y))/2=w(y)\) for all $y\in [0,1]$, then equality must hold in~\eqref{eq:u_i=w-case 2}.
In particular, \(u_1(c)=u_2(c)=w(c)\). Together with \(u_1(b)=u_2(b)=w(b)\), concavity implies that
\[
u_i(y)\geq
\frac{b-y}{b-c}u_i(c)+\frac{y-c}{b-c}u_i(b)
=
\frac{b-y}{b-c}w(c)+\frac{y-c}{b-c}w(b)
=w(y),\;\;  
\forall y\in[c,b],
\]
for $i=1,2$, where the last equality follows from the linearity of \(w\) on \([c,b]\). 
Using the fact that \((u_1(y)+u_2(y))/2=w(y)\), for all $y\in [0,1]$,
we obtain
$
u_1(y)=u_2(y)=w(y),\forall y\in[c,b].
$

\end{itemize}
Summarizing 
the discussions in the two cases,
we conclude that \(u_1=u_2=w\) on \([a,b]\), contradicting the choice of \([a,b]\) as the first subinterval on which \(u_1\) and \(u_2\) differ. 
This shows that equality in~\eqref{eq:stat-t-w} implies
$u_1=u_2=w$ as desired.
\end{proof}
}

The remainder of this section 
{focuses}
on developing a tractable formulation for solving~\eqref{eq:upper-bound-initial-integral}. 
Our basic idea is to show that~\eqref{eq:upper-bound-initial-integral} is equivalent to a program with the same objective but with the feasible set restricted to a specific class of piecewise linear utility functions within ${\cal U}_m$, and then to reformulate the latter into a program with a linear objective and second-order cone constraints.

Let $\mathcal{U}_m$ be defined as in~\eqref{eq:U_m-MUS}, and define $\mathbb{Y}_m=\{y^m_j\}_{j=1,\ldots,N_m}$ as in~\eqref{eq:Y_m}. 
For each $u\in\mathcal{U}_m$, let 
$$
\hat{u}_j(y):=u^{'}_{+}(y^m_j)(y-y^m_j)+u(y^m_j)
$$
be a tangent line of $u$ at $y^m_j$ based on the right derivative of $u$ at $y^m_j$, for $j=1,\ldots,N_m$, where $u^{'}_{+}(y^m_{N_m}):=0$ and $\hat{u}_{N_m}:=1$. 
By the monotonicity, concavity, and Lipschitz continuity of $u$, we know that 
$$
L\geq u^{'}_{+}(y^m_{j})\geq u^{'}_{+}(y^m_{j+1})\geq0, \; {\rm for }\; 
j=1,\ldots,N_m-1. 
$$
Consequently, the two adjacent tangent lines $\hat{u}_j(y)$ and $\hat{u}_{j+1}(y)$ either intersect at a unique point within the interval $[y^m_j,y^m_{j+1}]$, when $u^{'}_{+}(y^m_{j})> u^{'}_{+}(y^m_{j+1})$, or coincide in the interval $[y^m_j,y^m_{j+1}]$, when $u^{'}_{+}(y^m_{j})= u^{'}_{+}(y^m_{j+1})$. 
In the latter case, we set a dummy intersection point $(y^m_j+y^m_{j+1})/2$ without affecting the piecewise linear structure in~\eqref{eq:PUA}. Based on the above discussions, we denote the intersection point of the two adjacent tangent lines $\hat{u}_j(y)$ and $\hat{u}_{j+1}(y)$ by $z_{j+1}$ for $j=1,\ldots,N_m-1$, where $\hat{u}_j(z_{j+1})=\hat{u}_{j+1}(z_{j+1})$. 
Let $z_1=0$, $z_{N_m+1}=1$ and define $\mathbb{Z}:=\{{z}_j\}_{j=1,\ldots,{N}_m+1}$, satisfying 
\begin{equation}\label{eq:breakpoints-z}
    y^m_j\leq z_{j+1}\leq y^m_{j+1},\;\; j=1,\ldots,N_m-1,\quad 0\leq z_j\leq z_{j+1}\leq1,\;\; j=1,\ldots,N_m. 
\end{equation}

\begin{definition}[Piecewise-Linear Upper Approximation (PUA)]\label{def:PUA}
Let $\mathcal{U}_m$ and $\mathbb{Y}_m=\{y^m_j\}_{j=1,\ldots,N_m}$ be defined as in~\eqref{eq:U_m-MUS} and~\eqref{eq:Y_m} respectively.
Define operator 
$\overline{\mathcal{A}}:\mathcal{L}^1([0,1])\to\mathcal{L}^1([0,1])$ 
which maps each $u\in\mathcal{U}_m$ to a piecewise-linear upper approximation (PUA) as follows{\rm :}
\begin{equation}\label{eq:PUA}
    \overline{\mathcal{A}}u(y)=\left\{ 
    \begin{aligned}
    & 0 && \text{\rm for}\quad y=z_1,\\
    & u^{'}_{+}(y^m_j)(y-y^m_j)+u(y^m_j) && \text{\rm for}\quad {z}_j< y \leq{z}_{j+1},\;\; j=1,\ldots,N_m,\\
    & 1 &&\text{\rm for}\quad 
    y={z}_{N_m+1},
    \end{aligned}
    \right.
\end{equation}
where $z_1=0$, $z_{N_m+1}=1$ and 
for
$j=1,\ldots,N_m-1$, 
\begin{equation*}
z_{j+1}
=\left\{\begin{aligned}
& \frac{u(y^m_{j+1})-u(y^m_j)-u^{'}_{+}(y^m_{j+1})y^m_{j+1}+u^{'}_{+}(y^m_j)y^m_j}{u^{'}_{+}(y^m_j)-u^{'}_{+}(y^m_{j+1})} && \text{\rm if}\;\; u^{'}_{+}(y^m_{j})> u^{'}_{+}(y^m_{j+1}),\\
& \frac{1}{2}(y^m_j+y^m_{j+1}) && \text{\rm if}\;\; u^{'}_{+}(y^m_{j})=u^{'}_{+}(y^m_{j+1}).
\end{aligned}\right.
\end{equation*}
Figure~\ref{fig:N-N+1} 
illustrates what $ \overline{\mathcal{A}}u$ may look like.
Let $\mathcal{U}_m^{\overline{\mathcal{A}}}:=\{\overline{\mathcal{A}}u: u\in\mathcal{U}_m\}$ denote the set of all piecewise-linear upper approximation utility functions in ${\cal U}_m$.  
\end{definition}


\begin{theorem}
\label{thm:upper-bound-re-N+1}

Let 
$\mathcal{U}_m^{\overline{\mathcal{A}}}$ be defined as above, and $\mathcal{U}_m^{\underline{\mathcal{A}}}$ be defined as in Definition~\ref{def:PLU}. 
Consider programs 
\bgeq 
{\rm (B)}\quad
  \displaystyle\max\limits_{u\in 
  \mathcal{U}_m^{\overline{\mathcal{A}}}
  }\int_{0}^{1} u(y)dy
  \qquad {\rm \rm and}
  \qquad
{\rm (C)}\quad
\displaystyle\max\limits_{u\in 
\mathcal{U}_m^{\overline{\mathcal{A}}},\ 
\underline{u}\in
\mathcal{U}_m^{\underline{\mathcal{A}}}
}&\; \displaystyle\int_{0}^{1} u(y)dy \\
 {\rm s.t.}&\; u(y_j^m)=\underline{u}(y_j^m),\;\; j=1,\ldots,N_m.
\edeq
Then the following assertions hold.

\begin{enumerate}

\item[(i)] $\mathcal{U}_m^{\overline{\mathcal{A}}}
\subseteq\mathcal{U}_m$.

\item[(ii)] {\rm (A)} is equivalent to {\rm (B)}
in 
terms of the optimal values and optimal solutions.

\item[(iii)] {\rm (B)} is equivalent to {\rm (C)} in 
terms of the optimal values and optimal solutions.

\end{enumerate}
\end{theorem}

\begin{proof}

\uline{Part~(i).}
By the definition, $\overline{\mathcal{A}}u(y^m_j)=u(y^m_j)$, for $j=1,\ldots,N_m$. 
Note that $\overline{\mathcal{A}}u$ preserves monotonicity, concavity, and Lipschitz continuity of $u$. 
We select $u^1_j=u^2_j=u$ for $j=1,\ldots,N_m$, then it follows by Lemma~\ref{lemma:mus-u-U_m} that $\overline{\mathcal{A}}u\in\mathcal{U}_m$.

\vspace{0.1cm}
\uline{Part~(ii).}
Denote the optimal 
{values}
of {\rm (A)} and {\rm (B)} by $v_A$ and $v_B$, respectively. 
We show $v_A=v_B$. 
Since $\mathcal{U}_m^{\overline{\mathcal{A}}}\subseteq\mathcal{U}_m$, we have $v_A\geq v_B$.
Conversely, for any ${u}\in\mathcal{U}_m$, the concavity and Lipschitz continuity of ${u}$ imply that $0\leq{u}^{'}_{+}(y^m_{j+1})\leq{u}^{'}_{+}(y)\leq{u}^{'}_{+}(y^m_{j})\leq L$ for any $y\in[y^m_j,y^m_{j+1}]$, $j=1,\ldots,N_m-1$. 
Thus, $\overline{\mathcal{A}}{u}(y)\geq{u}(y)$, $\forall y\in[y^m_j,y^m_{j+1}]$, $j=1,\ldots,N_m-1$, which means that $\overline{\mathcal{A}}{u}$ provides an upper bound of ${u}$ over $[0,1]$. 
Taking $\tilde{u}$ as an optimal solution of {\rm (A)}, we conclude from the above that $v_B\geq\int_{0}^{1} \overline{\mathcal{A}}\tilde{u}(y)dy\geq\int_{0}^{1} \tilde{u}(y)dy=v_A$. 
See Figure~\ref{fig:N-N+1} for a geometric illustration.

\begin{figure}[htbp]
    \centering
    \subfigure[
    $u$ is smooth at $y_j^m$
    ]{
    \label{fig:N-N+1-smooth}
    \includegraphics[width=0.4\linewidth]{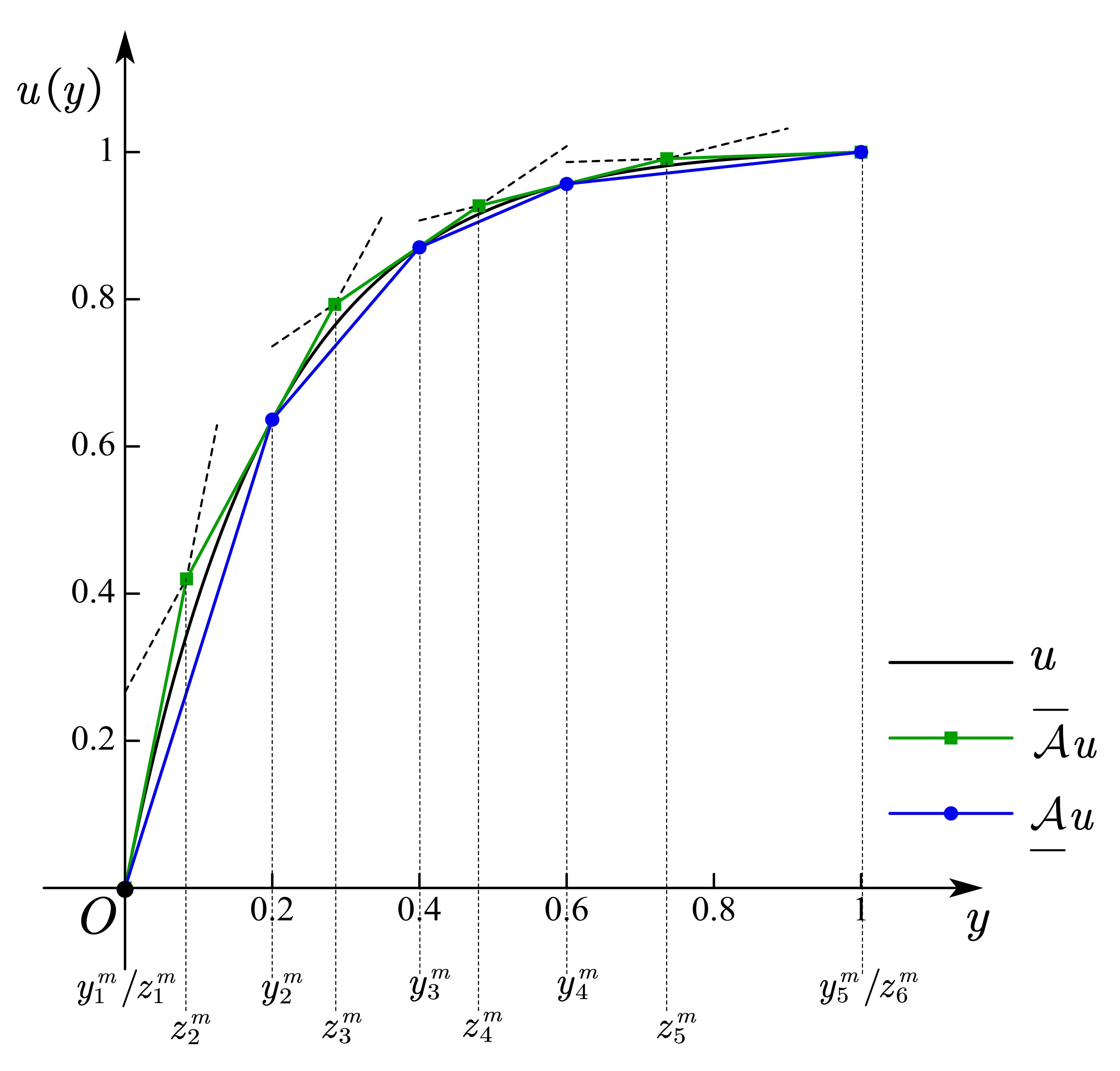}
    }
    \quad\quad
    \subfigure[
    $u$ is non-smooth at $y_j^m$ 
    ]{
    \label{fig:N-N+1-nonsmooth}
    \includegraphics[width=0.4\linewidth]{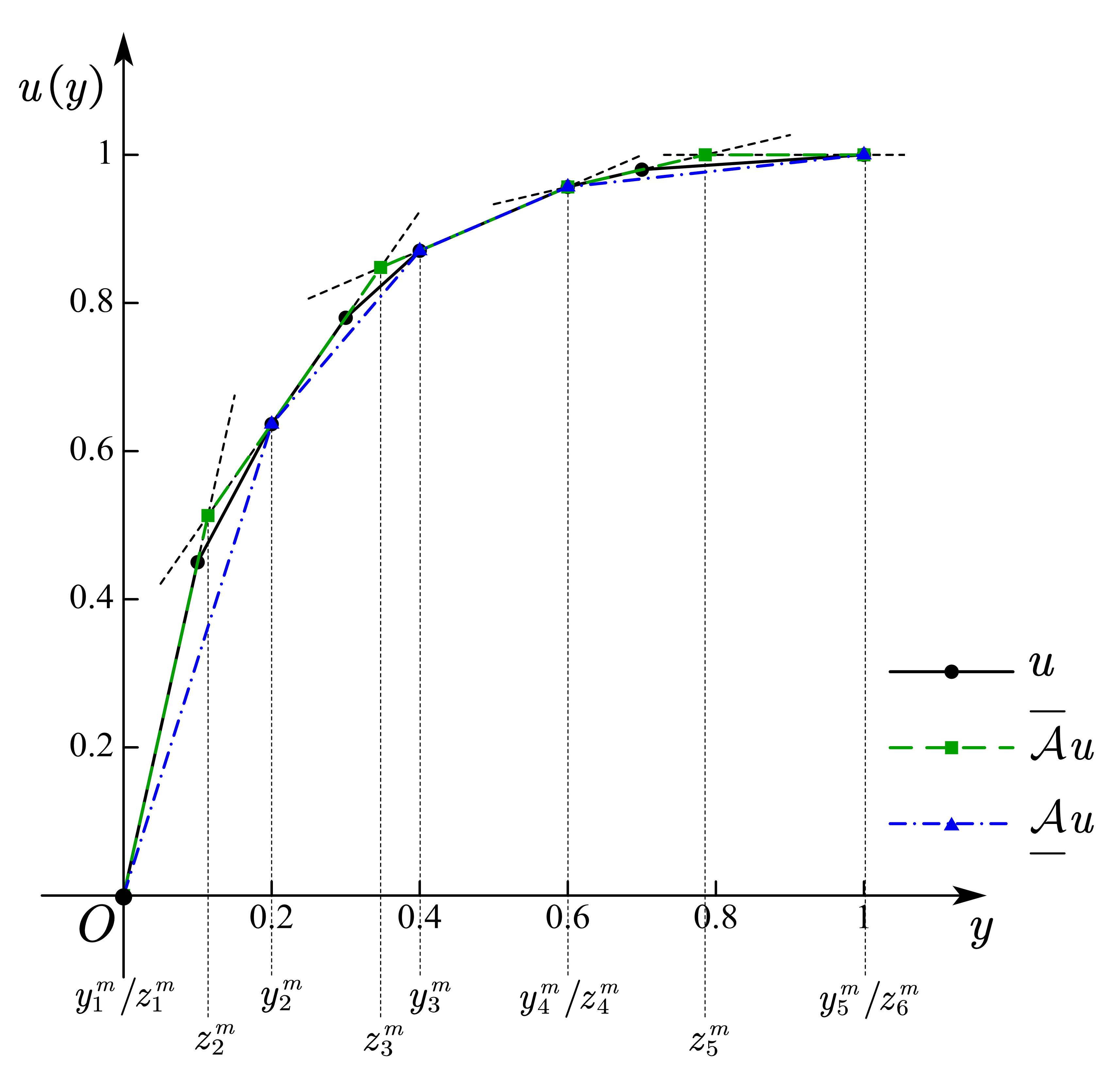}
    }
    \caption{Geometric interpretation of Theorem~\ref{thm:upper-bound-re-N+1} with $N_m=5$. 
    We can observe that $\underline{\mathcal{A}}u=\underline{\mathcal{A}}\overline{\mathcal{A}}u=\overline{\mathcal{A}}\underline{\mathcal{A}}u$. 
    }
    \label{fig:N-N+1}
\end{figure}

\vspace{0.1cm}
\uline{Part~(iii).}
For any $u\in\mathcal{U}_m^{\overline{\mathcal{A}}}$, consider its PLA $\underline{\mathcal{A}}u$ as $\underline{u}$, which satisfies $u(y_j^m)= \underline{u}(y_j^m)$ for $j=1,\cdots,N_m$ by definition, then $\underline{u}\in\mathcal{U}_m^{\underline{\mathcal{A}}}$. 
Thus, the constraints $\underline{u}\in \mathcal{U}_m^{\underline{\mathcal{A}}}$ with $u(y_j^m)=\underline{u}(y_j^m)$, for $j=1,\ldots,N_m$ do not impose 
{any}
additional constraint on the feasible set $\mathcal{U}_m^{\overline{\mathcal{A}}}$, which means that programs (B) and (C) have an identical feasible set with respect to variable $u$.
The equivalence is therefore straightforward. 
\end{proof}


By Theorem~\ref{thm:upper-bound-re-N+1}, we can solve program~(A) by solving program~(C).
For $\overline{\mathcal{A}}u\in\mathcal{U}_m^{\overline{\mathcal{A}}}$, let ${\alpha}_j:=\overline{\mathcal{A}}u(
{z_j}
)
$, 
$j=1,\ldots,N_m+1$, ${\beta}_j=({\alpha}_{j+1}-{\alpha}_{j})/(z_{j+1}-z_j)$, $j=1,\ldots,N_m$, and define 
\begin{equation*}
    \overline{\mathcal{A}}u(z)
    =\left\{ 
    \begin{aligned}
    &0 && \text{for}\quad z= {z}_{1}=0,\\
    &{\beta}_j(z-{z}_j)+{\alpha}_j && \text{for}\quad {z}_j< z \leq{z}_{j+1},\ j=1,\ldots,N_m,\\
    & 1 && \text{for}\quad z={z}_{N_m+1}=1.
    \end{aligned}
    \right.
\end{equation*}
Then we can write down program~(C) in detail: 
\begin{subequations}\label{eq:upper-bound-re-1}
\begin{align}
\max\limits_{\alpha, \beta,z, {\alpha}^m, {\beta}^m}\
& \sum_{j=1}^{N_m}\frac{1}{2}(z_{j+1}-z_j)(\alpha_{j+1}+\alpha_{j}) \\ 
\text{s.t.}\ 
\label{const:upper-bound-re-1-N+1-passing score}
& \alpha_{j+1}-{\alpha}^m_{j}=\beta_{j}(z_{j+1}-{y}^m_j),\ j=1,\ldots,{N}_m,\\
\label{const:upper-bound-re-1-N+1-end point}
&z_{1}=0,\ z_{{N}_m+1}=1,\\
\label{const:upper-bound-re-1-N+1-z}
& y^m_j\leq z_{j+1}\leq y^m_{j+1},\ j=1,\ldots,N_m-1,\\
\label{const:upper-bound-re-1-N+1-slope}
& \beta_{j}=(\alpha_{j+1}-\alpha_j)/(z_{j+1}-z_j),\ j=1,\ldots,{N}_m,\\
\label{const:upper-bound-re-1-N+1-concave}
& 0\leq\beta_{j+1}\leq \beta_{j}\leq L,\ j=1,\ldots,{N}_m-1,\\
& \alpha_{1} =0,\ \alpha_{N_m+1}=1, 
\label{const:upper-bound-re-1-N+1-domain}
\ \alpha \in \mathbb{R}^{{N}_m+1},\ \beta\in\mathbb{R}_+^{{N}_m},\\
\label{const:upper-bound-re-1-N-pairwise}
& \eqref{const:mus-pairwise}-\eqref{const:mus-domain}, 
\end{align}
\end{subequations}
where~\eqref{const:upper-bound-re-1-N+1-passing score} guarantees that the optimal solution shares the same value with $\underline{u}\in\mathcal{U}_m^{\underline{\mathcal{A}}}$ at ${y}^m_j,\ j=1,\ldots,{N}_m$; 
\eqref{const:upper-bound-re-1-N+1-end point}~and~\eqref{const:upper-bound-re-1-N+1-z} ensure the unknown breakpoints $\{z_j\}_{j=1,\ldots,N_m+1}$ satisfy~\eqref{eq:breakpoints-z}; 
\eqref{const:upper-bound-re-1-N+1-slope}--\eqref{const:upper-bound-re-1-N+1-domain} specify the basic shape of the largest utility function $u$ in a piecewise-linear form constructed by $\alpha$ and $\beta$, as described in $\mathcal{U}_{cv}$; 
\eqref{const:upper-bound-re-1-N-pairwise} reflects that $\underline{u}\in\mathcal{U}_m^{\underline{\mathcal{A}}}$.

\begin{proposition}\label{prop:upper-bound-final}
Problem~\eqref{eq:upper-bound-initial-integral}
can be reformulated as a second-order cone-constrained program{\rm :}
\begin{subequations}\label{eq:upper-bound-step2-re-socp}
\begin{align}
\max\limits_{\beta,\lambda,{\alpha}^m,{\beta}^m}\
& \frac{1}{2}\left(2-\beta_{{N}_m}-\sum_{j=2}^{{N}_m}\lambda_j\right)\\
{\rm s.t.}\ &  
\left[\begin{array}{c}
     2\left({\alpha}^m_{j}-{\alpha}^m_{j-1}+\beta_{j-1}{y}^m_{j-1}-\beta_{j}{y}^m_{j}\right)\\
     \lambda_j-(\beta_{j-1}-\beta_{j})\\
     \lambda_j+(\beta_{j-1}-\beta_{j})
\end{array}\right] \in \text{\rm SOC}_{3},\ j=2,\ldots,{N}_m,\\ 
& {\beta}^m_{j}\leq \beta_{j}\leq {\beta}^m_{j-1},\ j=2,\ldots,{N}_m-1,\\
& \beta_{1}\geq{\beta}^m_{1},\ \beta_{N_m}\leq {\beta}^m_{N_m-1},\\
& 0\leq\beta_{j+1}\leq \beta_{j}\leq L,\ j=1,\ldots,{{N}_m-1},\\
& \beta\in\mathbb{R}_+^{N_m},\ 
\lambda\in\mathbb{R}_+^{{N}_m-1},\\
& \eqref{const:mus-pairwise}-\eqref{const:mus-domain}, 
\end{align}
\end{subequations}
where $\mathrm{SOC}_3 := \left\{ v \in \mathbb{R}^3 \,\middle|\, \left\| [v_1, v_2]^\top \right\| \le v_3 \right\}$ denotes the 3-dimensional second-order cone.
\end{proposition}

\begin{proof}

We start from the reformulation~\eqref{eq:upper-bound-re-1} of~\eqref{eq:upper-bound-initial-integral}. 
From~\eqref{const:upper-bound-re-1-N+1-passing score} and~\eqref{const:upper-bound-re-1-N+1-slope}, we obtain the following relations: 
\begin{align}
& z_j = \frac{{\alpha}^m_{j}-{\alpha}^m_{j-1}+\beta_{j-1}{y}^m_{j-1}-\beta_{j}{y}^m_j}{\beta_{j-1}-\beta_{j}},\;\; j=2,\ldots,{N}_m,
\label{eq:closed-z}\\
& \alpha_{j} = {\alpha}^m_{j-1}+\beta_{j-1}\frac{{\alpha}^m_{j}-{\alpha}^m_{j-1}-\beta_{j}{y}^m_j+\beta_{j}{y}^m_{j-1}}{\beta_{j-1}-\beta_{j}},\;\; j=2,\ldots,{N}_m. 
\label{eq:closed-alpha}
\end{align}
Then, \eqref{const:upper-bound-re-1-N+1-z} is equivalent to
\begin{equation*}
y^m_{j-1}\leq \frac{{\alpha}^m_{j}-{\alpha}^m_{j-1}+\beta_{j-1}{y}^m_{j-1}-\beta_{j}{y}^m_j}{\beta_{j-1}-\beta_{j}} \leq y^m_j,\;\; j=2,\ldots,{N}_m,
\end{equation*}
i.e.,  
\begin{equation}\label{eq:closed-beta-const}
{\beta}^m_{j}\leq \beta_{j}\leq {\beta}^m_{j-1},\;\; j=2,\ldots,{N}_m-1,\quad \beta_{1}\geq{\beta}^m_{1},\quad \beta_{N_m}\leq {\beta}^m_{N_m-1}. 
\end{equation}
Substituting equations~\eqref{eq:closed-z}, \eqref{eq:closed-alpha}~and~\eqref{eq:closed-beta-const} into problem~\eqref{eq:upper-bound-re-1}, it can be reformulated as 
\begin{subequations}
\label{eq:upper-bound-step2-re-frac}
\begin{align}
\max\limits_{\beta,{\alpha}^m,{\beta}^m}\
& \frac{1}{2}\left[2-\beta_{{N}_m}-\sum_{j=2}^{{N}_m}\frac{\left({\alpha}^m_{j}-{\alpha}^m_{j-1}+\beta_{j-1}{y}^m_{j-1}-\beta_{j}{y}^m_{j}\right)^2}{\beta_{j-1}-\beta_{j}}\right]\\
\text{s.t.}\ 
\label{const:beta-z}
& {\beta}^m_{j}\leq \beta_{j}\leq {\beta}^m_{j-1},\ j=2,\ldots,{N}_m-1,\\
& \beta_{1}\geq{\beta}^m_{1},\ \beta_{N_m}\leq {\beta}^m_{N_m-1},\\
& 0\leq\beta_{j+1}\leq \beta_{j}\leq L,\ j=1,\ldots,{{N}_m-1},\\
& \eqref{const:mus-pairwise}-\eqref{const:mus-domain}, \\
\label{const:N-domain}
& \beta\in\mathbb{R}_+^{N_m}.  
\end{align}
\end{subequations}
Since~\eqref{eq:upper-bound-step2-re-frac} is a fractional programming
{ problem}, we introduce auxiliary variables $\lambda_j$, $j=2,\ldots,{N}_m$ and formulate~\eqref{eq:upper-bound-step2-re-frac} as
\begin{subequations}\label{eq:upper-bound-step2-re-qcqp}
\begin{align}
\max\limits_{\beta,\lambda,{\alpha}^m,{\beta}^m}\
& \frac{1}{2}\left(2-\beta_{{N}_m}-\sum_{j=2}^{{N}_m}\lambda_j\right)\\
\label{eq:const--qcqp}
\text{s.t.}\ 
&\left({\alpha}^m_{j}-{\alpha}^m_{j-1}+\beta_{j-1}{y}^m_{j-1}-\beta_{j}{y}^m_{j}\right)^2\leq \lambda_j\left(\beta_{j-1}-\beta_{j}\right),\ j=2,\ldots,{N}_m,\\
& \eqref{const:beta-z}-\eqref{const:N-domain},\\
& {\lambda\in\mathbb{R}_+^{{N}_m-1}.}
\end{align}
\end{subequations}
The constraint~\eqref{eq:const--qcqp} is equivalent to 
$$
4({\alpha}^m_{j}-{\alpha}^m_{j-1}+\beta_{j-1}{y}^m_{j-1}-\beta_{j}{y}^m_{j})^2 + (\lambda_j-(\beta_{j-1}-\beta_{j}))^2\leq (\lambda_j+(\beta_{j-1}-\beta_{j}))^2, \;\; j=2,\ldots,{N}_m, 
$$
which can be rewritten as the following second-order cone constraint, leading to the desired conclusion: 
$$
\left[\begin{array}{c}
     2\left({\alpha}^m_{j}-{\alpha}^m_{j-1}+\beta_{j-1}{y}^m_{j-1}-\beta_{j}{y}^m_{j}\right)\\
     \lambda_j-(\beta_{j-1}-\beta_{j})\\
     \lambda_j+(\beta_{j-1}-\beta_{j})
\end{array}\right] \in \text{SOC}_{3},\;\; j=2,\ldots,{N}_m. 
$$
The proof is complete.
\end{proof}

Denote the $\beta$-component of the optimal solution to problem~\eqref{eq:upper-bound-step2-re-socp} by $\beta^L_m=[\beta^{L}_{m,1},\ldots,\beta^{L}_{m,{N}_m}]$. 
Then we can use it to calculate optimal $z$ and $\alpha$ via equations~\eqref{eq:closed-z}~and~\eqref{eq:closed-alpha} and denote them by $\mathbb{Z}_m=\{z^{m}_j\}_{j=1,\ldots,{N}_m+1}$ and $\alpha^{L}_m=[\alpha^{L}_{m,1},\ldots,\alpha^{L}_{m,{N}_m+1}]$, respectively. 
We can then use them to construct a piecewise linear utility function
\begin{equation}\label{eq:largest-u}
    u^L_m(y)=\left\{ 
    \begin{aligned}
    &0 && \text{for}\quad y= z^{m}_{1},\\
    &\beta^{L}_{m,j}(y-{z}^{m}_j)+\alpha^{L}_{m,j} && \text{for} \quad {z}^{m}_j< y \leq{z}^{m}_{j+1},\;\; j=1,\ldots,N_m,\\
    & 1 && \text{for}\quad y={z}^{m}_{N_m+1},\\
    \end{aligned}
    \right.
\end{equation}
which is the tight upper bound for the utility function in $\mathcal{U}_m$ under the Kantorovich norm. 
Unlike $u_m^S$ in~\eqref{eq:mus-min-utility}, 
$u_m^L$ is in general 
not necessarily a point-wise upper bound for all utility functions in $\mathcal{U}_m$, 
that is, there might exist a utility function $\hat{u}\in\mathcal{U}_m$ and a point $\hat{y}\in[0,1]$ such that $u^L_m(\hat{y}) < \hat{u}(\hat{y})$. 

\subsection{Identification of the nominal utility}
\label{sec:center-estimation}

{
In a decision-making problem that aims to maximize the expected utility of a reward function, using the smallest utility function \(u_m^S\) as an estimate of the true unknown utility function \(u^*\) may lead to an overly pessimistic and conservative decision, particularly when the ambiguity set \(\mathcal U_m\) is not sufficiently small. Likewise, using the largest utility function \(u_m^L\) may lead to an overly optimistic decision. To avoid these two extremes, we consider a nominal utility function that lies between \(u_m^S\) and \(u_m^L\), and identify such a utility function} 
by solving the following program: 
\begin{equation}\label{eq:center-minmax} 
\min\limits_{u\in \mathcal{U}_m}\max\Big\{\dd_K(u,u^L_m),\dd_K(u,u^S_m)\Big\}.
\end{equation}
We will show in Theorem~\ref{thm:opt-nominal} that $u^C_m:=\frac{1}{2}(u^L_m+u^S_m)$ is an optimal solution to problem~\eqref{eq:center-minmax} and that the Kantorovich ball $\mathbb{B}_K(u^C_m,3r)$ with $r=\frac{1}{2} \dd_K(u^S_m,u^L_m)$ contains the ambiguity set $\mathcal{U}_m$. 
To prove this major conclusion, we derive a closed-form formulation of the Kantorovich distance between two piecewise linear functions in Theorem~\ref{thm:closed-form-kantor}.

Since $u^L_m$ is the $N_m$-piecewise linear function with breakpoints in $\mathbb{Z}_m=\{z^m_j\}_{j=1,\ldots,N_m+1}$,  
and $u^S_m$ is the $(N_m-1)$-piecewise linear function with breakpoints in $\mathbb{Y}_m=\{y^m_j\}_{j=1,\ldots,N_m}$,  
we rewrite, for convenience in the following discussion, both functions in a unified piecewise linear structure. 
Let $\mathcal{L}_m:=\mathbb{Z}_m\cup\mathbb{Y}_m$, and denote by $\widetilde{N}_m:=|\mathcal{L}_m|$ the cardinality of the set $\mathcal{L}_m$, where $N_m+1\leq \widetilde{N}_m\leq 2N_m+1$. 
Let $\mathbb{X}_m:=\{x^m_j\}_{j=1,\ldots,\widetilde{N}_m}$ be the ordered sequence of points in $\mathcal{L}_m$ with fixed $x^m_1=0$ and $x^m_{\widetilde{N}_m}=1$. 
Then, $u^S_m$ and $u^L_m$ can be interpreted as $(\widetilde{N}_m-1)$-piecewise linear functions with breakpoints in $\mathbb{X}_m=\{x^m_j\}_{j=1,\ldots,\widetilde{N}_m}$. 
Accordingly, we use $\tilde{\beta}^S_m,\ \tilde{\beta}^L_m\in\mathbb{R}_+^{\widetilde{N}_m-1}$ to denote the vectors of slopes of $u^S_m$ and $u^L_m$ respectively, and use $\tilde{\alpha}^S_m, \tilde{\alpha}^L_m\in\mathbb{R}^{\widetilde{N}_m}$ to denote the utility function values at the breakpoints in $\mathbb{X}_m$ respectively.

We consider a subset $\mathcal{U}_m^{\widetilde{N}_m}\subseteq\mathcal{U}_m$, which consists of piecewise linear utility functions with breakpoints in $\mathbb{X}_m=\{x^m_j\}_{j=1,\ldots,\widetilde{N}_m}$. 
Under these circumstances, 
\begin{equation*}
\int_{0}^{1}g(z)du(z)=\sum\limits_{j=1}^{\widetilde{N}_m-1}\beta_j
\int_{x^m_j}^{x^m_{j+1}}g(t)dt
\end{equation*}
and 
\begin{equation}\label{eq:Kantorovich--piecewise}
\dd_K(u,\tilde{u})=\sup\limits_{g\in\mathscr{G}_L}\sum\limits_{j=1}^{\widetilde{N}_m-1}(\beta_j-\tilde{\beta}_j)\int_{x^m_j}^{x^m_{j+1}}g(t)dt,
\end{equation}
where $g\in\mathscr{G}_L$ as defined in~\eqref{eq:G-kant-a}, $\beta_j$ ($\tilde{\beta_j}$, resp.) is the slope of $u$ ($\tilde{u}$, resp.) in the interval $[x^m_j,x^m_{j+1}]$, $j=1,\ldots,\widetilde{N}_m-1$, and $u,\ \tilde{u}\in\mathcal{U}_m^{\widetilde{N}_m}$.

\begin{theorem}\label{thm:closed-form-kantor}
Given piecewise linear functions $u,\ \tilde{u}\in\mathcal{U}_m^{\widetilde{N}_m}$ with slopes $\beta=[\beta_1,\ldots,\beta_{\widetilde{N}_m-1}]$ and $\tilde{\beta}=[\tilde{\beta}_1,\ldots,\tilde{\beta}_{\widetilde{N}_m-1}]$ respectively, and with function values at breakpoints $\alpha=[\alpha_1,\ldots,\alpha_{\widetilde{N}_m}]$ and $\tilde{\alpha}=[\tilde{\alpha}_1,\ldots,\tilde{\alpha}_{\widetilde{N}_m}]$ respectively, where $\beta_j=\frac{\alpha_{j+1}-\alpha_j}{x^m_{j+1}-x^m_j}$ and $\tilde{\beta}_j=\frac{\tilde{\alpha}_{j+1}-\tilde{\alpha}_j}{x^m_{j+1}-x^m_j}$, $j=1,\ldots,\widetilde{N}_m-1$,  
define $c_j=|\tilde{\alpha}_{j+1}-\alpha_{j+1}+\tilde{\alpha}_j-\alpha_j|$ and $d_j=|\alpha_{j+1}-\tilde{\alpha}_{j+1}+\tilde{\alpha}_j-\alpha_j|$, for $j=1,\ldots,\widetilde{N}_m-1$. 
Then the optimal value of~\eqref{eq:Kantorovich--piecewise} is
\bgeq 
\dd_K(u, \tilde{u})=\frac{1}{4}\sum_{j=1}^{\widetilde{N}_m-1} (x^m_{j+1}-x^m_j) \left[\max\{c_j, d_j\}+\frac{c_j^2}{\max\{c_j, d_j\}}\right].
\edeq
\end{theorem}

\begin{proof}

Let 
\bgeq 
w_j :=\int_{x^m_{j}}^{x^m_{j+1}} g(t)dt\quad {\rm for}\;\; j=1,\dots,\widetilde{N}_m-1, \quad
p_j:=g(x^m_j) \quad
{\rm for} \;\; 
j=
{1,\ldots,\widetilde{N}_m-1,} 
\edeq 
and 
$p_{{\widetilde N_m}}=g(x^m_{{\widetilde N_m}})=g({1})=0$, {where $g\in\mathscr{G}_L$ as defined in~\eqref{eq:G-kant-a}.}
By~\cite[Proposition~6]{liu2025preference}, we have the following quadratically constrained quadratic program (QCQP) to calculate the Kantorovich distance
\begin{subequations}\label{eq:Kant-qcqp}
\begin{align}
\noalign{$\dd_K(u, \tilde{u})=$}
\max\limits_{\substack{w_1, \ldots, w_{\widetilde{N}_m-1}\\ p_1, \ldots, p_{\widetilde{N}_m}}}\ &\sum_{j=1}^{\widetilde{N}_m-1}  (\beta_j-\tilde{\beta}_j)w_j\\
\text{s.t.}\  &  
 w_{j}\leq \frac{1}{4} (x^m_{j+1}- x^m_{j})^2 -\frac{1}{4}( p_{j+1}-p_{j})^2 + \frac{1}{2}(x^m_{j+1}- x^m_{j})( p_{j+1}+p_{j}),\\
&-w_{j}\leq \frac{1}{4}(x^m_{j+1}- x^m_{j})^2 -\frac{1}{4}(p_{j+1}-p_{j})^2 - \frac{1}{2}(x^m_{j+1}- x^m_{j})(p_{j+1}+p_{j}),\\
&\text{for}  \;\; j=1,\ldots,\widetilde{N}_m-1. \nonumber
\end{align}
\end{subequations}
To simplify the quadratic constraints, 
we introduce $q_j=p_{j+1}-p_j$, $j=1,\ldots,\widetilde{N}_m-1$, which implies $p_j=-\sum_{i=j}^{\widetilde N_m-1}q_i$. 
Thus, we can 
rewrite problem~\eqref{eq:Kant-qcqp} as follows: 
\begin{subequations}\label{eq:Kant-qcqp-x}
\begin{align}
\noalign{$\dd_K(u, \tilde{u})=$}
\max\limits_{\substack{w_1, \ldots, w_{\widetilde{N}_m-1}\\ q_1, \ldots, q_{\widetilde{N}_m-1}}}\ &\sum_{j=1}^{\widetilde{N}_m-1}  (\beta_j-\tilde{\beta}_j)w_j\\
\text{s.t.}\  &  
w_{j}\leq \frac{1}{4} (x^m_{j+1}- x^m_{j})^2 -\frac{1}{4}(q_{j})^2 
{-}
\frac{1}{2}(x^m_{j+1}- x^m_{j})
\left(2\sum\limits_{{i=j+1}}^{{\widetilde N_m-1}}q_i+q_j\right),\\
&-w_{j}\leq \frac{1}{4}(x^m_{j+1}- x^m_{j})^2 -\frac{1}{4}(q_{j})^2 
{+}
\frac{1}{2}(x^m_{j+1}- x^m_{j})
\left(2\sum\limits_{{i=j+1}}^{{\widetilde N_m-1}}q_i+q_j\right),\\
&\text{for}  \;\; j=1,\ldots,\widetilde{N}_m-1.\nonumber
\end{align}
\end{subequations}
{We next derive the Lagrangian dual of the convex QCQP~\eqref{eq:Kant-qcqp-x}:}
\begin{subequations}\label{dual K}
\begin{align}
\noalign{$\dd_K(u, \tilde{u})=$}
\min\limits_{\lambda,\mu}\
& \frac{1}{4}\sum_{j=1}^{\widetilde{N}_m-1} (\lambda_j+\mu_j)({x}^m_{j+1}-{x}^m_{j})^2+\sum_{j=1}^{\widetilde{N}_m-1}\frac{\left(\frac{(\lambda_j-\mu_j)({x}^m_{j+1}-{x}^m_j)}{2}+
\sum\limits_{{i=1}}^{{j-1}}
(\lambda_i-\mu_i)({x}^m_{i+1}-{x}^m_i)\right)^2}{\lambda_j+\mu_j} \\
\text{s.t.}\ \  
\label{cons:kant-dual-1}
& \beta_{j}-\tilde{\beta}_{j}-\lambda_j+\mu_j=0,\ j=1,\ldots,\widetilde{N}_m-1,\\
\label{cons:kant-dual-2}
& \frac{1}{2}\sum_{j=1}^{\widetilde{N}_m-1}(\lambda_j-\mu_j)({x}^m_{j+1}-{x}^m_j)=0,\\
& \lambda_j\geq0,\ \mu_j\geq0,\ j=1,\ldots,\widetilde{N}_m-1. 
\end{align}
\end{subequations}
Equivalence between~\eqref{eq:Kant-qcqp-x} and~\eqref{dual K} 
follows from strong duality under Slater’s condition, because the primal problem has a strictly feasible solution and a finite optimal value. 
Specifically, since the breakpoints satisfy $x_1^m<\cdots<x_{\widetilde N_m}^m$, the point $w_j=q_j=0$, $j=1,\ldots,\widetilde N_m-1$, is strictly feasible for the convex problem~\eqref{eq:Kant-qcqp-x}. 

From~\eqref{cons:kant-dual-1} and~\eqref{cons:kant-dual-2}, we can see that if $\frac{1}{2}\sum_{j=1}^{\widetilde{N}_m-1}(\beta_{j}-\tilde{\beta}_{j})({x}^m_{j+1}-{x}^m_j)\neq 0$, then problem~\eqref{dual K} is infeasible. 
Thus we focus on the case when the equality holds. 
This problem is decomposable. 
Moreover, by eliminating the variables 
\begin{equation}\label{eq:change-lambda}
\lambda_j= \beta_j-\tilde{\beta}_j+\mu_j,\ j=1,\ldots,\widetilde{N}_m-1, 
\end{equation}
via the equality constraint~\eqref{cons:kant-dual-1}, 
we can solve~\eqref{dual K} by solving ${\widetilde{N}_m-1}$ univariate sub-programs:
\begin{equation*}
\min\limits_{
\substack{\mu_j\geq 0,\\
\beta_{j}-\tilde{\beta}_{j}+\mu_j\geq 0
}} 
a_j(2\mu_j+ \beta_{j}-\tilde{\beta}_{j})+\frac{b_j}{2\mu_j+ \beta_{j}-\tilde{\beta}_{j}},  
\end{equation*}
where $a_j=\frac{1}{4}({x}^m_{j+1}-{x}^m_{j})^2 $ and 
\begin{align*}
b_j&={\left(\frac{(\beta_{j}-\tilde{\beta}_{j})({x}^m_{j+1}-{x}^m_j)}{2}+
\sum\limits_{{i=1}}^{{j-1}}
(\beta_{i}-\tilde{\beta}_{i})({x}^m_{i+1}-{x}^m_i)\right)^2}\\
&=\left(\frac{\alpha_{j+1}-\alpha_j-\tilde{\alpha}_{j+1}+\tilde{\alpha}_j}{2}+
\sum\limits_{{i=1}}^{{j-1}}
(\alpha_{i+1}-\alpha_i-\tilde{\alpha}_{i+1}+\tilde{\alpha}_i)\right)^2=\frac{1}{4}(\tilde{\alpha}_{j+1}-\alpha_{j+1}+\tilde{\alpha}_j-\alpha_j)^2, 
\end{align*}
for $j=1,\ldots,\widetilde{N}_m-1$, and obtain
\begin{equation}\label{eq:kantor-closed}
\lambda_j^*+\mu_j^*
=\left\{\begin{array}{ll}
   \sqrt{\frac{b_j}{a_j}}\  & \text{if}\ \ \sqrt{\frac{b_j}{a_j}}\geq |\beta_{j}-\tilde{\beta}_{j}| \\
   |\beta_j-\tilde{\beta_j}|\  & \text{if}\ \ \sqrt{\frac{b_j}{a_j}}< |\beta_{j}-\tilde{\beta}_{j}|
\end{array}
\right.=\max\left\{\sqrt{ \frac{b_j}{a_j}},\ |\beta_{j}-\tilde{\beta}_{j}|\right\},
\;\; j=1,\ldots,\widetilde{N}_m-1.  
\end{equation}
Substituting the optimal $\lambda_j^*$, $\mu_j^*$ from~\eqref{eq:change-lambda} and~\eqref{eq:kantor-closed} back into~\eqref{dual K} yields the final result
\begin{align*}
\noalign{$\dd_K(u, \tilde{u})$}
=&\sum_{j=1}^{\widetilde{N}_m-1} \left(a_j\cdot\max\left\{\sqrt{ \frac{b_j}{a_j}},\ |\beta_{j}-\tilde{\beta}_{j}|\right\}+\frac{b_j}{\max\left\{\sqrt{\frac{b_j}{a_j}},\ |\beta_{j}-\tilde{\beta}_{j}|\right\}}\right)\\
=&\sum_{j=1}^{\widetilde{N}_m-1}\left(\frac{1}{4}({x}^m_{j+1}-{x}^m_{j})^2 \cdot \max\left\{\left|\frac{\tilde{\alpha}_{j+1}-\alpha_{j+1}+\tilde{\alpha}_j-\alpha_j}{{x}^m_{j+1}-{x}^m_{j}}\right|,\left|\frac{\alpha_{j+1}-\alpha_j-\tilde{\alpha}_{j+1}+\tilde{\alpha}_j}{{x}^m_{j+1}-{x}^m_{j}}\right|\right\}+\right.\\
& \left.\frac{(\tilde{\alpha}_{j+1}-\alpha_{j+1}+\tilde{\alpha}_j-\alpha_j)^2/4}{\max\left\{\left|\frac{\tilde{\alpha}_{j+1}-\alpha_{j+1}+\tilde{\alpha}_j-\alpha_j}{{x}^m_{j+1}-{x}^m_{j}}\right|,\left|\frac{\alpha_{j+1}-\alpha_j-\tilde{\alpha}_{j+1}+\tilde{\alpha}_j}{{x}^m_{j+1}-{x}^m_{j}}\right|\right\}}\right)\\
=&\frac{1}{4}\sum_{j=1}^{\widetilde{N}_m-1}({x}^{m}_{j+1}-{x}^{m}_{j}) \cdot\bigg[\max\{|\tilde{\alpha}_{j+1}-\alpha_{j+1}+\tilde{\alpha}_j-\alpha_j|, |\alpha_{j+1}-\tilde{\alpha}_{j+1}+\tilde{\alpha}_j-\alpha_j|\}+\\
&\frac{(\tilde{\alpha}_{j+1}-\alpha_{j+1}+\tilde{\alpha}_j-\alpha_j)^2}{\max\{|\tilde{\alpha}_{j+1}-\alpha_{j+1}+\tilde{\alpha}_j-\alpha_j|, |\alpha_{j+1}-\tilde{\alpha}_{j+1}+\tilde{\alpha}_j-\alpha_j|\}}\bigg]. 
\end{align*}
The proof is complete. 
\end{proof}

\begin{example}\label{example:compute-kantor}
We present two examples to intuitively illustrate the closed-form solution of the Kantorovich distance between two piecewise linear functions. In each example, we consider two functions $u$, $\tilde{u}$, defined on $[0,1]$, with three linear pieces between the $N=4$ breakpoints $y_1,y_2,y_3,y_4$. 
Denote the function 
{values}
at breakpoints by $\alpha_i=u(y_i)$ and $\tilde{\alpha}_i=\tilde{u}(y_i)$ for $i=1,\ldots,N$.  
We assume $\alpha_1=\tilde{\alpha}_1=0$ and $\alpha_4=\tilde{\alpha}_4=1$.

In the first example, the two functions exhibit a first-order dominance relationship, i.e., $\alpha_i\geq \tilde{\alpha}_i$, $i=1,\ldots,N$. 
Then, according to Theorem~\ref{thm:closed-form-kantor}, we observe that 
{\setlength{\abovedisplayskip}{0.3em}
 \setlength{\belowdisplayskip}{0.3em}
\bgeq 
\max\left\{|\tilde{\alpha}_{i+1}-\alpha_{i+1}+\tilde{\alpha}_i-\alpha_i|,|\alpha_{i+1}-\tilde{\alpha}_{i+1}+\tilde{\alpha}_i-\alpha_i|\right\}
&=|\tilde{\alpha}_{i+1}-\alpha_{i+1}+\tilde{\alpha}_i-\alpha_i|\\
&=\alpha_{i+1}-\tilde{\alpha}_{i+1}+\alpha_i-\tilde{\alpha}_i,
\edeq 
}
for $i=1,\ldots,N-1$, resulting in  $\dd_K(u,\tilde{u})=\sum_{i=1}^{N-1} e_i$ with 
{\setlength{\abovedisplayskip}{0.3em}
 \setlength{\belowdisplayskip}{0.3em}
\bgeq
e_i=\frac{1}{2}(y_{i+1}-y_i)(\alpha_{i+1}-\tilde{\alpha}_{i+1}+\alpha_i-\tilde{\alpha}_i), 
\edeq
}
which is exactly the area of the difference of their epigraphs.

In the second example, the two functions intersect at a point other than the endpoints, which is 
\bgeq 
(y_j+\frac{(\alpha_j-\tilde{\alpha}_j)(y_{j+1}-y_j)}{\tilde{\alpha}_{j+1}-\tilde{\alpha}_j-\alpha_{j+1}+\alpha_{j}},\ \frac{\tilde{\alpha}_{j+1}\alpha_j-\alpha_{j+1}\tilde{\alpha}_j}{\tilde{\alpha}_{j+1}-\tilde{\alpha}_j-\alpha_{j+1}+\alpha_{j}})
\edeq
with $j=2$, shown as the point $I$ in Figure~\ref{fig:compute-kantor}. 
Within the interval $[y_i,y_{i+1}]$, $i=1,\ldots,j-1,j+1,\ldots,N-1$, either $u(y)\geq \tilde{u}(y)$ or $\tilde{u}(y)\geq u(y)$, thus $\max\{|\tilde{\alpha}_{i+1}-\alpha_{i+1}+\tilde{\alpha}_i-\alpha_i|,|\alpha_{i+1}-\tilde{\alpha}_{i+1}+\tilde{\alpha}_i-\alpha_i|\}=|\alpha_{i+1}-\tilde{\alpha}_{i+1}+\alpha_i-\tilde{\alpha}_i|$. We denote $e_i=\frac{1}{2}(y_{i+1}-y_i)|\alpha_{i+1}-\tilde{\alpha}_{i+1}+\alpha_i-\tilde{\alpha}_i|$, which also represents the area of the difference between the epigraphs of $u$ and $\tilde{u}$ within $[y_i,y_{i+1}]$. 
Unlike the previous case, within the interval $[y_j,y_{j+1}]$, $u$ and $\tilde{u}$ intersect, with $\alpha_j>\tilde{\alpha}_j$ and $\tilde{\alpha}_{j+1}>\alpha_{j+1}$, thus 
\bgeq 
\max\Big\{|\tilde{\alpha}_{j+1}-\alpha_{j+1}+\tilde{\alpha}_j-\alpha_j|,|\alpha_{j+1}-\tilde{\alpha}_{j+1}+\tilde{\alpha}_j-\alpha_j|\Big\}
&=|\alpha_{j+1}-\tilde{\alpha}_{j+1}+\tilde{\alpha}_j-\alpha_j|\\
&=\alpha_j-\tilde{\alpha}_j+\tilde{\alpha}_{j+1}-\alpha_{j+1}.
\edeq
Then we denote
\begin{align*}
e_j&=\frac{1}{4}(y_{j+1}-y_j)\left[\alpha_j-\tilde{\alpha}_j+\tilde{\alpha}_{j+1}-\alpha_{j+1}+\frac{(\tilde{\alpha}_{j+1}-\alpha_{j+1}+\tilde{\alpha}_j-\alpha_j)^2}{\alpha_j-\tilde{\alpha}_j+\tilde{\alpha}_{j+1}-\alpha_{j+1}}\right]\\
&=\frac{1}{2}(y_{j+1}-y_j)\frac{(\alpha_j-\tilde{\alpha}_j)^2+(\tilde{\alpha}_{j+1}-\alpha_{j+1})^2}{\tilde{\alpha}_{j+1}-\tilde{\alpha}_j-\alpha_{j+1}+\alpha_j}\\
&=\frac{1}{2}\left[\frac{(\alpha_j-\tilde{\alpha}_j)(y_{j+1}-y_j)}{\tilde{\alpha}_{j+1}-\tilde{\alpha}_j-\alpha_{j+1}+\alpha_{j}}\right] (\alpha_j-\tilde{\alpha}_j)+\frac{1}{2}\left[y_{j+1}-\left(y_j+\frac{(\alpha_j-\tilde{\alpha}_j)(y_{j+1}-y_j)}{\tilde{\alpha}_{j+1}-\tilde{\alpha}_j-\alpha_{j+1}+\alpha_{j}}\right)\right](\tilde{\alpha}_{j+1}-\alpha_{j+1}), 
\end{align*}
which is exactly the area of the two triangles constructed by $\max\{u(y),\ \tilde{u}(y)\}-\min\{u(y),\ \tilde{u}(y)\}$ over the interval $[y_j,y_{j+1}]$. 
Therefore, we have $\dd_K(u,\tilde{u})=\sum_{i=1}^{N-1} e_i$ by Theorem~\ref{thm:closed-form-kantor}. 

\begin{figure}[htbp]
    \centering
    \subfigure
    {
    \includegraphics[width=0.42\linewidth]{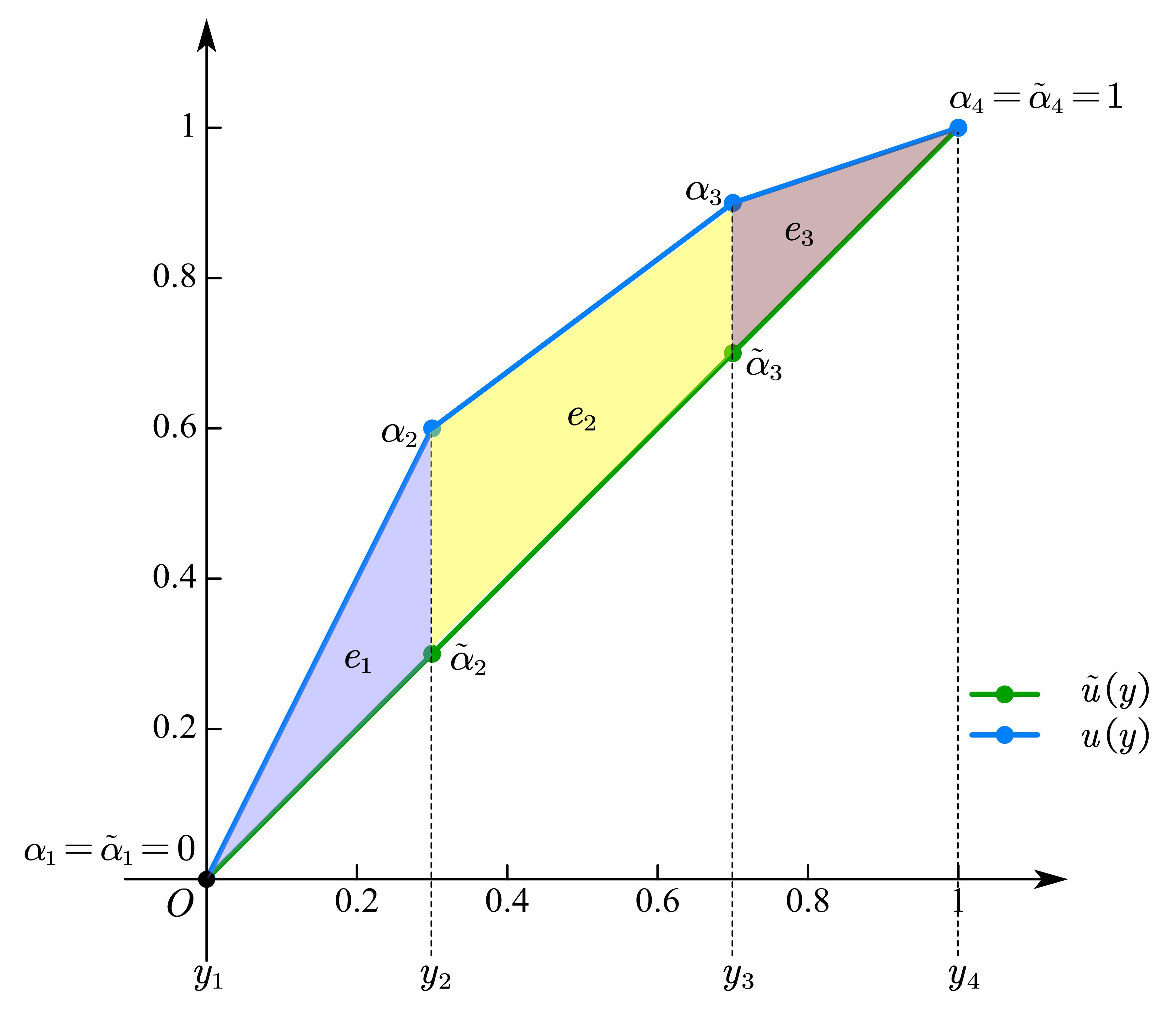}
    }
    \quad
    \subfigure
    {
    \includegraphics[width=0.42\linewidth]{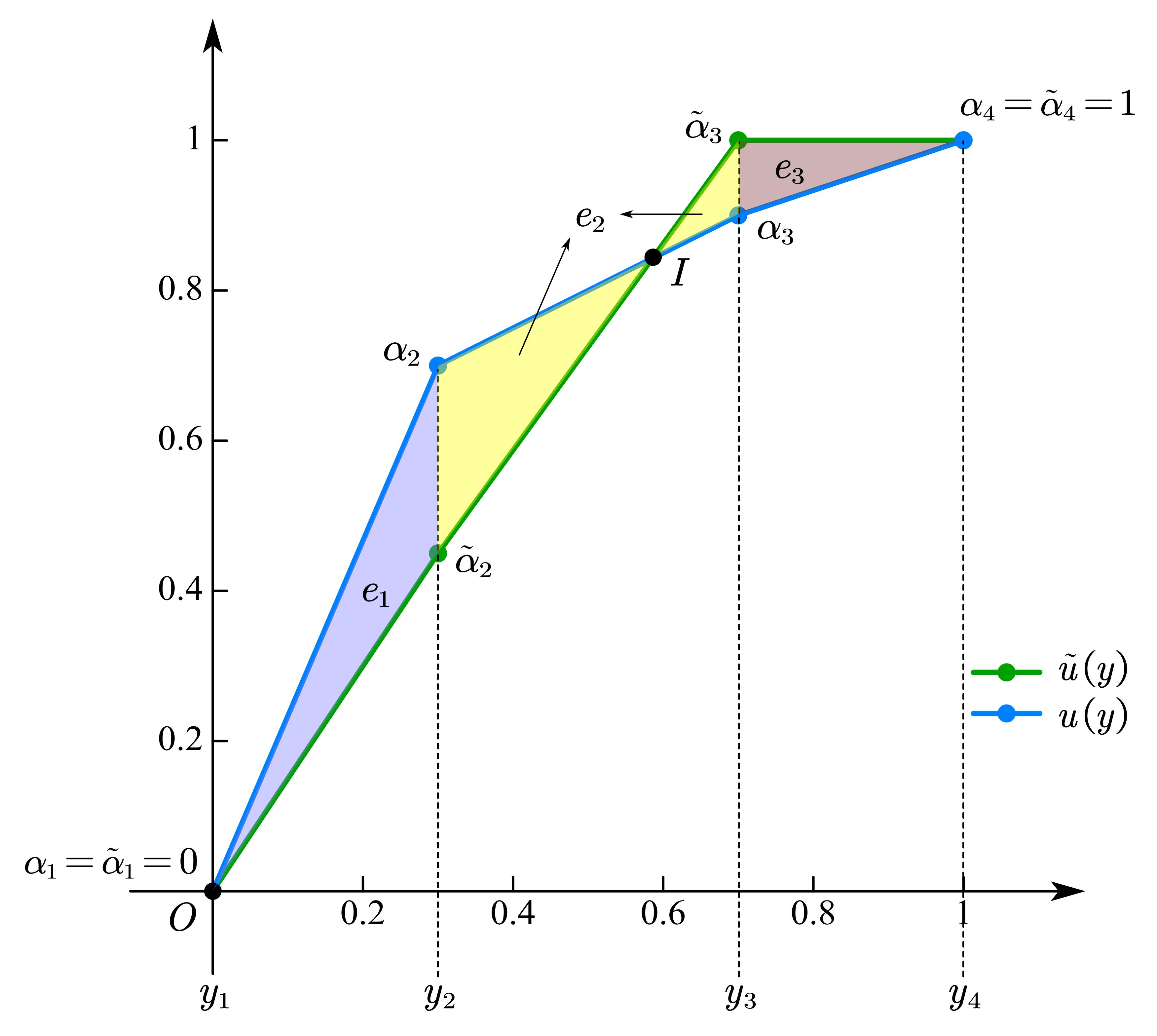}
    }
    \caption{Two intuitive examples of the computation of the Kantorovich distance.}
    \label{fig:compute-kantor}
\end{figure}

\end{example}

\begin{remark}\label{remark:compute-kantor}
Observing from Theorem~\ref{thm:closed-form-kantor} and Example~\ref{example:compute-kantor},
we can conclude that, given piecewise linear functions $u,\ \tilde{u}\in\mathcal{U}_m^{\widetilde{N}_m}$, the Kantorovich distance $\dd_K(u,\tilde{u})$ is exactly the area of the difference between $\operatorname{epi}(u)\cup \operatorname{epi}(\tilde{u})$ and $\operatorname{epi}(u)\cap \operatorname{epi}(\tilde{u})$, where $\operatorname{epi}(\cdot)$ denotes the epigraph of a given function. 
See Part~(ii) in Theorem~\ref{thm:mus-converg}. 
Thus
\bgeqn 
\label{eq:dd_K-u-tilde-u}
\dd_K(u,\tilde{u})=\int_0^1 \left(\max\{u(y),\tilde{u}(y)\}-\min\{u(y),\tilde{u}(y)\}\right)dy.
\edeqn 
\end{remark}

\begin{theorem}\label{thm:opt-nominal}
Let $u^L_m$ defined in~\eqref{eq:largest-u} be 
{the}
largest utility function in the ambiguity set $\mathcal{U}_m$, and $u^S_m$ defined in~\eqref{eq:mus-min-utility} be 
{the}
smallest utility function in 
$\mathcal{U}_m$. 
Let $u^C_m:=\frac{1}{2}(u^L_m+u^S_m)$, i.e., 
\begin{equation*}
    u^C_m(y)=\left\{ 
    \begin{aligned}
    &0 && {\rm for} \quad y= x^{m}_{1},\\
    &\frac{1}{2}(\tilde{\beta}^{L}_{m,j}+\tilde{\beta}^{S}_{m,j})(y-{x}^{m}_j)+\frac{1}{2}(\tilde{\alpha}^{L}_{m,j}+\tilde{\alpha}^{S}_{m,j}) && {\rm for}\quad {x}^{m}_j< y \leq{x}^{m}_{j+1},\;\; j=1,\ldots,\widetilde{N}_m-1,\\
    & 1 && {\rm for}\quad y={x}^{m}_{\widetilde{N}_m}.\\
    \end{aligned}
    \right.
\end{equation*}
Then the following assertions hold. 
\begin{itemize}
\item[(i)] $u^S_m(y)\leq u^C_m(y)\leq u^L_m(y)$, $\forall y\in[0,1]$, and $\dd_K(u^C_m,u^L_m)=\dd_K(u^C_m,u^S_m)=\frac{1}{2} \dd_K(u^S_m,u^L_m)$.

\item[(ii)] 
$u^C_m$ is an optimal solution of problem~\eqref{eq:center-minmax}.

\item[(iii)] Let $r:=\frac{1}{2} \dd_K(u^S_m,u^L_m)$  
{and suppose $r>0$}. 
Then 
\bgeq 
\dd_K(u,u^C_m)
{<}
3r,\;\; \forall u\in {\cal U}_m, 
\edeq 
i.e., ${\cal U}_m\subseteq \mathbb{B}_K(u^C_m,3r)$ under the Kantorovich distance.
\end{itemize}

\end{theorem}

\begin{proof}
\uline{Part~(i).}
By Theorem~\ref{thm:mus-lower-bound}, $u^S_m(y)\leq u(y)$, $\forall y\in[0,1]$, $\forall u\in\mathcal{U}_m$. 
Thus, $u^S_m(y)\leq u^L_m(y)$ and $u^S_m(y)\leq \frac{1}{2}(u^S_m(y)+ u^L_m(y))\leq u^L_m(y)$, $\forall y\in[0,1]$. By Theorem~\ref{thm:closed-form-kantor} (see~\eqref{eq:dd_K-u-tilde-u}), we can easily establish $\dd_K(u^C_m,u^L_m)=\dd_K(u^C_m,u^S_m)=\frac{1}{2} \dd_K(u^S_m,u^L_m)$. 

\vspace{0.1cm}
\uline{Part~(ii).}
Note that $u^C_m=\frac{1}{2}(u^S_m+u^L_m)$ preserves monotonicity, concavity, and Lipschitz continuity of both $u^S_m$ and $u^L_m$. 
Thus, $u^C_m\in\mathcal{U}_{cv}$ evidently. 
We select $u^1_j=u^S_m$ and $u^2_j=u^L_m$ for $j=1,\ldots,N_m$, and it then follows from Lemma~\ref{lemma:mus-u-U_m} that $u^C_m\in\mathcal{U}_m$.  
Assume for the sake of a contradiction that there exists a function $\hat{u}\in\mathcal{U}_m$ such that $\max\{\dd_K(\hat{u},u^L_m),\dd_K(\hat{u},u^S_m)\}<\frac{1}{2} \dd_K(u^S_m,u^L_m)$. 
Then, 
\begin{equation*}
\dd_K(\hat{u},u^L_m)+\dd_K(\hat{u},u^S_m)\leq2\max\{\dd_K(\hat{u},u^L_m),\dd_K(\hat{u},u^S_m)\}< \dd_K(u^S_m,u^L_m), 
\end{equation*}
which contradicts the triangle inequality of the Kantorovich metric.

\vspace{0.1cm}
\uline{Part~(iii).}
{Since $u_m^S:=\underline{f}$ is the pointwise lower-bound utility function by Theorem~\ref{thm:mus-lower-bound}, $0\leq u_m^S(y)\leq{u}(y)$ for any $y\in[0,1]$ and any $u\in\mathcal{U}_m$.
Then the Kantorovich distance satisfies
\[
\dd_K(u,u_m^S)=\int_0^1\bigl(u(y)-u_m^S(y)\bigr)\,dy,\;\; \forall u\in\mathcal U_m.
\] 
By the uniqueness of $u_m^L$ in Theorem~\ref{thm:largest-unique}, every \(u\in\mathcal U_m\setminus\{u_m^L\}\) satisfies $\int_0^1u(y)\,dy<\int_0^1u_m^L(y)\,dy$, thus
\[
\dd_K(u,u_m^S)=\int_0^1\bigl(u(y)-u_m^S(y)\bigr)\,dy<\int_0^1\bigl(u_m^L(y)-u_m^S(y)\bigr)\,dy=\dd_K(u_m^L,u_m^S)=2r.
\]} 
By the triangle inequality of the Kantorovich distance
{and Part~(i)}, 
\begin{align*}
\dd_K(u,u^C_m)\leq\dd_K(u,u^S_m)+\dd_K(u^S_m,u^C_m)
{<}
2r+\dd_K(u^S_m,u^C_m)=3r,\;\; \forall u\in
{\mathcal U_m\setminus\{u_m^L\}}.
\end{align*}
{For \(u=u_m^L\), Part~(i) gives \(\dd_K(u_m^L,u_m^C)=r<3r\). Consequently, $\dd_K(u,u_m^C)<3r$, $\forall u\in\mathcal U_m$. } 
\end{proof}

\section{Extensions}\label{sec-extensions}

The MUS scheme discussed in the preceding sections primarily focuses on concave utility functions. This naturally raises the question of whether the approach can be extended to non-concave functions. In this section, we outline how the method may be adapted to general non-concave utility functions, with particular emphasis on S-shaped utility functions, which are widely used in behavioral economics.


\subsection{General utility functions}
\label{sec:ec}

Let $\mathcal{U}_{IC}\subseteq\mathcal{L}^1([0,1])$ be an ambiguity set comprising utility functions that are monotonically increasing, Lipschitz continuous with modulus bounded by $L$ and normalized:
\begin{equation}\label{eq:U_ec}
\mathcal{U}_{IC}:= \left\{ u\in \mathcal{L}^1([0,1])\mid 
 {u(y_1)\leq u(y_2),\ \forall\,0\leq y_1<y_2\leq1},
 \ u(0)=0,\ u(1)=1,\ \text{Lip}(u)\leq L \right\}. 
\end{equation}
Let ${\mathcal{U}}^{IC}_0 :=\mathcal{U}_{IC}$ be the initial ambiguity set
and 
${\mathcal{U}}^{IC}_m\subseteq\mathcal{U}_{IC}$ be the ambiguity set updated with paired gambling questionnaires generated by the MUS scheme, $m=1,2,\cdots$, i.e., 
\begin{equation}\label{eq:U_ec-m}
    \mathcal{U}^{IC}_m =\left\{u\in\mathcal{U}_{IC}\mid Z_k\cdot u(r_2^{
{k}
})\geq Z_k\cdot p^k,\ k=1,\ldots,m\right\}.  
\end{equation} 
Analogous to the MUS in concave case, we use  
the questionnaires in the form of~\eqref{eq:W-Y} as in Section~\ref{sec:questionnaire}.
Let $\mathcal{S}_m :=\{0\}\cup \bigcup\limits_{k=1}^{m}\{r_2^k\}\cup\{1\}$ and let $\mathbb{Y}_m:= \{{y}^m_{j}\}_{j=1,\ldots,N_m}$ be the set of all points in $\mathcal{S}_m$ sorted in increasing order with fixed $y^m_1=0$ and $y^m_{N_m}=1$. 
For each $r_2\in[0,1]$, we define the upper bound and lower bound functions of the ambiguity set $\mathcal{U}^{IC}_m$ as: 
\begin{equation}\label{eq:U_ec-upper-lower-func}
\bar{f}^{IC}(r_2):=\max\limits_{u\in\mathcal{U}^{IC}_m} u(r_2)\quad\text{and}\quad
\underline{f}^{IC}(r_2):=\min\limits_{u\in\mathcal{U}^{IC}_m} u(r_2).  
\end{equation}
A key step of the MUS is to identify 
$r_2$ which solves
the following maximization problem
\begin{equation*}
\mathop{\max}\limits_{r_2\in [0,1]} f^{IC}(r_2):= \left[\max\limits_{u\in \mathcal{U}_{m}^{IC}}u(r_2)-\min\limits_{u\in \mathcal{U}_{m}^{IC}}u(r_2)\right]=\bar{f}^{IC}(r_2)-\underline{f}^{IC}(r_2). 
\end{equation*} 
Similar to the concave utility function case, 
the problems in~\eqref{eq:U_ec-upper-lower-func} 
also admit closed-form solutions. Both $\underline{f}^{IC}$ and $\bar{f}^{IC}$ consist of
{at most} two pieces between any pair of adjacent elicited points in $\mathbb{Y}_m$, with slopes 
equal 
either to zero or 
to 
the corresponding Lipschitz modulus $L$, as stated in the following theorem. 
\begin{theorem}[Semi-closed forms of $\underline{f}^{IC}$ and  $\bar{f}^{IC}$]\label{thm:ex-upper-lower}
Consider the ambiguity set  ${\mathcal{U}}^{IC}_m\subseteq\mathcal{U}_{IC}$ generated by the MUS scheme, as in~\eqref{eq:U_ec-m}. 
The lower bound function $\underline{f}^{IC}$ and the upper bound function $\bar{f}^{IC}$ defined in~\eqref{eq:U_ec-upper-lower-func} are both piecewise linear with the following respective structures over $[y^m_k,y^m_{k+1}]$ for $k=1,\ldots,N_m-1:$ 

\begin{equation}\label{eq:U_ex-mus-min-utility}
\underline{f}^{IC}(y)
    =\left\{ 
    \begin{aligned}
    &\underline{\alpha}^{IC}_{m,k} && {\rm for}\quad y^m_k\leq y \leq \underline{y}^{m}_{k},\\
    &\underline{\alpha}^{IC}_{m,k+1}+L(y-y^m_{k+1}) && {\rm for} \quad\underline{y}^{m}_{k} < y \leq y^m_{k+1},
    \end{aligned}
    \right.
\end{equation}

\begin{equation}\label{eq:U_ex-mus-max-utility}
\bar{f}^{IC}(y)
    =\left\{ 
    \begin{aligned}
    &\bar{\alpha}^{IC}_{m,k} + L(y-y^m_k) && {\rm for} \quad y^m_k\leq y \leq \bar{y}^{m}_{k},\\
    &\bar{\alpha}^{IC}_{m,k+1}&& {\rm for} \quad\bar{y}^{m}_{k} < y \leq y^m_{k+1},
    \end{aligned}
    \right.
\end{equation}
where $\underline{\alpha}^{IC}_{m,k}:=\min\limits_{u\in\mathcal{U}^{IC}_m}u(y^m_k)$ and $\bar{\alpha}^{IC}_{m,k}:=\max\limits_{u\in\mathcal{U}^{IC}_m}u(y^m_{k})$; 
$\underline{y}^m_k:=y^m_{k+1}-\left(\underline{\alpha}^{
{IC}
}_{{m,}k+1}-\underline{\alpha}^{
{IC}
}_{{m,}k}\right)/{L}\in[y^m_k,y^m_{k+1}]$ and $\bar{y}^m_k=y^m_k+\left(\bar{\alpha}^{
{IC}
}_{{m,}k+1}-\bar{\alpha}^{
{IC}
}_{{m,}k}\right)/{L}\in[y^m_k,y^m_{k+1}]$; $L$ is the uniform Lipschitz modulus in~\eqref{eq:U_ec}.
See Figure~\ref{fig:EC-upper-lower} for an illustration. 
Therefore, $f^{IC}=\bar{f}^{IC}-\underline{f}^{IC}$ is also a piecewise linear function with at most three pieces over each  $[y^m_k,y^m_{k+1}]$, $k=1,\ldots,N_m-1$.    
\end{theorem}

\begin{figure}[htbp]
    \centering
\includegraphics[width=0.55\linewidth]{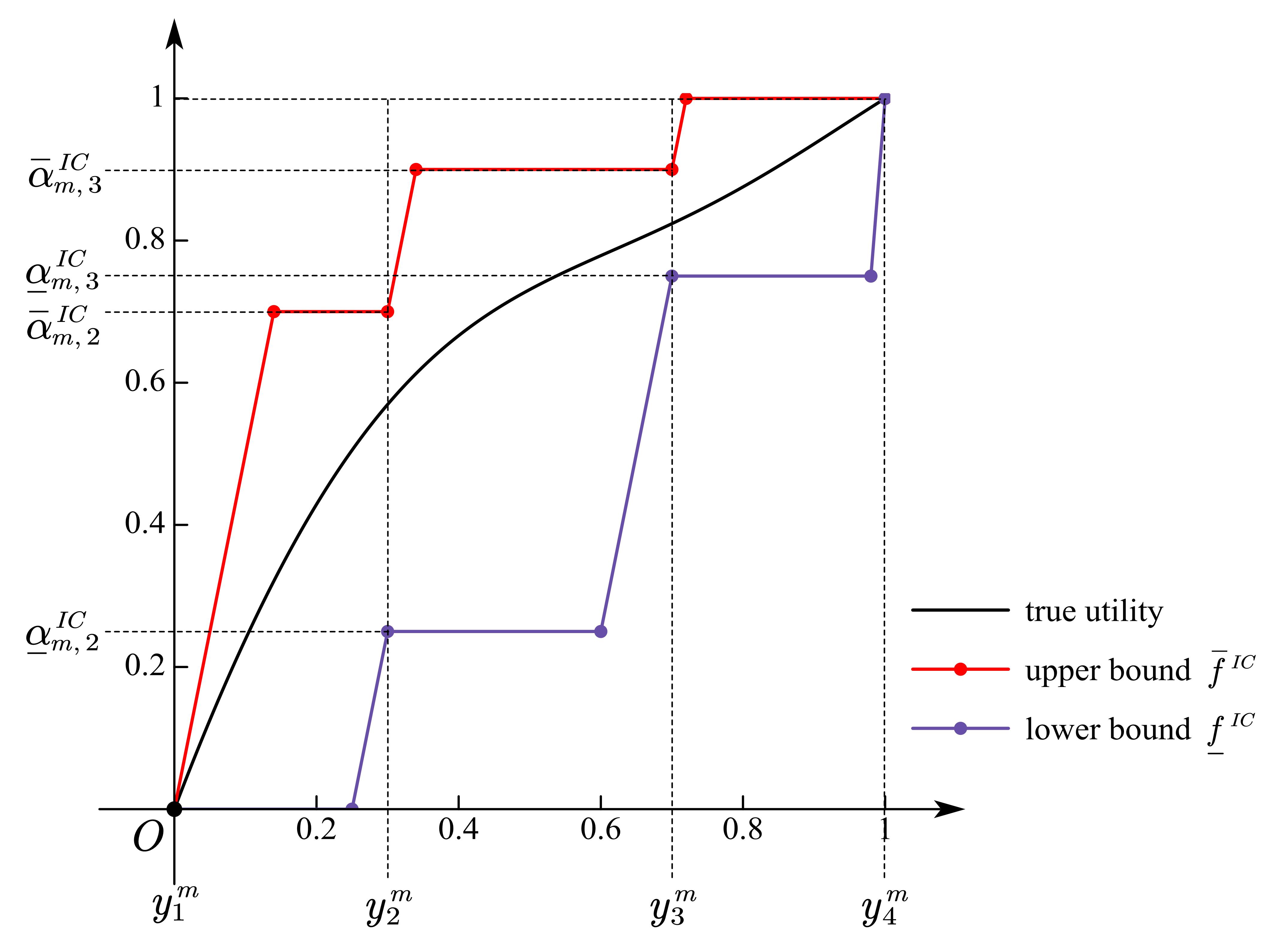}
    \caption{Illustration of 
    the lower bound function $\underline{f}^{IC}$ and the upper bound function $\bar{f}^{IC}$ with Lipschitz modulus $L=5$.
    }
    \label{fig:EC-upper-lower}
\end{figure}

\begin{proof}

We only prove the semi-closed form of $\underline{f}^{IC}$ as in~\eqref{eq:U_ex-mus-min-utility}, because the
semi-closed form of $\bar{f}^{IC}$ as in~\eqref{eq:U_ex-mus-max-utility} can be 
derived analogously. 
Define a function $\underline{u}\in\mathcal{L}^1([0,1])$ with the following structure over $[y^m_k,y^m_{k+1}]$, $k=1,\ldots,N_m-1$: 
\begin{equation}\label{eq:pf-thm8-u}
\underline{u}(y)
    =\left\{ 
    \begin{aligned}
    &\underline{\alpha}^{IC}_{m,k} && {\rm for} \quad y^m_k\leq y \leq \underline{y}^{m}_{k},\\
    &\underline{\alpha}^{IC}_{m,k+1}+L(y-y^m_{k+1}) && {\rm for} \quad \underline{y}^{m}_{k} < y \leq y^m_{k+1}, 
    \end{aligned}
    \right.
\end{equation}
where $\underline{y}^{m}_{k}$
is defined as in~\eqref{eq:U_ex-mus-max-utility}. 
We proceed with the proof in two steps:
first $\underline{u}\in\mathcal{U}^{IC}_m$, and then $\underline{f}^{IC}(y):=\min\limits_{u\in\mathcal{U}^{IC}_m} u(y)=\underline{u}(y)$, $\forall y\in[0,1]$. 

\vspace{0.1cm}
\noindent
\uline{Step~1}. 
Since each $u\in {\cal U}^{IC}_m$ is monotonically increasing and the min operation preserves the monotonicity,
then
$\underline{f}^{IC}$ is increasing over $[0,1]$, i.e., $\underline{f}^{IC}(y^m_k)=\underline{u}(y^m_k)=\underline{\alpha}^{IC}_{m,k}\leq\underline{\alpha}^{IC}_{m,k+1}=\underline{u}(y^m_{k+1})=\underline{f}^{IC}(y^m_{k+1})$, $k=1,\ldots,N_m-1$. 
Moreover, for any $0\leq y_1\leq y_2\leq1$, since 
\begin{align*}
\left|\underline{f}^{IC}(y_2)-\underline{f}^{IC}(y_1)\right|
=\left|\min\limits_{u\in\mathcal{U}^{IC}_m}u(y_2)-
\min\limits_{u\in\mathcal{U}^{IC}_m}u(y_1)\right|
\leq\max\limits_{u\in\mathcal{U}^{IC}_m}\left|u(y_2)-
u(y_1)\right| \leq L|y_2-y_1|,
\end{align*}
$\underline{f}^{IC}$ is Lipschitz continuous with modulus bounded by $L$. 
Then, $\left|\underline{\alpha}^{IC}_{m,k+1}-\underline{\alpha}^{IC}_{m,k}\right|\leq L |y^m_{k+1}-y^m_k|$ and 
$\underline{y}^{m}_{k}=y^m_{k+1}-\left(\underline{\alpha}^{IC}_{m,k+1}-\underline{\alpha}^{IC}_{m,k}\right)/{L}\in[y^m_k,y^m_{k+1}]$, which implies that~\eqref{eq:pf-thm8-u} is well-defined. 
From~\eqref{eq:pf-thm8-u}, it is obvious that $\underline{u}$ is monotonically increasing and Lipschitz continuous with modulus $L$, thus, $\underline{u}\in\mathcal{U}_{IC}$. 
Furthermore, let $u^1_j,u^2_j\in\mathop{\arg\min}\limits_{u\in\mathcal{U}^{IC}_m}u(y^m_j)$. 
Then $u^1_j(y^m_{j})=\underline{u}(y^m_{j})= u^2_j(y^m_{j})$, for $j=1,\ldots,N_m$.  
By Lemma~\ref{lemma:mus-u-U_m}, $\underline{u}\in\mathcal{U}^{IC}_m$(
Lemma~\ref{lemma:mus-u-U_m} remains valid when $\mathcal U_{cv}$ is replaced by
$\mathcal U
_{{IC}}$),
and we indeed have $
{\underline{u}}
\in\mathcal U^{IC}_m$.

\vspace{0.1cm}
\noindent
\uline{Step~2}.
Since $\underline{u}\in\mathcal{U}^{IC}_m$, we have $\underline{f}^{IC}(y)=\min\limits_{u\in\mathcal{U}^{IC}_m}u(y)\leq \underline{u}(y)$, $\forall y\in[0,1]$. 
Thus it suffices to show that $\underline{u}(y)\leq \underline{f}^{IC}(y)$. 
For any $y\in[y^m_{k}, \underline{y}^{m}_{k}]$, by the monotonicity of $\underline{f}^{IC}$, it is obvious that 
\begin{align*}
    \underline{f}^{IC}(y)\geq\underline{f}^{IC}(y^m_{k})=\underline{\alpha}^{IC}_{m,k}=\underline{u}(y). 
\end{align*}
For any $y\in[\underline{y}^{m}_{k}, y^m_{k+1}]$, by the monotonicity and Lipschitz continuity of $\underline{f}^{IC}$, 
we have 
\begin{align*}
\underline{f}^{IC}(y)\geq\underline{f}^{IC}(y^m_{k+1})
{-} L(y^m_{k+1}-y)=\underline{\alpha}^{IC}_{m,k+1}
{+}  
L(y-y^m_{k+1})=\underline{u}(y). 
\end{align*}
The two inequalities hold for $k=1,\ldots,N_m-1$.
\end{proof}

Observe that both $\underline{f}^{IC}$ and $\bar{f}^{IC}$ belong to $\mathcal{U}_m^{IC}$.
This implies that one can identify a unique utility function within $\mathcal{U}_m^{IC}$ that acts as their uniform or pointwise upper (lower) bound. 
This coincides with the lower-bound result (Theorem~\ref{thm:mus-lower-bound}) in the concave case $\mathcal{U}_m$, but not with the corresponding upper-bound result (Theorem~\ref{thm:mus-upper-bound}). 
The underlying mathematical reason is that the maximum or minimum of two increasing or Lipschitz-continuous functions still preserves these properties, whereas the maximum of two concave functions is generally no longer concave—unlike the minimum. 
This suggests that under incomplete information about the true utility function, the utility preference against the worst-case scenario of two risk-averse investors may still
{exhibit} a risk-averse consensus, while their attitudes toward the best-case scenario may differ. 
Beyond risk-aversion, it is more common to reach consensus on rationality or bounded rationality.



\subsection{S-shaped utility functions}\label{sec:s-shape}


We now turn to the case where the true utility function is S-shaped.
Let $\mathcal{U}_{S}\subseteq\mathcal{L}^1([-1,1])$ be an ambiguity set of normalized utility functions that are monotonically increasing, Lipschitz continuous with modulus bounded by $L$,  convex over $[-1,0]$ and concave over $[0,1]$, i.e., 
\begin{align}\label{eq:U_s}
\nonumber
\mathcal{U}_{S}:= \{u\in & \mathcal{L}^1([-1,1])\mid \ 
 u^{'}_{+}(y)\geq0, u^{'}_{-}(y)\geq0,\forall y\in
 [-1,1], 
 u(-1)=-1, u(0)=0, u(1)=1, \\
 \nonumber
&\text{Lip}(u)\leq L, u^{'}_{+}(y)\leq u^{'}_{-}(y), \forall y\in
{(0,1)}, 
u^{'}_{-}(y_2)\leq u^{'}_{+}(y_1), \forall 0\leq y_1<y_2\leq1,\\
&u^{'}_{+}(y)\geq u^{'}_{-}(y), \forall y\in
{(-1,0)}, 
u^{'}_{-}(y_2)\geq u^{'}_{+}(y_1), \forall -1\leq y_1<y_2\leq0\}.  
\end{align}
Compared to the ambiguity set $\mathcal{U}_{cv}$ in~\eqref{eq:U_c}, we extend the domain of utility functions to $[-1,1]$ in order to account for losses, and we assume that $y=0$ is the 
prespecified reference point for the S-shaped utility function.

The MUS scheme for $S$-shaped utility functions specifies a paired gambling questionnaire $(W_m,Y_m)$ as in~\eqref{eq:W-Y} by setting $r^m_1=-1$ and $r^m_3=1$ for all $m\in\mathbb{Z}$, thus $\mathbb{E}[u(W_m)]=p^m u(r^m_3)+(1-p^m) u(r^m_1)=2p^m-1$ and $\mathbb{E}[u(Y_m)]=u(r^m_2)$. 
Let ${\mathcal{U}}^S_0 :=\mathcal{U}_{S}$ and let ${\mathcal{U}}^{S}_m\subseteq\mathcal{U}_{S}$ be the ambiguity set updated with paired gambling questionnaires generated by the MUS scheme, $m=1,2,\cdots$, i.e., 
\begin{equation}\label{eq:U_s-m}
    \mathcal{U}^{S}_m =\Big\{u\in\mathcal{U}_{S}\mid Z_k\cdot u(r_2^{
    k
    })\geq Z_k\cdot (2p^k-1),\ k=1,\ldots,m\Big\}. 
\end{equation} 
Let $\mathcal{S}_m :=\{-1,0,1\}\cup \bigcup\limits_{k=1}^{m}\{r_2^k\}$ and 
$\mathbb{Y}_m:= \{{y}^m_{j}\}_{j=1,\ldots,N_m}$ be the set of all points in $\mathcal{S}_m$ sorted in increasing order with fixed $y^m_1=-1$, $y^m_{N_m}=1$ and $y^m_{j_0}=0$ 
for some $j_0\in\{2,\ldots,N_m-1\}$.

For each $r_2\in[-1,1]$, we define the upper bound and lower bound functions of the ambiguity set $\mathcal{U}^{S}_m$ as: 
\begin{equation}\label{eq:U_s-upper-lower-func}
\bar{f}^{S}(r_2):=\max\limits_{u\in\mathcal{U}^{S}_m} u(r_2) 
\quad\text{and}\quad
\underline{f}^{S}(r_2):=\min\limits_{u\in\mathcal{U}^{S}_m} u(r_2). 
\end{equation}
Thus, the 
main computational 
task in the MUS scheme 
is 
to solve 
the following problem to determine $r^{m+1}_2$
\begin{equation}
    \mathop{\max}\limits_{r_2\in [-1,1]} f^{S}(r_2):= 
    \bar{f}^{S}(r_2)-\underline{f}^{S}(r_2) = \left[\max\limits_{u\in \mathcal{U}_{m}^{S}}u(r_2)-\min\limits_{u\in \mathcal{U}_{m}^{S}}u(r_2)\right], 
\end{equation} 
and specifies $p^{m+1}$ as
\begin{equation}
p^{m+1}=\frac{1}{2}\left[\frac{1}{2}\left( \bar{f}^S(r^{m+1}_2) + \underline{f}^S (r^{m+1}_2)\right)+1\right].
\end{equation}



\begin{theorem}[Semi-closed forms of $\underline{f}^{S}$ and  $\bar{f}^{S}$]\label{thm:s-upper-lower}
Consider the ambiguity set  ${\mathcal{U}}^{S}_m\subseteq\mathcal{U}_{S}$ generated by the MUS scheme, as in~\eqref{eq:U_s-m}. 
Let
\begin{align*}
&
\underline{\alpha}^{S}_{m,k}:=\min\limits_{u\in\mathcal{U}^{S}_m}u(y^m_k),\; 
\bar{\alpha}^{S}_{m,k}:=\max\limits_{u\in\mathcal{U}^{S}_m}u(y^m_{k}),\;\; k=1,\ldots,N_m,\\
& \underline{\beta}^{S}_{m,k}:=\min\limits_{u\in\{u\in\mathcal{U}^{S}_m\mid u(y^m_k)=\underline{\alpha}^{S}_{m,k}\}} u^{'}_{+}(y^m_{k}),\quad 
\bar{\beta}^{S}_{m,k+1}:=\max\limits_{u\in\{u\in\mathcal{U}^{S}_m\mid 
{u(y^m_{k+1})=\underline{\alpha}^{S}_{m,k+1}}
\}} u^{'}_{-}(y^m_{k+1}),\;\; k=1,\ldots,j_0-1,\\
& \bar{\beta}^{S}_{m,k}:=\max\limits_{u\in\{u\in\mathcal{U}^{S}_m\mid u(y^m_k)=\bar{\alpha}^{S}_{m,k}\}} u^{'}_{+}(y^m_k),\quad
\underline{\beta}^{S}_{m,k+1}:=\min\limits_{u\in\{u\in\mathcal{U}^{S}_m\mid 
{u(y^m_{k+1})=\bar{\alpha}^{S}_{m,k+1}}
\}}u^{'}_{-}(y^m_{k+1}),\;\; k=j_0,\ldots,N_m-1. 
\end{align*}
Then the following assertions hold. 
See Figure~\ref{fig:SS-upper-lower-two-piece} for an illustration. 
\begin{itemize}
 
\item[(i)] The lower bound function $\underline{f}^{S}$ is piecewise linear and concave over $[0,1]${\rm :} 
\begin{equation}
\label{eq:U_s-lower-1}
\underline{f}^{S}(y)=
\left\{
\begin{aligned}
    & 0 && {\rm for} \quad y=0,\\
    &\frac{\underline{\alpha}^{S}_{m,k+1}-\underline{\alpha}^{S}_{m,k}}{y^m_{k+1}-y^m_{k}}(y-y^m_{k})+\underline{\alpha}^{S}_{m,k} && {\rm for} \quad y^m_{k} < y \leq y^m_{k+1},\;\; k=j_0,\ldots,N_m-1,\\
    & 1 && {\rm for} \quad y=1,\\
\end{aligned}
\right. 
\end{equation}
and is piecewise linear and convex over $[y^m_k,y^m_{k+1}]\subset[-1,0]$ for $k=1,\ldots,j_0-1$ with the following structure{\rm :} 
\begin{itemize}

\item 
if  $\underline{\beta}^{S}_{m,k}=\bar{\beta}^{S}_{m,k+1}$, then  
\bgeqn 
\label{eq:U_s-lower-2}
\underline{f}^{S}(y)=
\underline{\alpha}^{S}_{m,k}+\frac{\underline{\alpha}^{S}_{m,k+1}-\underline{\alpha}^{S}_{m,k}}{y^m_{k+1}-y^m_{k}}(y-y^m_k)
\quad
\text{\rm for}\quad  y_k^m\leq y \leq y^m_{k+1};
\edeqn

\item 
if $\underline{\beta}^{S}_{m,k}<\bar{\beta}^{S}_{m,k+1}$, then 
\begin{equation}
\label{eq:U_s-lower-3}
\underline{f}^{S}(y)=\left\{ 
\begin{aligned}
& \underline{\alpha}^{S}_{m,k}+\underline{\beta}^{S}_{m,k}(y-y^m_k) 
&& \text{\rm for}\quad  y_k^m\leq y \leq y^*_{m,k},\\
&\underline{\alpha}^{S}_{m,k+1}+\bar{\beta}^{S}_{m,k+1}(y-y_{k+1}^m) 
&& \text{\rm for}\quad y^*_{m,k}< y \leq y_{k+1}^m, 
\end{aligned}
\right.
\end{equation}
where
\begin{equation*}
y^*_{m,k}=\frac{\underline{\alpha}^{S}_{m,k}-\underline{\alpha}^{S}_{m,k+1}-\underline{\beta}^{S}_{m,k}y^m_{k}+\bar{\beta}^{S}_{m,k+1} y^m_{k+1}}{\bar{\beta}^{S}_{m,k+1}-\underline{\beta}^{S}_{m,k}}\in (y^m_k,y^m_{k+1}). 
\end{equation*}

\end{itemize}

\item[(ii)]
The upper bound function $\bar{f}^{S}$ is piecewise linear and convex over $[-1,0]:$
\begin{equation}\label{eq:U_s-upper-1}
\bar{f}^{S}(y)=
\left\{
\begin{aligned} 
    & -1 && {\rm for}\quad y=-1,\\
    &\frac{\bar{\alpha}^{S}_{m,k+1}-\bar{\alpha}^{S}_{m,k}}{y^m_{k+1}-y^m_{k}}(y-y^m_{k})+\bar{\alpha}^{S}_{m,k} && {\rm for}\quad y^m_{k} < y \leq y^m_{k+1},\;\; k=1,\ldots,j_0-1,\\
    & 0 && {\rm for} \quad y=0,\\
\end{aligned}
\right.    
\end{equation}
and is piecewise linear and concave over $[y^m_k,y^m_{k+1}]\subset[0,1]$ for $k=j_0,\ldots,N_m-1$ with the following structure{\rm :}

\begin{itemize}
\item 
if $\bar{\beta}^{S}_{m,k}=\underline{\beta}^{S}_{m,k+1}$, then  
\bgeqn 
\label{eq:U_s-upper-2}
\bar{f}^{S}(y)=
\bar{\alpha}^{S}_{m,k}+\frac{\bar{\alpha}^{S}_{m,k+1}-\bar{\alpha}^{S}_{m,k}}{y^m_{k+1}-y^m_{k}}(y-y^m_k)
\quad
\text{\rm for}\quad  y_k^m\leq y \leq y^m_{k+1}; 
\edeqn

\item 
if $\bar{\beta}^{S}_{m,k}>\underline{\beta}^{S}_{m,k+1}$, then 
\begin{equation}
\label{eq:U_s-upper-3}
\bar{f}^{S}(y)=\left\{ 
\begin{aligned}
& \bar{\alpha}^{S}_{m,k}+\bar{\beta}^{S}_{m,k}(y-y^m_k) 
&& \text{\rm for}\quad  y_k^m\leq y \leq y^*_{m,k},\\
&\bar{\alpha}^{S}_{m,k+1}+\underline{\beta}^{S}_{m,k+1}(y-y_{k+1}^m) 
&& \text{\rm for}\quad y^*_{m,k}< y \leq y_{k+1}^m, 
\end{aligned}
\right.
\end{equation}
where
\bgeqn 
y^*_{m,k}=\frac{\bar{\alpha}^{S}_{m,k+1}-\bar{\alpha}^{S}_{m,k}-\underline{\beta}^{S}_{m,k+1}y^m_{k+1}+\bar{\beta}^{S}_{m,k}y^m_k}{\bar{\beta}^{S}_{m,k}-\underline{\beta}^{S}_{m,k+1}}\in (y^m_k,y^m_{k+1}).
\edeqn 

\end{itemize}

\item[(iii)] $f^{S}=\bar{f}^{S}-\underline{f}^{S}$ is also a piecewise linear function, with at most three pieces over each interval $[y^m_k,y^m_{k+1}]$, $k=1,\ldots,N_m-1$. 
\end{itemize}
\end{theorem}

\begin{figure}[htbp]
\subfigure[{$\bar{\beta}^S_{m,5}=\underline{\beta}^S_{m,6}=\frac{\bar{\alpha}^S_6-\bar{\alpha}^S_5}{y^m_6-y^m_5}$ and $\underline{\beta}^S_{m,2}=\bar{\beta}^S_{m,3}=\frac{\underline{\alpha}^S_{m,3}-\underline{\alpha}^S_{m,2}}{y^m_3-y^m_2}$}]{
    \includegraphics[width=0.45\linewidth]{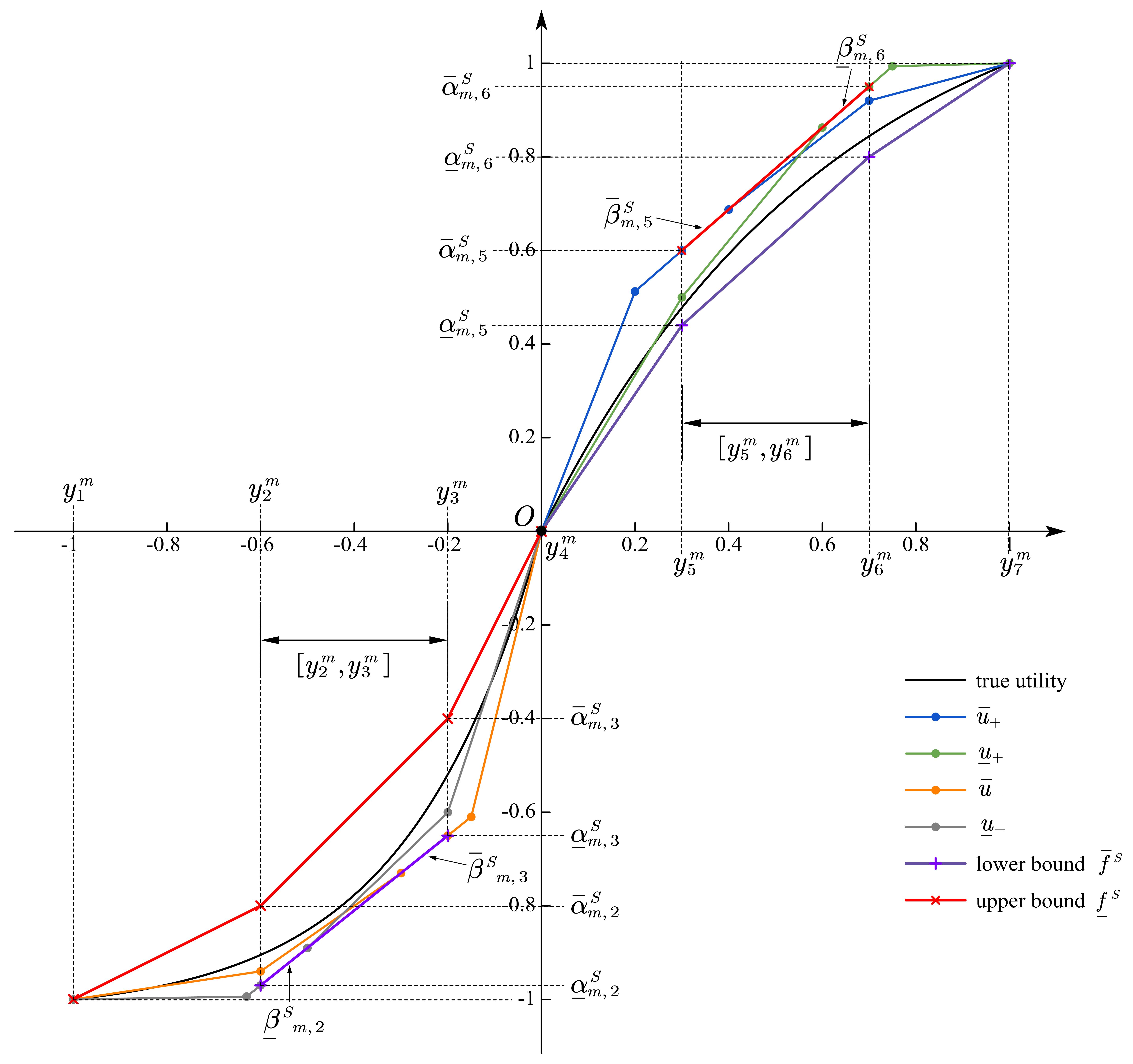}
}
\quad
\subfigure[$\bar{\beta}^S_{m,5}>\frac{\bar{\alpha}^S_6-\bar{\alpha}^S_5}{y^m_6-y^m_5}>\underline{\beta}^S_{m,6}$ and $\underline{\beta}^S_{m,2}<\frac{\underline{\alpha}^S_{m,3}-\underline{\alpha}^S_{m,2}}{y^m_3-y^m_2}<\bar{\beta}^S_{m,3}$]{
    \includegraphics[width=0.45\linewidth]{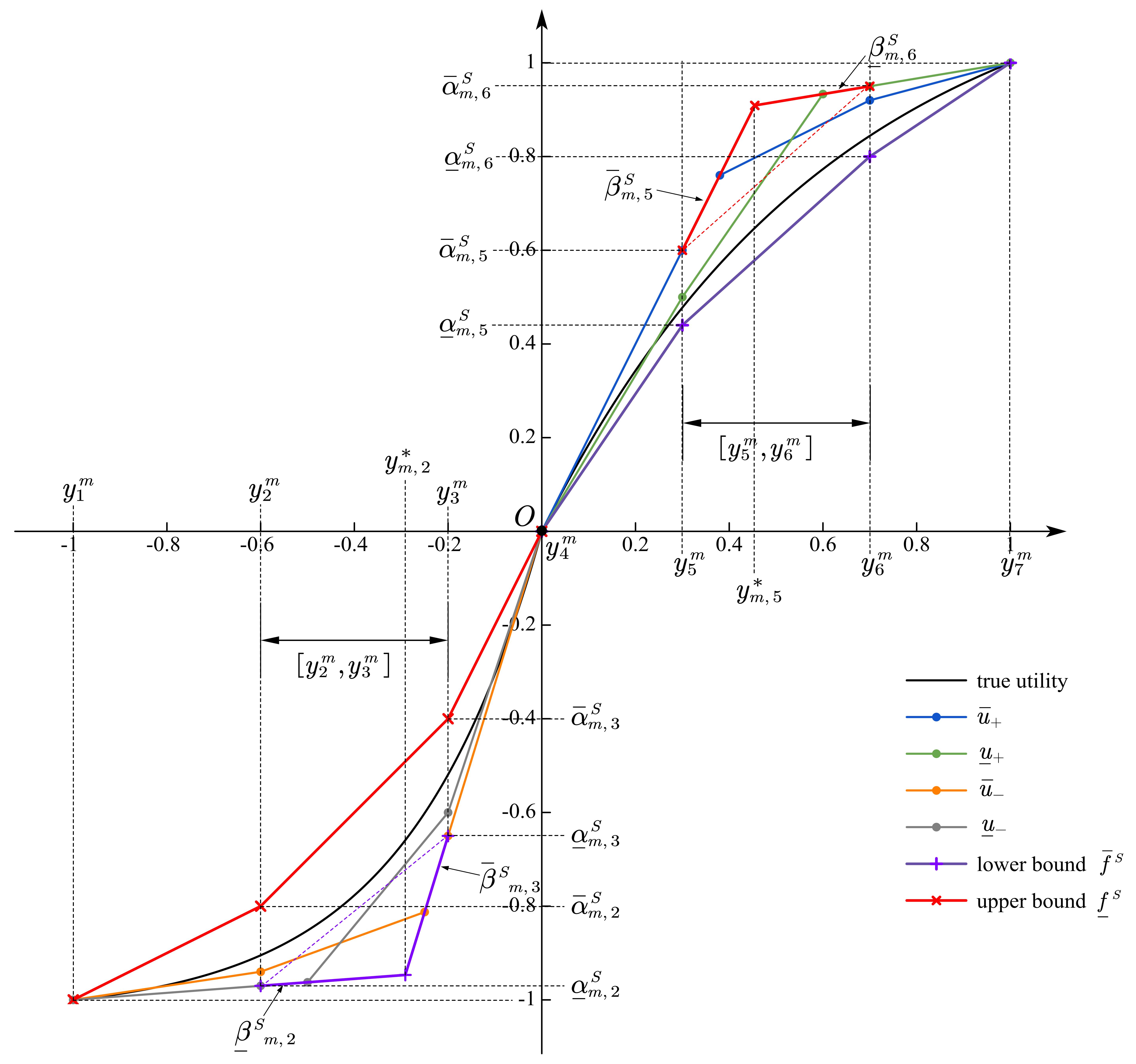}
}
\caption{
Illustration of the graphs of the lower bound function $\underline{f}^S$ over intervals $[y^m_2,y^m_3]$ and $[0,1]$, and the upper bound function $\bar{f}^S$ over intervals $[-1,0]$ and $[y^m_5,y^m_6]$. \rm $\underline{f}^S$ is either linear or convex comprising two linear pieces over $[y^m_2,y^m_3]$ depending on the relationship between $\underline{\beta}^S_{m,2}$ and $\bar{\beta}^S_{m,3}$, and the purple segment over $[y^m_2,y^m_3]$ is the graph of $\underline{f}^S$ restricted to $[y^m_2,y^m_3]$, which is constructed from $\bar{u}_{-}$ and $\underline{u}_{-}$; $\bar{f}^S$ is either linear or concave comprising two linear pieces over $[y^m_5,y^m_6]$ depending on the relationship between $\bar{\beta}^S_{m,5}$ and $\underline{\beta}^S_{m,6}$, and the red segment over $[y^m_5,y^m_6]$ is the graph of $\bar{f}^S$ restricted to $[y^m_5,y^m_6]$, which is constructed from $\bar{u}_{+}$ and $\underline{u}_{+}$. 
There are two additional cases: 
(c) $\bar{\beta}^S_{m,5}>\frac{\bar{\alpha}^S_6-\bar{\alpha}^S_5}{y^m_6-y^m_5}>\underline{\beta}^S_{m,6}$ and $\underline{\beta}^S_{m,2}=\bar{\beta}^S_{m,3}=\frac{\underline{\alpha}^S_{m,3}-\underline{\alpha}^S_{m,2}}{y^m_3-y^m_2}$; and 
(d) $\bar{\beta}^S_{m,5}=\underline{\beta}^S_{m,6}=\frac{\bar{\alpha}^S_6-\bar{\alpha}^S_5}{y^m_6-y^m_5}$ and $\underline{\beta}^S_{m,2}<\frac{\underline{\alpha}^S_{m,3}-\underline{\alpha}^S_{m,2}}{y^m_3-y^m_2}<\bar{\beta}^S_{m,3}$. 
These cases can be easily visualized by combining the figures from cases~(a)~and~(b), and are therefore omitted here.}
\label{fig:SS-upper-lower-two-piece}
\end{figure}

\begin{proof}

{To improve the clarity of the proof, we first introduce the following definitions. }

\begin{definition}
Define operator $\oplus: \mathcal{L}^1([-1,0]) \times \mathcal{L}^1([0,1])\to\mathcal{L}^1([-1,1])$ which maps each pair of convex and concave functions $(u^{-},u^{+})$ to 
{an}
S-shaped function as follows{\rm :} 
\begin{equation*}
(u^{-}\oplus u^{+})(y)=
\left\{
\begin{aligned}
    & u^{-}(y) && {\rm for} \quad -1\leq y \leq 0,\\
    & u^{+}(y) && {\rm for} \quad 0< y \leq 1. 
\end{aligned}
\right.
\end{equation*}
Let $\mathcal{U}_{S}^{+}:=\mathcal{U}_{cv}$ be defined in~\eqref{eq:U_c}, and let 
\begin{align*}
\mathcal{U}_{S}^{-}:= \Big\{ u\in\   \mathcal{L}^1([-1,0])\mid \ &
 u^{'}_{+}(y)\geq0,\ u^{'}_{-}(y)\geq0,\ \forall y\in[-1,0],\ u(-1)=-1,\ u(0)=0,\ \text{Lip}(u)\leq L,\\
& u^{'}_{+}(y)\geq u^{'}_{-}(y),\ \forall y\in[-1,0],\ 
u^{'}_{-}(y_2)\geq u^{'}_{+}(y_1),\ \forall -1\leq y_1<y_2\leq0 \Big\}. 
\end{align*}
Then, it holds that  $\mathcal{U}_S=\left\{u^{-}\oplus u^{+}\mid u^{-}\in\mathcal{U}_S^{-},\ u^{+}\in\mathcal{U}_S^{+}\right\}:=\mathcal{U}_S^{-}\oplus\mathcal{U}_S^{+}$. 
Moreover,
since $\mathbb{Y}_m= \{{y}^m_{j}\}_{j=1,\ldots,N_m}$ is the set of all points in $\mathcal{S}_m=\{-1,0,1\}\cup \bigcup\limits_{k=1}^{m}\{r_2^k\}$ sorted in increasing order with fixed $y^m_1=-1$, $y^m_{j_0}=0$ and $y^m_{N_m}=1$, 
let $r_2^{k_j}\in\mathcal{S}_m$ be such that $y^m_j=r_2^{k_j}\in\mathbb{Y}_m$ for some $k_j\in\{1,2,\ldots,N_m\}$. Then $\mathcal{U}^{S}_m$ in~\eqref{eq:U_s-m} can be recast as 
\begin{equation*}
    \mathcal{U}^{S}_m=\left\{u\in\mathcal{U}_S\mid Z_{k_j} \cdot u(y^m_{j})\geq Z_{k_j}\cdot (2p^{k_j}-1),\ j=1,\ldots,N_m \right\}. 
\end{equation*} 
Let 
\begin{align}
& \mathcal{U}_m^{S-}:=\left\{u\in\mathcal{U}_S^{-}\mid Z_{k_j} \cdot u(y^m_{j})\geq Z_{k_j}\cdot(2p^{k_j}-1),\ 
j=1,\ldots,j_0\right\},  
\label{eq:U_m_s-}\\
& \mathcal{U}_m^{S+}:=\left\{u\in\mathcal{U}_S^{+}\mid Z_{k_j} \cdot u(y^m_{j})\geq Z_{k_j}\cdot (2p^{k_j}-1),\ j=j_0,\ldots,N_m\right\}.
\label{eq:U_m_s+}
\end{align}
Then we have 
$\mathcal{U}^{S}_m=\mathcal{U}_m^{S-}\oplus\mathcal{U}_m^{S+}$. 

\end{definition}
With the operator defined above, for
any $r_2\in[-1,1]$, the upper and lower bound functions defined in~\eqref{eq:U_s-upper-lower-func} can be recast as 
\begin{equation}
\bar{f}^{S}(r_2)=\max\limits_{u\in \mathcal{U}^{S}_{m}}u(r_2)= \max\limits_{u^{-}\in \mathcal{U}_{m}^{S-},u^{+}\in\mathcal{U}_{m}^{S+}}\ (u^{-}\oplus u^{+})(r_2)=
\left\{
\begin{aligned}
& \max\limits_{u\in\mathcal{U}_m^{S-}} u(r_2) && {\rm for}\quad -1\leq r_2\leq 0,\\
& \max\limits_{u\in\mathcal{U}_m^{S+}} u(r_2) && {\rm for}\quad 0< r_2\leq 1,
\end{aligned}
\right.
\end{equation}
and
\begin{equation}
\underline{f}^{S}(r_2)=\min\limits_{u\in \mathcal{U}^{S}_{m}}u(r_2)=\min\limits_{u^{-}\in \mathcal{U}_{m}^{S-},u^{+}\in\mathcal{U}_{m}^{S+}}(u^{-}\oplus u^{+})(r_2)=
\left\{
\begin{aligned}
& \min\limits_{u\in\mathcal{U}_m^{S-}} u(r_2) && {\rm for}\quad -1\leq r_2\leq 0,\\
& \min\limits_{u\in\mathcal{U}_m^{S+}} u(r_2) && {\rm for}\quad 0< r_2\leq 1,
\end{aligned}
\right.
\end{equation}
which indicates that we can consider the two bound functions over $[-1,0]$ and $[0,1]$ separately.

\begin{definition}\label{def:rotate}
Define a one-to-one mapping $\mathcal{R}: \mathcal{L}^1([-1,0])\to\mathcal{L}^{1}([0,1])$ that maps each $u\in\mathcal{U}_m^{S-}$ to a function $\mathcal{R}u$ in $\mathcal{U}^{+}_S$ as follows{\rm :} 
\begin{equation*}
    \mathcal{R}u(y)=-u(-y), \;\; \forall y\in[0,1]. 
\end{equation*}
Since each $u\in\mathcal{U}_m^{S-}$ is monotonically increasing, convex, and Lipschitz continuous with modulus bounded by $L$ over $[-1,0]$, it follows directly that $\mathcal{R}u$ is monotonically increasing, concave, and Lipschitz continuous over $[0,1]$. 
Let
\begin{equation}\label{eq:U_rm}
\mathcal{U}_m^{\mathcal{R}-}:=\left\{\mathcal{R}u:\  u\in\mathcal{U}_m^{S-}\right\}
=\left\{u\in\mathcal{U}_{cv} \mid -Z_{k_j} \cdot u(-y^m_{j})\geq Z_{k_j}\cdot(2p^{k_j}-1),\ 
j=1,\ldots,j_0\right\}. 
\end{equation}

\end{definition}


With the above preparations, we prove the theorem in two parts.

\vspace{0.1cm}
\noindent
\underline{Part~(i): Proof of the semi-closed form of $\underline{f}^{S}$ in~\eqref{eq:U_s-lower-1}--\eqref{eq:U_s-lower-3}.}

First, we have $\underline{f}^{S}(y)=\min\limits_{u\in\mathcal{U}_m^{S+}} u(y)$ for $y\in[0,1]$. 
Since $\mathcal{U}_m^{S+}$ in~\eqref{eq:U_m_s+} and $\mathcal{U}_m$ in~\eqref{eq:U_m-MUS-0} possess the same structural properties, determining $\underline{f}^{S}$ on $[0,1]$ is equivalent to determining the lower bound function $\underline{f}$ studied in Section~\ref{sec:lower-bound-f}. Hence, by Theorem~\ref{thm:mus-lower-bound}, $\underline{f}^{S}$ takes the form given in~\eqref{eq:U_s-lower-1} on $[0,1]$.

Next, for $y\in[-1,0]$, we have
\[
\underline{f}^{S}(y)=\min\limits_{u\in\mathcal{U}_m^{S-}} u(y)
=-\max\limits_{u\in\mathcal{U}_m^{S-}}\bigl(-u(y)\bigr)
=-\max\limits_{u\in\mathcal{U}_m^{S-}}\mathcal{R}u(-y)
=-\max\limits_{u\in\mathcal{U}_m^{\mathcal{R}-}} u(-y),
\]
where the last equality follows from the correspondence between $\mathcal{U}_m^{S-}$ and $\mathcal{U}_m^{\mathcal{R}-}$ in~\eqref{eq:U_rm}. Therefore, it remains to characterize $\max\limits_{u\in\mathcal{U}_m^{\mathcal{R}-}} u(y)$
for $y\in[0,1]$. Because $\mathcal{U}_m^{\mathcal{R}-}$ in~\eqref{eq:U_rm} and $\mathcal{U}_m$ in~\eqref{eq:U_m-MUS-0} again share identical structural features, determining $\underline{f}^{S}$ on $[-1,0]$ reduces to determining the upper bound function $\bar{f}$ on $[0,1]$ as analyzed in Section~\ref{sec:upper-bound-f}. By Theorem~\ref{thm:mus-upper-bound}, $\underline{f}^{S}$ has the form stated in~\eqref{eq:U_s-lower-2} or~\eqref{eq:U_s-lower-3} on each interval $[y_k^m,\,y_{k+1}^m]$, for $k=1,\ldots,j_0-1$. 


\vspace{0.1cm}

\noindent
\underline{Part~(ii): Proof of the semi-closed form of $\bar{f}^{S}$ in~\eqref{eq:U_s-upper-1}--\eqref{eq:U_s-upper-3}.}

First, we have $\bar{f}^{S}(y)=\max\limits_{u\in\mathcal{U}_m^{S+}} u(y)$ for $y\in[0,1]$.  
Since $\mathcal{U}_m^{S+}$ in~\eqref{eq:U_m_s+} and $\mathcal{U}_m$ in~\eqref{eq:U_m-MUS-0} possess identical structural properties, determining $\bar{f}^{S}$ on $[0,1]$ is equivalent to determining the upper bound function $\bar{f}$ studied in Section~\ref{sec:upper-bound-f}. Consequently, by Theorem~\ref{thm:mus-upper-bound}, $\bar{f}^{S}$ takes the form given in~\eqref{eq:U_s-upper-2} or~\eqref{eq:U_s-upper-3} on each interval $[y_k^m,\,y_{k+1}^m]$, for $k=j_0,\ldots,N_m-1$.

Next, for $y\in[-1,0]$, we have
\[
\bar{f}^{S}(y)=\max\limits_{u\in\mathcal{U}_m^{S-}} u(y)
=-\min\limits_{u\in\mathcal{U}_m^{S-}}\bigl(-u(y)\bigr)
=-\min\limits_{u\in\mathcal{U}_m^{S-}}\mathcal{R}u(-y)
=-\min\limits_{u\in\mathcal{U}_m^{\mathcal{R}-}} u(-y),
\]
where the last equality follows from the correspondence between $\mathcal{U}_m^{S-}$ and $\mathcal{U}_m^{\mathcal{R}-}$ in~\eqref{eq:U_rm}. Thus, it remains to characterize $\min\limits_{u\in\mathcal{U}_m^{\mathcal{R}-}} u(y)$
for $y\in[0,1]$. Because $\mathcal{U}_m^{\mathcal{R}-}$ in~\eqref{eq:U_rm} and $\mathcal{U}_m$ in~\eqref{eq:U_m-MUS-0} again share identical structural properties, determining $\bar{f}^{S}$ on $[-1,0]$ reduces to determining the lower bound function $\underline{f}$ on $[0,1]$ as analyzed in Section~\ref{sec:lower-bound-f}. By Theorem~\ref{thm:mus-lower-bound}, $\bar{f}^{S}$ therefore has the form stated in~\eqref{eq:U_s-upper-1} on $[-1,0]$.
\end{proof}

Since $\underline{f}^S$ is piecewise linear and piecewise convex over the interval $[-1,0]$, and $\bar{f}^S$ is piecewise linear and piecewise concave over the interval $[0,1]$, neither $\underline{f}^S$ nor $\bar{f}^S$ necessarily belongs to $\mathcal{U}^S_m$. 
Note 
that in the definition of the initial ambiguity set~\eqref{eq:U_s}, we 
implicitly assume that the true utility function changes from  convex to concave
at $0$. This assumption aligns with
the consensus in behavioral economics that a DM tends to be risk-averse in the domain of gains and risk-seeking in the domain of losses.
In practical applications, the turning point (also referred to as the reference point or reflection point in prospect theory~\citep{prospect}) may not be known in advance, or it may vary dynamically (see, e.g.,~\citet{SCY15}). Consequently, it is necessary to identify the reference point through pairwise comparisons before implementing the MUS. Since this issue falls outside the scope of the current paper, we leave it for future research.


\section{Numerical results}
\label{sec:numerical}

We have carried out some comparative numerical analyses on the proposed MUS
scheme in comparison with RUS~\citep{armbruster2015decision}, RRUS~\citep{armbruster2015decision}, and the polyhedral method~\citep{zhang2025modified}, by performing a number of tests on an academic example in Section~\ref{sec-academic} and a robo-advisor problem in Section~\ref{sec:robo-advisor}. 
All tests are carried out using Python 3.9.17 and PyCharm 2024.1.4 on a PC with 32~GB RAM and a 3.5~GHz Intel Core i5-13600KF processor. 
We use GUROBI to solve the problems in Section~\ref{sec-academic} and COPT for those in Section~\ref{sec:robo-advisor}.

\subsection{An academic example}\label{sec-academic}

We consider a DM whose preferences are described by the true utility function 
\begin{equation}
\label{eq:u*-true-test}
    u^*(y)=\frac{1}{1-e^{-6}}(1-e^{-6 y}),\quad \forall y\in[0,1],
\end{equation}
referred to as $u_{\rm true}$, which is monotonically increasing, concave, Lipschitz continuous and normalized to $[0,1]$.
The Lipschitz modulus of utility functions in $\mathcal{U}_{cv}$ is set to $L = 10$.

\subsubsection{Convergence of the nominal utility function to the true utility function}

\begin{figure}[tp]
    \centering
    \subfigure
    {
    \includegraphics[width=0.47\linewidth]{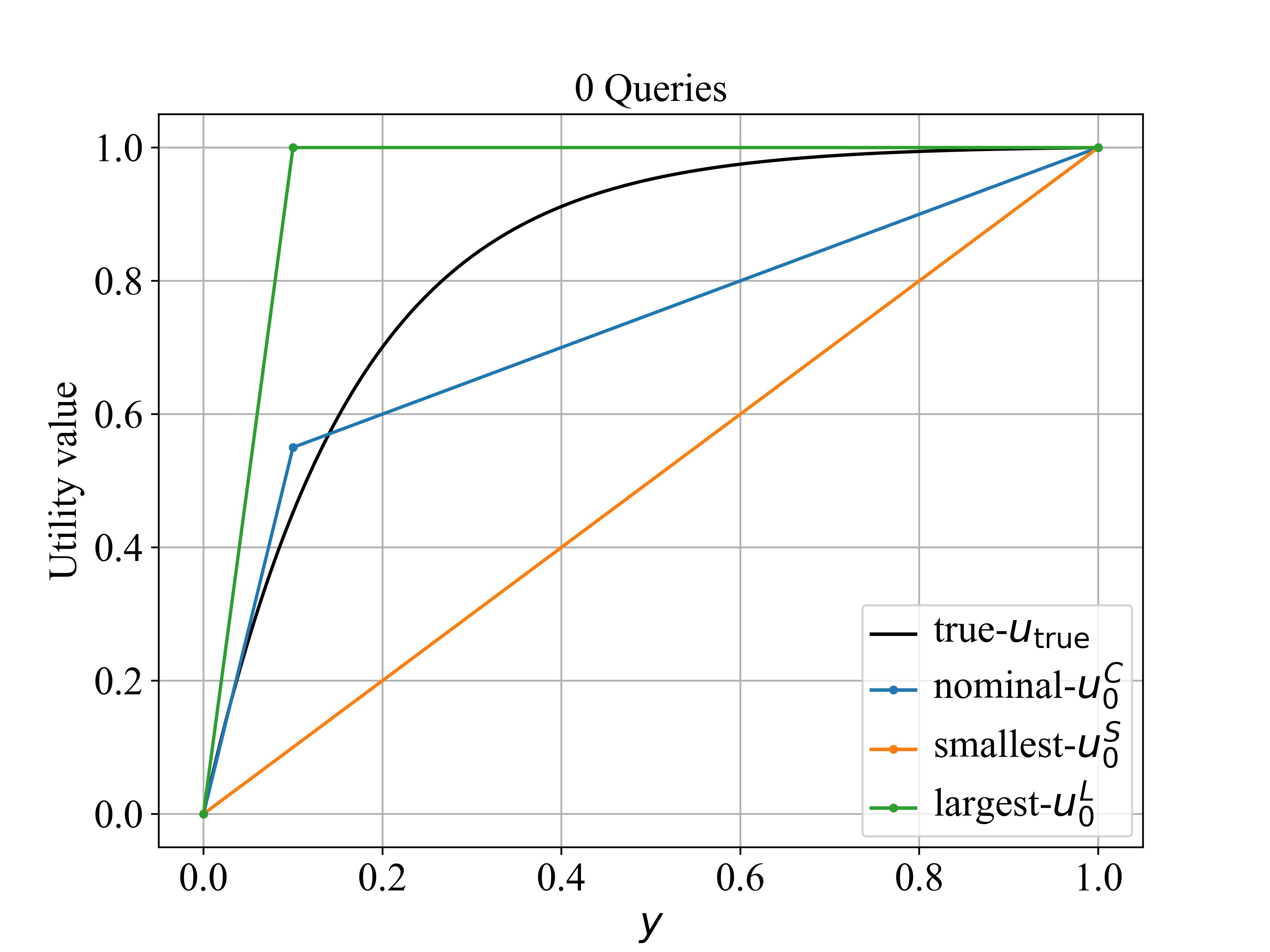}
    }
    \;
    \subfigure
    {
    \includegraphics[width=0.47\linewidth]{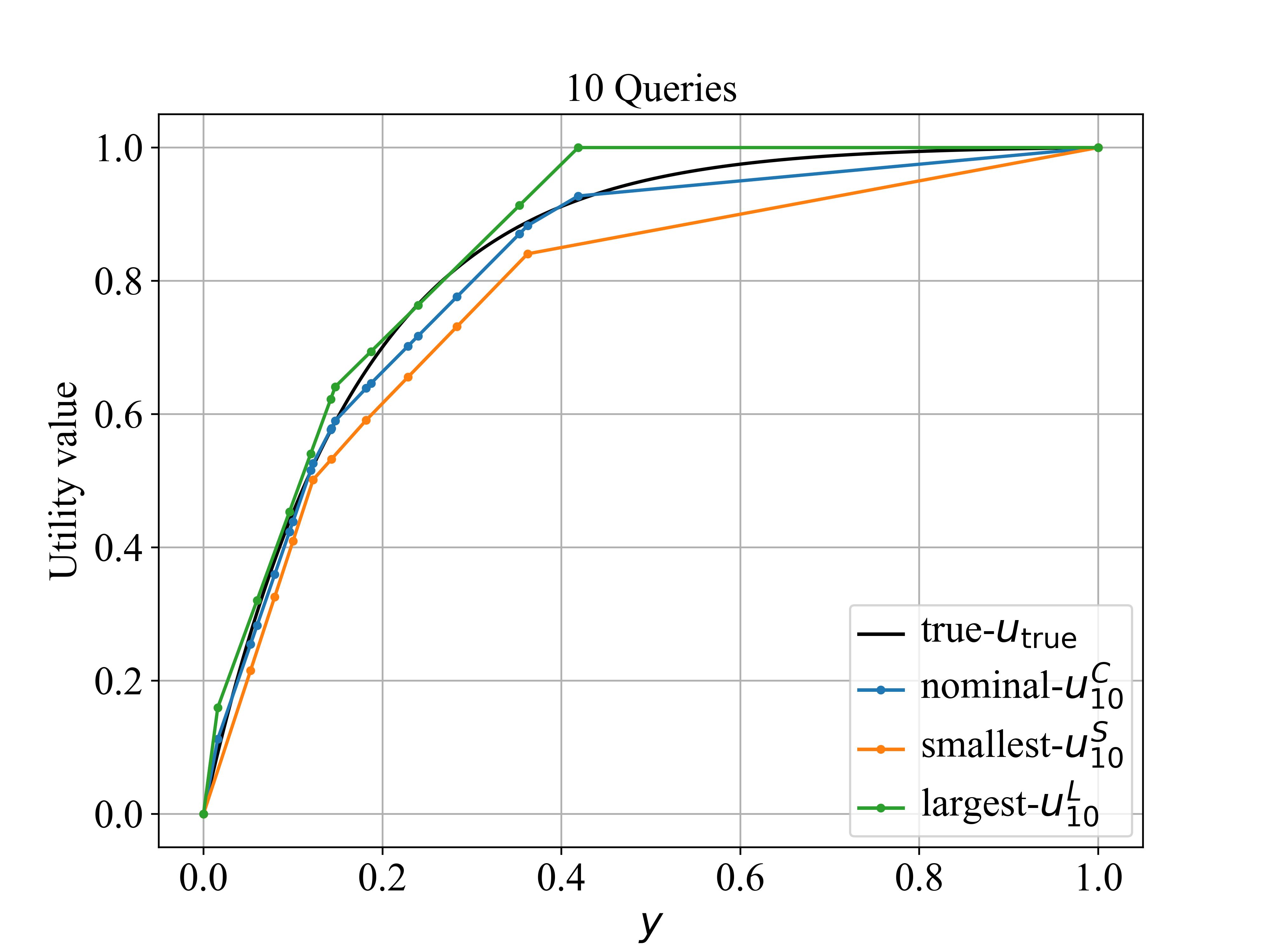}
    }\\
    \subfigure
    { 
    \includegraphics[width=0.47\linewidth]{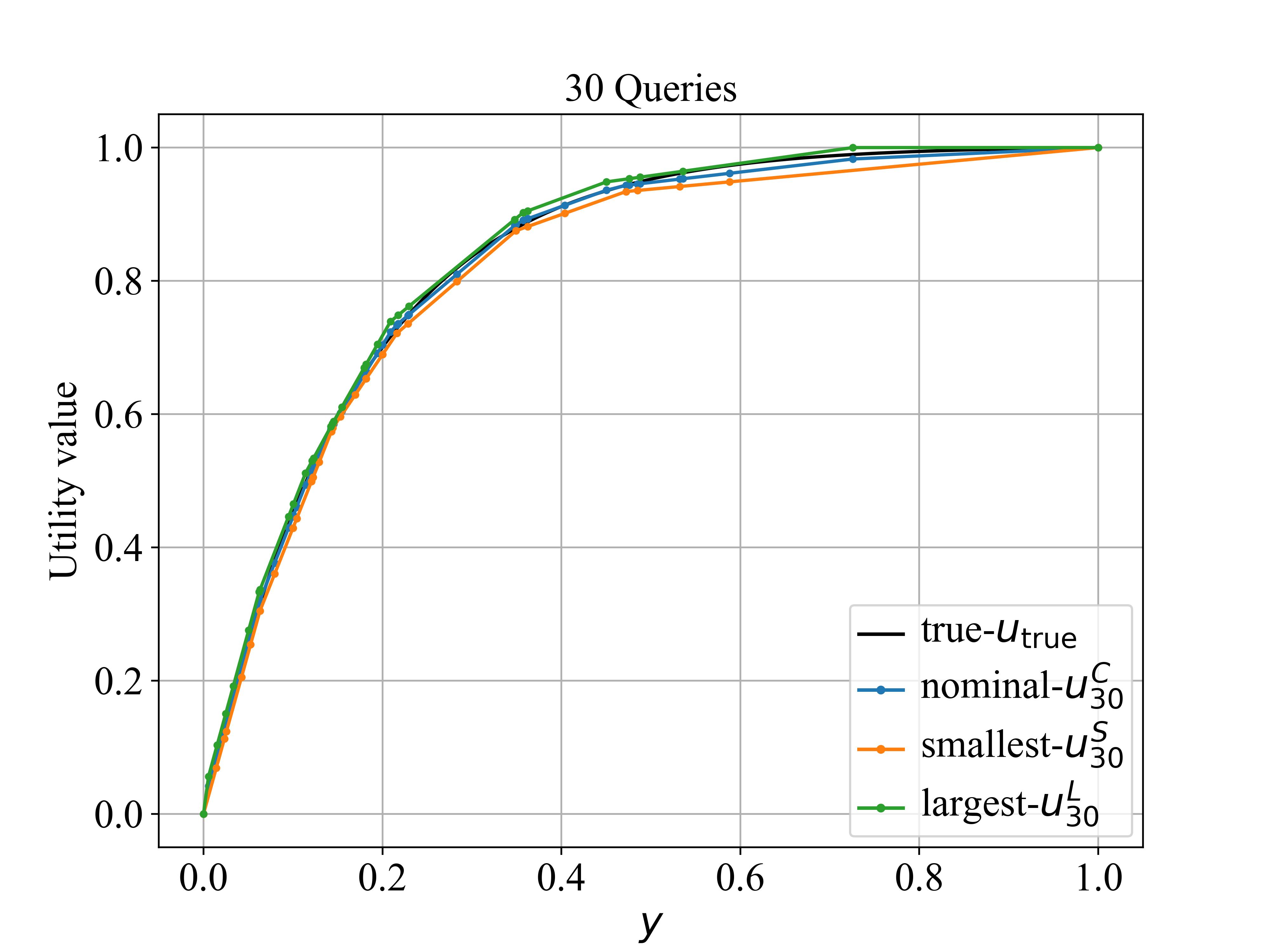}
    }
    \;
    \subfigure
    {
    \includegraphics[width=0.47\linewidth]{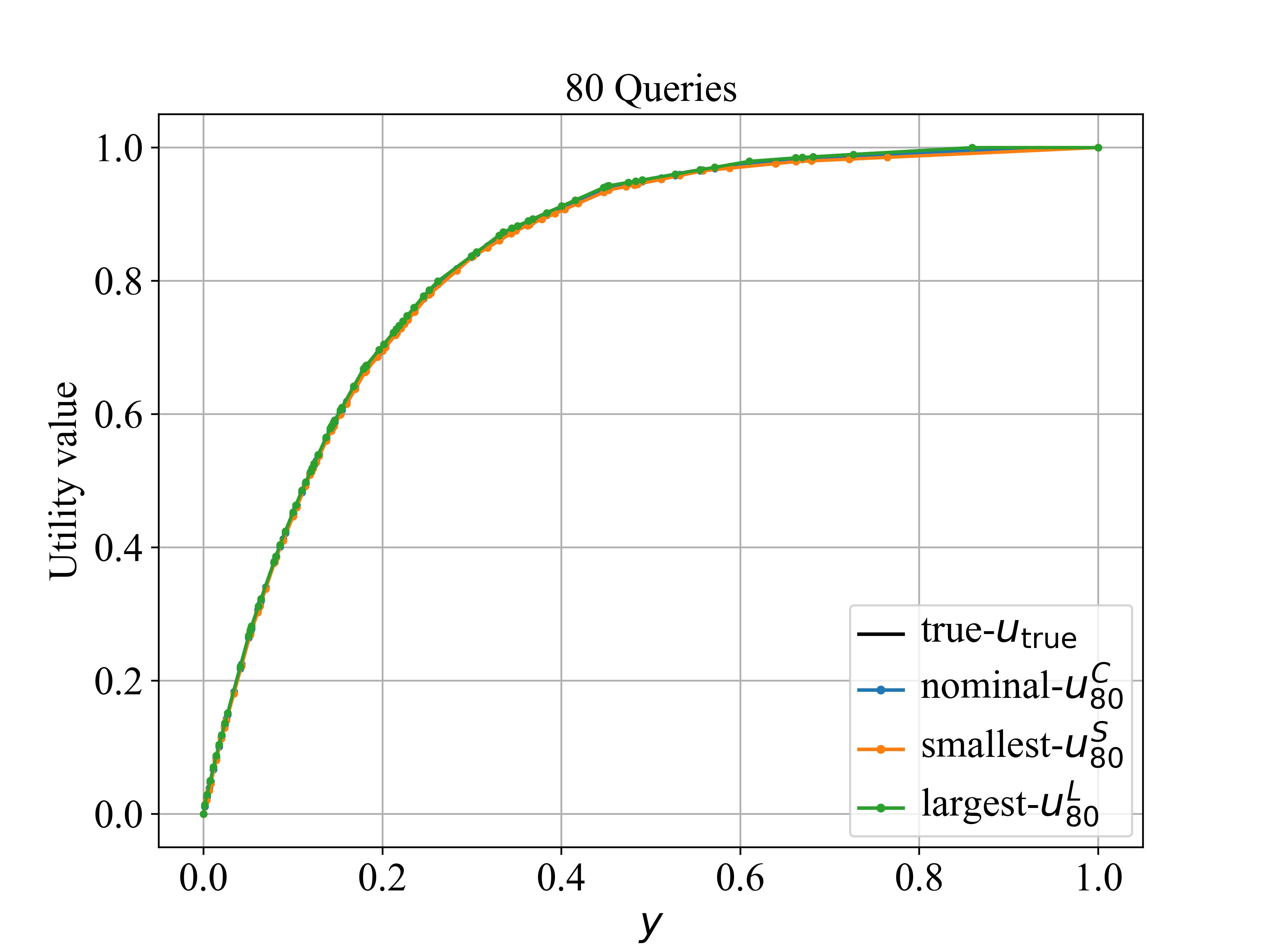}
    }
    \caption{
    Convergence of the smallest, the largest, and the nominal utility functions of $\mathcal{U}_m$ as $m$ varies from $0$ to $10, 30$ and $80$.
    }
    \label{fig:num-mus-fit}
\end{figure}

The aim of the tests in this subsection is to report
the preference elicitation procedure under the MUS scheme proposed in Section~\ref{sec:MUS}, by examining the nominal utility functions identified from the resulting ambiguity sets. 
We use the MUS scheme to adaptively generate a sequence of 
queries to 80 
during the
interaction with the DM.
At the 0-, 10-, 30- and 80-th round of interaction,
we compute 
the smallest, largest, and nominal utility functions as described in Section~\ref{sec:utility}, and present the results in Figure~\ref{fig:num-mus-fit}. 
We observe that the smallest, largest, and nominal utility functions all converge to the true utility function as the number of queries increases.
This is because the ambiguity set converges to the true utility function, as theoretically guaranteed by Theorem~\ref{thm:mus-converg}, and all three elicited utility functions are selected from this set. 
The smallest utility functions consistently lie below the true utility function, while the largest ones lie above it.

\subsubsection{MUS versus RUS, RRUS, and polyhedral method}\label{sec:num-mus-vs}

In this subsection, we compare the elicitation efficiency of the proposed MUS scheme with that of the RUS and RRUS schemes, as well as the polyhedral method in terms of convergence of the elicited ambiguity set
as the number of queries increases. 
We start with the initial ambiguity set $\mathcal{U}_0:=\mathcal{U}_{cv}$ and use MUS, RUS, RRUS, and polyhedral methods 
to adaptively generate a sequence of 50 queries to interact with the DM.\footnotemark 
After each interaction, we identify the smallest, largest, and nominal utility functions from the current ambiguity set $\mathcal{U}_m$, which is updated using adaptive queries generated by each of the four methods.  
We use the following distances to measure convergence of the elicited ambiguity set.

\footnotetext{
When applying the polyhedral method, we impose additional constraints on the vector of increments of the piecewise linear utility function in the initial polyhedron. These constraints correspond to the concavity and Lipschitz continuity of utility functions considered in this paper. 
We consider two cases of the modified polyhedral method: the first uses an initial breakpoint set of $[0, 0.25, 0.75, 1]$ (referred to as $\text{Poly}_1$), and the second uses an initial breakpoint set of $[0, \frac{1}{L}, 1]$ (referred to as $\text{Poly}_2$).} 

\begin{itemize}
    \item \underline{$\dd_K(u^L,u^S)$ and $\dd_I(u^L,u^S)$}, 
    which 
    measure
    the Kantorovich distance and Kolmogorov distance  
    between
    the largest utility function and the smallest utility function
 in $\mathcal{U}_m$.  
 Figure~\ref{fig:dist-large-small} depicts
 reduction of the distances as the number of queries increases. 
 We can see that MUS displays
 {the}
 fastest reduction in both distances; $\dd_K(u^L,u^S)$ is monotonically decreasing whereas 
 $\dd_I(u^L,u^S)$ is not. 
 This is because the largest utility function is defined in terms of Kantorovich distance and it is not reduced pointwise; RUS outperforms RRUS purely because of the concave structure of the utility function;
 the performance of 
 {the}
 polyhedral method depends on the choice of 
 {the}
 initial set of breakpoints. As we can see that the method starting with
 initial set of breakpoints $[0, \frac{1}{L}, 1]$ displays much faster convergence than the same method beginning with
  $[0, 0.25, 0.75, 1]$.

\begin{figure}[htbp]
    \centering
    \subfigure
    []
    {
    \label{fig:kanto-large-small}
    \includegraphics[width=0.45\linewidth]{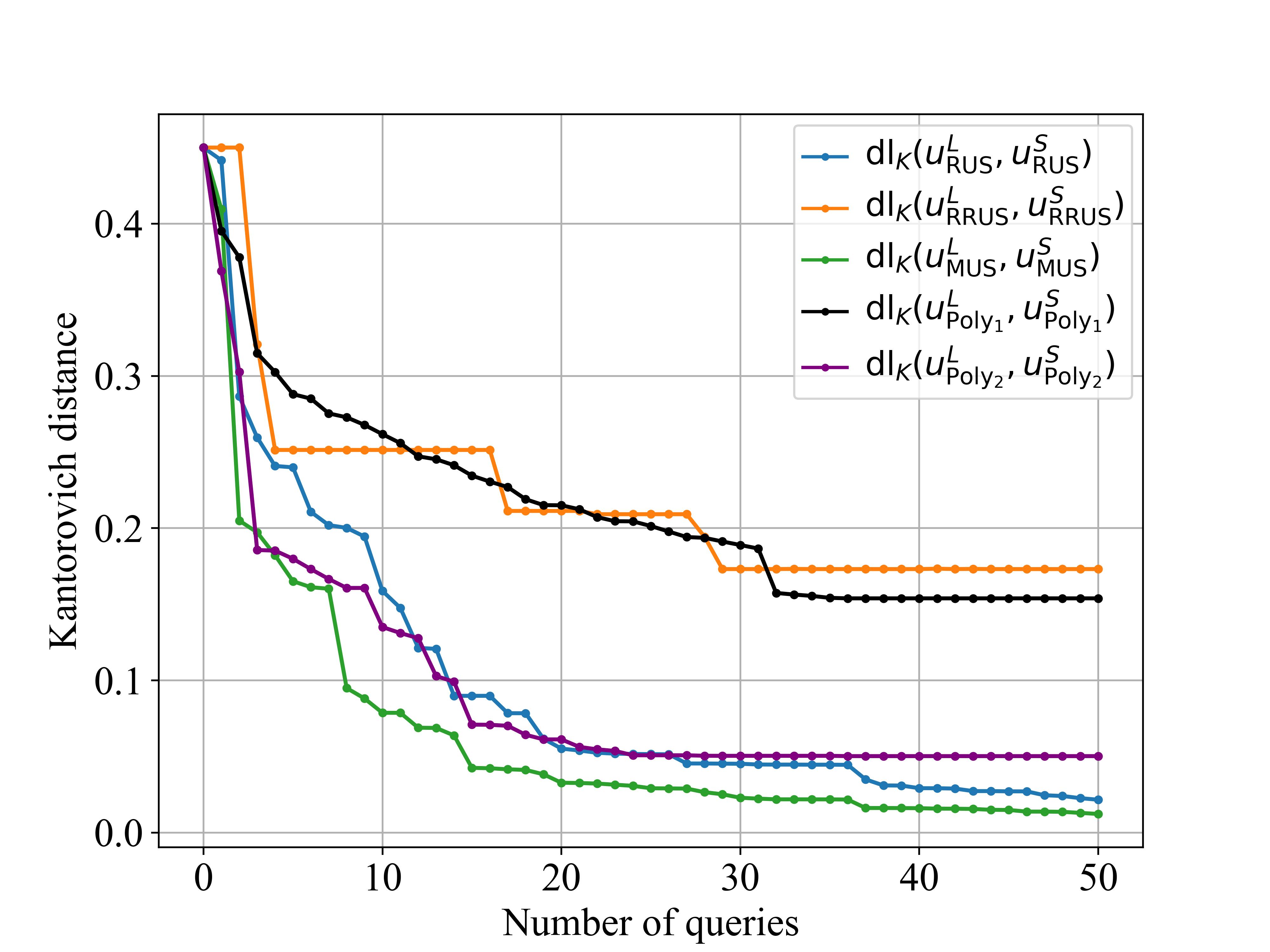}
    }
    \;
    \subfigure
    []
    {
    \label{fig:kolmo-large-small}
    \includegraphics[width=0.45\linewidth]{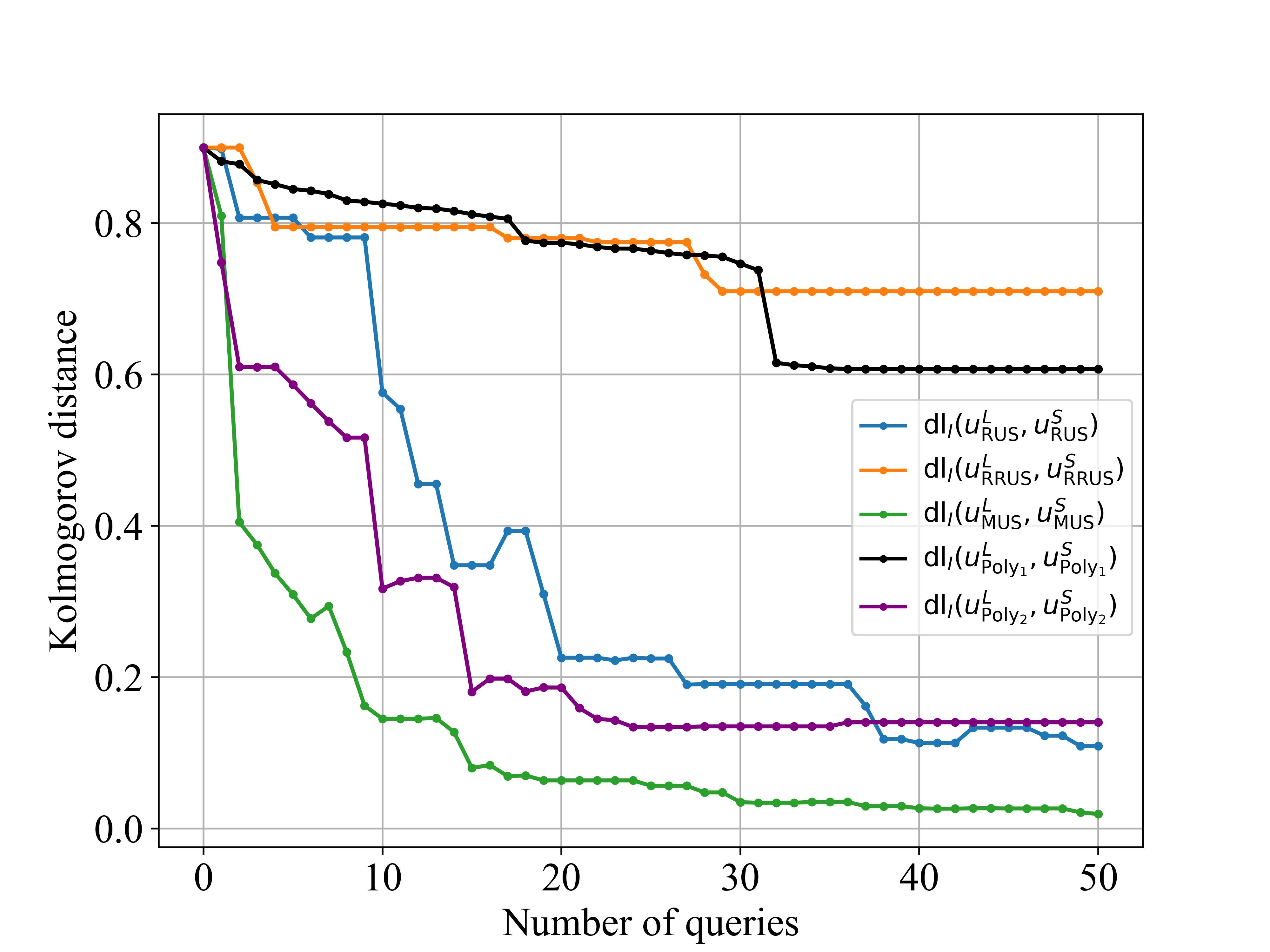}
    }
    \caption{Changes of the Kantorovich distance and the Kolmogorov distance between the largest utility functions and the smallest utility functions in ${\cal U}_m$ as $m$ increases, 
    with 
    queries generated by MUS, RUS, RRUS, and Poly methods respectively.}
    \label{fig:dist-large-small}
\end{figure}
 
    \item \underline{$\dd_K(u^C,u_{\text{true}})$ and $\dd_I(u^C,u_{\text{true}})$}, which  quantify the distance between the nominal utility function and the true utility function. Figure~\ref{fig:dist-center-true} depicts
 reduction of the distances as $m$ increases.
 We can see tendency 
 of reduction in both
  distances 
 despite neither is
monotonically decreasing.
This is 
because the {\em relative location} of $u_{\text{true}}$
 in ${\cal U}_m$ varies as $m$ increases albeit the size of the ambiguity set is reduced at each iteration (after a cut). 
 
\begin{figure}[htbp]
    \centering
    \subfigure
    []
    {
    \label{fig:kantor-center-true}
    \includegraphics[width=0.45\linewidth]{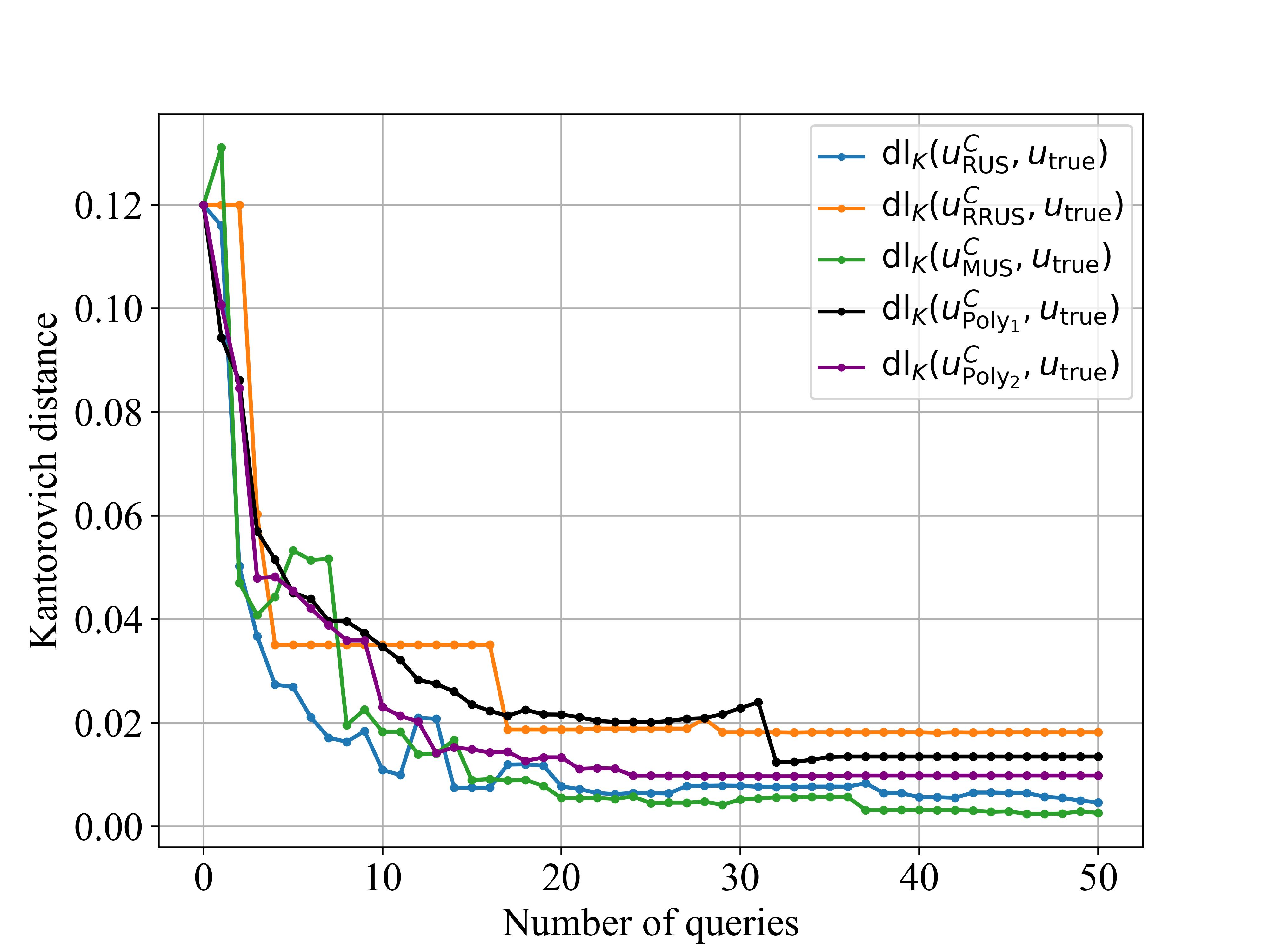}
    }
    \;
    \subfigure
    []
    {
    \label{fig:kolmo-center-true}
    \includegraphics[width=0.45\linewidth]{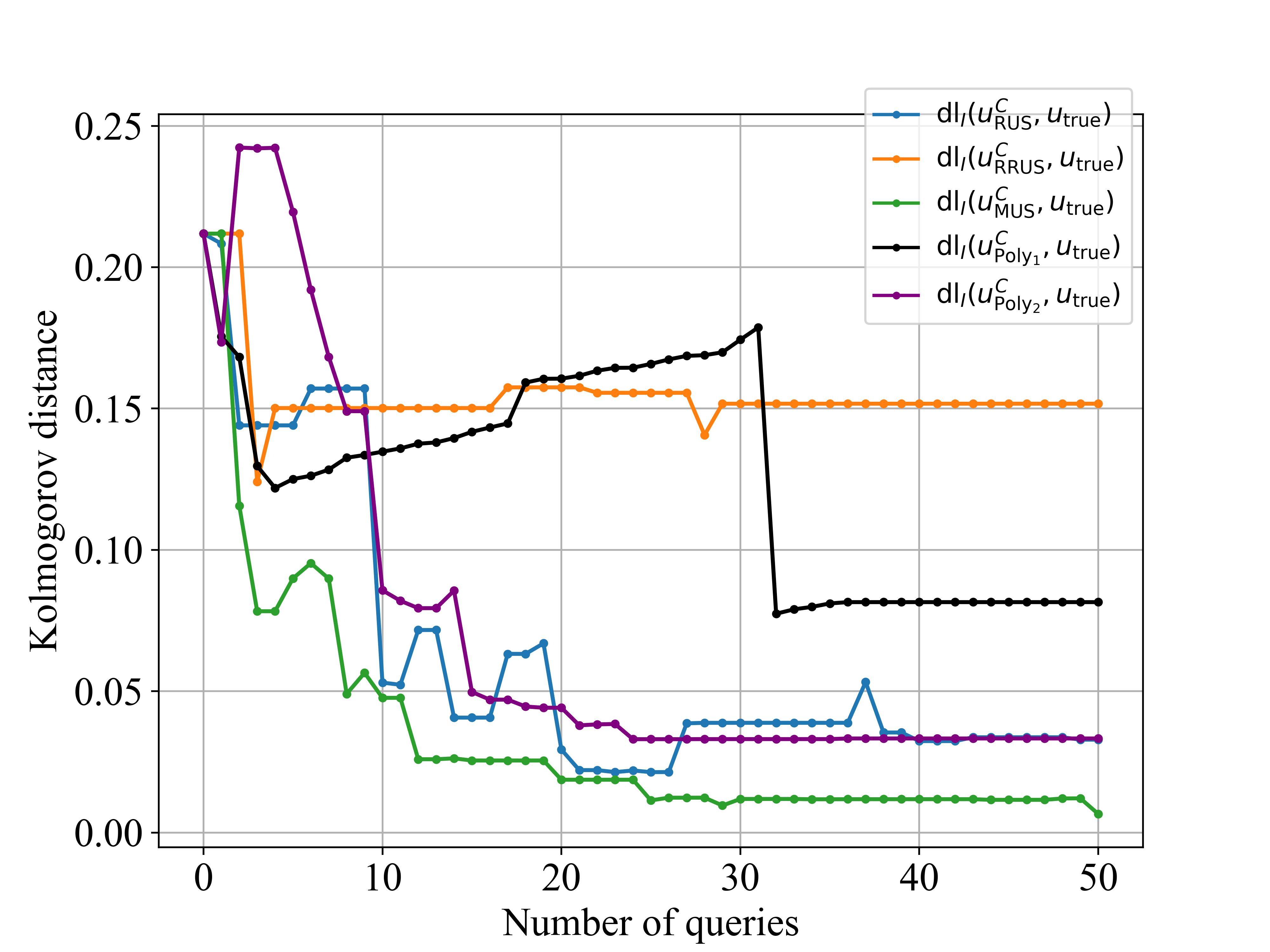}
    }
    \caption{Comparison of 
    MUS, RUS, RRUS and Poly methods
    in terms of
    the Kantorovich distance and Kolmogorov distance 
    between the nominal utility functions and true utility function as 
    the number of queries increases.
    }
    \label{fig:dist-center-true}
\end{figure}

    \item \underline{$\dd_K(u^S,u_{\text{true}})$ and $\dd_I(u^S,u_{\text{true}})$}, which  quantify the distance between the smallest utility function and the true utility function. Figure~\ref{fig:dist-small-true} 
    depicts reduction of the distances. Both distances are monotonically decreasing because the smallest utility function increases pointwise as $m$ increases (the smallest utility function coincides with $\underline{f}$).

\begin{figure}[htbp]
    \centering
    \subfigure
    []
    {
    \label{fig:kanto-small-true}
    \includegraphics[width=0.45\linewidth]{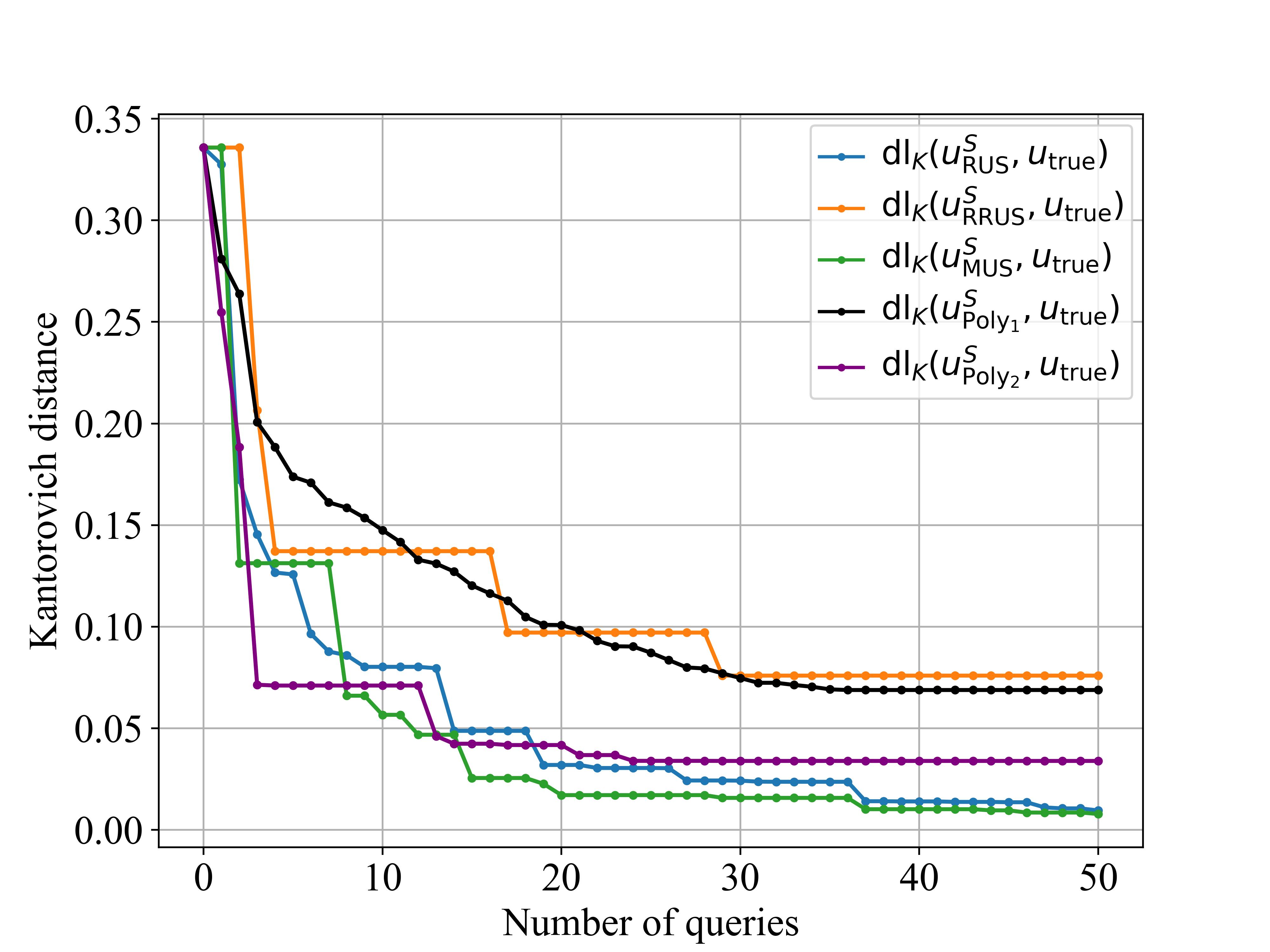}
    }
    \;
    \subfigure
    []
    {
    \label{fig:kolmo-small-true}
    \includegraphics[width=0.45\linewidth]{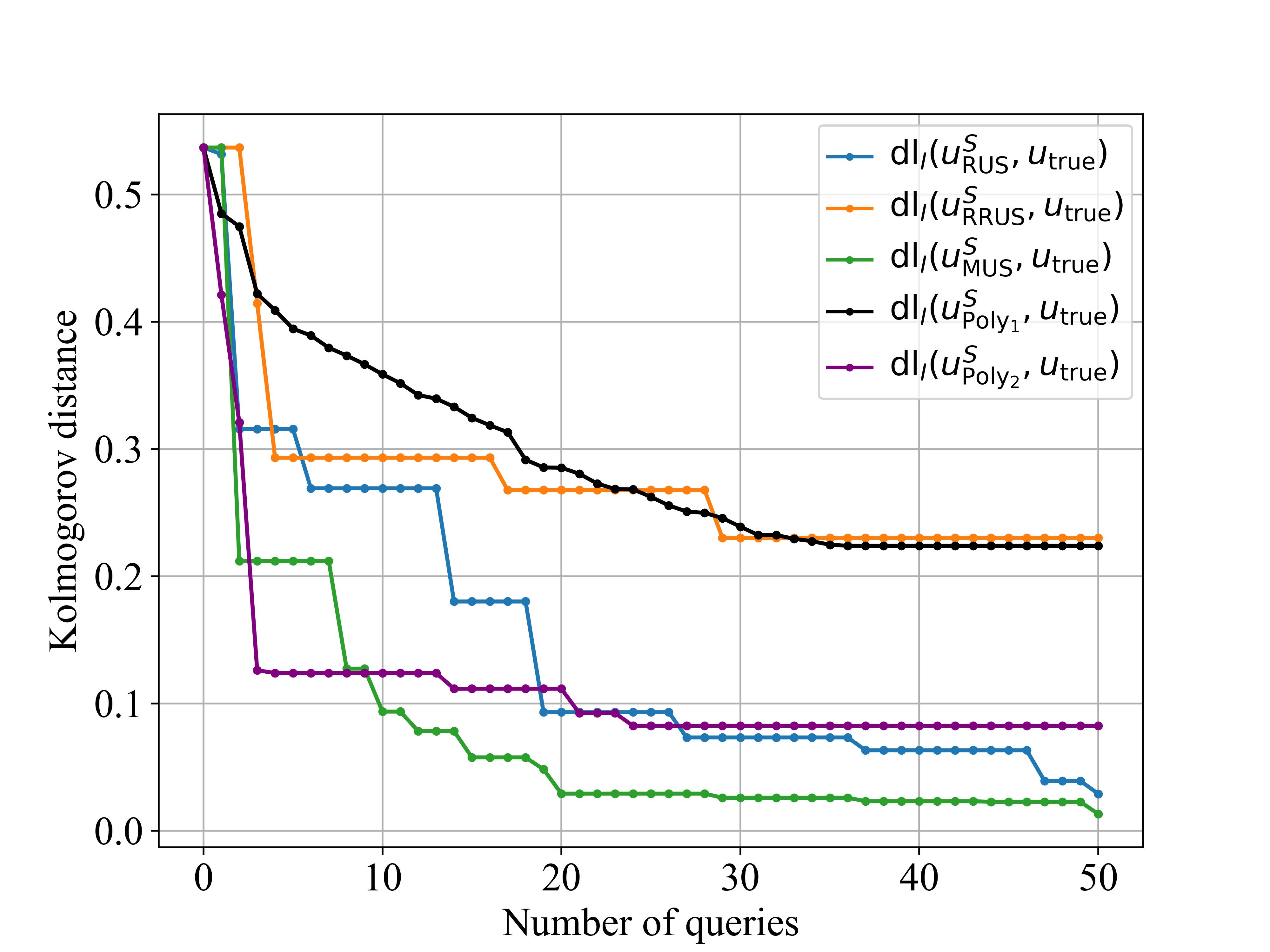}
    }
    \caption{Kantorovich distance and Kolmogorov distance between the smallest utility functions and the true utility function, 
    with increasing numbers of queries generated by MUS, RUS, RRUS, and Poly methods.}
    \label{fig:dist-small-true}
\end{figure}

    \item \underline{$\dd_K(u^L,u_{\text{true}})$ and $\dd_I(u^L,u_{\text{true}})$}, which  quantify the distance between the largest utility function and the true utility function. Figure~\ref{fig:dist-large-true}
    depicts reduction of the distances.
$\dd_K(u^L,u_{\text{true}})$ is monotonically decreasing whereas 
 $\dd_I(u^L,u_{\text{true}})$ is not.
 Unlike $\dd_I(u^S,u_{\text{true}})$
 where $u^S$ provides a uniform lower bound in $\mathcal{U}_m$, 
 $u^L$ is not a uniform upper bound, as discussed at the end of Section~\ref{sec:upper-bound}.
 This is the main reason behind non-monotonicity of $\dd_I(u^L,u^S)$ in Figure~\ref{fig:kolmo-large-small}. 
 In particular, the largest utility function is not improved
 in the first 
 20 
 iterations 
 under RRUS and 
 $\text{Poly}_1$ schemes. 
It might be helpful to explain why the largest utility function is not improved within first 20 queries generated by RRUS and $\text{Poly}_1$ methods (see Figure~\ref{fig:dist-large-true}). We make some comments on their poor performance.
Consider a piecewise linear function with Lipschitz modulus $L$ in $\mathcal{U}_{cv}$:
\begin{equation*}
    u^L_0(y)=\left\{ 
    \begin{aligned}
    &Ly && {\rm for} \quad 0\leq y \leq \frac{1}{L},\\
    & 1 && {\rm for} \quad \frac{1}{L}<y\leq 1,\\
    \end{aligned}
    \right.
\end{equation*}
which is exactly the largest utility function in $\mathcal{U}_0:=\mathcal{U}_{cv}$. 
Assume $u^L_m=u^L_0$ for some $m\in\{0,1,\cdots\}$, and consider the next pairwise comparison query $(W_{m+1}, Y_{m+1})$ as defined in~\eqref{eq:W-Y}.  There are three 
cases 
{in which}
$u^L_{m+1}$ remains unchanged from $u^L_m$.
\begin{enumerate}
\item [(a)] If $r_1^{m+1}\geq\frac{1}{L}$, then $u^L_m(r_1^{m+1})=u^L_m(r_2^{m+1})=u^L_m(r_3^{m+1})=1$, 
and subsequently
$$
(1-p^{m+1})u^L_m(r_1^{m+1})+p^{m+1}u^L_m(r_3^{m+1})=1=u^L_m(r_2^{m+1}),
$$
i.e., $\mathbb{E}[u^L_m(W_{m+1})]=\mathbb{E}[u^L_m(Y_{m+1})]$  holds. In this case, regardless of the DM’s choice between $W_{m+1}$ and $Y_{m+1}$, $u^L_m$ is consistent with the observed preference.  

\item[(b)] If $r_2^{m+1}>\frac{1}{L}>r_1^{m+1}$, then $u^L_m(r_2^{m+1})=u^L_m(r_3^{m+1})=1>u^L_m(r_1^{m+1})$, and we have 
$$
(1-p^{m+1})u^L_m(r_1^{m+1})+p^{m+1}u^L_m(r_3^{m+1})\leq1=u^L_m(r_2^{m+1}),
$$
i.e., $\mathbb{E}[u^L_m(W_{m+1})]\leq\mathbb{E}[u^L_m(Y_{m+1})]$ always holds. In this case, if the DM prefers $Y_{m+1}$ over $W_{m+1}$, the preference is consistent with $u^L_m$. 

\item[(c)] If $Z_{m+1}\cdot\mathbb{E}[u^L_m(W_{m+1})]\geq Z_{m+1}\cdot\mathbb{E}[u^L_m(Y_{m+1})]$, where $Z_{m+1}=1$ if the DM prefers 
{$Y_{m+1}$}
over 
{$W_{m+1}$},
and $Z_{m+1}=-1$ otherwise, then $u^L_m$ still satisfies the corresponding pairwise comparison preference.
\end{enumerate}
The first 21 queries generated by RRUS all fall into either case~(a)~or~(b), 
keeping $u^L_m$ unchanged from $u^L_0$ until $m = 22$. 
This is because $r_1^m$ is selected with probability $\frac{L-1}{L}$ to exceed $\frac{1}{L}$, thus falling into case~(a). 
For $\text{Poly}_1$, the first 31 queries all fall into cases~(a), (b), or~(c), among which 27 fall into cases~(a)~and~(b). 
In comparison, $\text{Poly}_2$ performs much better with even fewer initial breakpoints. 
The key point is that a prior selection of the breakpoint $\frac{1}{L}$ in $\text{Poly}_2$ plays an essential role in our problem setting. 
It indicates that the polyhedral method heavily depends on the proper selection of the initial breakpoints. 
In contrast, the queries generated by RUS never fall into case~(a) because $r_1^{m+1}$ is fixed at $0$, 
and those of MUS never fall into either case~(a)~or~(b), which significantly improves their performance in eliciting the largest utility function.

\begin{figure}[tbp]
    \centering
    \subfigure
    []
    {
    \label{fig:kanto-large-true}
    \includegraphics[width=0.45\linewidth]{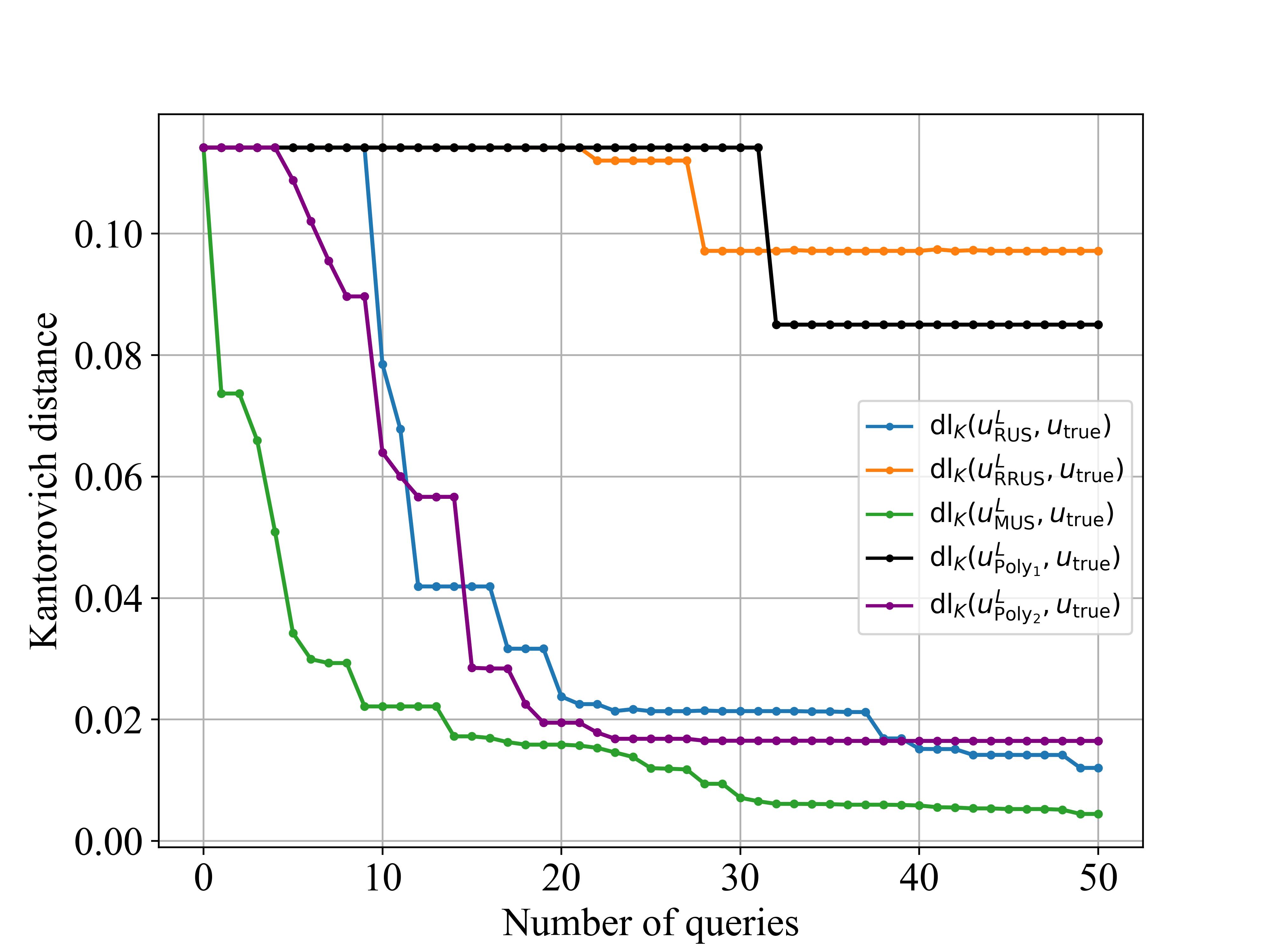}
    }
    \;
    \subfigure
    []
    {
    \label{fig:kolmo-large-true}
    \includegraphics[width=0.45\linewidth]{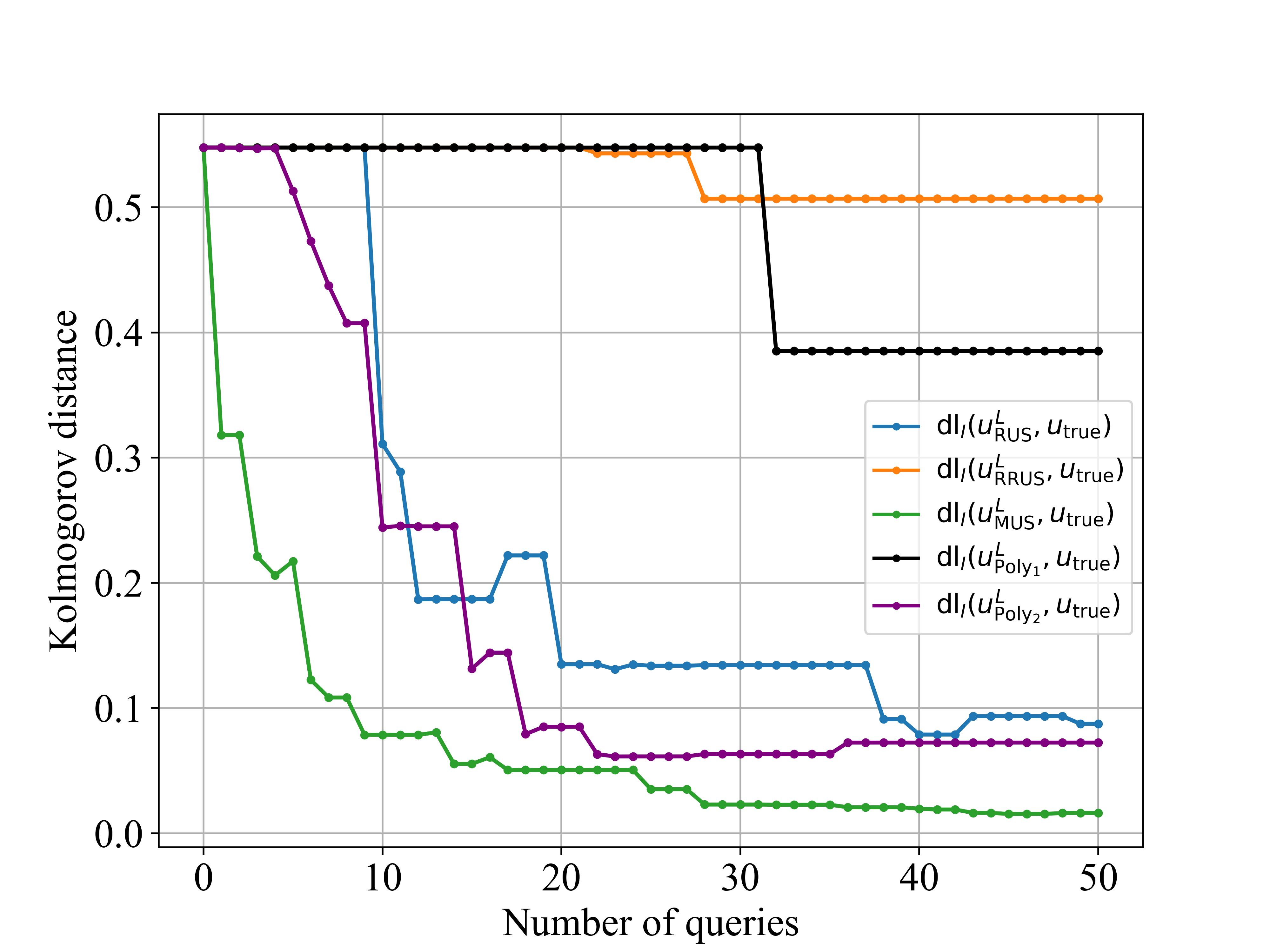}
    }
    \caption{Comparison of the Kantorovich distance and Kolmogorov distance between the largest utility functions and the true utility function, 
    with increasing numbers of queries generated by MUS, RUS, RRUS, and Poly methods.}
    \label{fig:dist-large-true}
\end{figure}

\end{itemize}

From Figures~\ref{fig:dist-large-small}, \ref{fig:dist-center-true}, \ref{fig:dist-small-true}, and~\ref{fig:dist-large-true}, we can draw the following conclusions, which confirm our expectations: the smallest, largest, and nominal utility functions obtained from the MUS method are all closer to the true utility function than those derived from the other three benchmark methods; the ambiguity set derived from MUS-generated queries exhibits the most rapid convergence.

We also compare the computational complexity of the four methods in terms of CPU time.
Table~\ref{tab:CPU time} reports the CPU time required by the four methods to generate 10, 20, and 40 queries, respectively. 
RUS and RRUS are the most time-efficient as they choose $r_1$, $r_2$ and $r_3$ deterministically or randomly 
without requiring to solve an optimization problem as MUS or 
polyhedral method.
Compared to the polyhedral method, 
the CPU time 
of MUS is much shorter
for this specific concave structured utility function. 


\begin{table}[htbp]
\centering
\caption{The CPU time (s) of MUS, RUS, RRUS, and polyhedral methods for generating different numbers of queries.}
\begin{tabular}{cccccc}
\toprule
    \multicolumn{1}{c}{Number of queries} & \multicolumn{1}{c}{MUS} & \multicolumn{1}{c}{RUS} & \multicolumn{1}{c}{RRUS} & \multicolumn{1}{c}{$\text{Poly}_1$} & \multicolumn{1}{c}{$\text{Poly}_2$} \\
\midrule
    10    & 2.67  & 0.32  & 0.56  & 17.97  & 15.57  \\
    20    & 23.65  & 1.10  & 2.15  & 104.64  & 80.13  \\
    40    & 249.82  & 4.11  & 8.37  & 6469.62   &  1058.87\\
\bottomrule
\end{tabular}
\label{tab:CPU time}
\end{table}

\subsection{Convergence in the case of S-shaped utility functions}

As discussed in Section~\ref{sec:s-shape}, the MUS scheme readily extends to the settings where the true utility function is S-shaped. 
We consider a DM whose preferences are represented by the true utility function 
\begin{equation}
\label{eq:u*-true-test-s}
    u^*_S(y) = 
    \begin{cases} 
    -\frac{1}{1-e^{-7}}(1-e^{7y}), & {\rm for}\quad y \in [-1, 0),\\
    \frac{1}{1-e^{-3}}(1-e^{-3y}), & {\rm for}\quad y \in [0, 1], 
    \end{cases}
\end{equation} 
which is monotonically increasing, Lipschitz continuous, convex on $[-1,0]$, concave on $[0,1]$, and normalized to $[-1,1]$.
The Lipschitz modulus of utility functions in $\mathcal{U}_{S}$ is set to $L = 10$.

In this subsection, we
examine convergence of the corresponding elicited ambiguity set $\mathcal{U}^S_m$ introduced in~\eqref{eq:U_s-m}, and compare the elicitation efficiency of the proposed MUS scheme with that of the RUS and RRUS schemes as the number of queries increases. 
We start from the initial ambiguity set $\mathcal{U}^S_0:=\mathcal{U}_S$. 
Using the MUS method,
implemented in its general non-concave formulation, 
as well as the RUS and RRUS methods, we adaptively generate 70 queries to interact with the DM, and after each query we update $\mathcal{U}^S_m$ and determine the smallest,  
largest, 
{and}
nominal utility functions it contains.

At the 0-, 10-, 30- and 70-th round of interaction, we elicit the smallest, largest, and the nominal S-shaped utility functions in $\mathcal{U}^S_m$, and present the results in Figure~\ref{fig:num-ss-mus-fit}. 
We can also observe that the smallest, largest, and nominal utility functions all converge to the S-shaped true utility as the number of queries increases. 
Moreover, 
{we}
compute the Kantorovich distance and Kolmogorov distance between the largest utility function and the smallest utility function in $\mathcal{U}^S_m$, denoted as $\dd_K(u^L,u^S)$ and $\dd_I(u^L,u^S)$ respectively. Figure~\ref{fig:s-shape-dist-large-small} illustrates the reduction of these distances as the number of queries increases. 
The results show that when the underlying true utility is S-shaped, MUS still achieves the fastest reduction in both metrics, demonstrating the most rapid convergence 
{compared with the}
other two benchmark schemes.

\begin{figure}[htbp]
    \centering
    \subfigure
    {
    \includegraphics[width=0.45\linewidth]{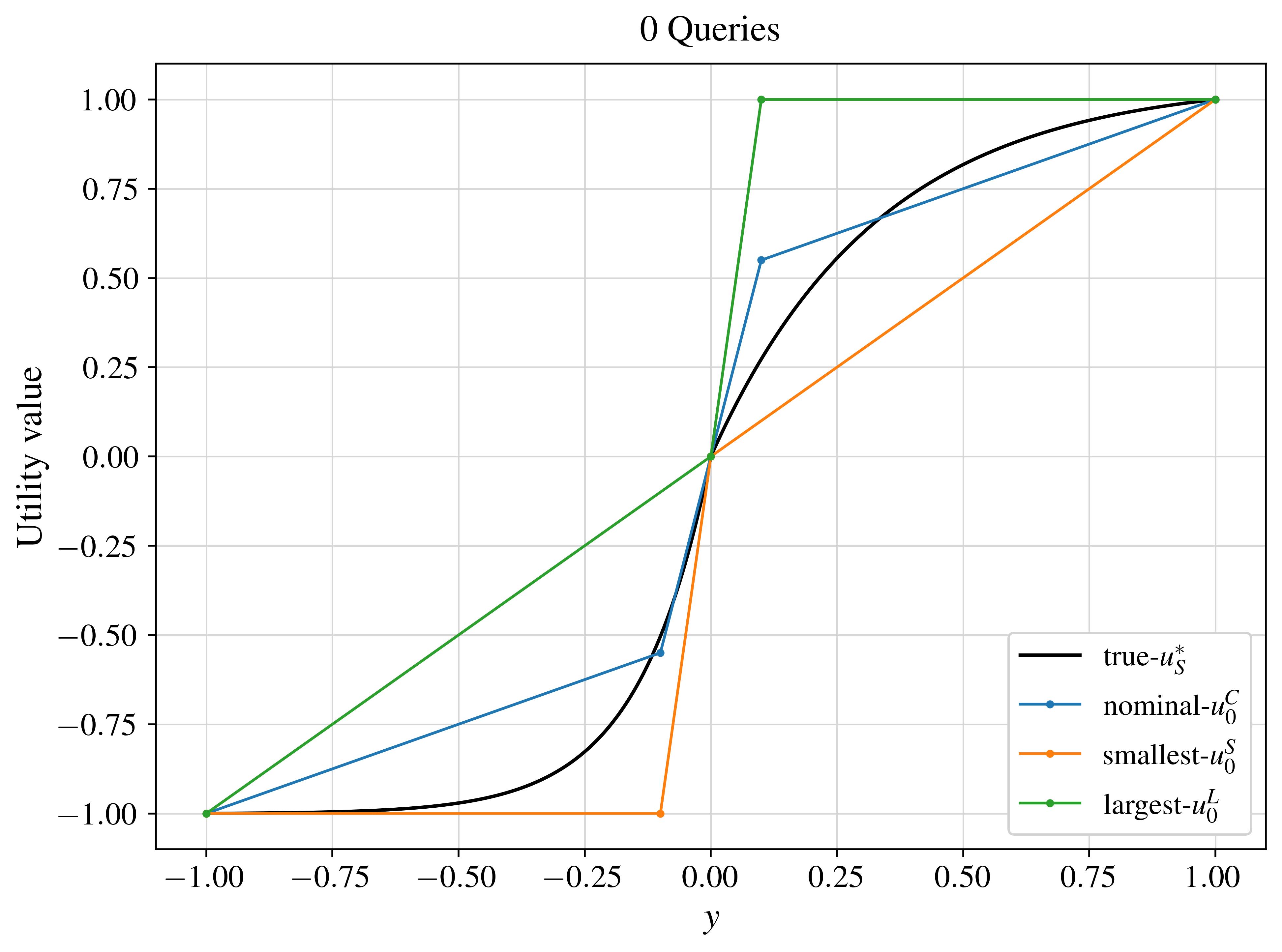}
    }
    \;
    \subfigure
    {
    \includegraphics[width=0.45\linewidth]{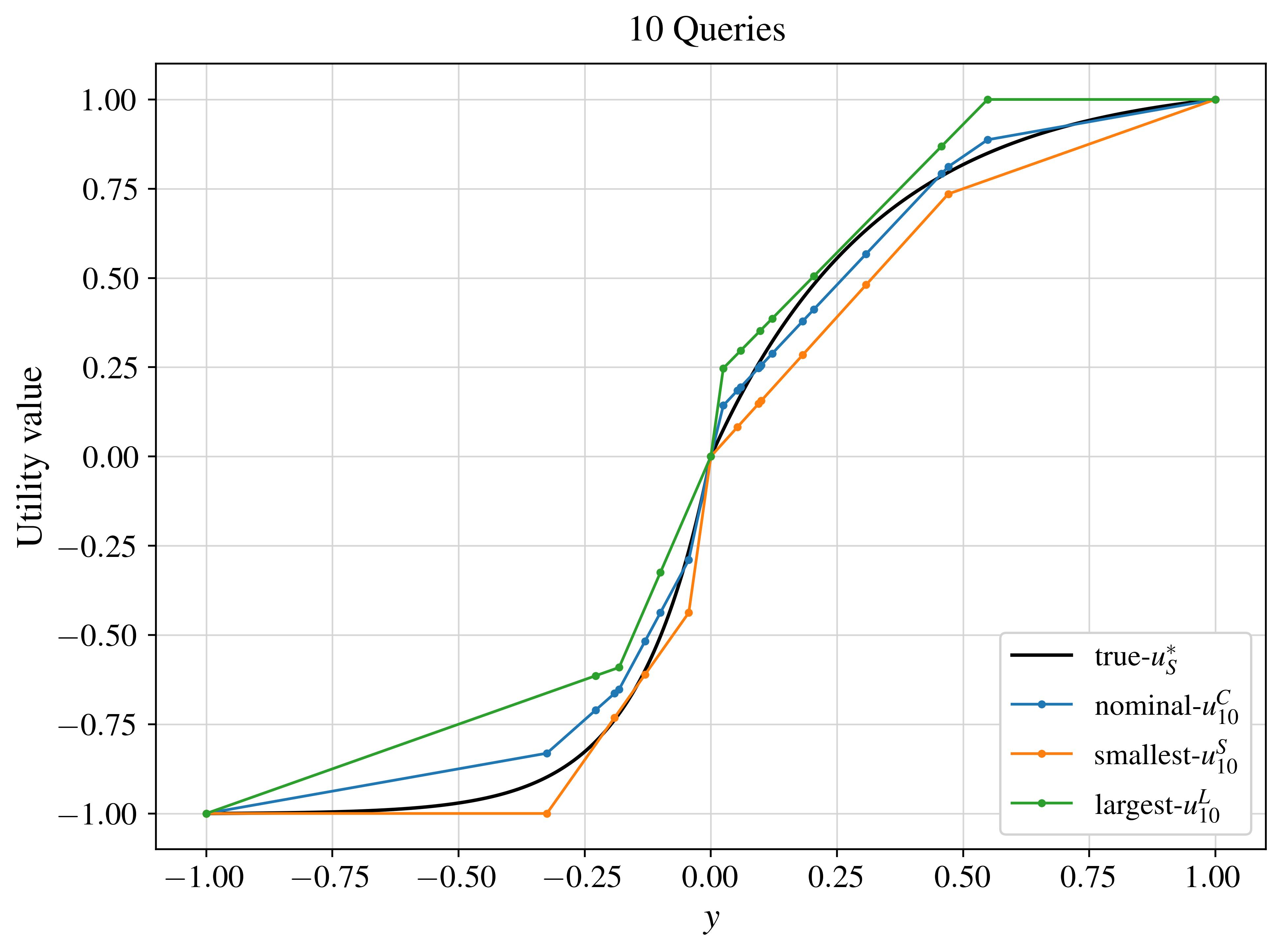}
    }\\
    \subfigure
    { 
    \includegraphics[width=0.45\linewidth]{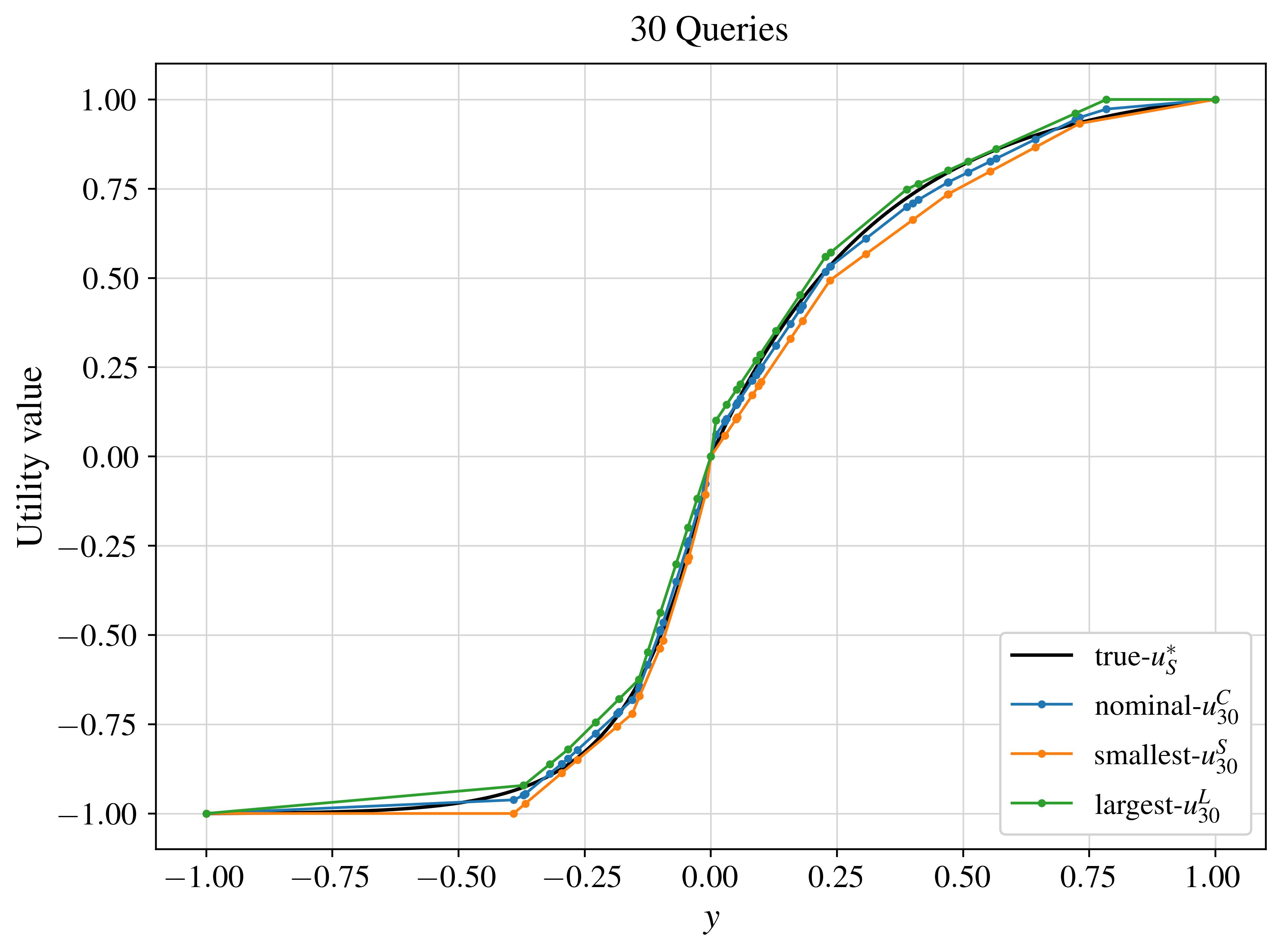}
    }
    \;
    \subfigure
    {
    \includegraphics[width=0.45\linewidth]{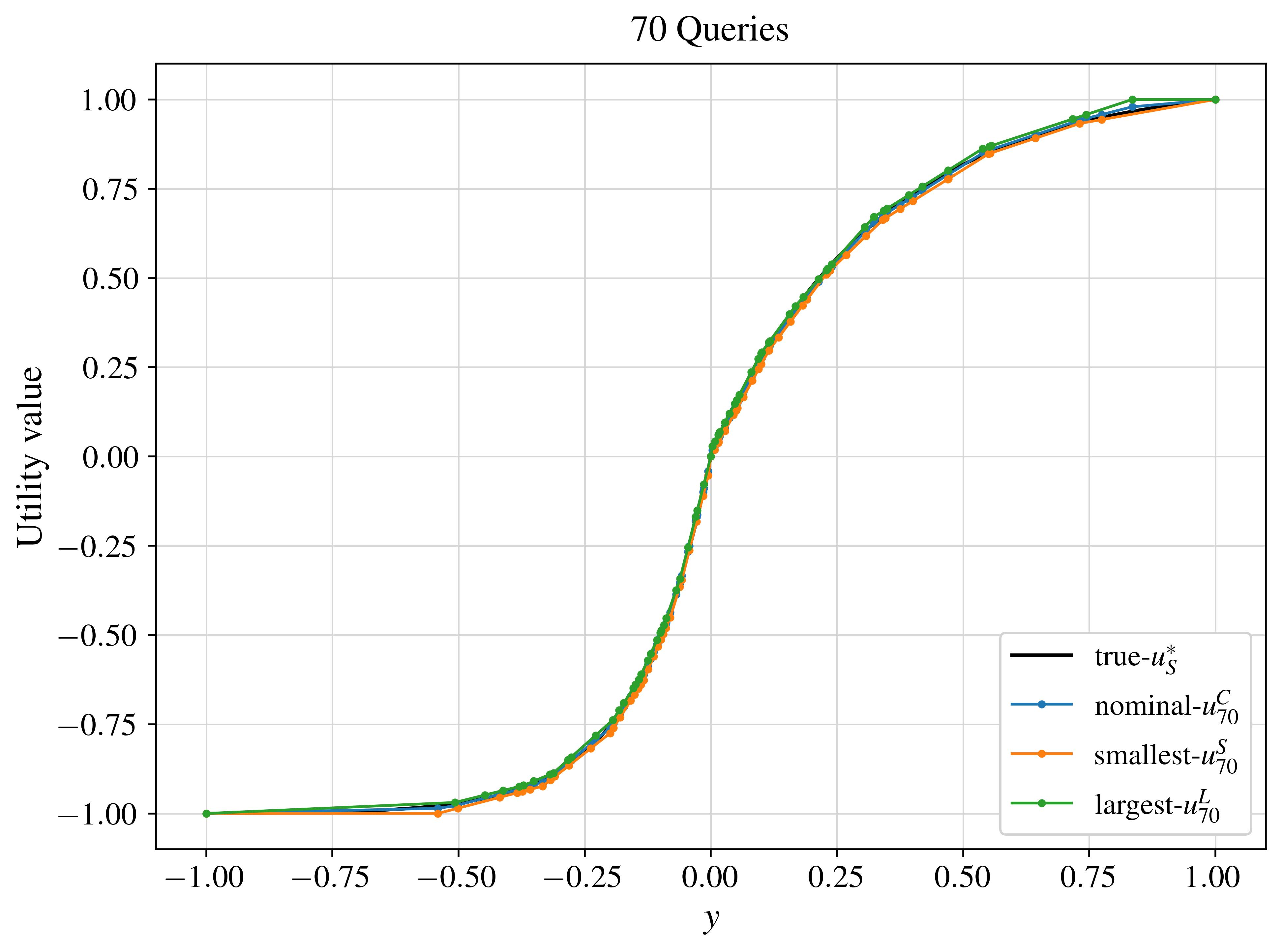}
    }
    \caption{Convergence of the smallest, the largest, and the nominal utility functions of $\mathcal{U}^{S}_m$ as $m$ varies from $0$ to $10, 30$ and $70$.
    }
    \label{fig:num-ss-mus-fit}
\end{figure}

\begin{figure}[htbp]
    \centering
    \subfigure
    []
    {
    \label{fig:s-shape-kanto-large-small}
    \includegraphics[width=0.45\linewidth]{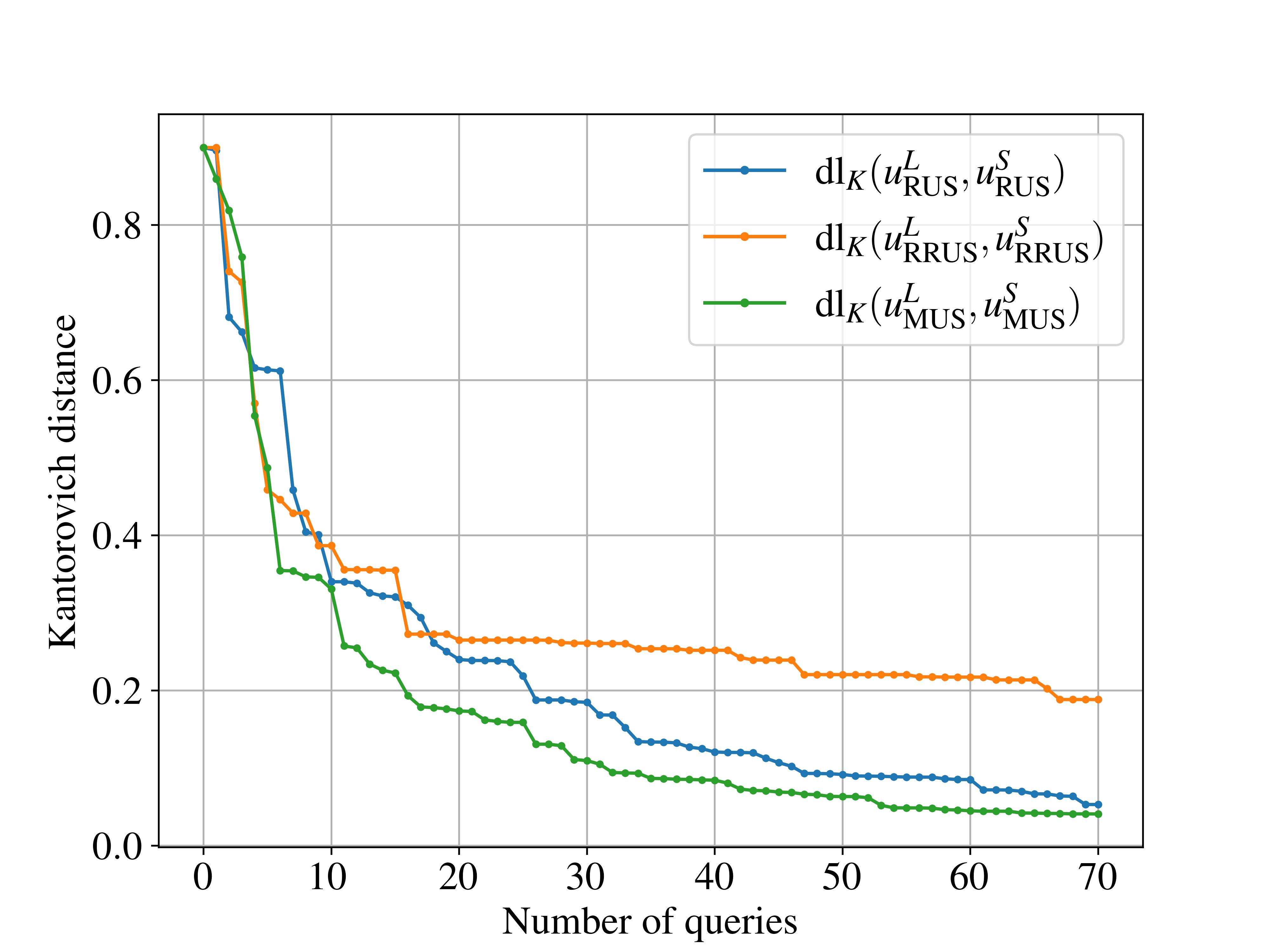}
    }
    \;
    \subfigure
    []
    {
    \label{fig:s-shape-kolmo-large-small}
    \includegraphics[width=0.45\linewidth]{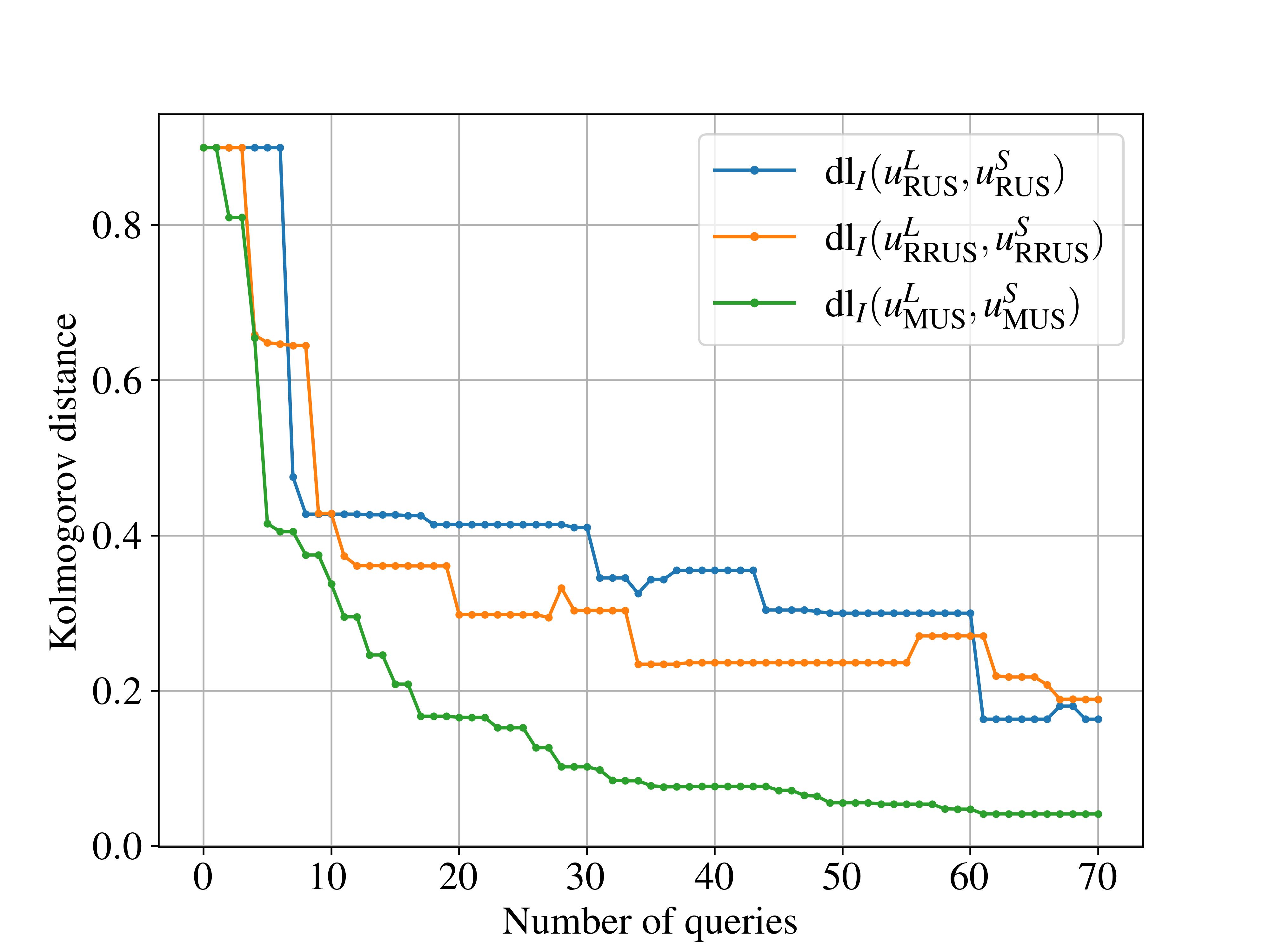}
    }
    \caption{Changes of the Kantorovich distance and the Kolmogorov distance between the largest utility functions and the smallest utility functions in ${\cal U}^S_m$ as $m$ increases, 
    with queries generated by MUS, RUS, and RRUS methods respectively.}
    \label{fig:s-shape-dist-large-small}
\end{figure}

\subsection{A robo-advisor problem}\label{sec:robo-advisor}

A robo-advisor usually refers to an automated system that provides financial advice and investment management online with minimal human intervention~\citep{chen2026robo,faloon2017individualization}. 
Typically, the robo-advisor interacts dynamically with the user to learn the user's risk preference from her/his choices or feedback, and then offers personalized investment advice based on a portfolio optimization model. 
We assume the user visits the robo-advisor repeatedly. 
Thus, the adaptive preference elicitation mechanism proposed in this paper is suitable for the robo-advisor setting. 
We embed the proposed MUS scheme into such a system and examine its effectiveness.

\subsubsection{Mechanism of the robo-advisor system}

We assume that the user's preference can be characterized by a true utility function $u^*$ 
as defined in~\eqref{eq:u*-true-test},
which is unobservable to the robo-advisor.  
The robo-advisor stores all the user's answers to queries from previous visits (forming the ambiguity set $\mathcal{U}_{m}$) and maintains a nominal utility function $u^C_m\in\mathcal{U}_m$. 
At each visit (say $(m+1)$-th), the robo-advisor generates a new pairwise comparison query $(W_{m+1},Y_{m+1})$ based on $\mathcal{U}_m$ using the MUS scheme (see Section~\ref{sec:MUS}) and presents it to the user. 
If the user responds, it is assumed that she/he chooses between the lotteries based on her/his true utility values, without response error. 
After receiving the response, the robo-advisor updates the ambiguity set from $\mathcal{U}_m$ to $\mathcal{U}_{m+1}$ by adding a new pairwise comparison constraint, and identifies a new nominal utility function $u^C_{m+1}$ (see Section~\ref{sec:utility}). 
{If no response is received, both the ambiguity set and nominal utility remain unchanged. }
Finally, using the nominal utility function and historical return data of risky assets, the robo-advisor solves a portfolio optimization problem that maximizes expected nominal utility (see~\eqref{eq:port_org}), and recommends a personalized portfolio to the user.  
Figure~\ref{fig:flowchart-robo-advisor} depicts how the system operates.

\begin{figure}[htbp]
    \centering
    \includegraphics[width=0.8\textwidth]{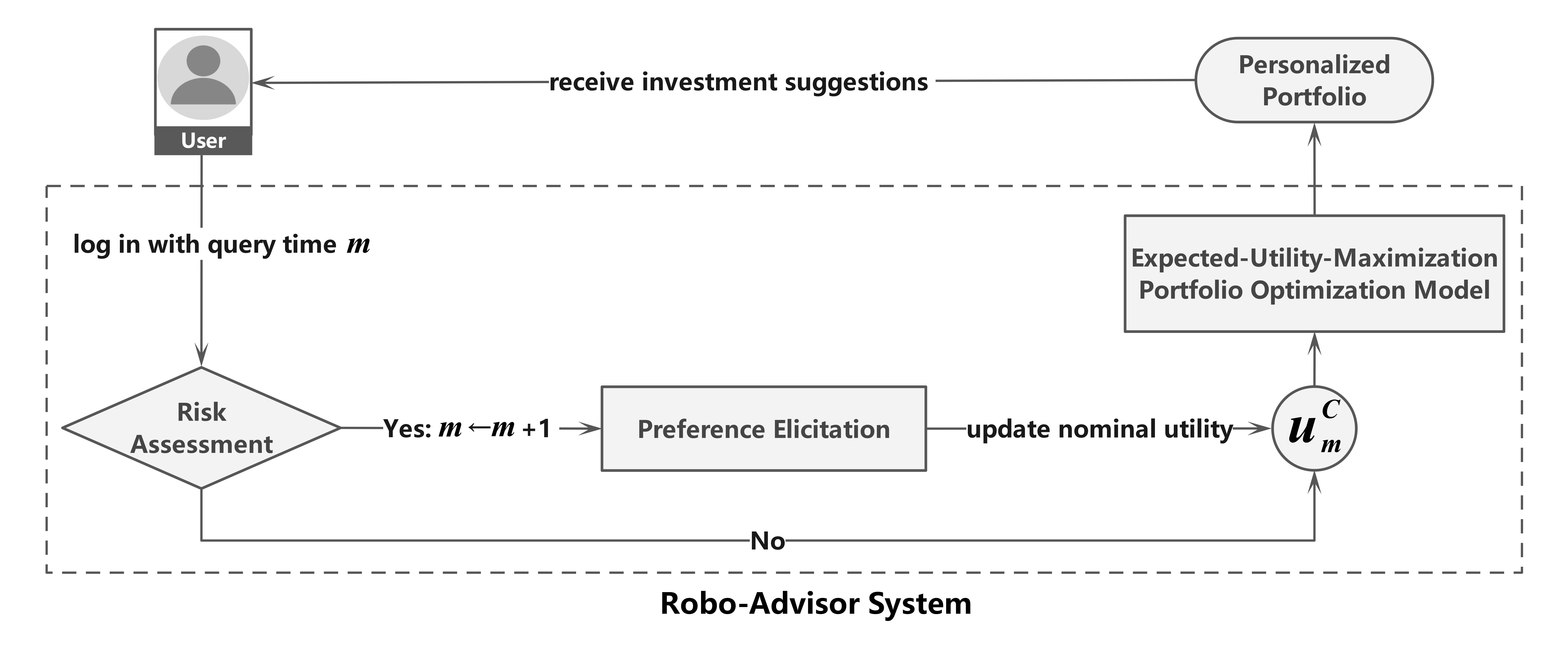}
    \caption{Flowchart of the robo-advisor system.}
    \label{fig:flowchart-robo-advisor}
\end{figure}

Assume that the user specifies a maximal outcome $\bar{b}$ of interest and an initial investment budget $W_0$. 
We consider a simple portfolio optimization problem for the user. 
Suppose there are $S$ risky assets with random return rates $\xi_1, \xi_2, \ldots, \xi_S$, and a risk-free asset with return rate $\xi_0 = 0$. 
Denote $\xi := [\xi_0, \xi_1, \ldots, \xi_S]^{\top}$. 
Let $x:=[x_0,x_1,\ldots,x_S]^\top\in\mathcal{X}$ be the portfolio, where the decision space $\mathcal{X}$ is a closed convex set defined as 
\begin{equation*}
    \mathcal{X}:=\left\{x\in\mathbb{R}^{S+1}\ \bigg|\ 
    \sum_{s=0}^{S}x_s=W_0,\ 0\leq x_0\leq W_0, \  0\leq x_s\leq c_s W_0,\ s=1,\ldots,S \right\}. 
\end{equation*}
Here, 
each $c_s$ denotes the maximum proportion allocable to the $s$-th risky asset. 
To fit into the MUS-based preference elicitation procedure and the corresponding computational scheme proposed in the preceding sections, 
we assume that all realizations of the return rates satisfy $\xi_s\geq-1$, a.e., for all $s=1,\ldots,S$, and that the maximum outcome $\bar{b}$ is sufficiently large such that $x^{\top}(\textbf{1}+\xi)/\bar{b}\leq1$ for all $x\in\mathcal{X}$ and realizations of $\xi$, where $\mathbf{1}$ denotes the vector of ones. 
The robo-advisor determines the optimal portfolio by maximizing the user's expected nominal utility of the total wealth relative to the maximal outcome (i.e., $x^{\top}(\mathbf{1}+\xi)/\bar{b}$), formulated as 
\begin{equation}\label{eq:port_org}
    \mathop{\max}\limits_{x\in \mathcal{X}}\ \mathbb{E}\left[u^C_m(x^{\top}(\mathbf{1}+\xi)/\bar{b})\right].   
\end{equation}
To solve problem~\eqref{eq:port_org} computationally, 
assume we have collected $T$ i.i.d. historical return rate samples of $\xi$, denoted as $\xi^{(1)},\xi^{(2)},\ldots,\xi^{(T)}$. 
The empirical distribution of $\xi$ is given by $\mathbb{P}[\xi=\xi^{(t)}]=\frac{1}{T}$, $t=1,\ldots,T$. 
Then, \eqref{eq:port_org} is approximated by
\begin{equation}\label{eq:port_saa}
    \max\limits_{x\in\mathcal{X}}\frac{1}{T}\sum_{t=1}^T u^C_m(x^{\top}(\mathbf{1}+\xi^{(t)})/\bar{b}).
\end{equation}
Note that $u^C_m$ is a piecewise linear function with slopes $\beta^{C}_m=[\beta_{m,1}^{C},\ldots,\beta_{m,\widetilde{N}_m-1}^{C}]$ and utility values $\alpha^{C}_m=[\alpha_{m,1}^{C},\ldots,\alpha_{m,\widetilde{N}_m}^{C}]$. By the concavity, 
\begin{equation*}
u^C_m(y)=\min\limits_{j=1,\ldots,\widetilde{N}_m-1}\left\{ \beta^{C}_{m,j}(y-{x}^m_j)+\alpha_{m,j}^{C}\right\} . 
\end{equation*}
By introducing auxiliary variables $\eta=[\eta_1, \eta_2, \ldots, \eta_T]^\top$, 
we can reformulate~\eqref{eq:port_saa} as a linear program: 
\begin{subequations}\label{eq:port_final}
\begin{align}
\max\limits_{x\in \mathbb{R}^{S+1},\ \eta\in\mathbb{R}^T}\
& \frac{1}{T}\sum_{t=1}^{T} \eta_t \\
 \text{s.t.}\ \qquad  & \eta_t\leq \beta_{m,j}^{C}\left[x^{\top}(\mathbf{1}+\xi^{(t)})/\bar{b}-{x}^m_j\right]+\alpha_{m,j}^{C},\ t=1,\ldots, T,\ j=1,\ldots, \widetilde{N}_m-1,\\
& \sum_{s=0}^S x_s=W_0,\\
&0\leq x_0\leq W_0,\ 0\leq x_s\leq c_s W_0 ,\ s=1, \ldots, S.
\end{align}
\end{subequations}

\subsubsection{Investment simulation by the robo-advisor system}

In this set of tests, we simulate the long-term preference elicitation and investment process by the robo-advisor system introduced in the previous subsection where the queries are generated by 
the MUS, RUS, RRUS, and polyhedral methods.
We mainly compare the difference of the out-of-sample return of the portfolio recommended by the robo-advisor system, compared to the portfolio return derived from the true utility function of the user.


\textbf{Dataset.}
We consider $S = 5$ risky assets in the U.S. financial market, considered in~\cite{gomez2024multi}: 
XLE and XLK represent the energy and technology sectors within 
{the}
S\&P 500 Index; SPDR Gold Shares (GLD) tracks the performance of gold bullion; IEF is a fund for long-term (7--10 years) Treasury bonds; and USDU tracks the performance of the U.S. dollar. 
Additionally, a cash account with zero return is included as the risk-free asset. 
We collect weekly return rates of the five risky assets from January 5, 2020, to December 31, 2023.\footnote{The data were downloaded from https://cn.investing.com/ on June 7, 2025.}

The user's true utility function $u^*$ is defined as in~\eqref{eq:u*-true-test}, and the true optimal portfolio is determined by solving 
$\max\limits_{x\in\mathcal{X}}\frac{1}{T}\sum_{t=1}^{T}u^*[x^{\top}(\mathbf{1}+\xi^{(t)})/\bar{b}]$.
The feasible portfolio set $\mathcal{X}$ is specified by setting the maximal outcome $\bar{b}=100,000$, the initial investment budget $W_0=10,000$, the maximal investment proportion $c_s=0.5$ for the two stock-based assets (XLE and XLK), and $c_s=0.9$ for the remaining three less-risky assets. 
The investment process is simulated using a rolling-window strategy with 30 weeks of in-sample data (i.e., $T=30$), and the portfolio is rebalanced every 4 weeks.
For simplicity, we assume no transaction cost.

Every 4 weeks, a query is generated by the examined methods and provided to the user.
Considering different cooperation 
{levels}
of different types of users in real-world questionnaire surveys, we set a response rate parameter
$\varepsilon \in [0,1]$. 
It means that at any round of interaction between the user and the robo-advisor, the user answers the query with probability $\varepsilon$.
The larger $\varepsilon$,
the more queries the user answers 
{in total}
during the investment process.
In the case that the user does not answer the query, the ambiguity set and the nominal utility function keep no updating during the next 4 weeks and the same query would be asked again 4 weeks later.
If the user answers the query, the robo-advisor updates the ambiguity set as well as the nominal utility function.
In both cases, the robo-advisor system compute a portfolio based on the current nominal utility function $u^C_m$ and make investment with it.


We conduct four 
tests
with $\varepsilon = 0, 0.25, 0.5, 1$ and compute five sequences of out-of-sample weekly return rates derived by the robo-advisor system with queries generated from
MUS, RUS, RRUS, $\text{Poly}_1$, and $\text{Poly}_2$, respectively. We use the true utility to compute the true optimal portfolio and compute its out-of-sample weekly return rate series as a benchmark.
We then compute the mean squared error (MSE) as \(\frac{1}{N}\sum_{t=1}^{N}(r_t-r_t^*)^2\), where \(r_t\) and \(r_t^*\) denote the out-of-sample weekly return rates generated by each method (MUS, RUS, RRUS, $\text{Poly}_1$, or $\text{Poly}_2$) and the true utility maximization model, respectively, and \(N\) is the number of out-of-sample weeks.
The results are presented in Table~\ref{tab:dynamic-port-ARD}.

\begin{table}[htbp]
\centering
\caption{
Mean Squared Error (MSE) between the weekly return rates of portfolios generated by the user's true utility maximization model and those by the robo-advisor equipped with the MUS, RUS, RRUS, $\text{Poly}_1$, and $\text{Poly}_2$ methods under different response rates. 
}
\begingroup
    \begin{tabular}{lc|
    rrrrr}
\toprule[1pt]
&    & 
\multicolumn{5}{c}{MSE} \\
\midrule
Response rate & Answered queries 
& \multicolumn{1}{c}{MUS} & \multicolumn{1}{c}{RUS} & \multicolumn{1}{c}{RRUS} & \multicolumn{1}{c}{$\text{Poly}_1$} & \multicolumn{1}{c}{$\text{Poly}_2$} \\
\midrule[0.5pt]
$\varepsilon=0$ & 0     
& 5.86E-04 & 5.86E-04 & 5.86E-04 & 5.86E-04 & 5.86E-04 \\
$\varepsilon=0.25$ & 11   
& 9.12E-05 & 3.54E-04 & 4.27E-04 & 2.99E-04 & 2.43E-04 \\
$\varepsilon=0.5$ & 22    
& 4.60E-05 & 2.85E-04 & 3.73E-04 & 3.66E-04 & 1.63E-04 \\
$\varepsilon=1$ & 44   
& 3.47E-05 & 1.53E-04 & 3.16E-04 & 3.36E-04 & 1.10E-04 \\
\bottomrule[1pt]
\end{tabular}
\endgroup
\label{tab:dynamic-port-ARD}%
\end{table}%

From Table~\ref{tab:dynamic-port-ARD}, we observe that: 
(a) 
The MSE values derived from MUS, RUS, RRUS, $\text{Poly}_1$, and $\text{Poly}_2$ generally decrease as $\varepsilon$ increases, indicating that the more queries the user answers, 
the closer the weekly returns of the portfolio provided by the robo-advisor 
are to 
those generated by 
the true utility maximization model of the user. 
This reflects the fact that 
the nominal utility function elicited by the robo-advisor becomes closer to the true utility function as the number of queries increases, as already validated in Figure~\ref{fig:num-mus-fit}.  
(b) The 
MSE values derived from the MUS method are always smaller than those from the benchmark methods, which coincides with the utility learning results in Figures~\ref{fig:dist-large-small}--\ref{fig:dist-large-true}.

\section{Concluding remarks }\label{sec:conclusion}

In this paper, we propose a new approach (MUS) to generate  pairwise comparison questionnaires for preference elicitation.
Unlike the existing RUS scheme, the outcome of the deterministic lottery is non-randomly selected from the point where the graph of the set-valued mapping ${\cal U}_m$ (the currently elicited ambiguity set of utility functions) displays the largest range. 
Under some moderate conditions, we demonstrate that the elicitation procedure based on MUS reduces the graph of ${\cal U}_m$ to the true utility function as the number of questionnaires goes to infinity and explicitly derive the complexity of the procedure. 
A key step towards practical implementation of the MUS approach is to develop a tractable computational algorithm, which enables us to compute the point of the maximum range, and we have managed to do so by deriving a semi-closed form of the upper bound function and lower bound function in the ambiguity set.
Preliminary numerical tests show that the MUS method is noticeably more effective than RUS, RRUS, and polyhedral methods.
We begin our discussion with concave utility functions and then extend it to general utility functions particularly 
the S-shaped ones. 
Finally, we apply the proposed new preference elicitation approach to a robo-advisor system.

An important assumption
in the implementation of the proposed MUS is that
there is no error occurring 
in the process of preference elicitation. 
This assumption may be undesirable in some practical applications where elicitation errors are inevitable either due to measurement/quantification errors 
or due to 
{the}
DM's erroneous responses~\citep{guo2024utility}.  
In that case, we 
may consider random errors as in~\cite{saure2019ellipsoidal}
to obtain a probabilistic 
convergence to the true utility function. 
Another interesting direction for future research is elicitation of 
{the}
reference point 
in 
{the}
$S$-shaped utility function case. 
As we commented earlier, it will significantly 
{broaden the}
application of MUS if it is incorporated with a strategy to identify the reference point when a decision maker's utility function changes from convex to concave. 
{
Finally, we note that
the MUS method developed in this paper is designed to elicit the utility preference of a single DM. 
It remains a
promising
question whether the proposed MUS method can be extended to a group of DMs with heterogeneous preferences.
We leave all these for future exploration.
}

\backmatter


\section*{Acknowledgements}

This work was funded by the National Key R\&D Program of China (No. 2022YFA1004000) and
National Natural Science Foundation of China (No. 12371324).
The authors would like to thank
the associate editor and anonymous reviewers for their insightful 
comments and 
constructive 
suggestions, 
which help us improve the paper significantly in both contents and style.

\begingroup
\def\UrlBreaks{\do\/\do\-}
\def\UrlBigBreaks{}
\def\UrlNoBreaks{\do\:\do\(\do\[\do\{\do\<}
\bibliography{ref}
\endgroup

\end{document}